\documentclass[psamsfonts,reqno]{amsart}

\newtheorem{proposition}{Proposition}[section]
\theoremstyle{remark}

\usepackage{accents}
\usepackage{fullpage}
\usepackage[dvips]{graphicx}
\usepackage{subfig}
\usepackage{caption}
\theoremstyle{definition}

\newcommand{\abs}[1]{\lvert#1\rvert}
\newcommand{\ve}{\varepsilon}
\newcommand{\wt}{\widetilde}
\newcommand{\acr}{\accentset{\circ}}
\newcommand{\mc}{\mathcal}
\numberwithin{equation}{section}
\numberwithin{table}{section}
\numberwithin{figure}{section}

\def\tblhead#1{\hline\\[-9pt]#1\\[1.5pt]\hline\\[-9.75pt]}
\newcounter{tblcap}

\newsavebox{\tabbox}
\newlength{\tablen}

\begin{document}
	\title{NOTES ON GALERKIN-FINITE ELEMENT METHODS FOR THE SHALLOW WATER EQUATIONS WITH 
		CHARACTERISTIC  BOUNDARY  CONDITIONS.}
	\author{D.C. Antonopoulos}
	\author{V.A. Dougalis}
	\address{Department of Mathematics, University of Athens, 15784 Zographou, Greece, and
		Institute of Applied and Computational Mathematics, FORTH, 70013 Heraklion, Greece}
	\email{antonod@math.uoa.gr\,,  doug@math.uoa.gr}
	\thanks{Work supported by  the project PEFYKA of the action KRIPIS of GSRT at IACM/FORTH. The project is
		funded by Greece and the European Development Fund of EU under the NSRF and the O.P. Competitiveness 
		and Entepreneurship}
	\begin{abstract}
		We consider the shallow water equations in the supercritical and subcritical cases in one space variable,
		posed in a finite spatial interval with characteristic boundary conditions at the endpoints, which, as is well known,
		are transparent, i.e. allow outgoing waves to exit without generating spurious reflected waves. Assuming that the 
		resulting  initial-boundary-value problems have smooth solutions, we approximate them in  space using standard
		Galerkin-finite element methods  and prove $L^{2}$ error estimates for the semidiscrete problems on 
		quasiuniform meshes.  We discretize the problems in the temporal variable using an explicit, fourth-order accurate 
		Runge-Kutta scheme and check, by means of numerical experiment, that the resulting fully discrete schemes have
		excellent absorption properties.
	\end{abstract}
	\maketitle
\section{Introduction}
In this paper we  consider the system of {\it{shallow water equations}} 
\begin{equation}
	\begin{aligned}
		\eta_{t} & + u_{x} + (\eta u)_{x} = 0\,,\\
		u_{t} & + \eta_{x} + uu_{x}  = 0\,,
	\end{aligned}
	\label{eq11}
\end{equation}
a well known approximation of the two-dimensional Euler equations of water-wave theory, modelling two-way
propagation of long surface waves of finite amplitude in a uniform horizontal channel of finite depth, \cite{wh}. The
variables in (\ref{eq11}) are non-dimensional and unscaled; $x\in \mathbb{R}$ and $t\geq 0$ are proportional 
to position along the channel and time, respectively, while $\eta=\eta(x,t)$ and $u=u(x,t)$ are proportional to the 
elevation of the free surface above a level of rest and to the horizontal velocity of the fluid, respectively. The latter 
is depth-independent to the order of approximation represented by the scaled analog of (\ref{eq11}). In the
variables of (\ref{eq11}) the bottom of the channel lies at a depth equal to $-1$. \par
It is well known that the initial-value problem for  (\ref{eq11}), posed with smooth initial data $\eta(x,0)$ and
$u(x,0)$ for $x \in \mathbb{R}$, has, in general, smooth solutions only locally in $t$, cf. e.g. \cite{ma}, Ch. 2. 
In this paper we will pose (\ref{eq11}) in the finite `channel' $[0,L]$ with given initial values at $t=0$, 
\begin{equation}
	\eta(x,0)=\eta^{0}(x), \quad u(x,0) = u^{0}(x), \quad 0\leq x\leq L\,,
	\label{eq12}
\end{equation}
and consider \emph{transparent boundary conditions}  at $x=0$ and $x=L$, i.e. conditions that permit  the 
waves to exit the `computational' domain $[0,L]$ without generating spurious reflected waves that pollute
the solution inside $[0,L]$. The transparent boundary conditions that we will use are 
\emph{nonlinear characteristic boundary conditions} for subcritical and supercritical flows governed by (\ref{eq11}).
Such conditions were first used, to our knowledge, by Nycander \emph{et al.}, \cite{nmf}, in numerical experiments
with finite difference discretizations of shallow water models. In the paper at hand we will 
analyze Galerkin-finite element approximations for smooth solutions of the initial-boundary-value problems (ibvp's)
resulting from the application of characteristic boundary conditions to (\ref{eq11}). The well-posedness of these 
ibvp's was analyzed in Petcu \& Temam, \cite{pt1} and Huang \emph{et al.}, \cite{hpt}. \par
The characteristic boundary conditions may be derived as follows. We write the system (\ref{eq11}) as
\[
\begin{pmatrix}\eta_{t}\\u_{t}\end{pmatrix} + A\begin{pmatrix}\eta_{x} \\ u_{x}\end{pmatrix} = 
\begin{pmatrix}0\\0\end{pmatrix}.
\]
The matrix  $A=\begin{pmatrix}u & 1+\eta \\ 1 & u \end{pmatrix}$ has  eigenvalues 
$\lambda_{1}=u+\sqrt{1+\eta}$,  $\lambda_{2} = u-\sqrt{1+\eta}$. Assuming always that $\eta > -1$, we 
consider two types  of flows: 
\begin{align*}
	\text{\emph{Supercritical}} :  \quad & u>\sqrt{1+\eta}\,, \quad 0<\lambda_{2} < \lambda_{1}\,, \\
	\text{\emph{Subcritical}} : \quad & u < \sqrt{1+\eta}\,, \quad \lambda_{2}<0<\lambda_{1}\,.
\end{align*} 
It is well known , cf. e.g. \cite{wh}, that along the family of \emph{characteristic curves} with 
$\dot{x}(t)=\lambda_{1}=u+\sqrt{1+\eta}$, the quantity $r_{1}:=u+2\sqrt{1+\eta}$ is constant, while
along the curves with $\dot{x}(t)=\lambda_{2}=u-\sqrt{1+\eta}$, $r_{2}:=u-2\sqrt{1+\eta}$ is preserved. If
$\eta$ and $u$ are expressed in terms of the \emph{Riemann invariants} $r_{1}$, $r_{2}$ one obtains the system
of ordinary differential equations (ode's) 
\begin{align*}
	\frac{dr_{1}}{dt} = 0 & \quad \text{on}\quad x(t) : \quad \frac{dx}{dt}=\lambda_{1}(r_{1},r_{2})=\frac{3r_{1}+r_{2}}{4},\\ 
	\frac{dr_{2}}{dt}=0 & \quad \text{on}\quad x(t) :   \quad \frac{dx}{dt}=\lambda_{2}(r_{1},r_{2})=\frac{r_{1}+3r_{2}}{4},
\end{align*}
which is  equivalent to the original pde system (\ref{eq11}) and whose discretization yields the classical 
\emph{method of characteristics}  for solving  (\ref{eq11}). If we pose (\ref{eq11}) in the spatial interval $[0,L]$ it is 
straightforword to see, cf. \cite{wh}, Section 5.4, that the temporal integration of the ode system requires in the
supercritical case that $r_{1}(0,t)$ and $r_{2}(0,t)$ be given for $t\geq 0$, as both  families of characteristics are 
incoming at $x=0$; this is equivalent to prescribing $u(0,t)$ and $\eta(0,t)$ for $t\geq 0$. In the subcritical case
$r_{1}(0,t)$ and $r_{2}(L,t)$ should be given for $t\geq 0$, since they correspond to the incoming characteristics
at $x=0$ and at $x=L$ respectively.\par
Following \cite{nmf} we assume that outside the interval $[0,L]$ the flow is uniform and is given by 
$\eta(x,t)=\eta_{0}$, $u(x,t)=u_{0}$, where $\eta_{0}$, $u_{0}$ are known constants. Therefore, in the 
\emph{supercritical case} the characteristic boundary conditions are simply 
\begin{equation}
	\eta(0,t) = \eta_{0}\,, \quad u(0,t)=u_{0}\,, 
	\label{eq13}
\end{equation} 
with $u_{0}>\sqrt{1+\eta_{0}}$ ,  while, in the \emph{subcritical case}, they are of the form 
\begin{equation}
	\begin{aligned}
		u(0,t)  + 2\sqrt{1 + \eta(0,t)} & = u_{0} + 2\sqrt{1+\eta_{0}},\\ 
		u(L,t)  -  2\sqrt{1 +  \eta(L,t)} & = u_{0} - 2\sqrt{1+\eta_{0}},
	\end{aligned}
	\label{eq14}
\end{equation}
where now it is assumed that $u_{0}^{2} < 1+\eta_{0}$. In both cases we may view the solution 
$(\eta,u)$ of (\ref{eq11}) for  $0\leq x\leq L$, $t\geq 0$, generated by the initial conditions (\ref{eq12}),
as a perturbation of the uniform flow $(\eta_{0},u_{0})$  to which the solution inside the computational
domain $[0,L]$ will revert once the waves generated by the initial conditions exit this interval. It is 
straightforward to check, using the definitions of characteristics and Riemann invariants and considering
e.g. initial conditions that differ from $\eta_{0}$, $u_{0}$ in a subinterval of $[0,L]$, that the boundary conditions
(\ref{eq13}) and (\ref{eq14}) are transparent. \par
As previously mentioned, the characteristic boundary conditions (\ref{eq13}) and (\ref{eq14}) were used by Nycander \emph{et al.}, 
 \cite{nmf}, in finite difference simulations of the shallow water equations, in one 
space dimension, actually in more complicated instances of hydraulic and geophysical interest, including single-
and two-layer flows in channels of variable width and variable bottom topography, time-dependent forcing
in the boundary conditions, examples where transcritical flows develop, \emph{et al.} (In the case of two-layer
flows an approximate SW system was used in which the barotropic and baroclinic modes are decoupled;
this allows using the analogs of the (local) characteristic boundary conditions in this case too. Also, for reasons
of numerical stability, a diffusive term with a small viscosity coefficient was added in the momentum equations.)
The finite difference spatial discretization was effected on a staggered grid and the leap-frog scheme was used 
for time stepping. Similar model equations and characteristic boundary conditions were applied in simulations 
of two-layer hydraulic exchange flows in \cite{fm}.\par
The characteristic boundary conditions (\ref{eq13}) and (\ref{eq14}) were also used by Shiue \emph{et al.}, 
\cite{sltt}, in the case of the one-dimensional, single-layer shallow water equations in channels of variable
bottom topography in the presence of Coriolis terms and with the addition of a cross-velocity variable that 
depends only on $x$. The system, written in balance law form, was discretized in space using midpoint quadrature   
for the source cell integral and a `central-upwind', \cite{knp}, \cite{kp}, Godunov-type approximation of the flux
term; a second-order, explicit Runge-Kutta method was used for time stepping. Many numerical experiments
performed with this scheme are reported in \cite{sltt}; they simulate interesting cases of subcritical,
transcritical and supercritical flows over variable-bottom topographies. The characteristic boundary conditions
and the same numerical scheme were subsequently used in \cite{bpstt} in the case of two-layer problems
in one dimension under the decoupling assumptions of \cite{nmf}. (The local well-posedness of this
two-layer problem was studied in \cite{pt2}.)
 \par 
In case the elevation of the free surface $\eta$ is a small perturbation of the steady state $\eta_{0}$ one may
derive \emph{linearized} approximations to the characteristic boundary conditions (\ref{eq14}) in the 
subcritical case. These linearized conditions are also considered
in \cite{nmf} and in \cite{sltt}, where they are compared to the nonlinear exact conditions and found in general
to cause spurious reflections that enter the computational domain. (The linearized boundary conditions are
easily seem to be (exactly) transparent for the \emph{linearized} shallow water equations obtained by linearizing
(\ref{eq11}) about the steady state $(\eta_{0},u_{0})$. In \cite{sltt} it is shown that the ibvp for the linearized system
supplemented by the linearized boundary conditions is well posed.) As pointed out in \cite{nmf}, \cite{sltt}, and in
\cite{nd}, the linearized characteristic boundary conditions have been extensively used in the computational 
fluid dynamics literature; the last reference contains a review of several other absorbing boundary conditions 
for the shallow water equations at artificial boundaries, including absorbing (`sponge') layer 
conditions \emph{et al.}.   \par
Of particular interest for our purposes is the rigorous analysis of the well-posedness of the ibvp's (locally
in time) for the
shallow water equations with characteristic boundary conditions carried out in Huang \emph{et al.}, \cite{hpt}, and Petcu \& Temam, \cite{pt1}.  
The ibvp in the supercritical case, was studied by Huang \emph{et al.}, 
\cite{hpt}, in fact in the more general setting of shallow water supercritical flows
over a variable bottom in the presence of Coriolis terms and a lateral component of the horizontal velocity 
depending on $x$, and also with nonhomogeneous boundary conditions satisfying 
appropriate compatibility conditions. 
 The hypotheses of \cite{hpt} on $u_{0}$, $\eta_{0}$ and the initial data, briefly reviewed in 
section 2 in the sequel, guarantee the existence and uniqueness, locally in time, of a smooth solution of the ibvp
(\ref{eq11})-(\ref{eq13}) with positive $1+\eta$, satisfying the strong supercriticality property 
$u^{2} - (1+\eta)\geq c_{0}^{2}$ for some positive constant $c_{0}$. The well-posedness of the ibvp in the 
subcritical case, i.e. of the ibvp (\ref{eq11}), (\ref{eq12}),  (\ref{eq14}), was studied by Petcu \& Temam, \cite{pt1}.
The assumptions of \cite{pt1} (reviewed in section 3 below) imply the existence and uniqueness, locally in time,
of a smooth solution of the ibvp with positive $1+\eta$ and satisfying the strong subcriticality condition
$u^{2} - (1+\eta) \leq - c_{0}^{2}$,  where $c_{0}$ is a positive constant.\par
In this paper we will analyze standard Galerkin-finite element spatial discretizations of the ibvp's (\ref{eq11})-(\ref{eq13})
and (\ref{eq11}), (\ref{eq12}), (\ref{eq14}), under the hypothesis that they have smooth solutions. In both cases 
the basic approximation will be effected by $C^{r-2}$ functions which are piecewise polynomials of degree $r-1$, 
$r\geq 2$, on quasiuniform partitions of $[0,L]$. In section 2 we consider the supercritical case and prove an $L^{2}$-error 
estimate of $O(h^{r-1})$ accuracy for the Galerkin approximations of $\eta$ and $u$. (It is well known that this
is the expected best order of convergence in $L^{2}$ for standard Galerkin semidiscretizations of first-order
hyperbolic problems on general quasiuniform meshes. For \emph{uniform meshes}, better results hold, cf. \cite{d1}
for the analysis in the case of a linear model problem. In \cite{ad2} it was proved that the order of convergence
in $L^{2}$ for piecewise linear continuous elements on a uniform mesh is equal to 2 in the case of an ibvp for
(\ref{eq11}) with the homogeneous boundary conditions $u(0,t)=u(L,t)=0$. This superaccuracy result is expected
to hold for the ibvp's under consideration as well and this is indeed what the numerical experiments of section 4
indicate.) For the proof of the error estimate we assume that a strengthened supercriticality condition holds for
the solution of (\ref{eq11})-(\ref{eq13}); cf. (H1)-(H3) in section 2.  The proof also requires that $r\geq 3$ so that 
a certain bootstrap argument, based on the boundedness of the $\|\cdot\|_{1,\infty}$ norm of an error term, goes 
through as in \cite{d2}, \cite{ad2}. \par
In section 3 we turn to the subcritical case. We write the ibvp (\ref{eq11}), (\ref{eq12}), (\ref{eq14}) in its classical
diagonal form in which the new unknowns are analogs of the two Riemann invariants in the context of the ibvp
at hand and satisfy homogeneous Dirichlet boundary conditions one at $x=0$ and the other at $x=L$.
The diagonal system is discretized in space on a quasiuniform mesh by the same type of standard Galerkin 
method as before, and a $L^{2}$ error estimate of $O(h^{r-1})$ is proved for both components of the solution.
A change of variables of this semidiscrete approximation yields approximations of the original unknowns $\eta$
and $u$ of $O(h^{r-1})$ accuracy. The proof requires that a strengthened form of the subcriticality property holds
for the solution of the ibvp (\ref{eq11}), (\ref{eq12}), (\ref{eq14})  (cf. (\ref{eqy1}), (\ref{eqy2}) in section 3), and 
the technical assumption that $r\geq 3$. \par
Section 4 is a report of various numerical experiments that we performed with the Galerkin-finite element 
methods of sections 2 and 3 and some of their variants. We use spatial discretizations with piecewise linear 
continuous functions on uniform meshes and discretize them in the temporal variable by the `classical', explicit, 
four-stage, fourth-order Runge-Kutta scheme. The resulting fully discrete methods are stable under a Courant 
number restriction. (Stability and convergence of high order explicit Runge-Kutta methods was established for 
closely related pde systems in \cite{ad1} and \cite{ad2}.)   Our main purpose in the numerical experiments is to 
check the stability and the numerical order of convergence of the fully discrete Galerkin methods and study by computational 
means their absorption properties. Although the full discretizations of the characteristic boundary 
conditions are not exactly transparent of course, the numerical experiments show that they are practically 
transparent, in contrast to the analogous Galerkin schemes with the linearized boundary conditions that we also implement in the subcritical case; the latter are absorbing but in general allow spurious reflections to form and enter the computational domain. 
In the subcritical case we also implement the analogous fully discrete
Galerkin method for the original, nondiagonal form (\ref{eq11}), (\ref{eq12}), (\ref{eq14}) of the system and check
that it gives results close but somewhat inferior to those of the analogous discretization of the diagonal form of the system. \par
In the error estimates in the sequel, we let the spatial interval be $[0,1]$ for simplicity. We let $C^{k}=C^{k}[0,1]$,
$k=0,1,2,\dots$, be the space of $k$ times continuously differentiable functions on $[0,1]$. The norm and inner 
product on $L^{2}=L^{2}(0,1)$ are denoted by $\|\cdot\|$, $(\cdot,\cdot)$, respectively. For integer $k\geq 0$,
$H^{k}$, $\|\cdot\|_{k}$ will denote the usual, $L^{2}$-based Sobolev spaces of classes of functions and the 
associated norms. The norms on $L^{\infty}=L^{\infty}(0,1)$ and on the $L^{\infty}$-based Sobolev spaces
$W_{\infty}^{k}=W_{\infty}^{k}(0,1)$ will be denoted by $\|\cdot\|_{\infty}$, $\|\cdot\|_{k,\infty}$, respectively.
Finally, we let $\mathbb{P}_{r}$ be the polynomials of degree $\leq r$.
\section{Semidiscretization of the supercritical shallow water equations}
In this section we consider the shallow water equations with characteristic boundary conditions in the
\emph{supercritical} case. Specifically, for $(x,t)\in [0,1]\times[0,T]$ we seek $\eta=\eta(x,t)$ and $u=u(x,t)$
satisfying the ibvp
\begin{equation}
	\begin{aligned} 
		\begin{aligned}
			\eta_{t} & + u_{x} + (\eta u)_{x} = 0, 
			\\
			u_{t} & + \eta_{x} + uu_{x} = 0,  
		\end{aligned}
		\quad  0\leq x\leq 1,\,\, & \, 0\leq t\leq T,  
		\\
		\eta(x,0) =\eta^{0}(x), \quad u(x,0)=u^{0}(x), \quad 0 & \leq x\leq 1, 
		\\
		\eta(0,t) = \eta_{0}, \quad u(0,t)=u_{0}, \quad 0  \leq t\leq T, \hspace{-7pt}  &
	\end{aligned}
	\tag{SW1}
	\label{eqsw1}
\end{equation}
where $\eta^{0}$, $u^{0}$ are given functions on $[0,1]$ and $\eta_{0}$, $u_{0}$ constants such that
$1+\eta_{0}>0$, $u_{0}>0$, $u_{0}>\sqrt{1+\eta_{0}}$.  \par
As mentioned in the Introduction, the ibvp (\ref{eqsw1}) was studied by Huang \emph{et al.}, \cite{hpt}, in
fact in the more general case of a shallow water supercritical flow with nonhomogeneous boundary conditions over a variable bottom for a nonzero 
Coriolis parameter and also in the presence of a lateral component of the horizontal velocity depending
on $x$ only. In the simpler case of (\ref{eqsw1}), the proof of the main result of \cite{hpt} amounts to the
selection of a suitable constant solution $(\eta_{0},u_{0})$ of (\ref{eqsw1}) and of
sufficiently smooth initial conditions close to the constant solution and satisfying appropriate compatibility 
relations at $x=0$. Under these hypotheses the conclusion of \cite{hpt} is that given positive constants
$c_{0}$, $\alpha_{0}$, $\underline{\zeta}_{0}$, and $\overline{\zeta}_{0}$, there exists a $T>0$ such that a
sufficiently smooth solution of (\ref{eqsw1}) exists satisfying for $(x,t)\in [0,1]\times[0,T]$ the strong
supercriticality properties 
\begin{align}
	& u^{2} - (1+\eta)  \geq c_{0}^{2}, \tag{P1} \label{eqp1}\\
	& u  \geq \alpha_{0},   \tag{P2} \label{eqp2} \\
	& \underline{\zeta}_{0}  \leq (1+\eta)\leq \overline{\zeta}_{0}. \tag{P3} \label{eqp3} 
\end{align}
\indent
For the purposes of the error estimation to follow we will assume that (\ref{eqsw1}) has a sufficiently smooth 
solution $(\eta,u)$ that satisfies a strengthened supercriticality condition of the following form: There exist 
positive constants $\alpha$ and $\beta$, such that for $(x,t)\in [0,1]\times [0,T]$ it holds that 
\begin{align}
	1 & +\eta  \geq \beta, \tag{H1} \label{eqh1}\\
	& u  \geq 2\alpha,   \tag{H2} \label{eqh2} \\
	1& +\eta \leq (u-\alpha)(u-\tfrac{2\alpha}{3}). \tag{H3} \label{eqh3} 
\end{align}
Obviously (\ref{eqh1}) and   (\ref{eqh3})  imply that $u>\sqrt{1+\eta}$. It is not hard to see that  (\ref{eqh3})
follows from  (\ref{eqp1})-(\ref{eqp3}) if e.g. $\alpha_{0}$ is taken sufficiently small and $c_{0}$ sufficiently large.
We also remark here that in the error estimates to follow  (\ref{eqh3}) will be needed only at $x=1$ for $t\in [0,T]$. \par
We will approximate the solution of  (\ref{eqsw1}) in a slightly tranformed form. We let $\wt{\eta}=\eta-\eta_{0}$,
$\wt{u}=u-u_{0}$ and rewrite (\ref{eqsw1}) as an ibvp for $\wt{\eta}$ and $\wt{u}$ with homogeneous boundary
conditions. Dropping the tildes we obtain the system
\begin{equation}
	\begin{aligned} 
		\begin{aligned}
			\eta_{t} & + u_{0}\eta_{x} + (1+\eta_{0})u_{x} + (\eta u)_{x} = 0, 
			\\
			u_{t} & + \eta_{x} + u_{0}u_{x} + uu_{x} = 0,  
		\end{aligned}
		\quad  0\leq x\leq 1,\,\, & \, 0\leq t\leq T,  
		\\
		\eta(x,0) =\eta^{0}(x)-\eta_{0}, \hspace{15pt}  u(x,0)=u^{0}(x)-u_{0}, \quad 0 & \leq x\leq 1, 
		\\
		\eta(0,t) = 0, \quad u(0,t)=0, \quad 0  \leq t\leq T. \hspace{50pt}  &
	\end{aligned}
	\tag{SW1a}
	\label{eqsw1a}
\end{equation}
In terms of the new variables $\eta$ and $u$, our hypotheses (\ref{eqh1})-(\ref{eqh3}) become 
\begin{align}
	1 & +\eta  + \eta_{0} \geq \beta, \tag{H1a} \label{eqh1a}\\
	u & + u_{0} \geq 2\alpha,   \tag{H2a} \label{eqh2a} \\
	1& +\eta + \eta_{0} \leq (u+u_{0}-\alpha)(u+u_{0}-\tfrac{2\alpha}{3}). \tag{H3a} \label{eqh3a} 
\end{align}
\indent
In the sequel, for integer $k\geq 0$, let $\acr{C}^{k}=\{v\in C^{k}[0,1]:  v(0)=0\}$,  and
$\acr{H}^{k+1}=\{v\in H^{k+1}(0,1) : v(0)=0\}$. For a positive integer $N$ let  
$0=x_{1}<x_{2}<\dots<x_{N+1}=1$ be a quasiuniform partition of $[0,1]$ with
$h:=\max_{i}(x_{i+1}-x_{i})$, and for integer $r\geq 2$ define 
$\acr{S}_{h}=\{\phi \in \acr{C}^{r-2} : \phi\big|_{[x_{j},x_{j+1}]} \in \mathbb{P}_{r-1}\,, 1\leq j\leq N\}$.
It is well known that if $v\in \acr{H}^{r}$ , there exists $\chi \in \acr{S}_{h}$ such that
\begin{equation}
	\|v-\chi\| + h\|v' - \chi'\| \leq Ch^{r}\|v\|_{r},
	\label{eq21}
\end{equation}
and, cf. \cite{sch}, if $r\geq 3$, 
\begin{equation}
	\|v - \chi\|_{2} \leq C h^{r-2}\|v\|_{r}. 
	\label{eq22}
\end{equation}
(Here and in the sequel $C$ will denote a generic constant independent of $h$.)  In addition, if $P$ is the
$L^{2}$-projection operator onto $\acr{S}_{h}$, then it follows that, cf. \cite{ddw},  
\begin{align}
	& \|Pv\|_{\infty} \leq C \|v\|_{\infty}, \quad \text{if} \quad v \in L^{\infty}, 
	\label{eq23} \\
	& \|Pv - v\|_{\infty} \leq C h^{r} \|v\|_{r,\infty}, \quad \text{if} \quad v \in W^{r,\infty}\cap\acr{H}^{1}.
	\label{eq24}
\end{align}
As a consequence of the quasiuniformity of the mesh the inverse inequalities 
\begin{align}
	\|\chi\|_{1} & \leq Ch^{-1} \|\chi\|,
	\label{eq25} \\
	\|\chi\|_{j,\infty} & \leq C h^{-(j+1/2)} \|\chi\|, \quad j=0,1,
	\label{eq26}
\end{align}
hold for  $\chi\in \acr{S}_{h}$. (In (\ref{eq26}) $\|\cdot\|_{0,\infty}=\|\cdot\|_{\infty}$.) \par
The standard Galerkin semidiscretization of (\ref{eqsw1a}) is defined as follows: We seek 
$\eta_{h}$, $u_{h}: [0,T]\to \acr{S}_{h}$ such that for $0\leq t\leq T$ 
\begin{align}
	(\eta_{ht},\phi) + (u_{0}\eta_{hx},\phi) + ((1+\eta_{0})u_{hx},\phi) +  ((\eta_{h}u_{h})_{x},\phi) & = 0, \quad 
	\forall \phi \in \acr{S}_{h}, 
	\label{eq27}    \\
	(u_{ht},\phi) + (\eta_{hx},\phi) + (u_{0}u_{hx},\phi) + (u_{h}u_{hx},\phi) & = 0,\quad 
	\forall \phi \in \acr{S}_{h},
	\label{eq28}
\end{align}
with 
\begin{equation}
	\eta_{h}(0) = P(\eta^{0}(\cdot)-\eta_{0}), \quad u_{h}(0)=P(u^{0}(\cdot)-u_{0}).
	\label{eq29}
\end{equation}
The main result of this  section is:
\begin{proposition}
	Let $(\eta,u)$ be the solution of (\ref{eqsw1a}), and assume that the hypotheses (\ref{eqh1a}),  (\ref{eqh2a}),
	(\ref{eqh3a}) hold, that $r\geq 3$, and $h$ is sufficiently small.  Then the semidiscrete ivp 
	(\ref{eq27})-(\ref{eq29}) has a unique solution $(\eta_{h},u_{h})$ for $0\leq t\leq T$ satisfying
	\begin{equation}
		\max_{0\leq t\leq T}(\|\eta(t) - \eta_{h}(t)\|  +  \|u(t) - u_{h}(t)\|) \leq C h^{r-1}.
		\label{eq210} 
	\end{equation}
\end{proposition}
\begin{proof}
	Let $\rho = \eta - P\eta$, $\theta = P\eta - \eta_{h}$, $\sigma = u-Pu$, $\xi = Pu - u_{h}$. 
	After choosing a basis for $\acr{S}_{h}$, it is straightforward to see that the semidiscrete problem
	(\ref{eq27})-(\ref{eq29}) represents an ivp for an ode system which has a unique solution locally in time.
	While this solution exists, it follows from (\ref{eq27}), (\ref{eq28}) and the pde's in (\ref{eqsw1a}), that 
	\small
	\begin{align}
		(\theta_{t},\phi) + (u_{0}(\rho_{x}+\theta_{x}),\phi) + ((1+\eta_{0})(\sigma_{x} + \xi_{x}),\phi) +
		((\eta u - \eta_{h}u_{h})_{x},\phi) & = 0, \quad \forall \phi \in \acr{S}_{h},
		\label{eq211} \\
		(\xi_{t},\phi) + (\rho_{x}+\theta_{x},\phi) + (u_{0}(\sigma_{x}+\xi_{x}),\phi) + (uu_{x}-u_{h}u_{hx},\phi)
		& = 0, \quad \forall \phi \in \acr{S}_{h}.
		\label{eq212}
	\end{align}
	\normalsize
	Since $\eta u - \eta_{h}u_{h}=\eta(\sigma+\xi) + u(\rho+\theta) - (\rho+\theta)(\sigma+\xi)$,
	$uu_{x} - u_{h}u_{hx}=(u\sigma)_{x} +(u\xi)_{x} -(\sigma\xi)_{x} -\sigma\sigma_{x} - \xi\xi_{x}$, it
	follows that 
	\begin{align}
		(\eta u - \eta_{h}u_{h})_{x} & = (\eta\xi)_{x} + (u\theta)_{x} - (\theta\xi)_{x} + \wt{R}_{1}, 
		\label{eq213} \\
		uu_{x} - u_{h}u_{hx} &  = (u\xi)_{x} - \xi\xi_{x} + \wt{R}_{2}, 
		\label{eq214}
	\end{align}
	where
	\begin{align}
		\wt{R}_{1} & = (\eta\sigma)_{x} + (u\rho)_{x} - (\rho\sigma)_{x} -(\rho\xi)_{x} - (\theta\sigma)_{x},
		\label{eq215} \\
		\wt{R}_{2} & = (u\sigma)_{x} - (\sigma\xi)_{x} - \sigma\sigma_{x}.
		\label{eq216}
	\end{align}
	Therefore, the equations (\ref{eq211}), (\ref{eq212}) may be written as
	\small
	\begin{align}
		(\theta_{t},\phi) + (u_{0}\theta_{x},\phi) + (\gamma_{x},\phi) + ((u\theta)_{x},\phi) - ((\theta\xi)_{x},\phi)
		& = -(R_{1},\phi), \quad \forall \phi \in \acr{S}_{h},
		\label{eq217} \\
		(\xi_{t},\phi) + (\theta_{x},\phi) + (u_{0}\xi_{x},\phi) + ((u\xi)_{x},\phi)  - (\xi\xi_{x},\phi)
		& = -(R_{2},\phi), \quad \forall \phi \in \acr{S}_{h}.
		\label{eq218}
	\end{align}
	\normalsize
	where  $\gamma=(1+\eta_{0}+\eta)\xi$ and
	\begin{align}
		{R}_{1} & =u_{0}\rho_{x} + (1+\eta_{0})\sigma_{x} + \wt{R}_{1},
		\label{eq219} \\
		{R}_{2} & = \rho_{x} + u_{0}\sigma_{x} + \wt{R}_{2}.
		\label{eq220}
	\end{align}
	Putting $\phi=\theta$ in (\ref{eq217}), using integration by parts, and suppressing the dependence on $t$
	we have
	\small
	\begin{equation}
	\begin{aligned}
		\tfrac{1}{2}\tfrac{d}{dt}\|\theta\|^{2} - (\gamma,\theta_{x}) +\tfrac{1}{2}(u_{0}+u(1))\theta^{2}(1) & + 
		(1+\eta_{0}+\eta(1))\xi(1)\theta(1) \\ & - \tfrac{1}{2} \xi(1)\theta^{2}(1)  = 
		- \tfrac{1}{2} (u_{x}\theta,\theta) + \tfrac{1}{2}(\xi_{x}\theta,\theta) 
		 - (R_{1},\theta).
		\end{aligned}
		\label{eq221}
	\end{equation} 
	\normalsize
	Take now $\phi=P\gamma=P[(1+\eta_{0}+\eta)\xi]$ in (\ref{eq218}) and get 
	\small
	\begin{equation}
		(\xi_{t},\gamma) + (\theta_{x},\gamma) + (u_{0}\xi_{x},\gamma) + ((u\xi)_{x},\gamma)  - (\xi\xi_{x},\gamma) 
		= - (R_{3},P\gamma - \gamma) - (R_{2},P\gamma),
		\label{eq222}
	\end{equation} 
	\normalsize
	where
	\begin{equation}
		R_{3} = \theta_{x} + u_{0}\xi_{x} + (u\xi)_{x} - \xi\xi_{x}.
		\label{eq223}
	\end{equation}
	Integration by parts in various terms in (\ref{eq222}) gives
	\begin{align*}
		(u_{0}\xi_{x},\gamma) & = (u_{0}\xi_{x},(1+\eta_{0}+\eta)\xi) = \tfrac{1}{2}u_{0}(1+\eta_{0}+\eta(1))\xi^{2}(1)
		- \tfrac{1}{2} (u_{0}\eta_{x}\xi,\xi), \\
		((u\xi)_{x},\gamma) & = ((u\xi)_{x},(1+\eta_{0}+\eta)\xi) = (u_{x}\xi,(1+\eta_{0}+\eta)\xi) + (u\xi_{x},(1+\eta_{0}+\eta)\xi) \\
		& = \tfrac{1}{2}u(1)(1+\eta_{0}+\eta(1))\xi^{2}(1) + \tfrac{1}{2}(u_{x}(1+\eta_{0}+\eta),\xi^{2})
		- \tfrac{1}{2} (u\eta_{x}\xi,\xi), \\
		(\xi\xi_{x},\gamma) & = (\xi\xi_{x},(1+\eta_{0}+\eta)\xi) = \tfrac{1}{3}(1+\eta_{0}+\eta(1))\xi^{3}(1)-\tfrac{1}{3}(\eta_{x}\xi^{2},\xi).
	\end{align*}
	Hence (\ref{eq222})  becomes 
	\begin{equation}
		\begin{aligned}
			(\xi_{t},\gamma) + (\theta_{x},\gamma) & + \tfrac{1}{2}(u_{0}+u(1))(1+\eta_{0}+\eta(1))\xi^{2}(1) 
			- \tfrac{1}{3}(1+\eta_{0}+\eta(1))\xi^{3}(1) \\
			& = (R_{4},\xi) - (R_{3},P\gamma-\gamma) - (R_{2},P\gamma),
		\end{aligned}
		\label{eq224}
	\end{equation}
	where 
	\begin{equation}
		R_{4} = \tfrac{1}{2}u_{0}\eta_{x}\xi - \tfrac{1}{2}u_{x}(1+\eta_{0}+\eta)\xi + \tfrac{1}{2}u\eta_{x}\xi 
		- \tfrac{1}{3}\eta_{x}\xi^{2}.
		\label{eq225}
	\end{equation}
	Adding now (\ref{eq221}) and (\ref{eq224}) we obtain 
	\begin{equation}
		\begin{aligned}
			\tfrac{1}{2}\tfrac{d}{dt}[\|\theta\|^{2} & + ((1+\eta_{0}+\eta)\xi,\xi)] + \omega = \tfrac{1}{2}(\eta_{t}\xi,\xi)
			- \tfrac{1}{2}(u_{x}\theta,\theta) \\
			& + \tfrac{1}{2}(\xi_{x}\theta,\theta) - (R_{1},\theta) + (R_{4},\xi) - (R_{3},P\gamma-\gamma) - (R_{2},P\gamma),
		\end{aligned}
		\label{eq226}
	\end{equation}
	where 
	\begin{equation}
		\begin{aligned}
			\omega  = & \tfrac{1}{2}(u_{0}+u(1))\theta^{2}(1) + \tfrac{1}{2}(u_{0}+u(1))(1+\eta_{0}+\eta(1))\xi^{2}(1) \\
			& + (1+\eta_{0}+\eta(1))\xi(1)\theta(1) - \tfrac{1}{2}\xi(1)\theta^{2}(1) - \tfrac{1}{3}(1+\eta_{0}+\eta(1))\xi^{3}(1).
		\end{aligned}
		\label{eq227}
	\end{equation}
	In view of (\ref{eq29}), by continuity we conclude that there exists a maximal temporal instance $t_{h}>0$ 
	such that $(\eta_{h},u_{h})$ exist and $\|\xi_{x}\|_{\infty}\leq \alpha$ for $t\leq t_{h}$. Suppose that $t_{h}<T$.
	Then, since $\|\xi\|_{\infty}\leq \|\xi_{x}\|_{\infty}$, it follows from (\ref{eq227}) that for $t\in[0,t_{h}]$
	\begin{equation}
		\begin{aligned}
			\omega  \geq & \tfrac{1}{2}(u_{0}+u(1)-\alpha)\theta^{2}(1) 
			+ \tfrac{1}{2}(1+\eta_{0}+\eta(1))(u_{0}+u(1)-\tfrac{2\alpha}{3})\xi^{2}(1) \\
			& + (1+\eta_{0}+\eta(1))\xi(1)\theta(1) 
			= \tfrac{1}{2}(\theta(1),\xi(1))^{T}\begin{pmatrix} \mu & \lambda \\ \lambda & \lambda\nu \end{pmatrix}
			\begin{pmatrix} \theta(1) \\ \xi(1) \end{pmatrix},
		\end{aligned}
		\label{eq228}
	\end{equation}
	where $\mu=u_{0} + u(1) - \alpha$, $\lambda=1+\eta_{0}+\eta(1)$, $\nu=u_{0} + u(1) - \tfrac{2\alpha}{3}$.
	The hypotheses (\ref{eqh1a}) and (\ref{eqh2a}) give that $0<\mu<\nu$, $\lambda>0$. It is easy to see
	then that the matrix in (\ref{eq228}) will be positive semidefinite precisely when (\ref{eqh3a}) holds. We conclude
	from (\ref{eq228}) that $\omega \geq 0$. \\
	We now estimate the various terms in the right-hand side of (\ref{eq226}) for $0\leq t\leq t_{h}$. We obviously have
	\begin{equation}
		\abs{(\eta_{t}\xi,\xi)} + \abs{(u_{x}\theta,\theta)} \leq C(\|\xi\|^{2} + \|\theta\|^{2}),
		\label{eq229}
	\end{equation}
	and
	\begin{equation}
		\abs{(\xi_{x}\theta,\theta)} \leq \alpha \|\theta\|^{2}.
		\label{eq230}
	\end{equation}
	In addition, from (\ref{eq219}), (\ref{eq214}), and the inverse and approximation properties of $\acr{S}_{h}$ and
	(\ref{eq23}), (\ref{eq24}) we have 
	\begin{equation}
		\begin{aligned}
			\abs{(R_{1},\theta)} \leq & Ch^{r-1}\|\theta\| + \|\rho\|_{\infty}\|\xi_{x}\| \|\theta\| + \|\rho_{x}\|_{\infty} \|\xi\| \|\theta\| \\
			& + \|\sigma_{x}\|_{\infty} \|\theta\|^{2} + \|\sigma\|_{\infty} \|\theta_{x}\| \|\theta\| \\
			\leq & Ch^{r-1} \|\theta\| + C(\|\theta\|^{2} + \|\xi\|^{2}).
		\end{aligned}
		\label{eq231}
	\end{equation} 
	Also, from (\ref{eq225}) 
	\begin{equation}
		\abs{(R_{4},\xi)} \leq C\|\xi\|^{2} + C\|\xi\|_{\infty} \|\xi\|^{2} \leq C(1+\alpha)\|\xi\|^{2}.
		\label{eq232}
	\end{equation}
	By (\ref{eq220}), (\ref{eq216}), (\ref{eq23}), (\ref{eq24}) and the inverse and approximation properties of $\acr{S}_{h}$
	we have
	\begin{equation}
		\begin{aligned}
			\abs{(R_{2},P\gamma)} & \leq Ch^{r-1}\|\xi\| + C \|\sigma\|_{\infty} \|\xi_{x}\| \|\xi\| + C\|\sigma_{x}\|_{\infty} \|\xi\|^{2}\\
			&\leq Ch^{r-1} \|\xi\| + C \|\xi\|^{2}.
		\end{aligned}
		\label{eq233}
	\end{equation}
	Finally, using a well-known \emph{superapproximation} property of $\acr{S}_{h}$, cf. \cite{ddw}, \cite{d2}, 
	in order to estimate the term $P\gamma - \gamma$ by
	\begin{equation}
		\|P\gamma - \gamma\| = \|P[(1+\eta + \eta_{0})\xi] - (1+\eta + \eta_{0})\xi\| \leq Ch \|\xi\|,
		\label{eq234}
	\end{equation}
	we obtain by (\ref{eq223}) and the inverse properties of $\acr{S}_{h}$ that 
	\begin{equation}
		\begin{aligned}
			\abs{(R_{3},P\gamma - \gamma)} & \leq \abs{(\theta_{x},P\gamma - \gamma)} + \abs{(u_{0}\xi_{x},P\gamma - \gamma)} 
			+ \abs{((u\xi)_{x}, P\gamma-\gamma)} + \abs{(\xi\xi_{x},P\gamma-\gamma)} \\
			& \leq Ch\|\theta_{x}\| \|\xi\| + Ch \|\xi_{x}\| \|\xi\| + Ch \|\xi\|^{2} + Ch \|\xi\|_{\infty} \|\xi_{x}\| \|\xi\| \\
			& \leq C\|\theta\| \|\xi\| + C\|\xi\|^{2} + C\alpha \|\xi\|^{2} \\
			& \leq C\|\theta\|^{2} + C(1+\alpha)\|\xi\|^{2}.
		\end{aligned}
		\label{eq235}
	\end{equation} 
	Therefore, (\ref{eq226}), the fact that $\omega\geq 0$, and the inequalities  (\ref{eq230})-(\ref{eq233}), (\ref{eq235}) 
	give for $0\leq t\leq t_{h}$
	\[
	\tfrac{d}{dt}[\|\theta\|^{2} + ((1+\eta_{0}+\eta)\xi,\xi)] \leq Ch^{r-1}(\|\theta\| + \|\xi\|)
	+ C(\|\theta\|^{2} + \|\xi\|^{2}),
	\]
	where $C$ is a constant independent of $h$ and $t_{h}$. By (\ref{eqh1a}) the norm $((1+\eta_{0}+\eta)\cdot,\cdot)^{1/2}$
	is equivalent to that of $L^{2}$ uniformly for $t\in[0,T]$. Hence, Gronwall's inequality and the fact that 
	$\theta(0) =\xi(0)=0$ yield for a constant $C=C(T)$ 
	\begin{equation}
		\|\theta\| + \|\xi\| \leq Ch^{r-1} \quad \text{for} \quad 0\leq t\leq t_{h}.
		\label{eq236}
	\end{equation}
	We conclude from (\ref{eq26}) that $\|\xi_{x}\|_{\infty} \leq Ch^{r-5/2}$ for $0\leq t\leq t_{h}$, and, since $r\geq 3$, if $h$
	is taken sufficiently small, $t_{h}$ is not maximal. Hence we may take $t_{h}=T$ and (\ref{eq210}) follows from (\ref{eq236}).  
\end{proof}
\section{Semidiscretization of the subcritical shallow water equations}
In this section we consider the shallow water equations with characteristic boundary conditions in the 
\emph{subcritical} case. Specifically, for $(x,t)\in [0,1]\times[0,T]$ we seek $\eta=\eta(x,t)$ and $u=u(x,t)$
satisfying the ibvp 
\begin{equation}
	\begin{aligned} 
		\begin{aligned}
			\eta_{t} & + u_{x} + (\eta u)_{x} = 0, 
			\\
			u_{t} & + \eta_{x} + uu_{x} = 0,  
		\end{aligned}
		\quad  0\leq x\leq 1,\,\, & \, 0\leq t\leq T,  
		\\
		\eta(x,0) =\eta^{0}(x), \quad u(x,0)=u^{0}(x), \quad 0 & \leq x\leq 1, 
		\\
		u(0,t) + 2\sqrt{1+\eta(0,t)} = u_{0} + 2\sqrt{1+\eta_{0}},  \quad 0  \leq t\leq T, \hspace{-39pt}  & \\
		u(1,t) -  2\sqrt{1+\eta(1,t)} = u_{0} - 2\sqrt{1+\eta_{0}},  \quad 0  \leq t\leq T, \hspace{-39pt}  & \\
	\end{aligned}
	\tag{SW2}
	\label{eqsw2}
\end{equation}
where $\eta^{0}$, $u^{0}$ are given functions on $[0,1]$ and $\eta_{0}$, $u_{0}$ constants such that 
$1+\eta_{0}>0$ and $u_{0}^{2} < 1+\eta_{0}$. \par
As mentioned in the Introduction, the ibvp (\ref{eqsw2}) was studied by Petcu \& Temam, \cite{pt1}. They used 
the hypotheses that there exists a constant $c_{0}>0$ such that $u_{0}^{2}-(1+\eta_{0}) \leq -c_{0}^{2}$
and that the initial conditions $\eta^{0}(x)$ and $u^{0}(x)$ are sufficiently smooth and satisfy the condition 
$(u^{0}(x))^{2} - (1 + \eta^{0}(x)) \leq - c_{0}^{2}$ (with $1+\eta^{0}(x)$ positive) and suitable compatibility relations 
at $x=0$ and $x=1$. Under these assumptions one may infer from the theory of \cite{pt1} that there exists a $T>0$
such that a sufficiently smooth solution $(\eta,u)$ of (\ref{eqsw2}) exists for $(x,t)\in [0,1]\times[0,T]$ with the
properties that $1 + \eta$ is positive and the strong subcriticality condition 
\begin{equation}{\tag{$\Pi$}}
	u^{2} - (1 + \eta) \leq - c_{0}^{2},
	\label{eqp}
\end{equation}
holds for  $(x,t)\in [0,1]\times[0,T]$. For the purposes of the error estimation to follow we will assume that  
(\ref{eqsw2}) has a sufficiently smooth solution $(\eta,u)$ such that $1 + \eta >0$ and satisfies a stronger 
subcriticality condition. Specifically we assume that for some constant $c_{0}>0$ it holds that
\begin{equation}{\tag{Y1}}
	u_{0} + \sqrt{1 + \eta_{0}} \geq c_{0}, \quad u_{0} - \sqrt{1 + \eta_{0}} \leq - c_{0},
	\label{eqy1}
\end{equation}
and for $(x,t)\in [0,1]\times[0,T]$  that 
\begin{equation}{\tag{Y2}}
	u + \sqrt{1 + \eta} \geq c_{0}, \quad u - \sqrt{1 + \eta} \leq - c_{0}.
	\label{eqy2}
\end{equation}
Obviously (\ref{eqy1}) and (\ref{eqy2}) imply the subcriticality conditions 
$u_{0}^{2} - (1 + \eta_{0}) \leq - c_{0}^{2}$ and (\ref{eqp}) of \cite{pt1}, and approximate the latter better as 
$c_{0}$ decreases.  \par
In this section we will approximate the solution of (\ref{eqsw2}) with a Galerkin-finite element method after
transforming (\ref{eq11}) in its classical diagonal form. As in the Introduction, we write the system as 
\begin{equation}
	\begin{pmatrix} \eta_{t}\\ u_{t} \end{pmatrix} + A\begin{pmatrix} \eta_{x}\\u_{x}\end{pmatrix} = 0,
	\label{eq31}
\end{equation}
where  $A=\begin{pmatrix}  u & 1+\eta \\ 1 & u \end{pmatrix}$. The matrix $A$ has the eigenvalues
$\lambda_{1} = u + \sqrt{1+\eta}$, $\lambda_{2} = u-\sqrt{1+\eta}$, (note that by (\ref{eqy2}) $\lambda_{1}\geq c_{0}$ 
and $\lambda_{2} \leq -c_{0}$ in $[0,1]\times [0,T]$), with associated eigenvectors 
$X_{1} = (\sqrt{1+\eta},1)^{T}$, $X_{2}=(-\sqrt{1+\eta,}1)^{T}$. If $S$ is the matrix with columns $X_{1}$, $X_{2}$
it follows from (\ref{eq31}) that
\begin{equation}
	S^{-1}\begin{pmatrix} \eta_{t}\\u_{t}\end{pmatrix} 
	+ \begin{pmatrix} \lambda_{1} & 0 \\ 0 & \lambda_{2}\end{pmatrix}S^{-1}\begin{pmatrix}\eta_{x}\\u_{x}\end{pmatrix}=0.
	\label{eq32}
\end{equation}
If we try to define now functions $v$, $w$ on $[0,1]\times [0,T]$ by the equations 
$S^{-1}\begin{pmatrix} \eta_{t} \\ u_{t}\end{pmatrix} = \begin{pmatrix} v_{t} \\ w_{t}\end{pmatrix}$, 
$S^{-1}\begin{pmatrix}\eta_{x}\\ u_{x}\end{pmatrix} = \begin{pmatrix} v_{x} \\ w_{x}\end{pmatrix}$, we see that these equations are consistent and their solutions  are given by 
$v=\tfrac{1}{2} (u+2\sqrt{1 + \eta}) + c_{v}$, $w=\tfrac{1}{2}(u-2\sqrt{1 + \eta}) + c_{w}$, for arbitrary constants
$c_{v}$, $c_{w}$. Choosing the constants $c_{v}$, $c_{w}$ so that $v(0,t)=0$, $w(1,t)=0$, and using the boundary
conditions in (\ref{eqsw2}) we get 
\begin{equation}
	v = \tfrac{1}{2}[u-u_{0} + 2(\sqrt{1 + \eta} - \delta_{0})], \quad w = \tfrac{1}{2}[u-u_{0} - 2(\sqrt{1 + \eta} - \delta_{0})],
	\label{eq33}
\end{equation}
where $\delta_{0} = \sqrt{1 + \eta_{0}}$. The original variables $\eta$, $u$ are given in terms of $v$ and $w$ by the formulas
\begin{equation}
	\eta = [\tfrac{1}{2}(v-w) + \delta_{0}]^{2} - 1, \quad u = v + w + u_{0}.
	\label{eq34}
\end{equation}
Since
\begin{equation} 
	\lambda_{1} = u + \sqrt{1 + \eta}=u_{0} + \delta_{0} + \frac{3v + w}{2}, \quad 
	\lambda_{2} = u - \sqrt{1 + \eta} = u_{0} - \delta_{0} + \frac{v + 3w}{2}, 
	\label{eq35}
\end{equation}
we see that the ibvp (\ref{eqsw2}) becomes
\begin{equation}{\tag{SW2a}}
	\begin{aligned}
		\begin{pmatrix} v_{t}\\ w_{t}\end{pmatrix}
		+ \begin{pmatrix} u_{0}+\delta_{0} + \tfrac{3v+w}{2} & 0 \\ 0 & u_{0}-\delta_{0} + \tfrac{v + 3w}{2}\end{pmatrix}
		\begin{pmatrix} v_{x} \\ w_{x}\end{pmatrix} & =0, \quad  0\leq x \leq 1, \,\,\, 0\leq t\leq T,\\
		v(x,0) = v^{0}(x), \quad  w(x,0) = w^{0}(x),   \quad 0\leq x & \leq 1,\\
		v(0,t)=0, \quad w(1,t) = 0,  \quad  0\leq t  \leq T,  \qquad &  
	\end{aligned}
	\label{eqsw2a}
\end{equation} 
where $v^{0}(x)= \tfrac{1}{2}[u^{0}(x) - u_{0} + 2(\sqrt{1 + \eta^{0}(x)} - \delta_{0})]$, 
$w^{0}(x) = \tfrac{1}{2}[u^{0}(x) - u_{0} - 2(\sqrt{1 + \eta^{0}(x)} - \delta_{0})]$. Under our hypotheses (\ref{eqsw2a})
has a unique solution $(v,w)$ on $[0,1]\times[0,T]$ which will be assumed to be smooth enough for the
purposes of the error estimation that follows. Of course, $v$ and $w$ represent analogs of the Riemann invariants
of the shallow water system in the context of the ibvp at hand; the system of pde's in (\ref{eqsw2a}) and (\ref{eq34}),
(\ref{eq35}) imply that the solution $(\eta,u)$ of (\ref{eqsw2}) may be expressed in terms of two waves $v$ and $w$ 
that propagate to the right and left, respectively, with speeds $u + \sqrt{1 + \eta}$ and $u - \sqrt{1 + \eta}$. \par
Given a quasiuniform partition of $[0,1]$ as in section 2, in addition to the spaces defined there, let for
integer $k\geq 0$  \,\,$\acr{\mathcal{C}}^{k}=\{f\in C^{k}[0,1] : f(1)=0\}$, 
$\acr{\mathcal{H}}^{k+1} = \{f\in H^{k+1}(0,1), f(1)=0\}$, and, for integer $r\geq 2$, 
$\acr{\mathcal{S}}_{h} = \{\phi \in \acr{\mathcal{C}}^{r-2} : \phi\big|_{[x_{j},x_{j+1}]} \in \mathbb{P}_{r-1}\,, 1\leq j\leq N\}$.
Note that the analogs of the approximation and inverse properties (\ref{eq21}), (\ref{eq22}), (\ref{eq25}), (\ref{eq26})
hold for $\acr{\mathcal{S}}_{h}$ as well, and that (\ref{eq23}), (\ref{eq24}) are also valid for the $L^{2}$ projection
$\mathcal{P}$ onto $\acr{\mathcal{S}}_{h}$, \emph{mutatis mutandis}. \par
The (standard) Galerkin semidiscretization of (\ref{eqsw2a}) is then defined as follows: Seek 
$v_{h}:[0,T]\to \acr{S}_{h}$, $w_{h}:[0,T]\to \acr{\mc{S}}_{h}$, such that for $t\in [0,T]$
\begin{align}
	(v_{ht},\phi) & + ((u_{0} + \delta_{0})v_{hx},\phi) + \tfrac{3}{2}(v_{h}v_{hx},\phi) + \tfrac{1}{2}(w_{h}v_{hx},\phi)=0,
	\quad \forall  \phi\in \acr{S}_{h}, \label{eq36} \\
	(w_{ht},\chi) & + ((u_{0}-\delta_{0})w_{hx},\chi) + \tfrac{3}{2}(w_{h}w_{hx},\chi) + \tfrac{1}{2}(v_{h}w_{hx},\chi)=0,
	\quad \forall \chi \in \acr{\mc{S}}_{h}, \label{eq37}
\end{align} 
with 
\begin{equation}
	v_{h}(0)=P(v^{0}), \quad w_{h}(0) = \mc{P}(w^{0}). 
	\label{eq38}
\end{equation}
The main result of this section is 
\begin{proposition} Let $(v,w)$ be the solution of (\ref{eqsw2a}) and assume that the hypotheses (\ref{eqy1}) 
	and (\ref{eqy2}) hold, that $r\geq 3$, and that $h$ is sufficiently small. Then the semidiscrete ivp (\ref{eq36})-(\ref{eq38})
	has a unique solution $(v_{h},w_{h})$ for $0\leq t\leq T$ that satisfies
	\begin{equation}
		\max_{0\leq t\leq T} (\|v - v_{h}\| + \|w - w_{h}\|) \leq Ch^{r-1}.
		\label{eq39}
	\end{equation}
	If $(\eta,u)$ is the solution of (\ref{eqsw2}) and we define 
	\begin{equation}
		\eta_{h} = [\tfrac{1}{2}(v_{h} - w_{h}) + \delta_{0}]^{2} - 1, \quad u_{h} = v_{h} + w_{h} + u_{0},
		\label{eq310}
	\end{equation}
	then
	\[
	\max_{0\leq t\leq T} (\|\eta - \eta_{h}\| + \|u - u_{h}\|) \leq Ch^{r-1}.
	\]
\end{proposition}
\begin{proof}
	Let $\rho = v - Pv$, $\theta = Pv - v_{h}$, $\sigma = w - \mc{P}w$, $\xi = \mc{P}w - w_{h}$. After choosing
	bases for $\acr{S}_{h}$ and $\acr{\mc{S}}_{h}$ we see that the ode ivp (\ref{eq36})-(\ref{eq38}) has a unique solution
	locally in time. From (\ref{eqsw2a}) and (\ref{eq36}), (\ref{eq37}) we obtain, as long as the solution exists, 
	\begin{align}
		(\theta_{t},\phi) + ((u_{0}+\delta_{0})(\theta_{x} + \rho_{x}),\phi) +\tfrac{3}{2} (vv_{x}-v_{h}v_{hx},\phi) +
		\tfrac{1}{2}((wv_{x}-w_{h}v_{hx},\phi) & = 0, \quad \forall \phi \in \acr{S}_{h},
		\label{eq311} \\
		(\xi_{t},\chi) + ((u_{0}-\delta_{0})(\sigma_{x}+\xi_{x}),\chi) + \tfrac{3}{2}(ww_{x}-w_{h}w_{hx},\chi)
		+ \tfrac{1}{2}(vw_{x}-v_{h}w_{hx},\chi) & = 0, \quad \forall \chi \in \acr{\mc{S}}_{h}.
		\label{eq312}
	\end{align}
	Now, since
	\begin{align*}
		vv_{x} - v_{h}v_{hx} & = (v\rho)_{x} + (v\theta)_{x} - (\rho\theta)_{x} - \rho\rho_{x} - \theta\theta_{x}, \\
		wv_{x} - w_{h}v_{hx} & = w(\rho_{x} + \theta_{x}) + v_{x}(\sigma + \xi) - (\rho_{x} + \theta_{x})(\sigma + \xi), \\
		ww_{x} - w_{h}w_{hx} & = (w\sigma)_{x} + (w\xi)_{x} - (\sigma\xi)_{x} - \sigma\sigma_{x} - \xi\xi_{x}, \\
		vw_{x} - v_{h}w_{hx} & = v(\sigma_{x} + \xi_{x}) + w_{x}(\rho + \theta) - (\sigma_{x} + \xi_{x})(\rho + \theta),
	\end{align*} 
	it follows that 
	\begin{align}
		vv_{x} - v_{h}v_{hx} & = (v\theta)_{x} - \theta\theta_{x} + R_{11}, \hspace{17pt}  wv_{x} - w_{h}v_{hx} = -\theta_{x}\xi + R_{12},
		\label{eq313} \\
		ww_{x} - w_{h}w_{hx} & = (w\xi)_{x} - \xi\xi_{x} + R_{21}, \hspace{17pt} vw_{x} - v_{h}w_{hx} = -\xi_{x}\theta + R_{22},
		\label{eq314}
	\end{align} 
	where
	\begin{align}
		R_{11} & = (v\rho)_{x} - (\rho\theta)_{x} - \rho\rho_{x}, \hspace{17pt}  
		R_{12} = w\rho_{x} + w\theta_{x} + v_{x}\sigma + v_{x}\xi - \rho_{x}\sigma - \rho_{x}\xi - \theta_{x}\sigma,
		\label{eq315} \\
		R_{21} & = (w\sigma)_{x} - (\sigma\xi)_{x} - \sigma\sigma_{x}, \hspace{12pt} 
		R_{22} = v\sigma_{x} + v\xi_{x} + w_{x}\rho + w_{x}\theta - \sigma_{x}\rho - \sigma_{x}\theta - \xi_{x}\rho.
		\label{eq316}
	\end{align} 
	Putting now $\phi = \theta$ in (\ref{eq311}) and $\chi = \xi$ in (\ref{eq312}) we obtain
	\begin{equation} 
		\begin{aligned}
			\tfrac{1}{2}\tfrac{d}{dt}\|\theta\|^{2} + &   ((u_{0} + \delta_{0})\theta_{x},\theta) 
			+ \tfrac{3}{2}((v\theta)_{x},\theta) - \tfrac{3}{2}(\theta\theta_{x},\theta)  \\ 
			= &  -((u_{0} + \delta_{0})\rho_{x},\theta) - \tfrac{3}{2}(R_{11},\theta) 
			+ \tfrac{1}{2}(\theta_{x}\xi,\theta) - \tfrac{1}{2}(R_{12},\theta), 
		\end{aligned}
		\label{eq317}
	\end{equation}
	\begin{equation} 
		\begin{aligned}
			\tfrac{1}{2}\tfrac{d}{dt}\|\xi\|^{2} + &   ((u_{0} - \delta_{0})\xi_{x},\xi) 
			+ \tfrac{3}{2}((w\xi)_{x},\xi) - \tfrac{3}{2}(\xi\xi_{x},\xi)  \\ 
			= &  -((u_{0} - \delta_{0})\sigma_{x},\xi) - \tfrac{3}{2}(R_{21},\xi) 
			+ \tfrac{1}{2}(\xi_{x}\theta,\xi) - \tfrac{1}{2}(R_{22},\theta), 
		\end{aligned}
		\label{eq318}
	\end{equation}
	Integration by parts yields (we suppress the $t$-dependence)
	\begin{align*}
		((u_{0} & + \delta_{0})\theta_{x},\theta) = \frac{u_{0}+\delta_{0}}{2}\theta^{2}(1), \quad 
		((v\theta)_{x},\theta) = \frac{1}{2}(v_{x}\theta,\theta) + \frac{1}{2}v(1)\theta^{2}(1), \\
		& (\theta\theta_{x},\theta) = \frac{1}{3}\theta^{3}(1), \quad ((u_{0} - \delta_{0})\xi_{x},\xi) = - \frac{u_{0} - \delta_{0}}{2}\xi^{2}(0),\\
		& ((w\xi)_{x},\xi) = \frac{1}{2}(w_{x}\xi,\xi) - \frac{1}{2}w(0)\xi^{2}(0), \quad (\xi\xi_{x},\xi) = -\frac{1}{3}\xi^{3}(0).
	\end{align*}
	Hence, (\ref{eq317}) becomes
	\begin{align*}
		\tfrac{1}{2}\tfrac{d}{dt}\|\theta\|^{2} & + \tfrac{1}{2}(u_{0} + \delta_{0} + \tfrac{3}{2}v(1) - \theta(1))\theta^{2}(1)\\
		& = - ((u_{0} + \delta_{0})\rho_{x},\theta) - \tfrac{3}{4}(v_{x}\theta,\theta) + \tfrac{1}{2}(\theta_{x}\xi,\theta) 
		- \tfrac{3}{2}(R_{11},\theta) - \tfrac{1}{2}(R_{12},\theta).
	\end{align*}
	By (\ref{eqy2}) and (\ref{eq35}) we see that $u_{0} + \delta_{0} + \tfrac{3}{2}v(1)\geq c_{0}>0$.
	Therefore the above equation gives
	\begin{equation}
		\begin{aligned}
			\tfrac{1}{2}\tfrac{d}{dt}\|\theta\|^{2} & + \tfrac{1}{2}(c_{0} - \theta(1))\theta^{2}(1)
			\leq -((u_{0} + \delta_{0})\rho_{x},\theta)  \\
			& - \tfrac{3}{4}(v_{x}\theta,\theta) + \tfrac{1}{2}(\theta_{x}\xi, \theta) - \tfrac{3}{2}(R_{11},\theta) - \tfrac{1}{2}(R_{12},\theta).
		\end{aligned}
		\label{eq319}
	\end{equation}
	Similarly, from (\ref{eq318}) we obtain 
	\begin{align*}
		\tfrac{1}{2}\tfrac{d}{dt}\|\xi\|^{2} & + \tfrac{1}{2}(-(u_{0} - \delta_{0} + \tfrac{3}{2}w(0)) + \xi(0))\xi^{2}(0)\\
		& = - ((u_{0} - \delta_{0})\sigma_{x},\xi) + \tfrac{1}{2}(\xi_{x}\theta,\xi) - \tfrac{3}{4}(w_{x}\xi,\xi)  
		- \tfrac{3}{2}(R_{21},\xi) - \tfrac{1}{2}(R_{22},\xi).
	\end{align*}
	Again, by (\ref{eqy2}) and (\ref{eq35}) we get $u_{0} - \delta_{0} + \tfrac{3}{2}w(0) \leq - c_{0}<0$.
	We conclude that  
	\begin{equation}
		\begin{aligned}
			\tfrac{1}{2}\tfrac{d}{dt}\|\xi\|^{2} & + \tfrac{1}{2}(c_{0} + \xi(0))\xi^{2}(0)
			\leq -((u_{0} - \alpha)\sigma_{x},\xi)  \\
			& + \tfrac{1}{2}(\xi_{x}\theta, \xi) - \tfrac{3}{4}(w_{x}\xi,\xi)  - \tfrac{3}{2}(R_{21},\xi) - \tfrac{1}{2}(R_{22},\xi).
		\end{aligned}
		\label{eq320}
	\end{equation}
	Finally, adding (\ref{eq319}) and (\ref{eq320}) we get, as long as the solution of (\ref{eq36})-(\ref{eq38}) exists, that
	\begin{equation}
		\begin{aligned}
			\tfrac{1}{2}\tfrac{d}{dt} & (\|\theta\|^{2} + \|\xi\|^{2}) + \tfrac{1}{2}(c_{0} - \theta(1))\theta^{2}(1)
			+ \tfrac{1}{2}(c_{0} + \xi(0))\xi^{2}(0) \\  
			& \leq -((u_{0} + \delta_{0})\rho_{x},\theta) - ((u_{0} - \delta_{0})\sigma_{x},\xi) - \tfrac{3}{4}(v_{x}\theta,\theta)
			- \tfrac{3}{4}(w_{x}\xi,\xi) \\
			& + \tfrac{1}{2}(\theta_{x}\xi,\theta) + \tfrac{1}{2}(\xi_{x}\theta,\xi) - \tfrac{3}{2}(R_{11},\theta)
			- \tfrac{1}{2}(R_{12},\theta) - \tfrac{3}{2}(R_{21},\xi) - \tfrac{1}{2}(R_{22},\xi).
		\end{aligned}
		\label{eq321}
	\end{equation}
	In view of (\ref{eq38}), by continuity we conclude that there exists a maximal temporal instance $t_{h} > 0$ such 
	that $v_{h}$, $w_{h}$ exist for $t\leq t_{h}$ and 
	\begin{equation}
		\|\theta(t)\|_{1,\infty} + \|\xi(t)\|_{1,\infty} \leq c_{0}, \quad t \in [0,t_{h}].
		\label{eq322}
	\end{equation}
	Suppose that $t_{h} < T$. For $t\in [0,t_{h}]$ we have by (\ref{eq322})
	\begin{equation}
		\tfrac{1}{2}(c_{0} - \theta(1))\theta^{2}(1) + \tfrac{1}{2}(c_{0} + \xi(0))\xi^{2}(0) \geq 0,
		\label{eq323}
	\end{equation}
	and
	\begin{equation}
		\tfrac{1}{2}\abs{(\theta_{x}\xi,\theta)} + \tfrac{1}{2}\abs{(\xi_{x}\theta,\xi)} 
		\leq \frac{c_{0}}{2}\|\theta\| \|\xi\|.
		\label{eq324}
	\end{equation}
	We obviously have  
	\begin{equation}
		\abs{(v_{x}\theta,\theta)} + \abs{(w_{x}\xi,\xi)} 
		\leq C(\|\theta\|^{2} +  \|\xi\|^{2}).
		\label{eq325}
	\end{equation}
	Using now the approximation and inverse properties (\ref{eq21})-(\ref{eq26}) for $\acr{S}_{h}$ (and also
	for $\acr{\mc{S}}_{h}$) we estimate the rest of the terms in the right-hand side of (\ref{eq321}) as follows.
	We first clearly have
	\begin{equation}
		\abs{((u_{0} + \delta_{0})\rho_{x},\theta)} + \abs{((u_{0} - \delta_{0})\sigma_{x},\xi)} 
		\leq Ch^{r-1}(\|\theta\| + \|\xi\|).
		\label{eq326}
	\end{equation}
	Integrating by parts we see by (\ref{eq315}) that
	\begin{align*}
		(R_{11},\theta) & = ((v\rho)_{x},\theta) - ((\rho\theta)_{x},\theta) - (\rho\rho_{x},\theta) \\
		& = v(1)\rho(1)\theta(1) - (v\rho,\theta_{x}) - \rho(1)\theta^{2}(1) + (\rho\theta,\theta_{x}) - (\rho\rho_{x},\theta).
	\end{align*} 
	Therefore
	\begin{equation}
		\begin{aligned}
			\abs{(R_{11},\theta)} & \leq C\|\rho\|_{\infty} \|\theta\|_{\infty} + C\|\rho\|_{\infty} \|\theta_{x}\| 
			+ \|\rho\|_{\infty} \|\theta\|_{\infty}^{2} 
			+ \|\rho\|_{\infty} \|\theta\| \|\theta_{x}\| + \|\rho\|_{\infty} \|\rho_{x}\| \|\theta\| \\
			& \leq Ch^{r} \|\theta\|_{\infty} + Ch^{r} \|\theta_{x}\| + Ch^{r}\|\theta\|_{\infty}^{2} 
			+ Ch^{r} \|\theta\| \|\theta_{x}\|  + Ch^{2r-1} \|\theta\|  \\
			& \leq Ch^{r-1}(\|\theta\| + \|\theta\|^{2}).
		\end{aligned}
		\label{eq327}
	\end{equation}
	Integration by parts and (\ref{eq315}) yield for the $R_{12}$ term that 
	\[
	(R_{12},\theta) = (w\rho_{x},\theta) - \tfrac{1}{2}(w_{x}\theta,\theta) + (v_{x}\sigma,\theta)
	+ (v_{x}\xi,\theta)  - (\rho_{x}\sigma,\theta) - (\rho_{x}\xi,\theta) - (\theta_{x}\sigma,\theta).
	\]
	Hence, similarly as above 
	\begin{equation}
		\begin{aligned}
			\abs{(R_{12},\theta)} & \leq C h^{r-1}\|\theta\| + C \|\theta\|^{2} + C h^{r}\|\theta\|  
			+ C \|\xi\| \|\theta\| \\
			& \,\,\quad  + Ch^{2r-1} \|\theta\| + C h^{r-1} \|\xi\|_{\infty} \|\theta\| + C h^{r} \|\theta\|_{\infty} \|\theta_{x}\| \\ 
			& \leq Ch^{r-1}\|\theta\| + C \|\theta\|^{2} + C \|\xi\| \|\theta\|.
		\end{aligned}
		\label{eq328}
	\end{equation}
	Again, using integration by parts and (\ref{eq315}) for the $R_{21}$ term, we obtain
	\[
	(R_{21},\xi) = - (w\sigma,\xi_{x}) - w(0) \sigma(0) \xi(0) + (\sigma\xi,\xi_{x}) + \sigma(0)\xi^{2}(0)
	- (\sigma\sigma_{x},\xi).
	\]
	Therefore
	\begin{equation}
		\begin{aligned}
			\abs{(R_{21},\xi)} & \leq C\|\sigma\| \|\xi_{x}\| + C\|\sigma\|_{\infty} \|\xi\|_{\infty}  + \|\sigma\|_{\infty} \|\xi\| \|\xi_{x}\|  
			+ \|\sigma\|_{\infty} \|\xi\|_{\infty}^{2} + \|\sigma\|_{\infty} \|\sigma_{x}\| \|\xi\| \\
			& \leq Ch^{r} \|\xi_{x}\| + Ch^{r} \|\xi\|_{\infty} + Ch^{r}\|\xi\| \|\xi_{x}\|  + Ch^{r} \|\xi\|_{\infty}^{2}  
			+ Ch^{2r-1} \|\xi\|  \\
			& \leq Ch^{r-1}(\|\xi\| + \|\xi\|^{2}).
		\end{aligned}
		\label{eq329}
	\end{equation}
	Finally, by (\ref{eq315}) and integration by parts we have for the $R_{22}$ term
	\[
	(R_{22},\xi) = (v\sigma_{x},\xi) - \tfrac{1}{2} (v_{x}\xi,\xi) + (w_{x}\rho,\xi) + (w_{x}\theta,\xi)
	- (\sigma_{x}\rho,\xi) - (\sigma_{x}\theta,\xi) - (\rho\xi_{x},\xi).
	\]
	Hence, 
	\begin{equation}
		\begin{aligned}
			\abs{(R_{22},\xi)} & \leq C\|\sigma_{x}\| \|\xi\| + C \|\xi\|^{2} + C \|\rho\| \|\xi\| + C \|\theta\| \|\xi\|  \\
			& \hspace{11pt} + \|\sigma_{x}\| \|\rho\|_{\infty} \|\xi\| + \|\sigma_{x}\| \|\theta\|_{\infty} \|\xi\| + \|\rho\|_{\infty} \|\xi_{x}\| \|\xi\| \\
			& \leq C h^{r-1}\|\xi\| + C \|\xi\|^{2} + C h^{r}\|\xi\|  + C \|\theta\| \|\xi\|  + Ch^{2r-1}\|\xi\|  \\
			& \hspace{11pt} + C h^{r-1} \|\theta\|_{\infty} \|\xi\| + C h^{r} \|\xi_{x}\| \|\xi\| \\ 
			& \leq  Ch^{r-1} \|\xi\| + C\|\xi\|^{2}  + C \|\theta\| \|\xi\|.
		\end{aligned}
		\label{eq330}
	\end{equation}
	By (\ref{eq321}), taking into account (\ref{eq323})-(\ref{eq330}) we see that 
	\[
	\tfrac{1}{2}\tfrac{d}{dt}(\|\theta\|^{2} + \|\xi\|^{2}) \leq Ch^{r-1}(\|\theta\| + \|\xi\|) + C(\|\theta\|^{2} + \|\xi\|^{2}),
	\quad t \in [0,t_{h}].
	\]
	An application of Gronwall's Lemma and (\ref{eq38}) yield
	\begin{equation}
		\|\theta(t)\| + \|\xi(t)\| \leq Ch^{r-1}, \quad t\in [0,t_{h}],
		\label{eq331}
	\end{equation}
	from which by inverse assumptions it follows that $\|\theta\|_{1,\infty} + \|\xi\|_{1,\infty} \leq Ch^{r-5/2}$ for
	$t\in[0,t_{h}]$. Since it was assumed that $r\geq 3$ this contradicts the maximality of $t_{h}$ and (\ref{eq331})
	holds for $0\leq t\leq T$. The estimate (\ref{eq39}) follows. 
	Since now $\|v - v_{h}\|_{\infty} \leq \|\rho\|_{\infty} + \|\theta\|_{\infty} \leq Ch^{r-3/2}$ and similarly 
	$\|w - w_{h}\|_{\infty} \leq Ch^{r-3/2}$, and since 
	\[
	\eta - \eta_{h} = [\delta_{0} + \tfrac{1}{4}((v - w) + (v_{h} - w_{h}))] [(v - w) - (v_{h} - w_{h})],
	\]
	we conclude that $\|\eta - \eta_{h}\| \leq C(\|v - v_{h}\| + \|w - w_{h}\|) \leq Ch^{r-1}$.  Similarly
	$\|u - u_{h}\| \leq \|v - v_{h}\| + \|w - w_{h}\| \leq Ch^{r-1}$, and the proof of Proposition 3.1 is now complete.
\end{proof}
\section{Numerical implementation and experiments}
\subsection{Supercritical case}
We will implement the standard Galerkin method for the shallow water equations with characteristic boundary
conditions in the supercritical case using the space of piecewise linear continuous functions on a uniform 
mesh in $[0,1]$ in the usual manner that problems with nonhomogeneous Dirichlet boundary conditions
are approximated in practice. For this purpose we let $x_{i} = ih$, $0\leq i\leq N$, $Nh=1$, define 
$S_{h} = \{\phi \in C^{0}[0,1] : \phi\big|_{[x_{j},x_{j+1}]} \in \mathbb{P}_{1}\,, 0\leq j\leq N-1\}$, and let 
$\acr{S}_{h}$ consist of the functions in $S_{h}$ that vanish at $x=0$. We seek $\eta_{h}$, 
$u_{h} : [0,T] \to S_{h}$, the semidiscrete approximation of the solution of (\ref{eqsw1}), satisfying 
for $0\leq t\leq T$  $\eta_{h}(0,t)=\eta_{0}$, $u_{h}(0,t) = u_{0}$, and the system of ode's 
\begin{equation}
	\begin{aligned}
		(\eta_{ht},\phi) & + (u_{hx},\phi) + ((\eta_{h}u_{h})_{x},\phi) = 0, \quad \forall \phi\in \acr{S}_{h},\\
		(u_{ht},\phi) & + (\eta_{hx},\phi) + (u_{h}u_{hx},\phi) = 0, \quad \forall \phi\in \acr{S}_{h},
	\end{aligned}
	\label{eq41}
\end{equation}
with initial values  $\eta_{h}(0) = P\eta^{0}$,  $u_{h}(0) = Pu^{0}$, where $P$ is the $L^{2}$ projection
onto $S_{h}$. We discretize this ode ivp in time by the `classical' explicit 4$^{th}$-order accurate 
Runge-Kutta scheme, written in the case of the ode $y' = f(t,y)$, $0\leq t\leq T$, in the form
\begin{equation}
	\begin{aligned}
		y^{n,1} & = y^{n} + \tfrac{k}{2} f(t^{n} + \tfrac{k}{2},y^{n}), \\
		y^{n,2} & = y^{n} + \tfrac{k}{2} f(t^{n} + \tfrac{k}{2},y^{n,1}), \\
		y^{n,3} & = y^{n} + k f(t^{n} + k,y^{n,2}), \\
		y^{n+1} & = y^{n} + k\bigl(\tfrac{1}{6}f(t^{n},y^{n}) + \tfrac{1}{3}f(t^{n}+\tfrac{k}{2},y^{n,1}) \\
		& \hspace{44pt} + \tfrac{1}{3} f(t^{n}+\tfrac{k}{2},y^{n,2}) + \tfrac{1}{6}f(t^{n} + k,y^{n,3})\bigr),
	\end{aligned}
	\label{eq42}
\end{equation}
where $k$ is the time step, $t^{n}=nk$, $n=0,1,\dots,M-1$, $Mk=T$, and $y^{n}$ approximates $y(t^{n})$. Theoretical and numerical evidence from linear stability theory and previous work
by the authors, \cite{ad1}, \cite{ad2}, on similar nonlinear systems, suggests that the resulting
fully discrete scheme is stable under a Courant-number restriction of the form 
 $k/h \leq r_{0}$.  \par 
  In our first numerical experiment we check the spatial rate of convergence of this fully discrete scheme. We consider (\ref{eqsw1}) 
with $\eta_{0}=1$ and $u_{0} = 3$ and add right-hand sides to the pde's so that the exact solution of the 
ibvp is $\eta(x,t) = x\mathrm{e}^{-xt} + \eta_{0}$, $u(x,t) = (1 - x- \cos(\pi x))\mathrm{e}^{2t} + u_{0}$.
With $h=1/N$, $k=h/10$ (so that the temporal error is negligible), we obtain the $L^{2}$ errors and
associated rates of convergence of (essentially) the semidiscrete problem at $T=1$ shown in Table \ref{tbl41}.
\begin{table}[h!]
\setlength\tabcolsep{10.7pt}
\begin{tabular}{@{}ccccc@{}}
\tblhead{$N$   &  $\eta$   &  $order$    &      $u$         &  $order$} 
$40$   &  $1.243098(-3)$   &    --       &  $5.623510(-3)$  &  --          \\ 
$80$   &  $3.110525(-4)$   &  $1.99871$  &  $1.405648(-3)$  &  $2.00024$   \\ 
$160$  &  $7.778520(-5)$   &  $1.99959$  &  $3.513979(-4)$  &  $2.00006$   \\ 
$320$  &  $1.944737(-5)$   &  $1.99992$  &  $8.784876(-5)$  &  $2.00001$   \\ 
$480$  &  $8.643341(-6)$   &  $1.99998$  &  $3.904381(-5)$  &  $2.00001$   \\ 
$520$  &  $7.364768(-6)$   &  $1.99996$  &  $3.326806(-5)$  &  $2.00001$   \\ \hline 
\end{tabular}\vspace{3.1pt}
\small
\caption{$L^{2}$ errors and spatial orders of convergence, supercritical case.}
\label{tbl41}
\end{table}
\normalsize
The experimental rates of convergence for both components of the solution are clearly equal to 2, i.e. 
superaccurate, as the expected rates for a general quasiuniform mesh would be equal to 1. 
A numerical study of stability for this example indicates that the errors remain of the same order 
of magnitude at $T=1$ up to about $k/h = 0.13$. For larger values of $k/h$ blow-up eventually
occurs. It should be noted that for the same test problem the alternative standard Galerkin
formulation (\ref{eq27})-(\ref{eq29}) (analyzed in section 2) coupled with the same Runge-Kutta 
scheme gives $L^{2}$ errors that coincide with those of Table \ref{tbl41} to at least 5 significant
digits; the stability condition was also the same. (Recall that the result of Proposition 2.1
strictly holds for $r\geq 3$ due to the technical requirement in the proof for controlling the
$W^{1,\infty}$ norm of the error. The numerical results suggest that the scheme converges
for $r=2$ as well and that the superaccurate order of convergence for a uniform mesh for
$r=2$, proved in \cite{ad2} for an ibvp for the shallow water equations with homogeneous
Dirichlet boundary conditions on $u$, persists in the case of the ibvp (\ref{eqsw1}) too.)\par
Since the temporal error is much smaller than the spatial one, the experimental estimation
of the temporal order of convergence may be done in the following way, used in \cite{bdmk}.
\begin{table}[h!]
\setlength\tabcolsep{10.7pt}
	\begin{tabular}{@{}ccccc@{}}
\tblhead{$k/h$   &  $E^{*}(T)$   &  $order$    &      $E(T)$}  
$1/35$   &  $2.6618459890(-8)$   &    --       &  $1.9910684230(-4)$   \\  
$1/40$   &  $1.6020860073(-8)$   &  $3.8022$   &  $1.9910680992(-4)$   \\ 
$1/45$   &  $1.0112973048(-8)$   &  $3.9061$   &  $1.9910679532(-4)$   \\ 
$1/50$   &  $6.6717108025(-9)$   &  $3.9478$   &  $1.9910678792(-4)$   \\ 
$1/55$   &  $4.5726218272(-9)$   &  $3.9638$   &  $1.9910678370(-4)$   \\ 
$1/60$   &  $3.2362144361(-9)$   &  $3.9728$   &  $1.9910678102(-4)$   \\ 
$1/64$   &  $2.5020819256(-9)$   &  $3.9865$   &  $1.9910677950(-4)$   \\ 
$1/64.5$ &  $2.4254282105(-9)$   &  $3.9983$   &  $1.9910677934(-4)$   \\ 
$1/65$   &  $2.3516195603(-9)$   &  $4.0020$   &  $1.9910677918(-4)$   \\ \hline
\end{tabular}\vspace{3.1pt}
\small
\caption{Temporal order of convergence, supercritical case, scheme (\ref{eq41})-(\ref{eq42}),
	$h=1/100$, $T=1$, $k_{ref}=h/120$.}
\label{tbl42a}
\end{table}
\normalsize
Let $H_{h}^{n}$ be the fully discrete approximation of $\eta(t^{n})$. For a fixed value of $h$
we make a reference computation with a very small value $k=k_{ref}$. 
The approximate 
solution $H_{h}^{m}=H_{h}^{m}(h,k_{ref})$, where $mk_{ref}=T$, differs from the exact
solution $\eta(\cdot,T)$ by an amount which is practically the error of the spatial 
discretization. For the same value of $h$ we define a modified $L^{2}$ error for small values
of $k$, that are nevertheless considerably  larger than $k_{ref}$, by the formula
$E^{*}(T) = \|H_{h}^{n}(h,k) - H_{h}^{m}(h,k_{ref})\|$, where $nk = T$. Since taking the
difference $H_{h}^{n}(h,k) - H_{h}^{m}(h,k_{ref})$ essentially cancels the spatial error of 
$H_{h}^{n}(h,k)$, we expect that $E^{*}(T)$ will decrease at the temporal order of convergence 
of the scheme as $k$ decreases. This is illustrated in the case of the test problem under
consideration and the fully discrete scheme (\ref{eq41})-(\ref{eq42}) in Table \ref{tbl42a}, where
$h = 1/100$, $T=1$, $k_{ref}=h/120$, and $E(T)$ denotes the $L^{2}$ error 
$\|H_{h}^{n}(h,k) - \eta(t^{n})\|$. For this range of $k$'s the expected temporal order of 
convergence, equal to $4$, clearly emerges. The analogous experiment with the Galerkin 
method (\ref{eq27})-(\ref{eq29}) discretized in time with the same Runge-Kutta scheme yields
fourth-order temporal convergence in $L^{2}$ as well.  \par
In the next numerical experiment  we integrate (\ref{eqsw1}) with the fully discrete
scheme (\ref{eq41})-(\ref{eq42}), taking $h=1/N$, $N=2000$, $k=h/10$, $\eta_{0}=1$, $u_{0}=3$, and initial conditions
$\eta^{0}(x)=0.05\exp (-400(x-0.5)^{2}) + \eta_{0}$, $u^{0}(x)=0.1\exp (-400(x-0.5)^{2}) + u_{0}$,
$0\leq x\leq1$. (Small-amplitude initial conditions were taken to ensure that no discontinuities in the 
derivatives of the solution develop before the wave profiles exit the spatial interval of integration.)
The evolution of the numerical solution is depicted in Figure \ref{fig41}(a)-(f). 
(The approximate $\eta$-profiles are on the left and those of $u$ on the right.) 
The initial Gaussian perturbations of the uniform state $\eta_{0}=1$, $u_{0}=3$ evolve into two unequal pulses
for both $\eta$ and $u$ that travel to the right and exit the computational domain by about $t=0.4$ without
leaving any visible 
residue or backwards-travelling oscillations as is confirmed by Fig. \ref{fig41}(g) that shows
the time history of the quantities $\max_{x}\abs{\eta(x,t) - \eta_{0}}$ and $\max_{x}\abs{u(x,t) - u_{0}}$.
(Here $\eta$, $u$ denote the approximate solution.) Due to the presence of the spatial and temporal discretizations,
the numerical boundary conditions are not expected to be exactly transparent. 
\captionsetup[subfloat]{labelformat=empty,position=bottom,singlelinecheck=false}
\begin{figure}[h!]
	\begin{center}
		\subfloat[]{\includegraphics[totalheight=2.7in,width=1.825in,angle=-90]{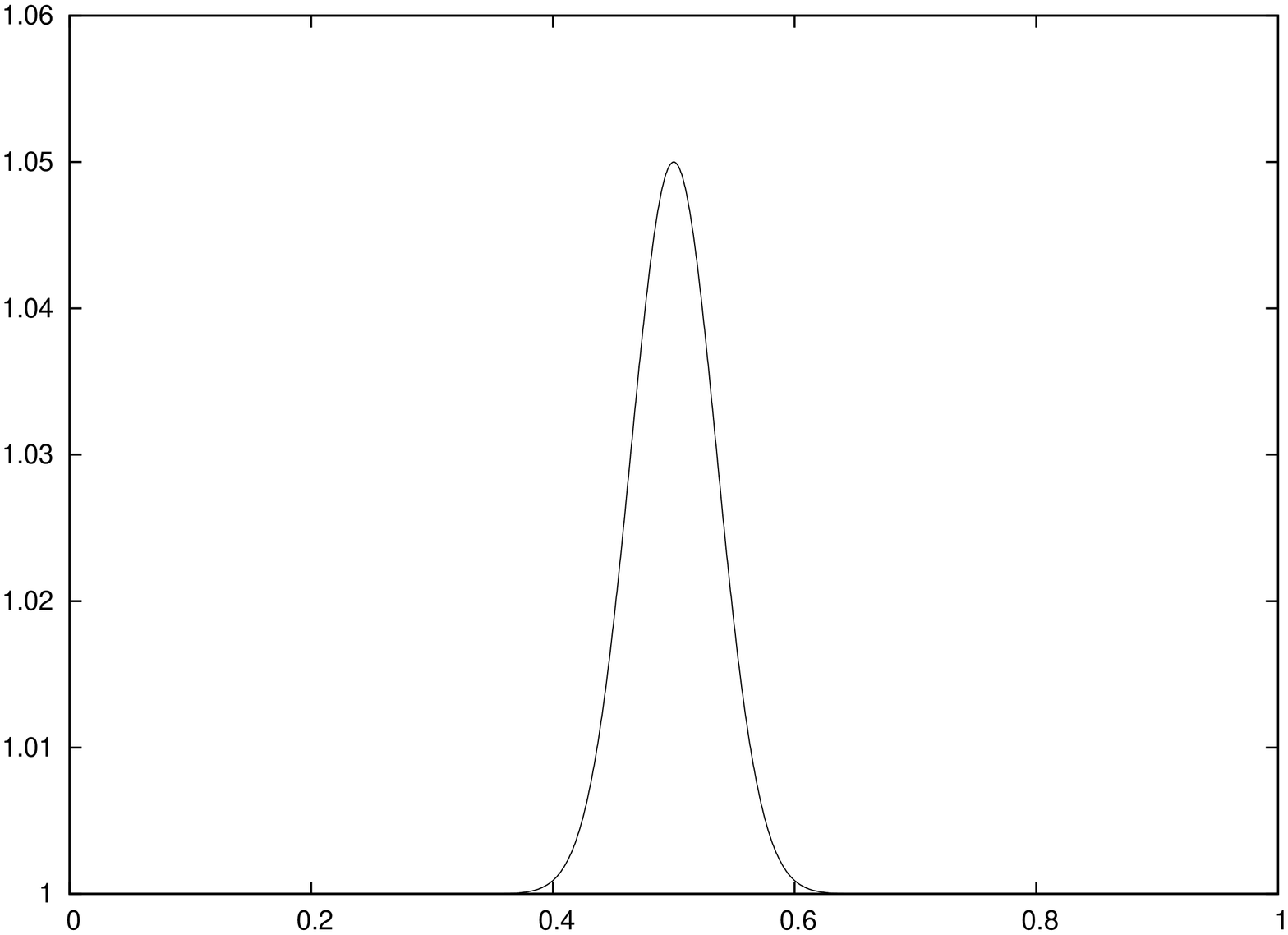}}  \quad \,
		\subfloat[]{\includegraphics[totalheight=2.7in,width=1.825in,angle=-90]{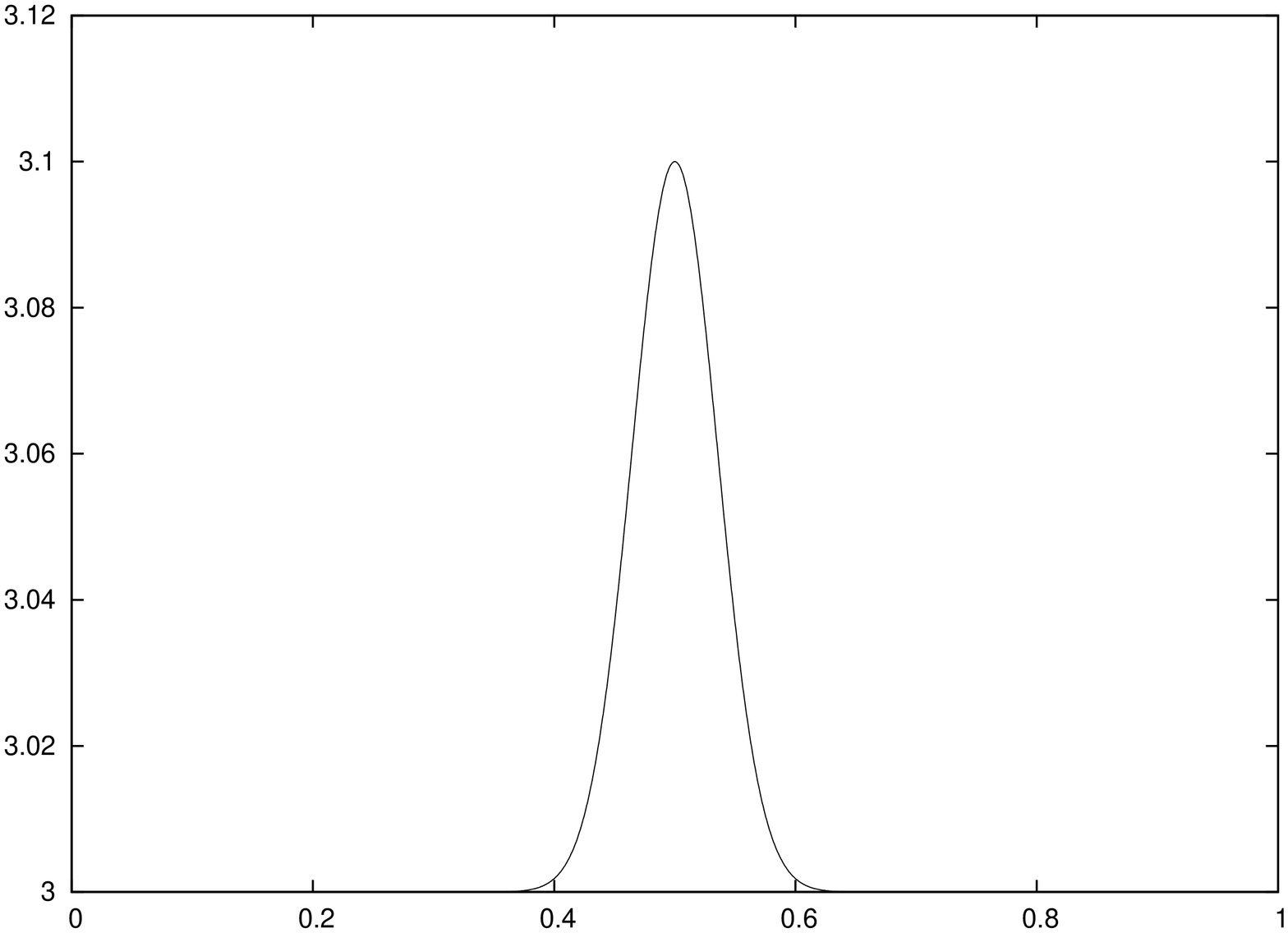}}\\
		(a)\,\,\,\,$\eta$ and $u$ at $t=0.0$  \vspace{12.5pt}\\
		\subfloat[]{\includegraphics[totalheight=1.825in,width=2.7in]{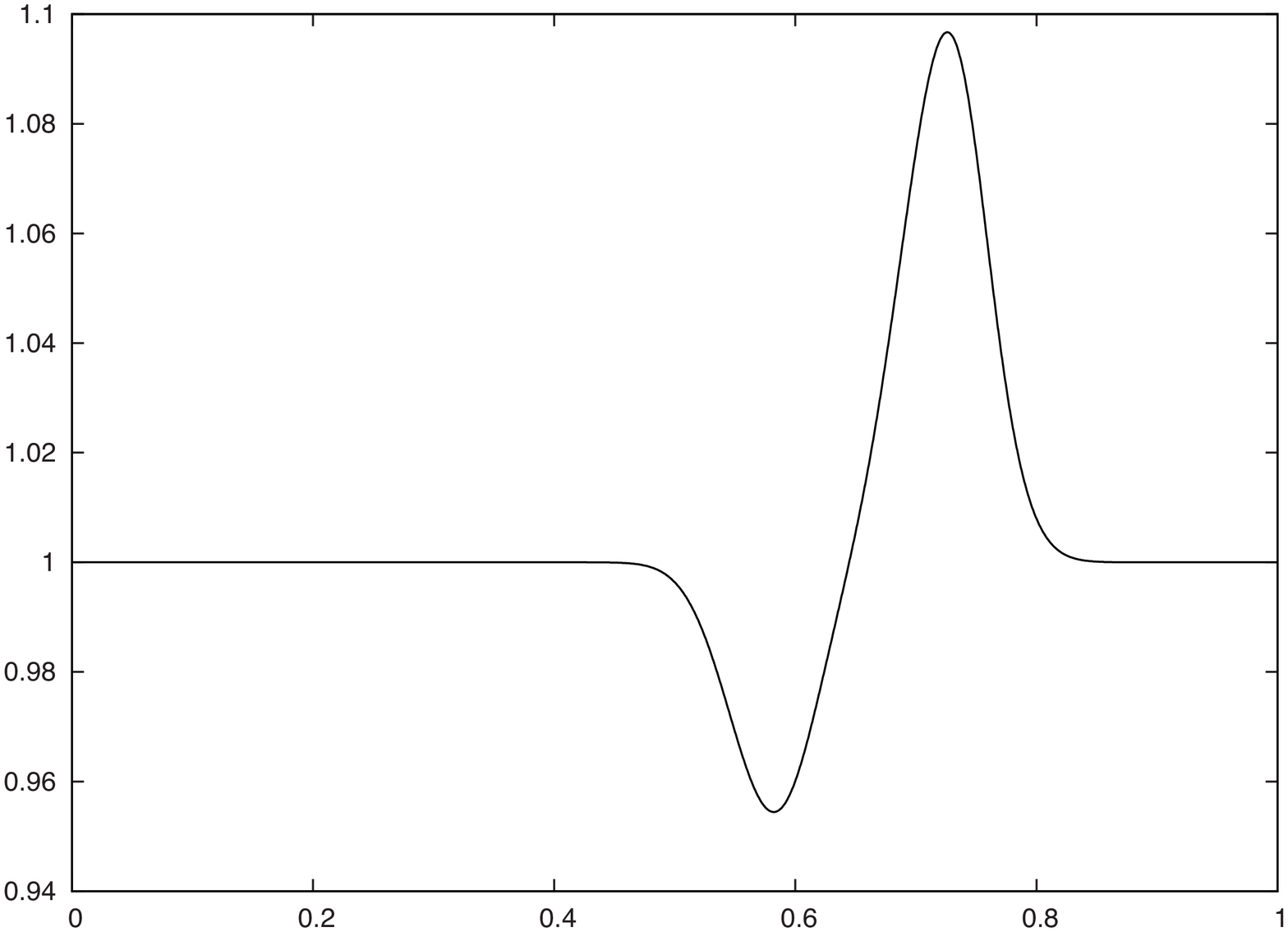}}  \quad \,
		\subfloat[]{\includegraphics[totalheight=1.825in,width=2.7in]{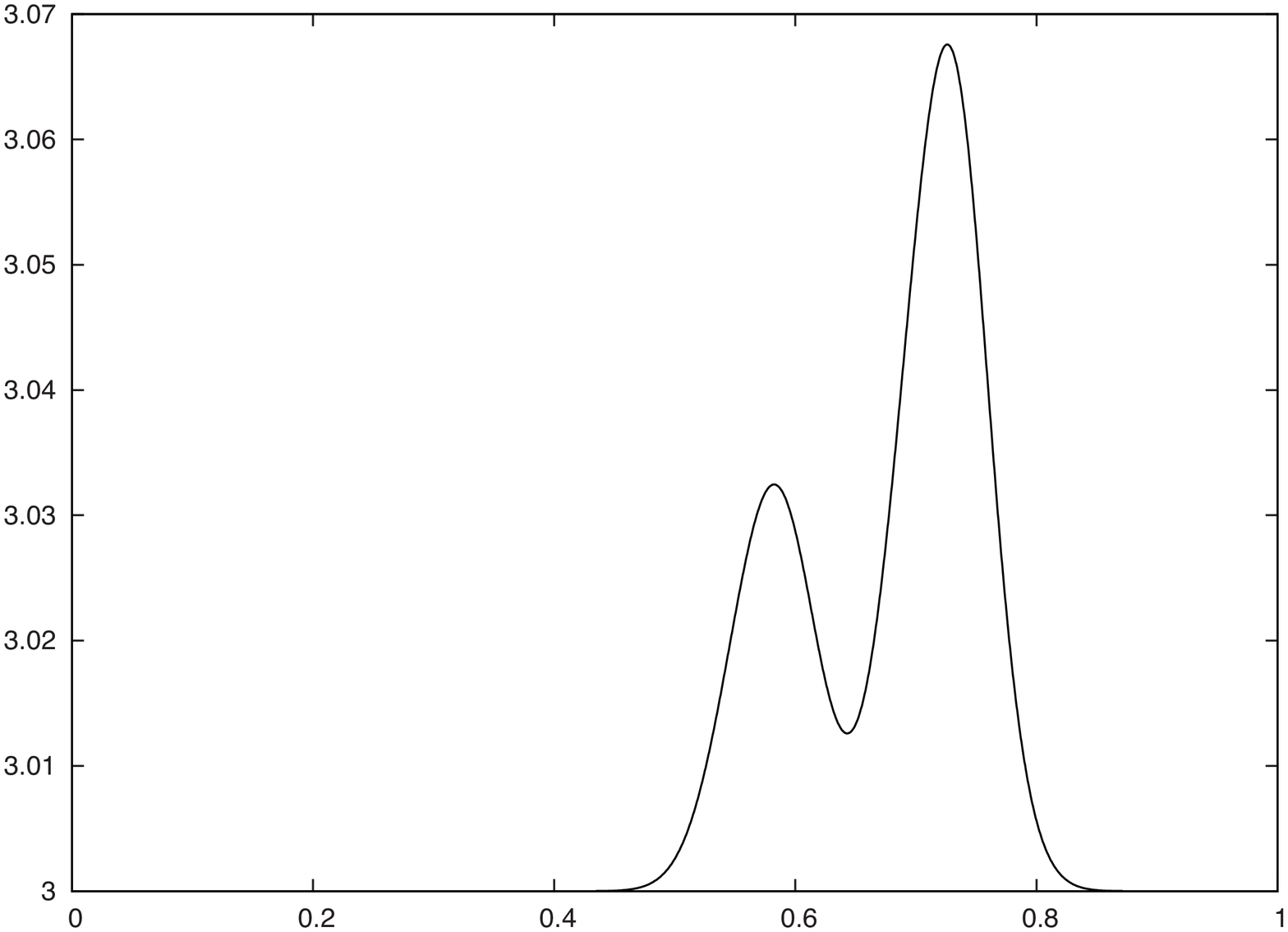}}\\
		(b)\,\,\,\,$\eta$ and $u$ at $t=0.05$  \vspace{12.5pt}\\
		\subfloat[]{\includegraphics[totalheight=1.825in,width=2.7in]{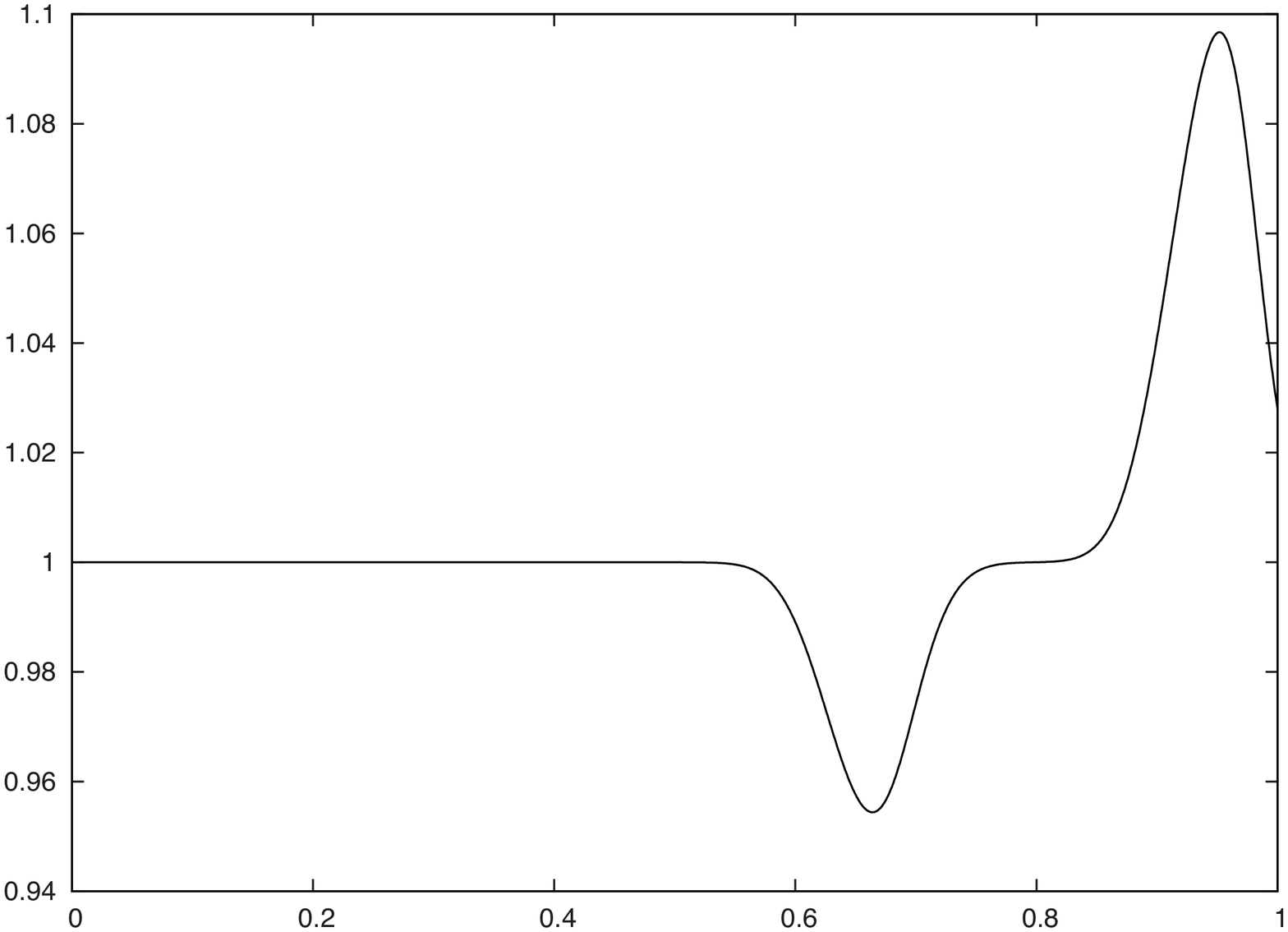}}  \quad \,
		\subfloat[]{\includegraphics[totalheight=1.825in,width=2.7in]{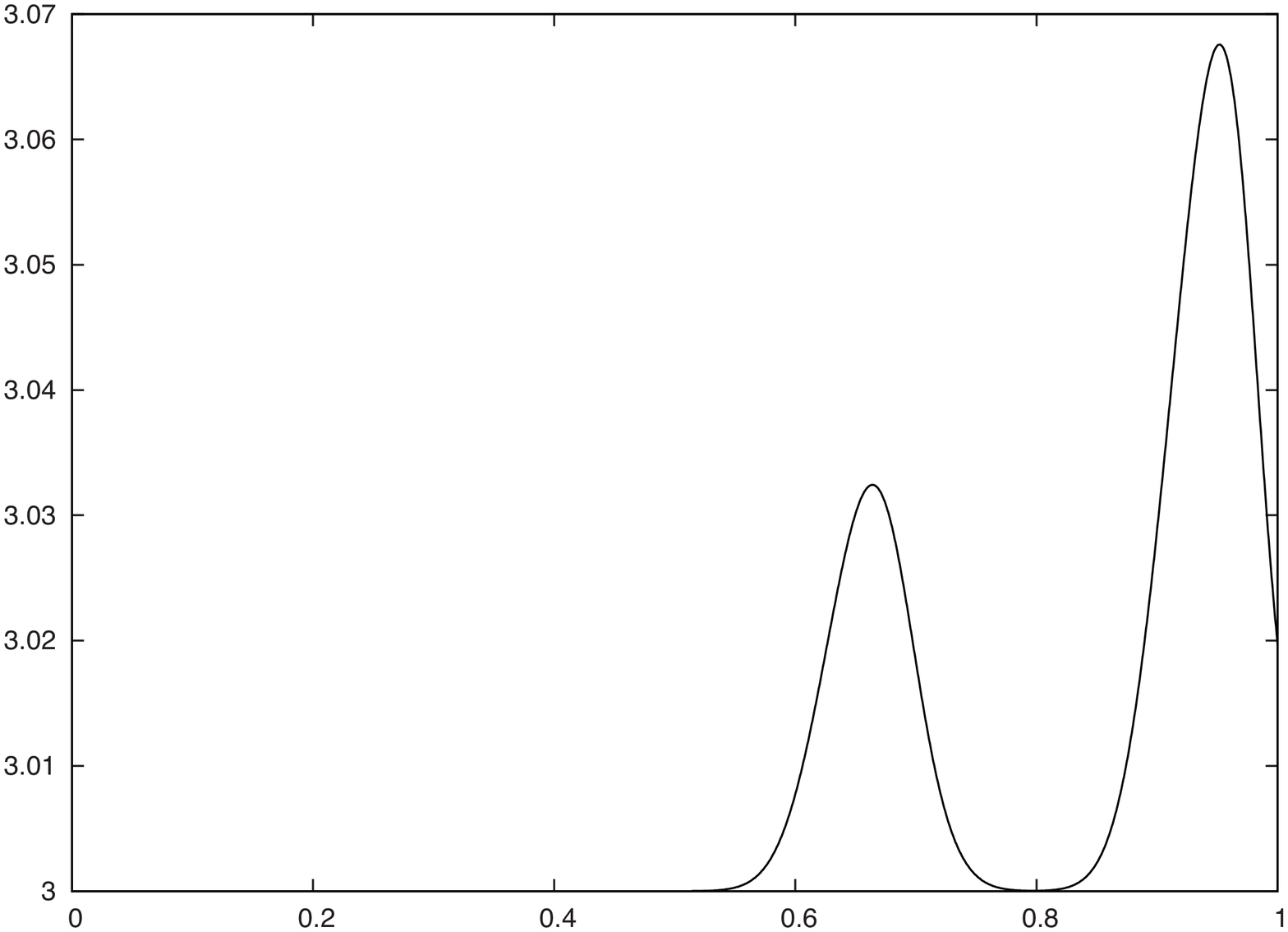}}\\
		(c)\,\,\,\,$\eta$ and $u$ at $t=0.1$ 
 \end{center}
\end{figure} 		
\clearpage
\begin{figure}[t!]
	\begin{center}
		\subfloat[]{\includegraphics[totalheight=1.825in,width=2.7in]{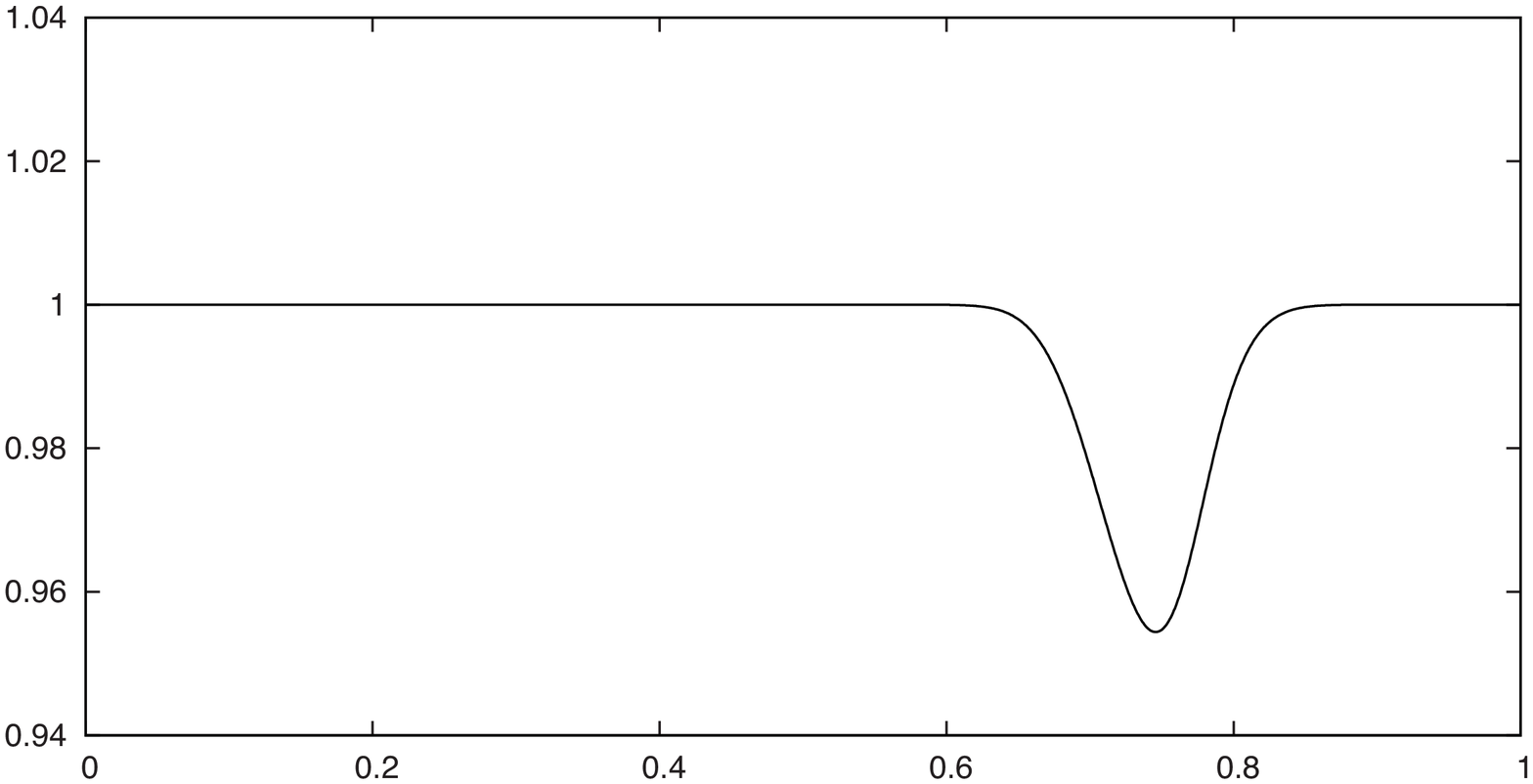}}  \quad \,
		\subfloat[]{\includegraphics[totalheight=1.825in,width=2.7in]{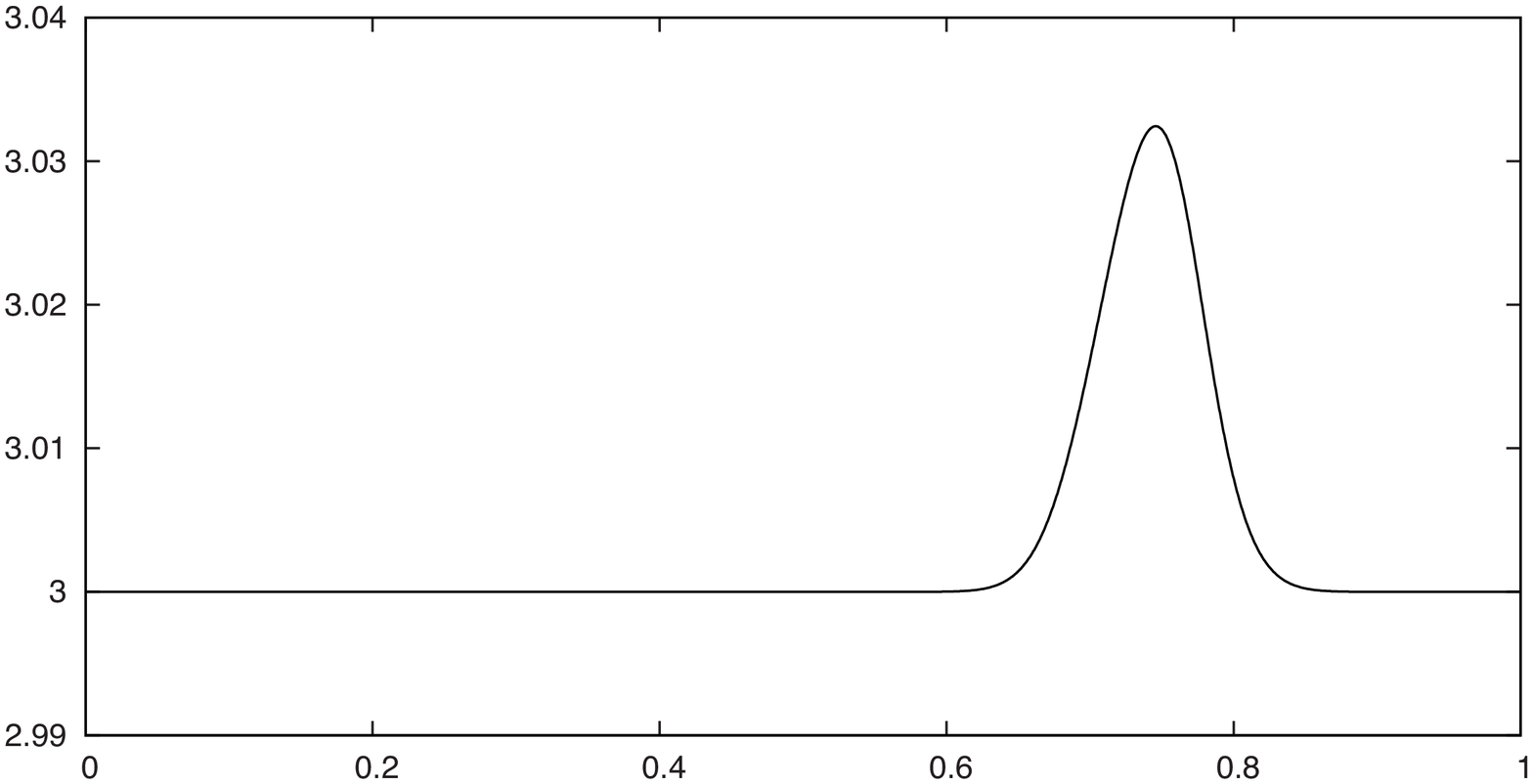}}\\
		(d)\,\,\,\,$\eta$ and $u$ at $t=0.15$ \vspace{12.5pt}\\
  	\subfloat[]{\includegraphics[totalheight=1.825in,width=2.7in]{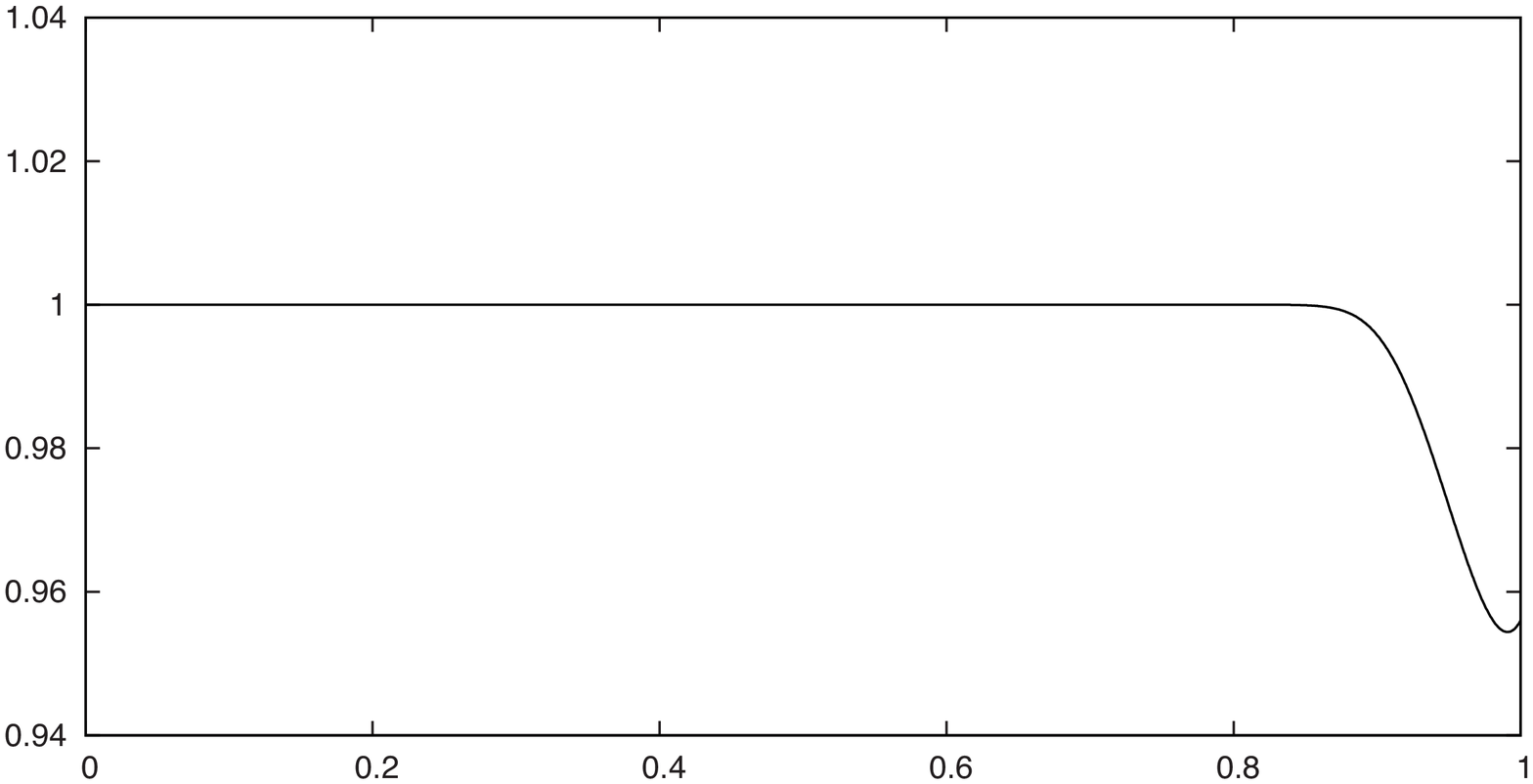}}  \quad \,
  	\subfloat[]{\includegraphics[totalheight=1.825in,width=2.7in]{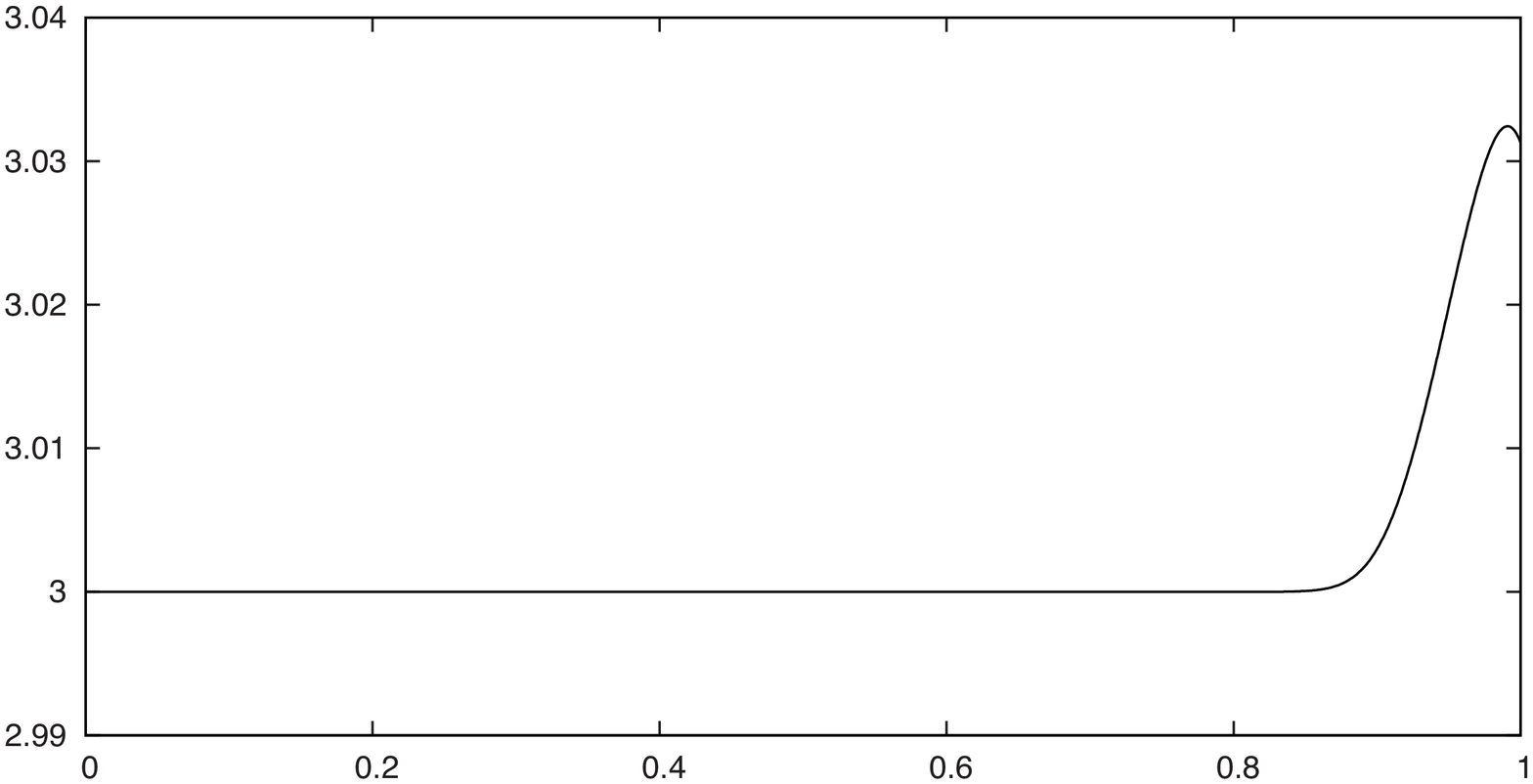}}\\
  	(e)\,\,\,\,$\eta$ and $u$ at $t=0.3$ \vspace{12.5pt}\\
  	\subfloat[]{\includegraphics[totalheight=1.825in,width=2.7in]{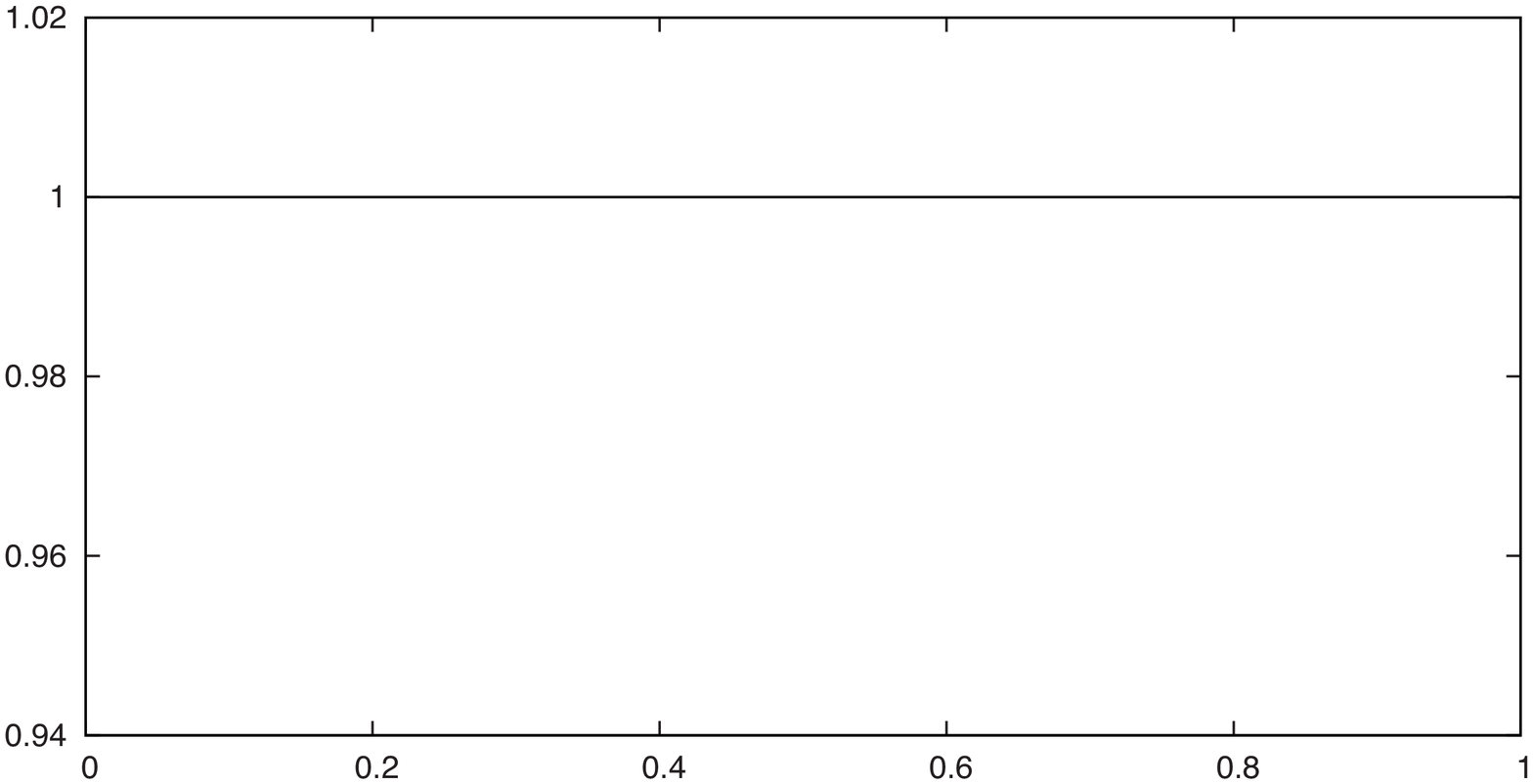}}  \quad \,
		\subfloat[]{\includegraphics[totalheight=1.825in,width=2.7in]{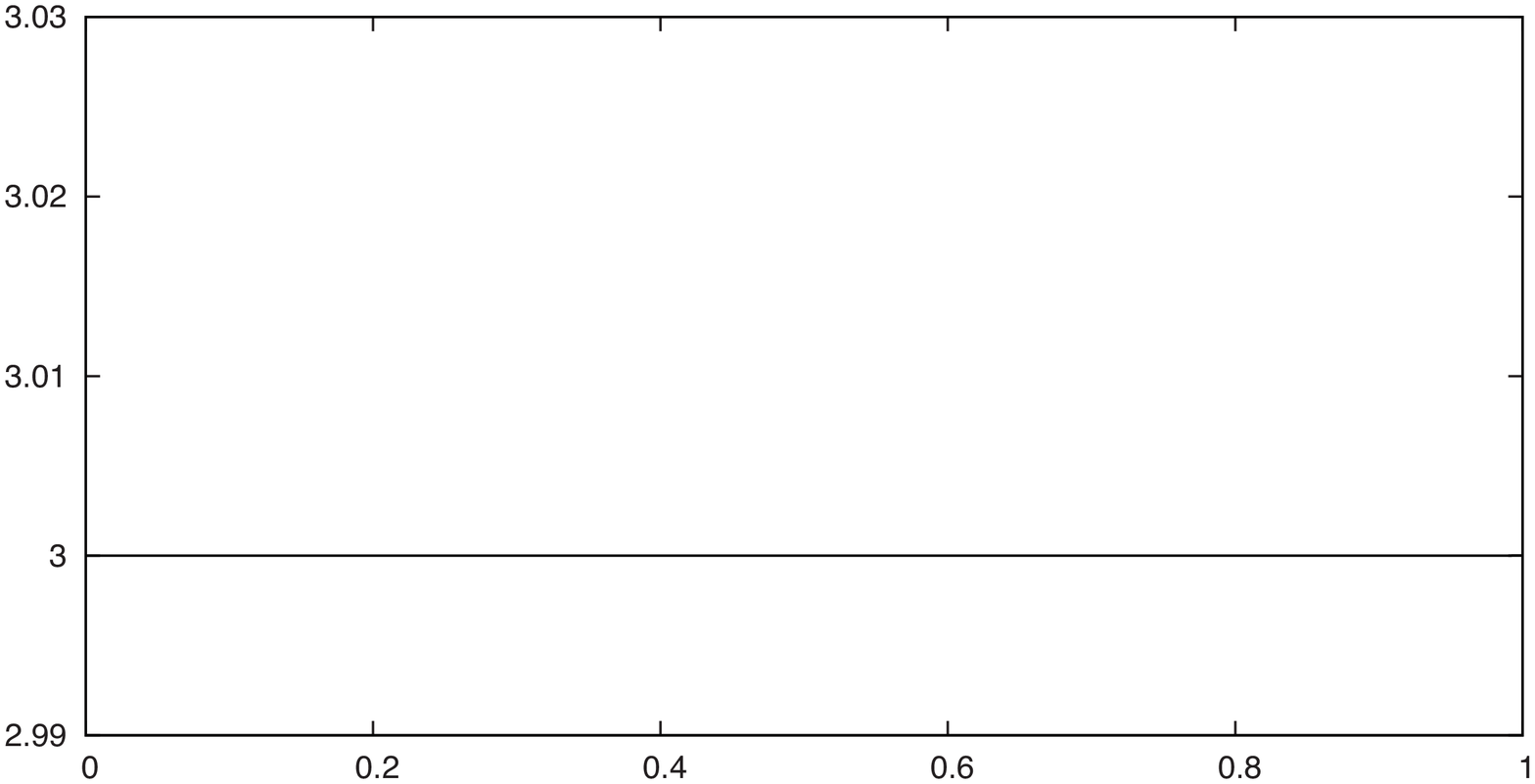}}\\
		(f)\,\,\,\,$\eta$ and $u$ at $t=0.4$ 
	\end{center}
\end{figure}
\clearpage
\begin{figure}[t!]
	\begin{center}
		\subfloat[]{\includegraphics[totalheight=1.825in,width=2.7in]{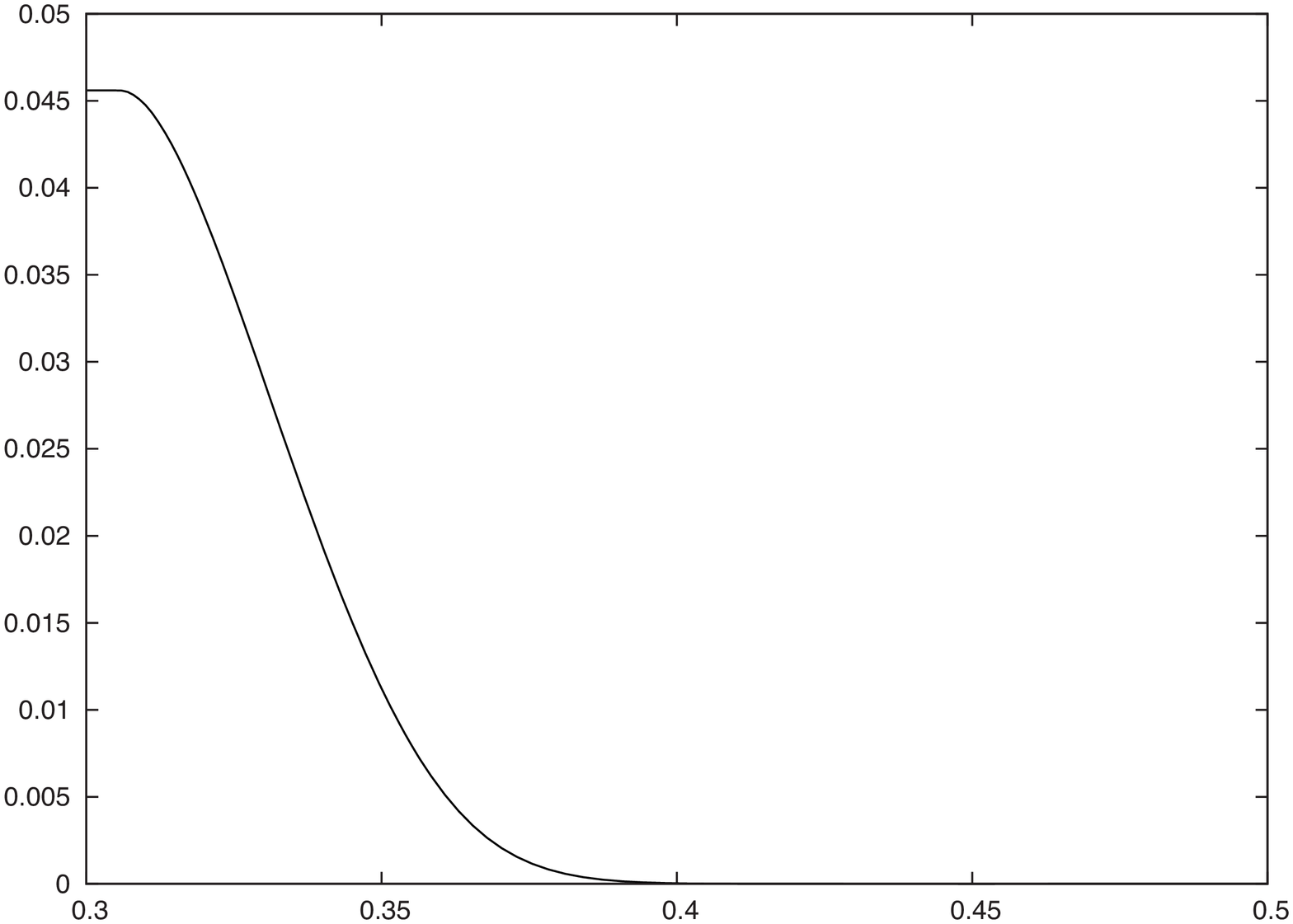}}  \quad \,
		\subfloat[]{\includegraphics[totalheight=1.825in,width=2.7in]{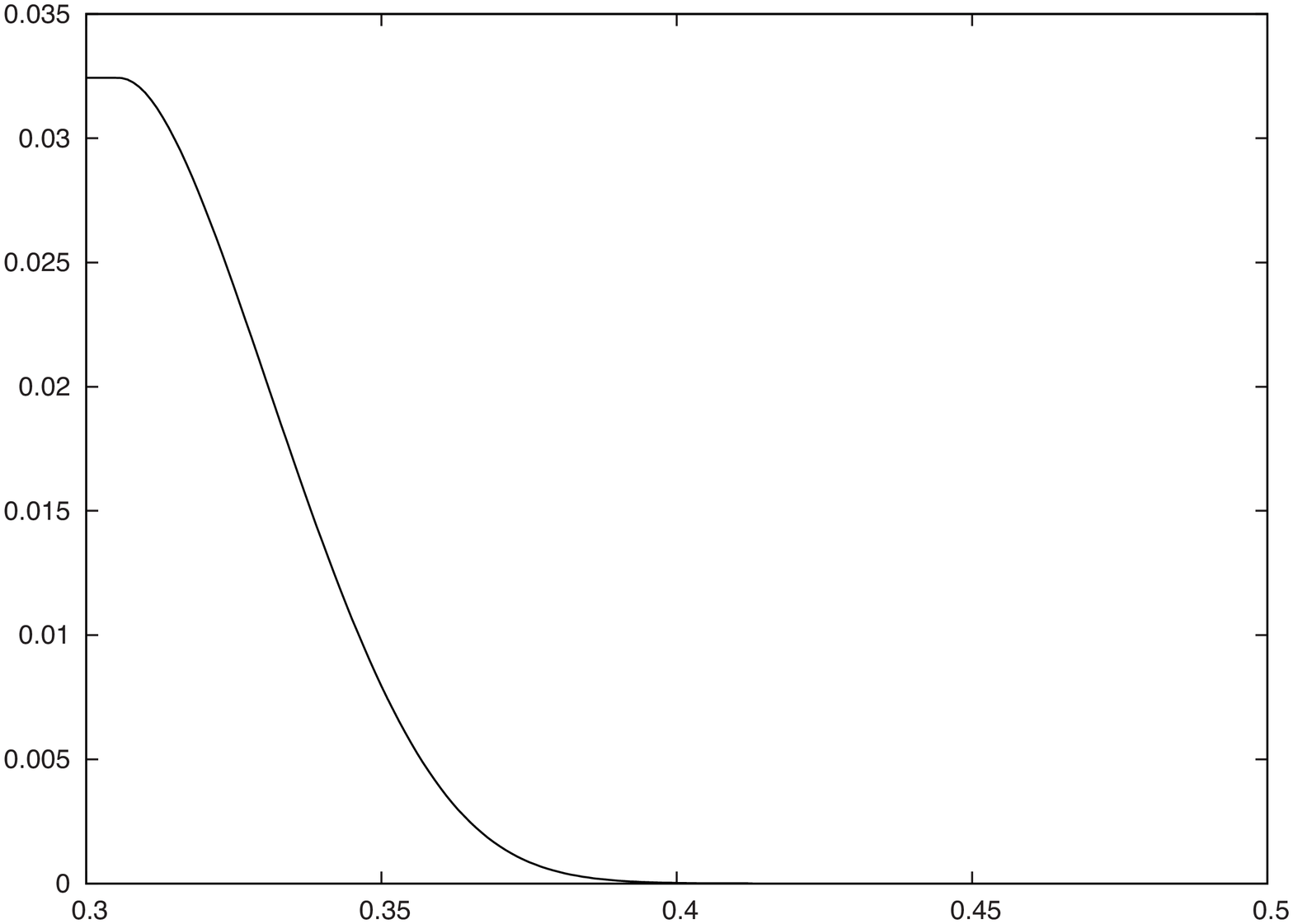}} \\
		(g)\,\,\,\,$\max_{x}\abs{\eta(x,t)-\eta_{0}}$ and $\max_{x}\abs{u(x,t)-u_{0}}$ vs. time \vspace{12.5pt}\\
		\subfloat[]{\includegraphics[totalheight=1.825in,width=2.7in]{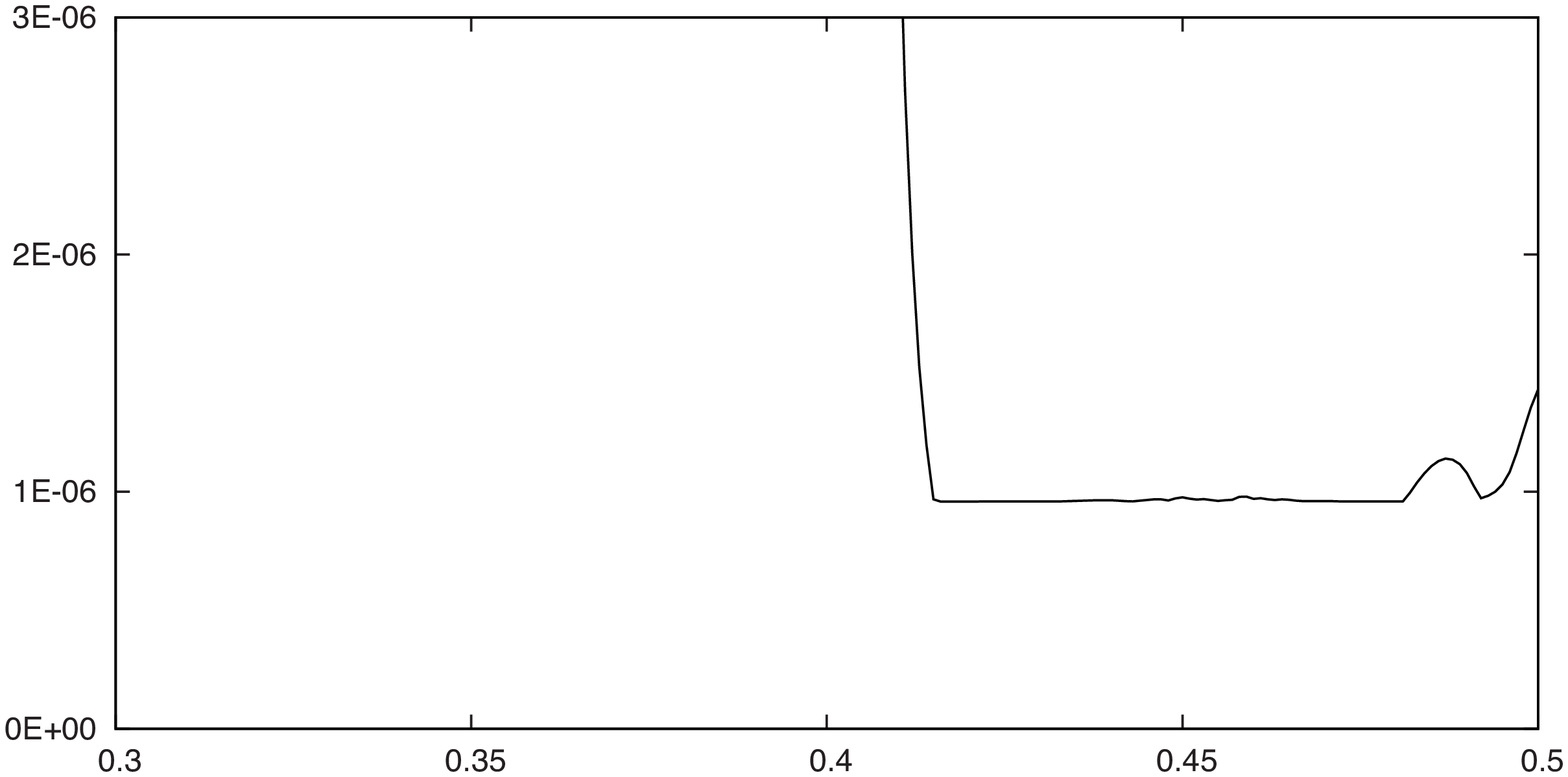}}  \quad \,
		\subfloat[]{\includegraphics[totalheight=1.825in,width=2.7in]{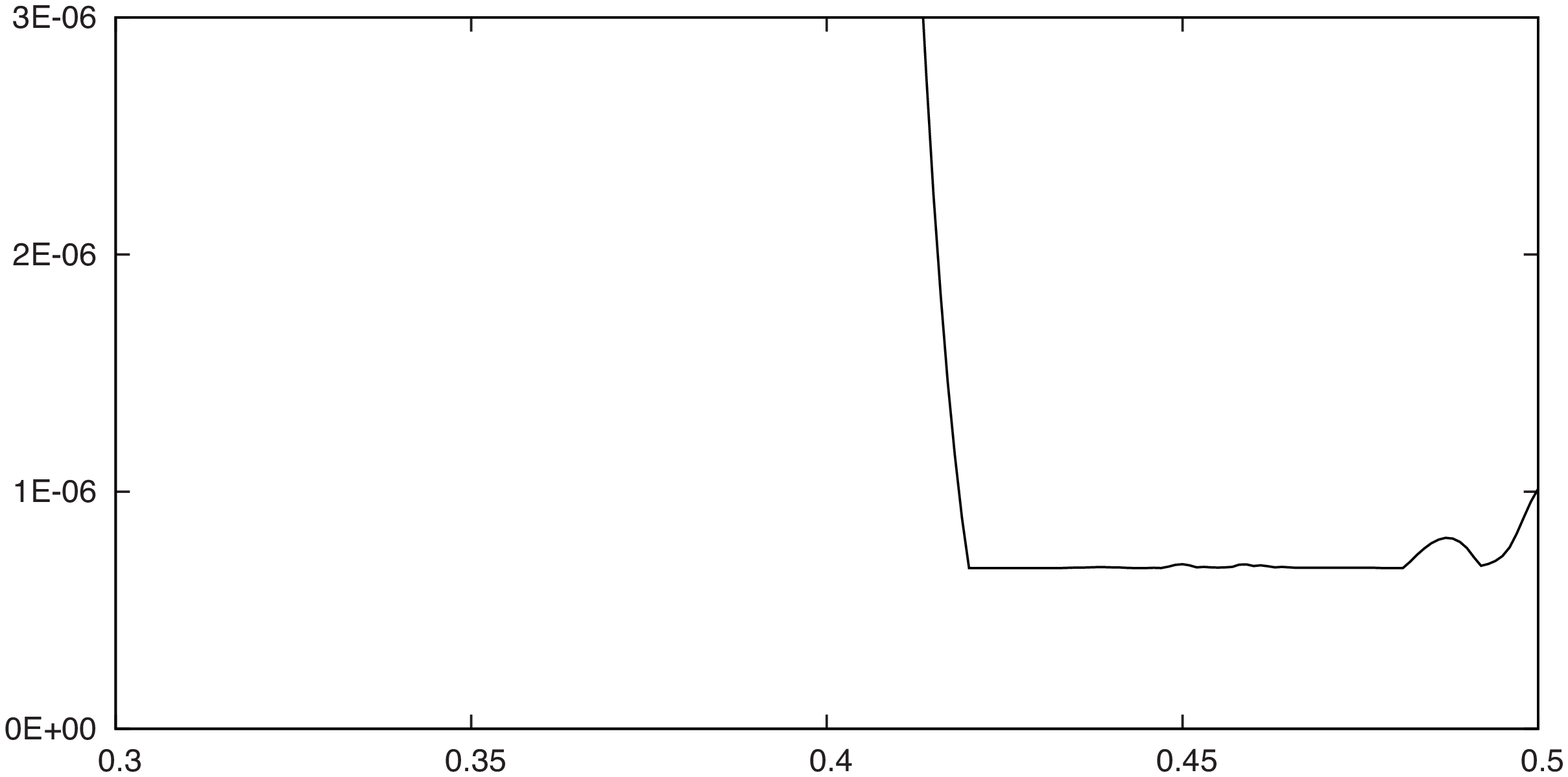}}\\
		(h)\,\,\,\,Magnification of (g)	\vspace{12.5pt} \\
		\subfloat[]{\includegraphics[totalheight=1.825in,width=2.7in]{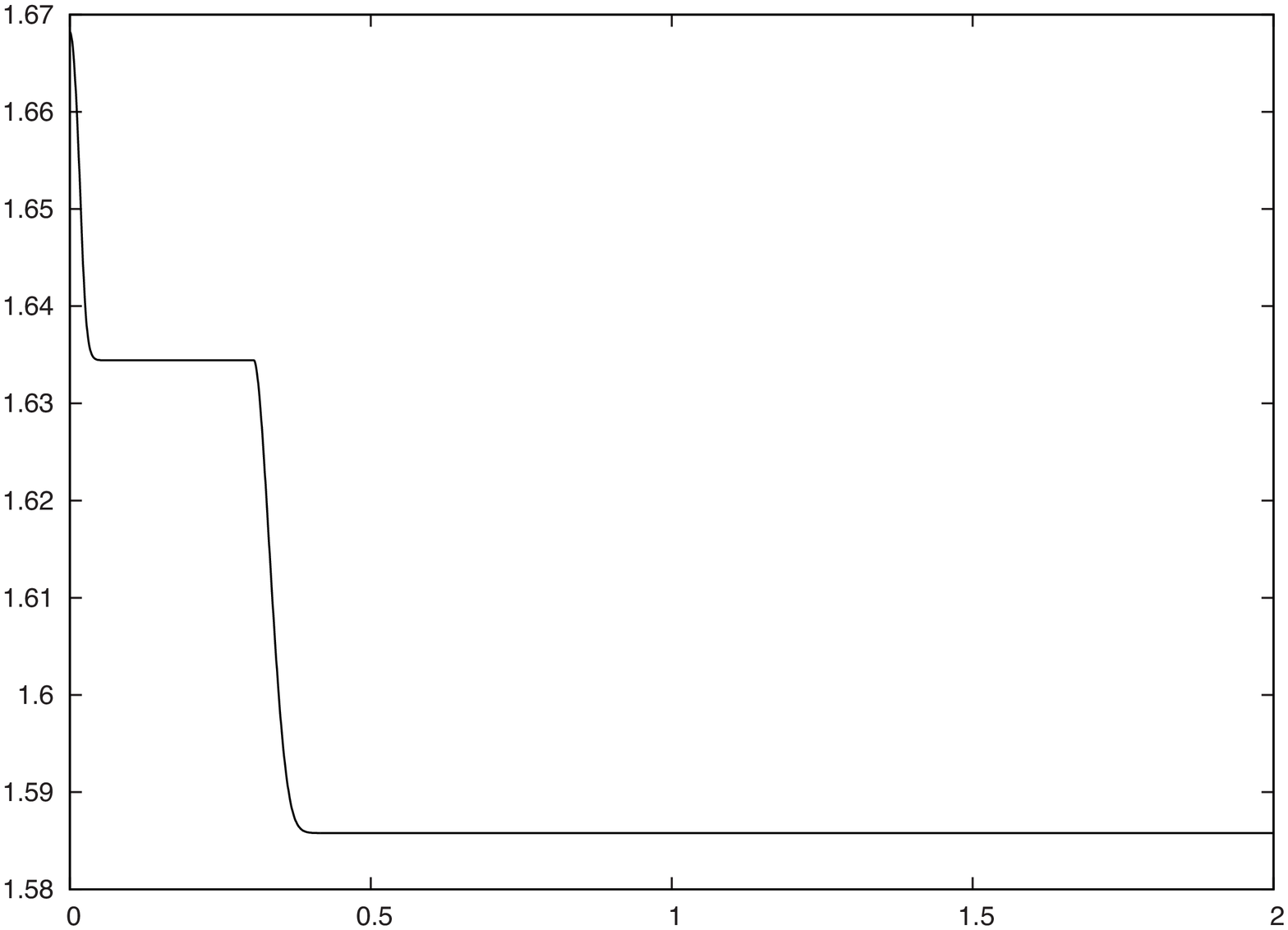}}\\
		(i)\,\,\,\,$\max_{x}(u-\sqrt{1+\eta})$  vs. time
	\end{center}
	\caption{Evolution of Gaussian initial profiles, supercritical case.}
	\label{fig41}
\end{figure}		
\noindent
However, they are highly absorbing; Figure \ref{fig41}(h) reveals  that the  
residue after the waves exit is of $O(10^{-6})$. The positivity of  $\max_{x}(u - \sqrt{1 + \eta})$ for all $t$  
checked in Fig.\ref{fig41}(i)  confirms that the numerical solution has
remained supercritical throughout the evolution. \par
In order to study numerically the stability of the fully discrete scheme (\ref{eq41})-(\ref{eq42}) for this
test problem, as there are no exact solutions, we took as a measure of error the residual quantity
$\max_{x}|\eta - \eta_{0}|$. This is plotted in Figure \ref{fig41}(g) and (h), stabilizes after the waves
exit the computational  domain, and has the value $9.76E-07$ at $t=0.45$. We then increased
$k$ and observed that this residual was conserved up to about $k/h=0.3695$ and started increasing
afterwards. Since the maximum wave speed $c$ is the speed of the higher rightward-travelling
pulse, which is equal to $u+\sqrt{1+\eta}\simeq 4.5$ for the duration of this experiment, we obtain
a Courant number restriction of about $ck/h\leq 1.67$. (Linear stability theory for this method applied
to the model problem $\eta_{t} + c\eta_{x}=0$ with periodic boundary conditions and $c$ constant
would give a Courant number restriction $ck/h \leq \sqrt{8/3}\simeq 1.633$ which is not far from the 
experimental result for this small-amplitude nonlinear propagation problem.) It should be noted that
the scheme (\ref{eq27})-(\ref{eq29}) discretized in time with the Runge-Kutta method (\ref{eq42})
yields practically the same numerical results for this test problem.
\subsection{Subcritical case} 
We implement the standard Galerkin method for the SW with characteristic boundary conditions in the
subcritical case using again piecewise linear continuous functions on a uniform mesh in $[0,1]$. In addition to
the notation introduced in the previous subsection for the mesh on $[0,1]$ and the space $S_{h}$, we let $S_{h,0}$
consist of the functions in $S_{h}$ that vanish at $x=0$ and $x=1$. We first consider the `direct' standard 
Galerkin semidiscretization of (\ref{eqsw2}), i.e. without reducing first the ibvp into the diagonal form (\ref{eqsw2a}).
We seek accordingly $\eta_{h}$, $u_{h} : [0,T] \to S_{h}$, the semidiscrete approximation of the solution of (\ref{eqsw2}),
satisfying for $0\leq t\leq T$ 
\begin{align}
	(\eta_{ht},\phi) & + (u_{hx},\phi) + ((\eta_{h}u_{h})_{x},\phi) = 0, \quad \forall \phi\in S_{h},
	\label{eq43} \\
	(\wt{u}_{ht},\chi) & + (\eta_{hx},\chi) + (\wt{u}_{h}\wt{u}_{hx},\chi) = 0, \quad \forall \chi\in S_{h,0},
	\label{eq44}
\end{align}
where $\wt{u}_{h}\in S_{h,0}$, and $u_{h}(x_{i},t) = \wt{u}_{h}(x_{i},t)$, $1\leq i\leq N-1$,
\begin{align}
	u_{h}(x_{0},t) & = -2\sqrt{1 + \eta_{h}(x_{0},t)} + u_{0} + 2\sqrt{1 + \eta_{0}}, 
	\label{eq45}\\
	u_{h}(x_{N},t) & = 2 \sqrt{1 + \eta_{h}(x_{N},t)} + u_{0} - 2\sqrt{1 + \eta_{0}}.
	\label{eq46}
\end{align}	 	
At $t=0$ we compute $u_{h}(0)$ and $\eta_{h}(0)$ as the $L^{2}$ projections of the  initial data $\eta^{0}$ , $u^{0}$ 
onto $S_{h}$. Although the semidiscrete ivp (\ref{eq43})-(\ref{eq46}) has nonlinear boundary conditions, it is easily 
discretized in time by an explicit scheme such as the 4$^{th}$-order RK method (\ref{eq42}), by first advancing
from $t^{n}$ to $t^{n+1}$ the approximations of $\eta_{h}$ and of the `interior' $\wt{u}_{h}$ using the temporal
discretizations of (\ref{eq43}) and (\ref{eq44}), and then updating the $u_{h}(x_{0},t)$ and $u_{h}(x_{N},t)$ values
at $t=t^{n+1}$ by (\ref{eq45}) and (\ref{eq46}). \par 
The convergence of this scheme was not analyzed in  section 3. The following experiment suggests that in the 
case of uniform mesh its $L^{2}$ errors are of $O(h^{2})$. 
We consider (\ref{eqsw2}) with $\eta_{0}=1$, $u_{0}=1$ 
\begin{table}[h!]
\setlength\tabcolsep{10.7pt}
\begin{tabular}{@{}ccccc@{}}
\tblhead{$N$   &  $\eta$   &  $order$    &      $u$          &  $order$}  
$40$  &  $4.847892(-3)$    &    --       &  $2.932354(-3)$   &    --        \\   
$80$  &  $1.207564(-3)$    &  $2.00526$  &  $7.414336(-4)$   &  $1.98367$   \\ 
$160$ &  $3.017313(-4)$    &  $2.00076$  &  $1.860285(-4)$   &  $1.99479$   \\ 
$320$ &  $7.544641(-5)$    &  $1.99974$  &  $4.657627(-5)$   &  $1.99786$   \\ 
$480$ &  $3.353298(-5)$    &  $1.99991$  &  $2.071174(-5)$   &  $1.99867$   \\ 
$520$ &  $2.857355(-5)$    &  $1.99953$  &  $1.764866(-5)$   &  $1.99944$   \\ \hline
\end{tabular}\vspace{3.1pt}
\small
\caption{$L^{2}$ errors and spatial orders of convergence, subcritical case,
		semidiscretization (\ref{eq43})-(\ref{eq46}).}
	\label{tbl42}
\end{table}
\normalsize
and its semidiscretization (\ref{eq43})-(\ref{eq46}) with piecewise linear continuous functions. We add appropriate 
right-hand sides to the pde's in (\ref{eqsw2}) so that the exact solution of the ibvp is $\eta(x,t)=(x+1)\mathrm{e}^{-xt}$, 
$u(x,t) = (2x + \cos (\pi x) - 1)\mathrm{e}^{t} + xA(t) + (1-x)B(t)$, where
$A(t) = 2\sqrt{1 + \eta(1,t)} + u_{0} - 2\sqrt{1 + \eta_{0}}$, $B(t) = -2\sqrt{1 + \eta(0,t)} + u_{0} + 2\sqrt{1 + \eta_{0}}$.
We consider uniform spatial and temporal  meshes with $h=1/N$, $k=h/10$, and discretize the semidiscrete problem
in time using again the 4$^{th}$-order `classical' RK scheme. (We checked that the temporal error is negligible for 
the range of $N$'s that we tried.)  The resulting $L^{2}$ errors and rates of convergence of (essentially) the
semidiscrete problem at $T=1$ are shown in Table \ref{tbl42}. The rates are practically equal to 2, i.e. superaccurate,
as in the supercritical case. An analogous temporal-order calculation to that appearing in
Table \ref{tbl42a} was not so robust and gave rates between $3.7$ and $3.9$; thus some sort  of
temporal order reduction cannot be ruled out for this scheme. The experimental Courant number 
restriction was $k/h\leq 0.53$. \par
In the following numerical experiment we integrate (\ref{eqsw2}) with the 
same fully discrete scheme, taking $h=1/N$, $N=2000$, $k=h/10$, $\eta_{0}=u_{0}=1$, and initial conditions 
$\eta^{0}(x) = 0.1\exp (-400(x-0.5)^{2}) + \eta_{0}$, $u^{0}(x) = 0.05\exp (-400(x-0.5)^{2}) + u_{0}$.
The evolution of the numerical solution is shown in Figure \ref{fig42} (a)-(j).%
\captionsetup[subfloat]{labelformat=empty,position=bottom,singlelinecheck=false}
\begin{figure}[ht!]
	\begin{center}
		\subfloat[]{\includegraphics[totalheight=2.7in,width=1.825in,angle=-90]{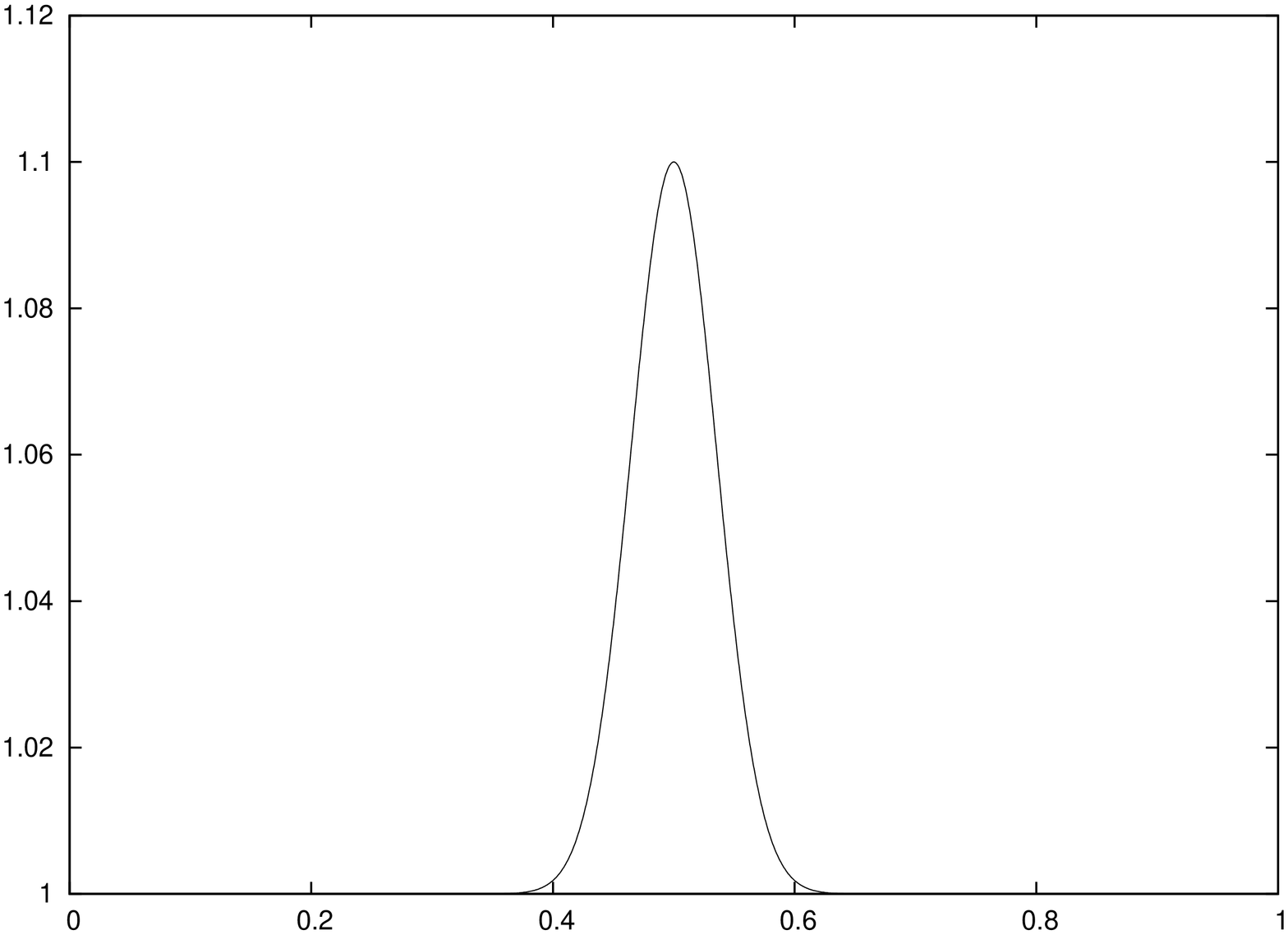}}  \quad \,
		\subfloat[]{\includegraphics[totalheight=2.7in,width=1.825in,angle=-90]{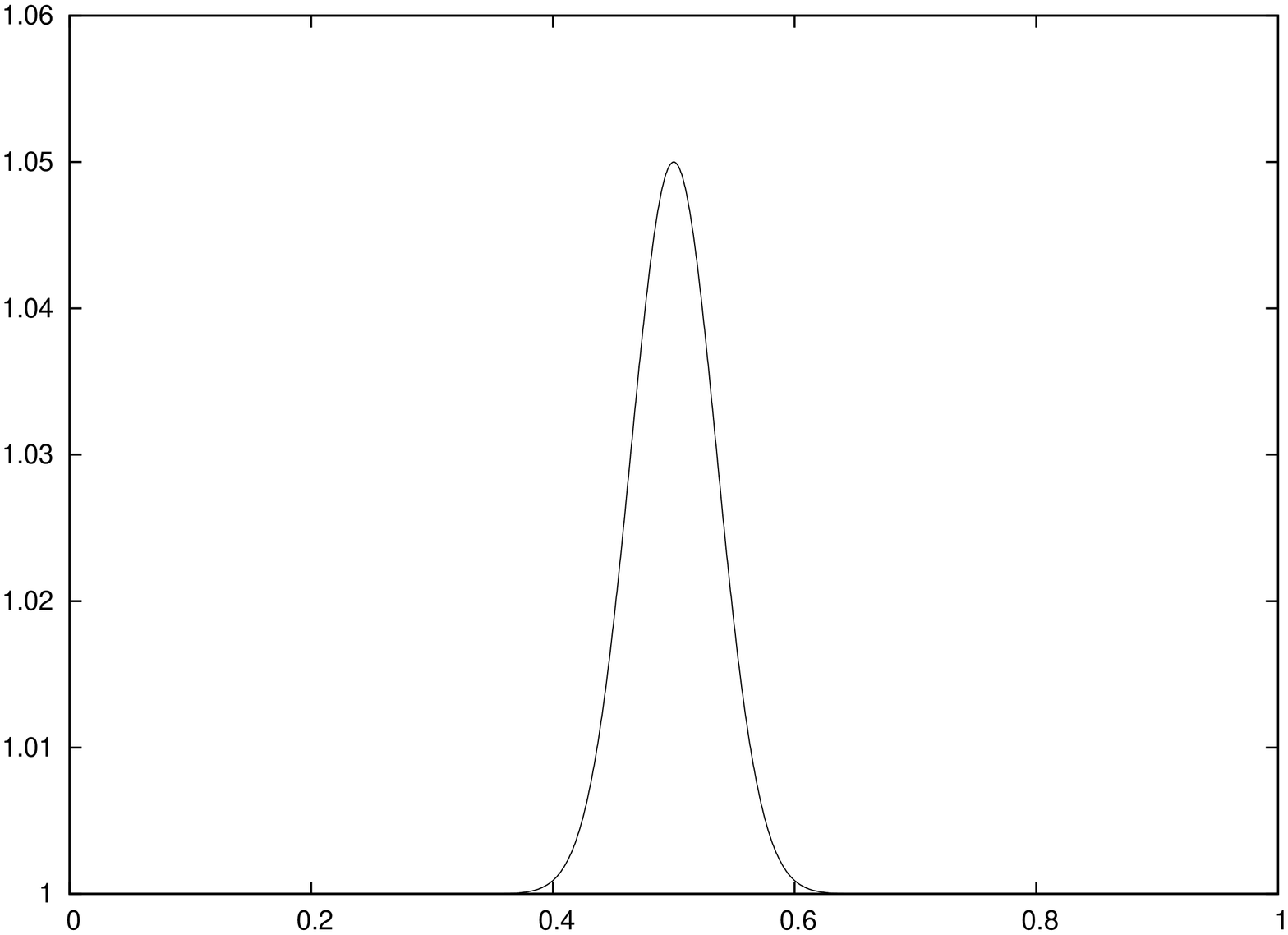}}\\
		(a)\,\,\,\,$\eta$ and $u$ at $t=0.0$  \vspace{3.1pt} \\
			\subfloat[]{\includegraphics[totalheight=2.7in,width=1.825in,angle=-90]{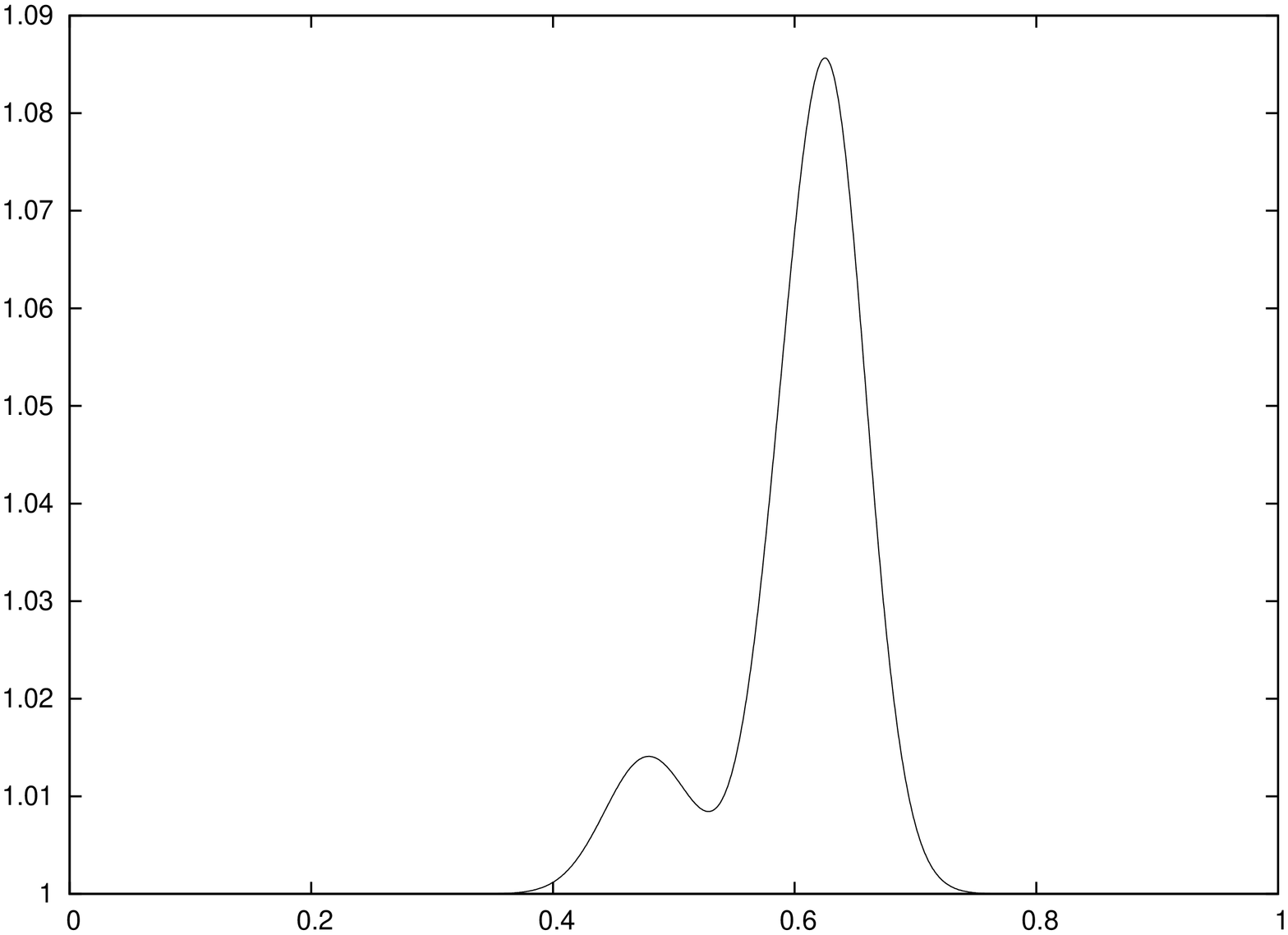}}  \quad \,
			\subfloat[]{\includegraphics[totalheight=2.7in,width=1.825in,angle=-90]{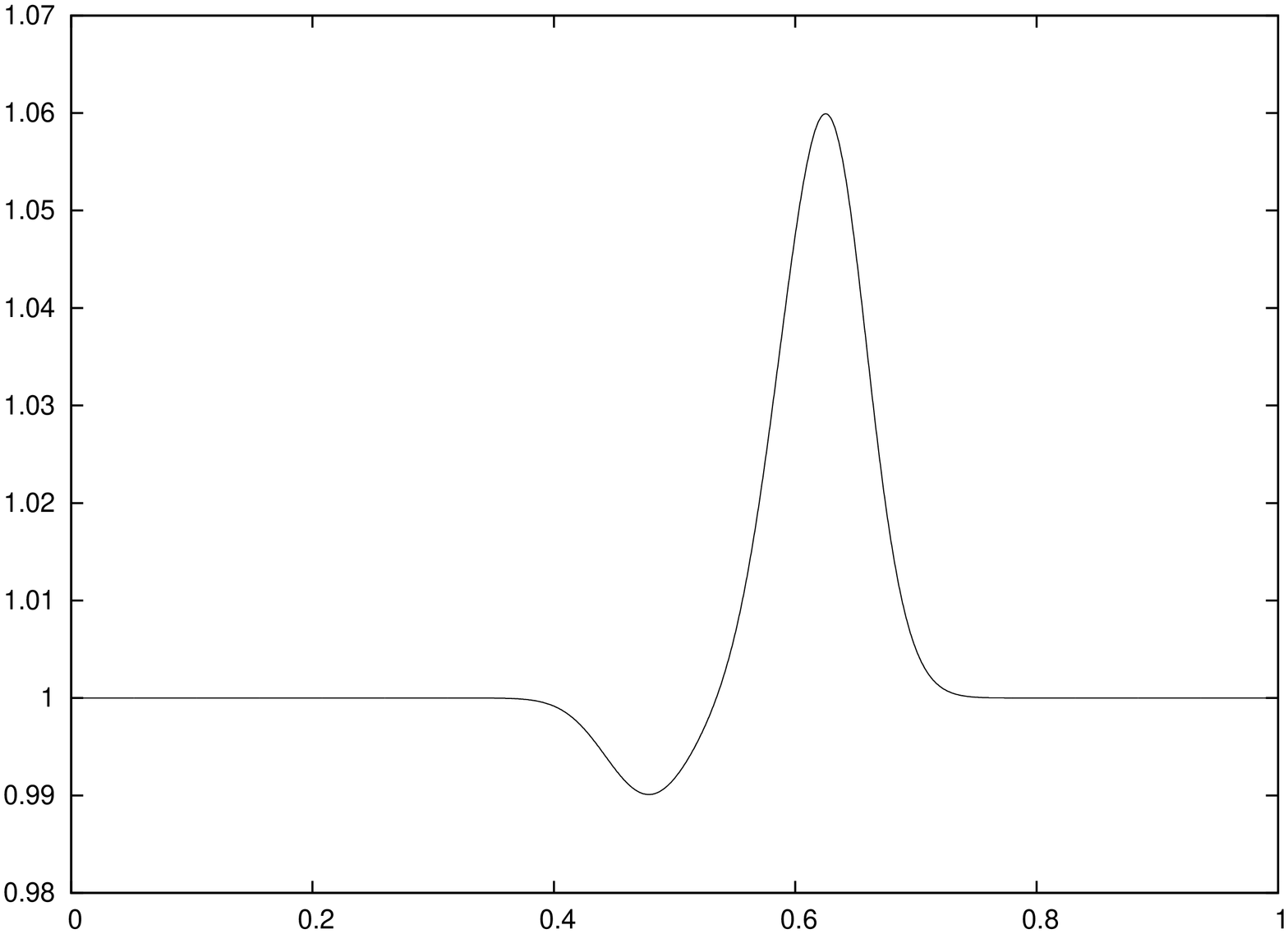}}\\
			(b)\,\,\,\,$\eta$ and $u$ at $t=0.05$ \vspace{3.1pt}\\
			\subfloat[]{\includegraphics[totalheight=2.7in,width=1.825in,angle=-90]{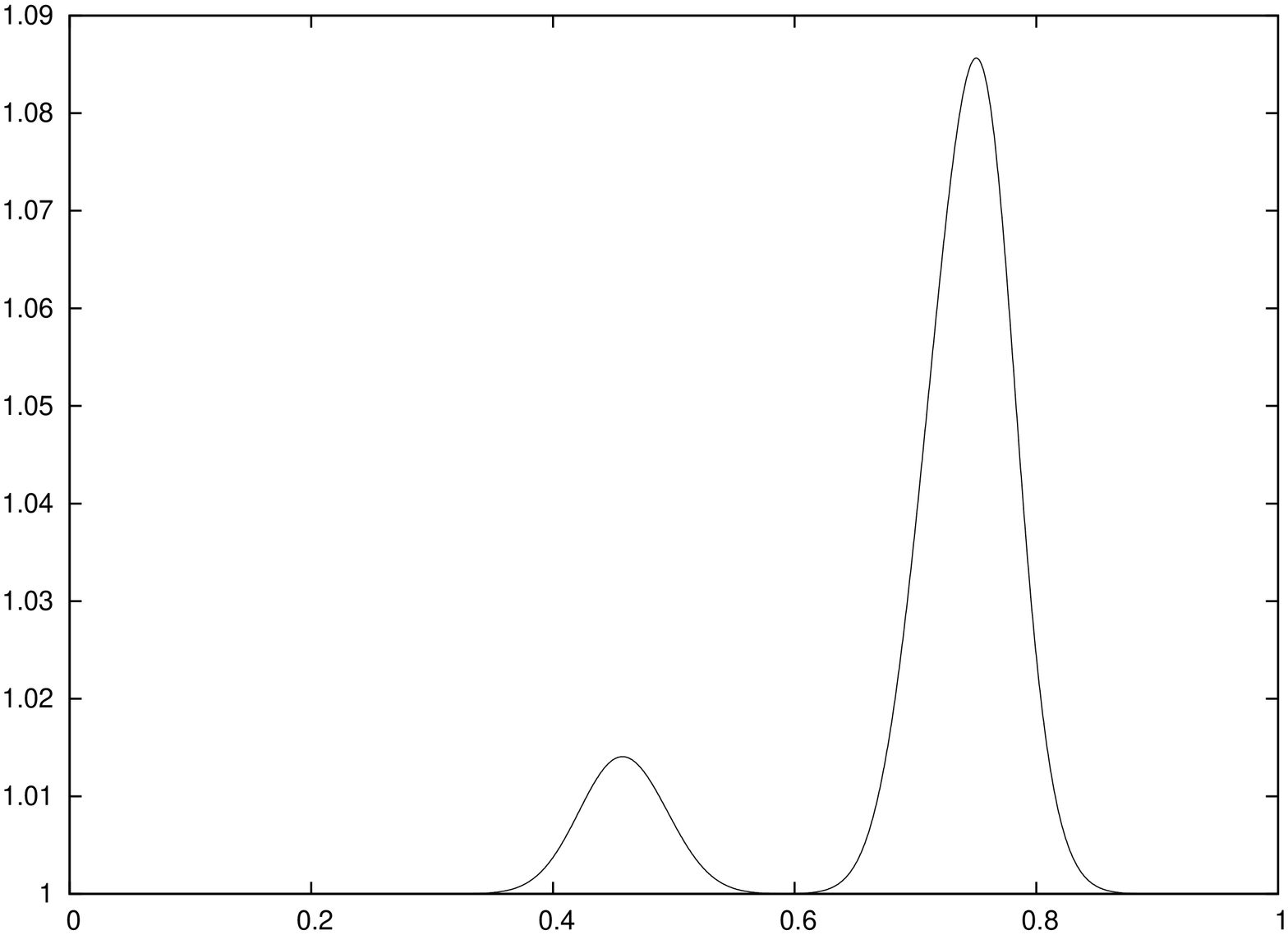}}  \quad \,
			\subfloat[]{\includegraphics[totalheight=2.7in,width=1.825in,angle=-90]{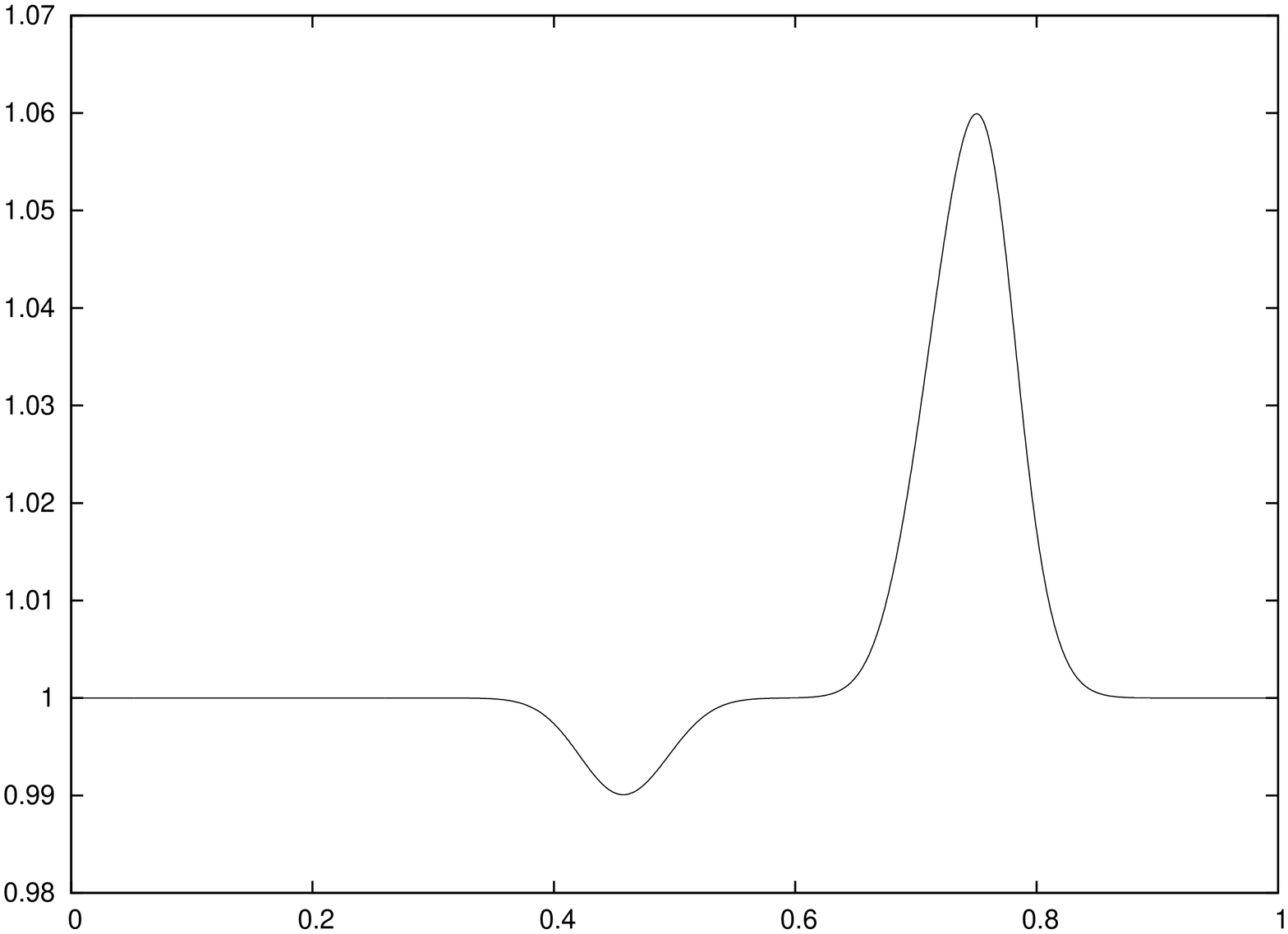}}\\
			(c)\,\,\,\,$\eta$ and $u$ at $t=0.1$ 
	\end{center}
	\end{figure}		
	\clearpage
\captionsetup[subfloat]{labelformat=empty,position=bottom,singlelinecheck=false}
\begin{figure}[t!]
	\begin{center}	
			\subfloat[]{\includegraphics[totalheight=2.7in,width=1.825in,angle=-90]{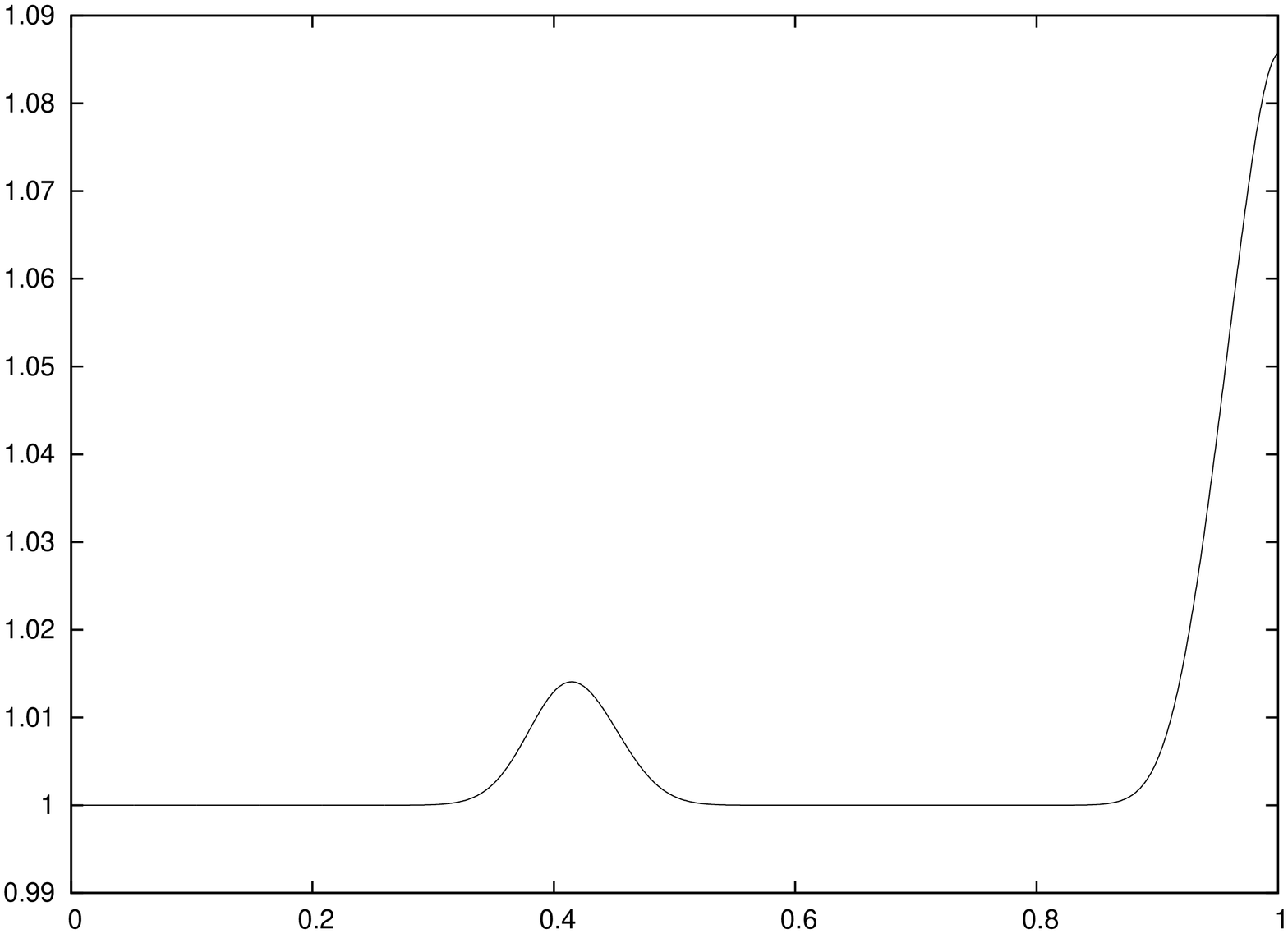}}  \quad \,
			\subfloat[]{\includegraphics[totalheight=2.7in,width=1.825in,angle=-90]{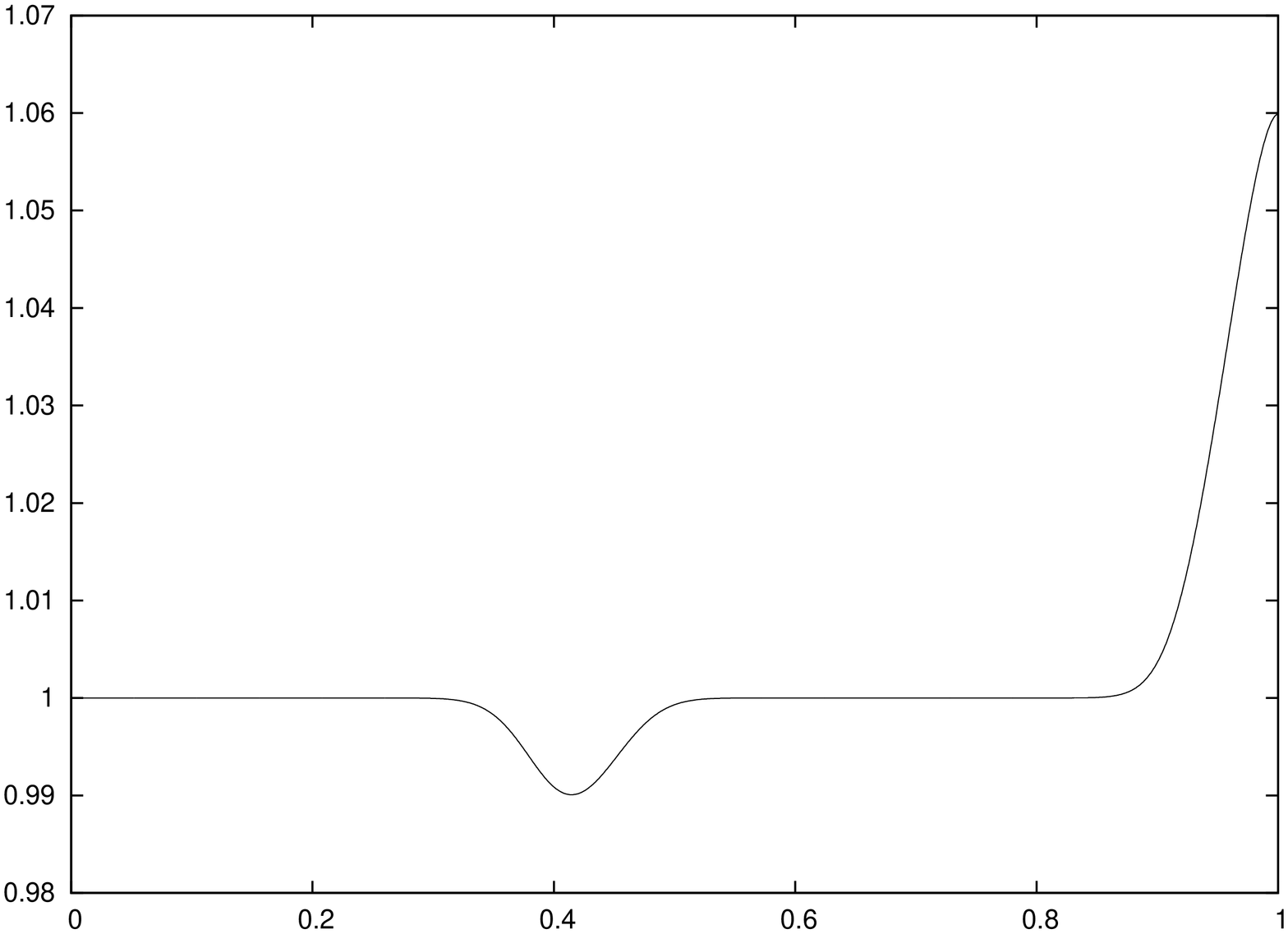}}\\
			(d)\,\,\,\,$\eta$ and $u$ at $t=0.2$\\
			\subfloat[]{\includegraphics[totalheight=2.7in,width=1.2in,angle=-90]{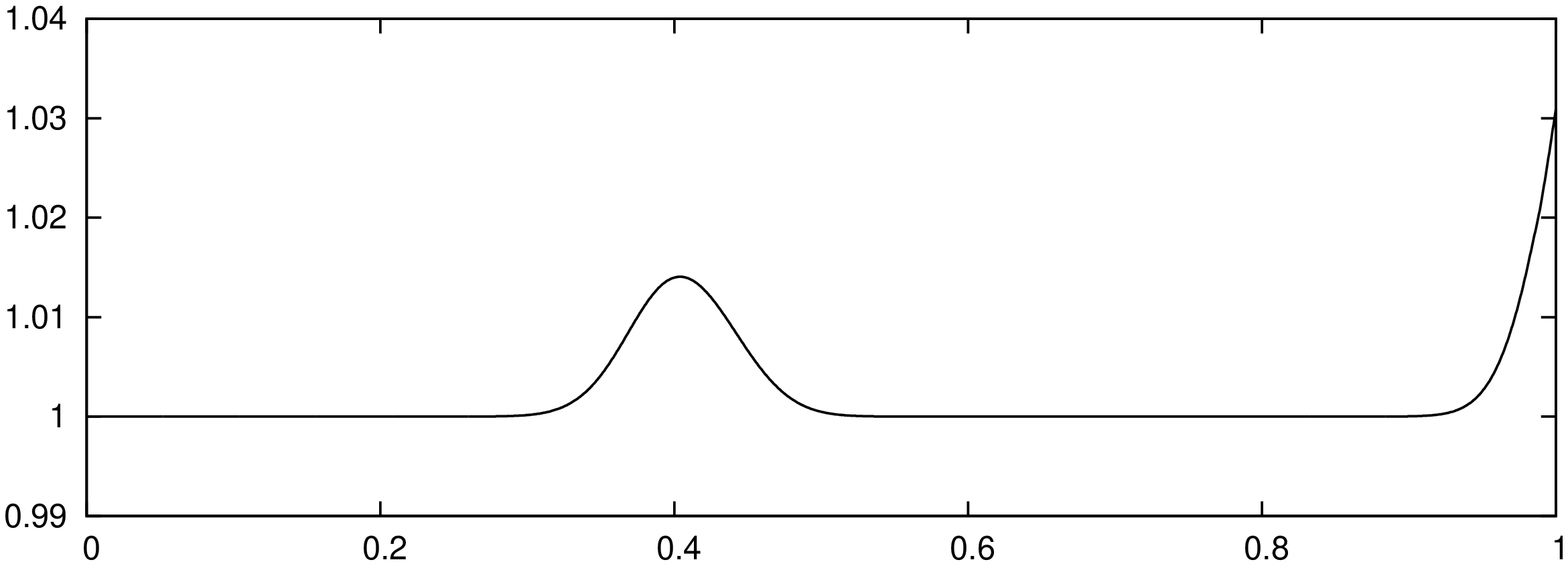}}  \quad \,
			\subfloat[]{\includegraphics[totalheight=2.7in,width=1.2in,angle=-90]{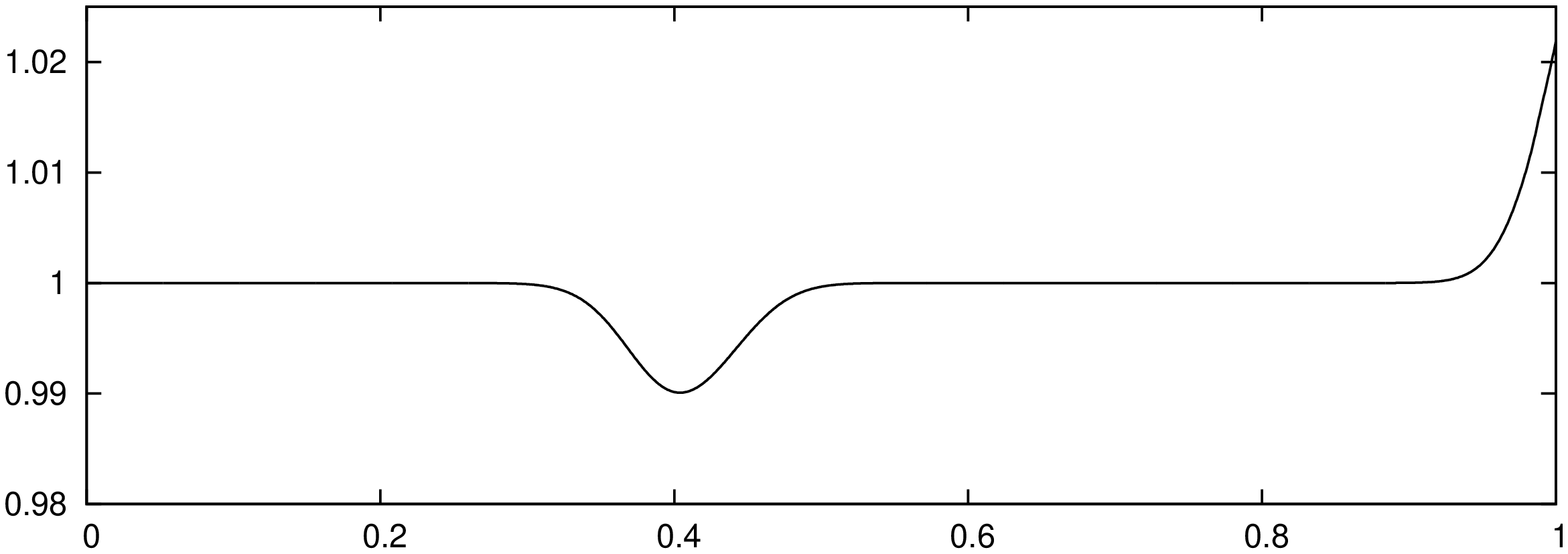}}\\
			(e)\,\,\,\,$\eta$ and $u$ at $t=0.225$ \\
			\subfloat[]{\includegraphics[totalheight=2.7in,width=1.2in,angle=-90]{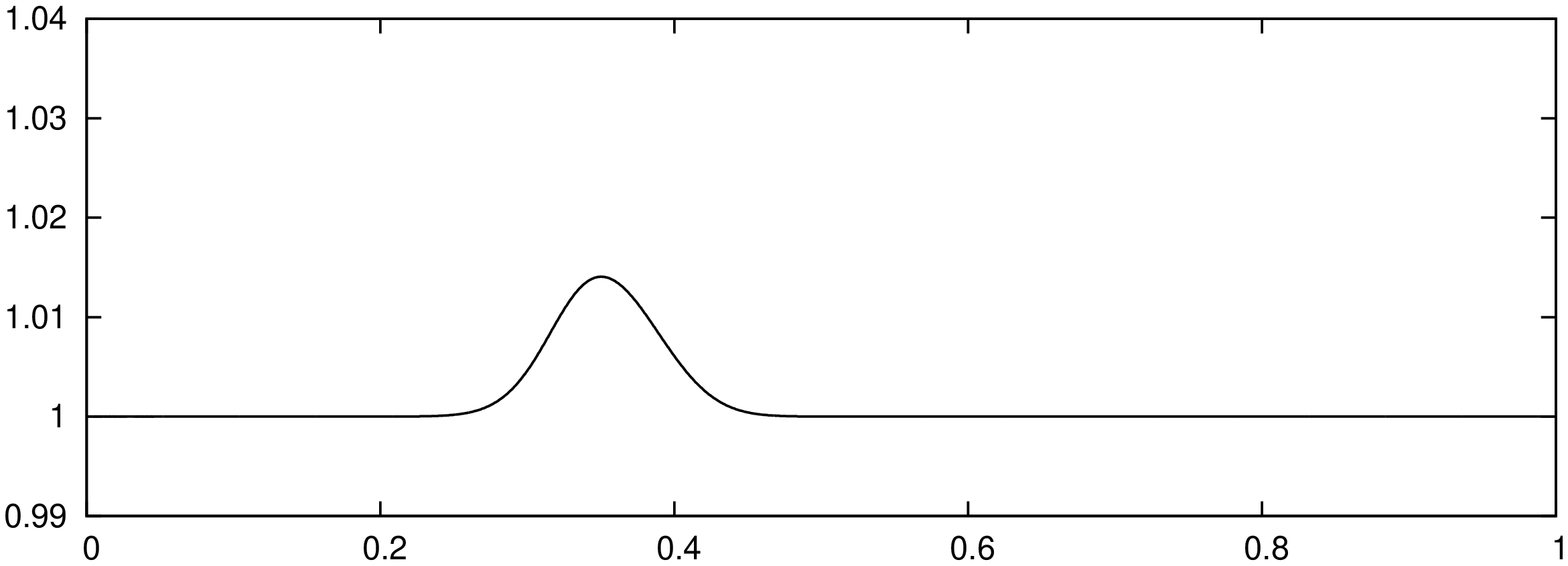}}  \quad \,
			\subfloat[]{\includegraphics[totalheight=2.7in,width=1.2in,angle=-90]{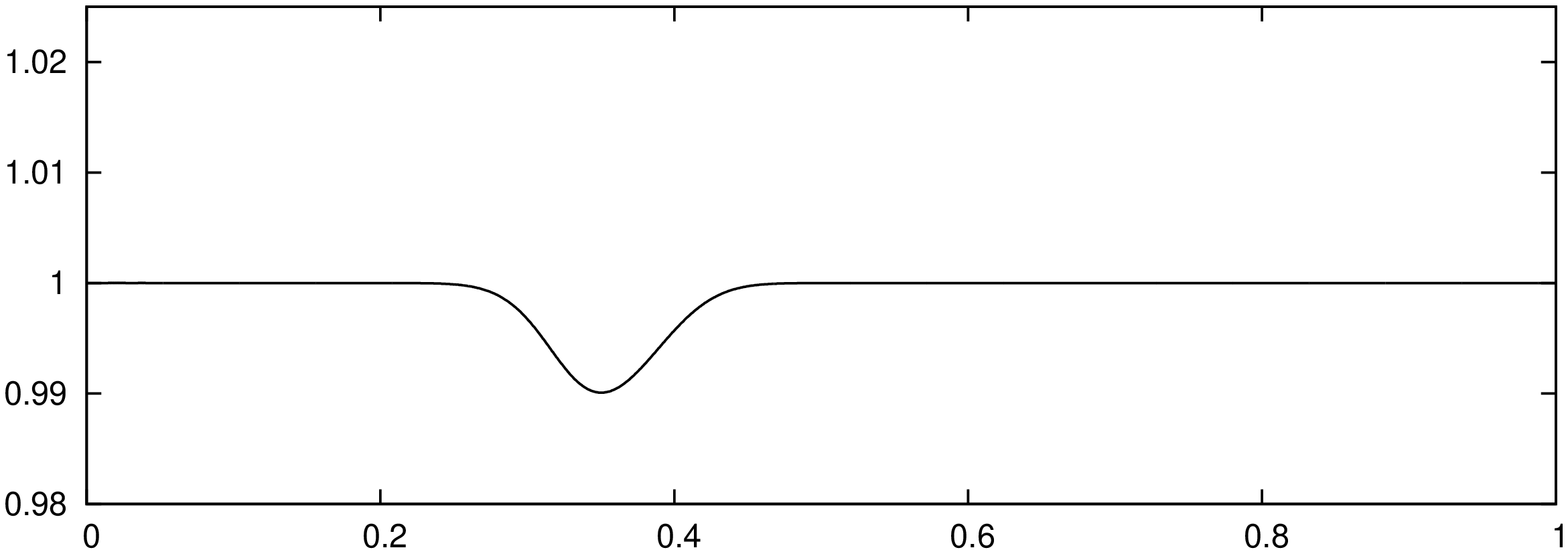}}\\
			(f)\,\,\,\,$\eta$ and $u$ at $t=0.35$	\\
			\subfloat[]{\includegraphics[totalheight=2.7in,width=1.2in,angle=-90]{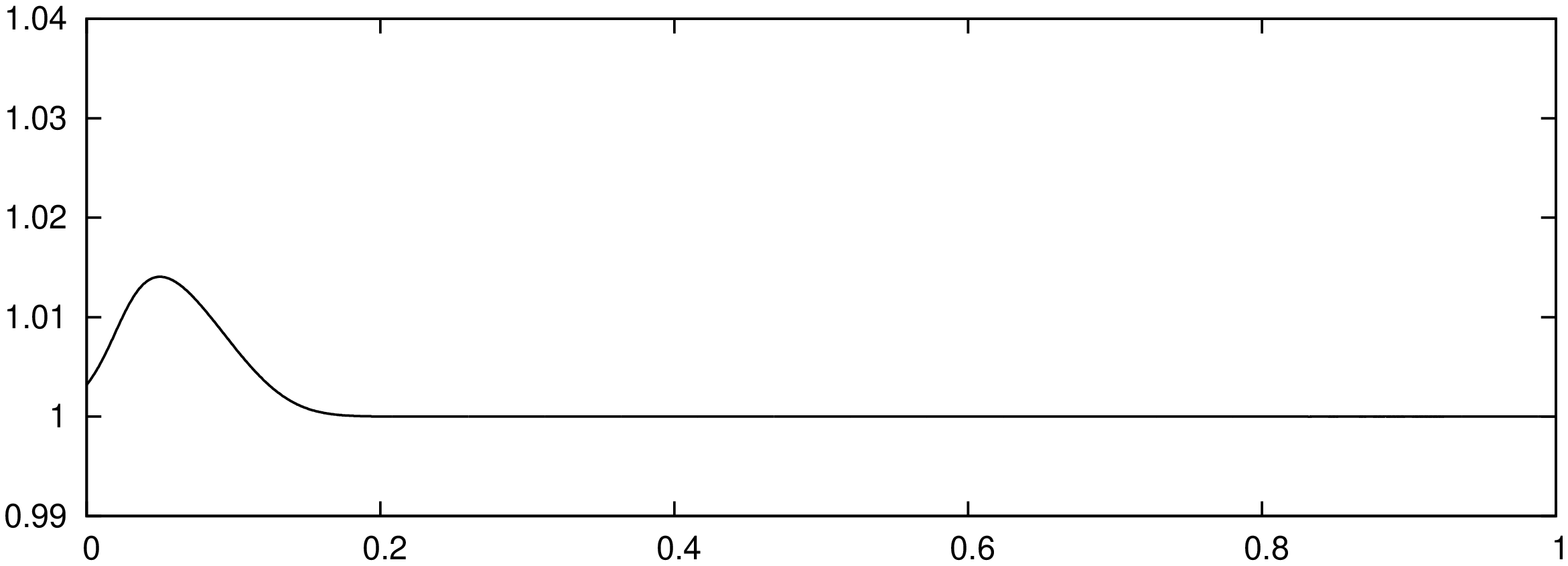}}  \quad \,
				\subfloat[]{\includegraphics[totalheight=2.7in,width=1.2in,angle=-90]{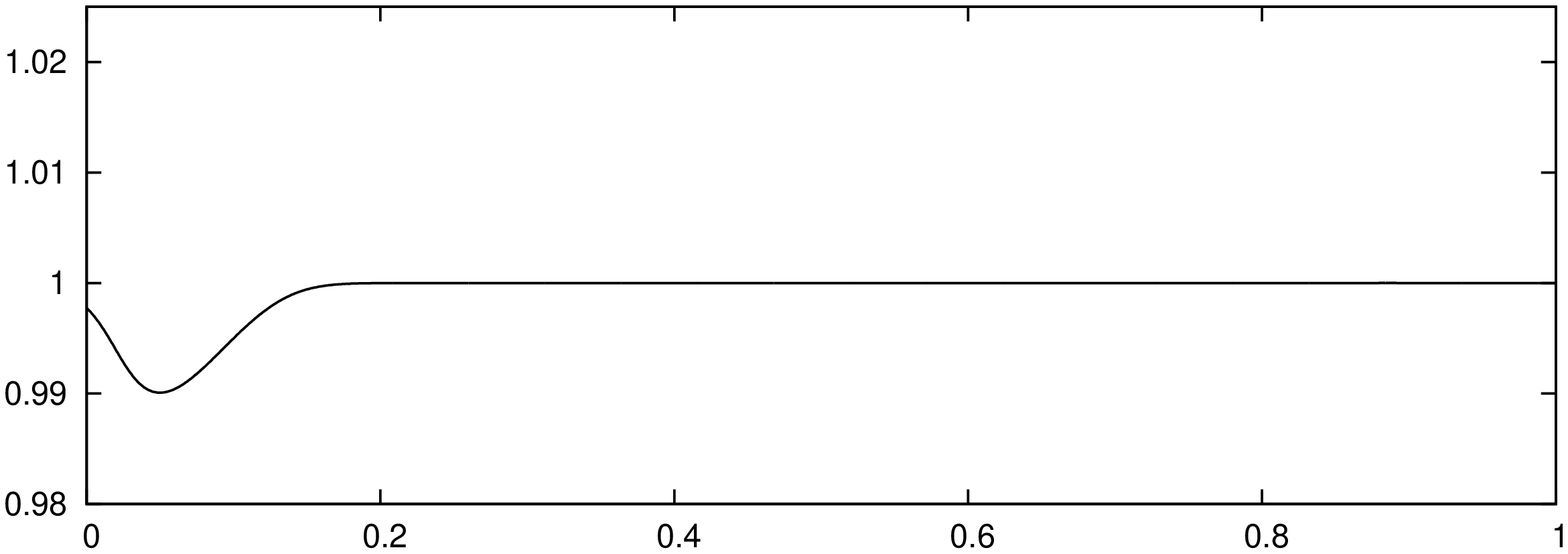}}\\
			(g)\,\,\,\,$\eta$ and $u$ at $t=1.05$ 
	\end{center}
	\end{figure}	
	\captionsetup[subfloat]{labelformat=empty,position=bottom,singlelinecheck=false}	
	\begin{figure}[t]
		\begin{center}
				\subfloat[]{\includegraphics[totalheight=2.7in,width=1.2in,angle=-90]{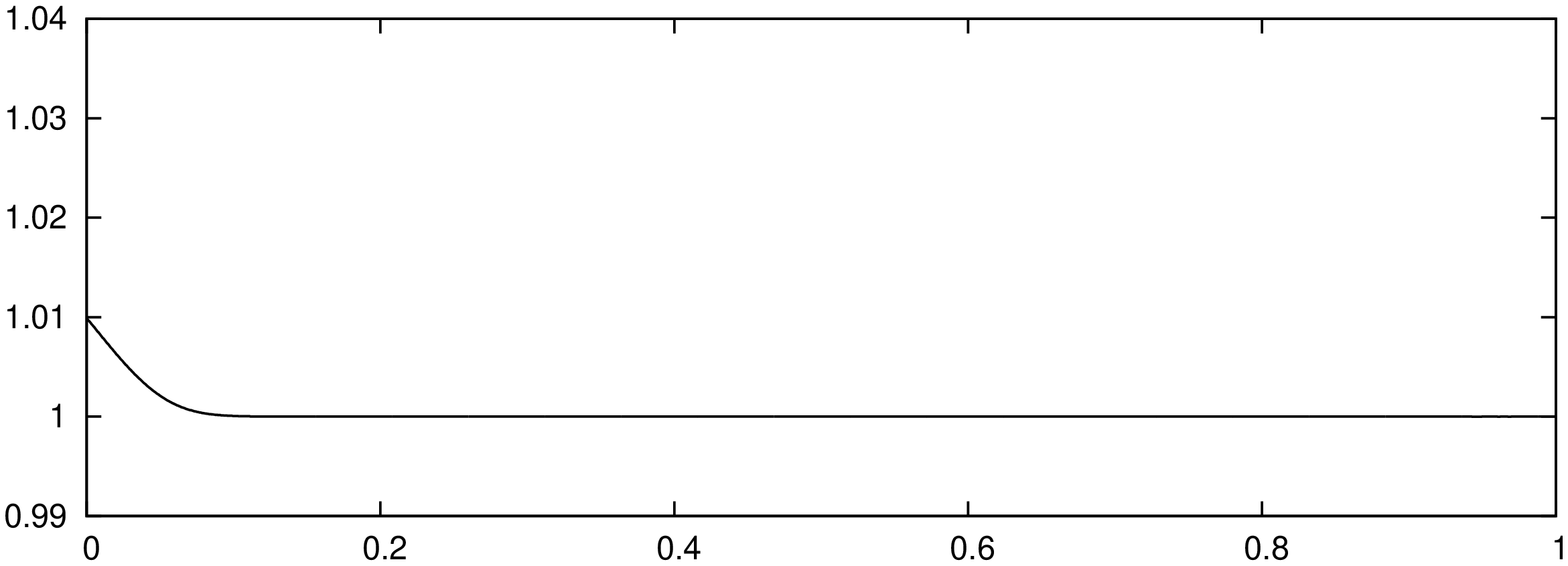}}  \quad \,
				\subfloat[]{\includegraphics[totalheight=2.7in,width=1.2in,angle=-90]{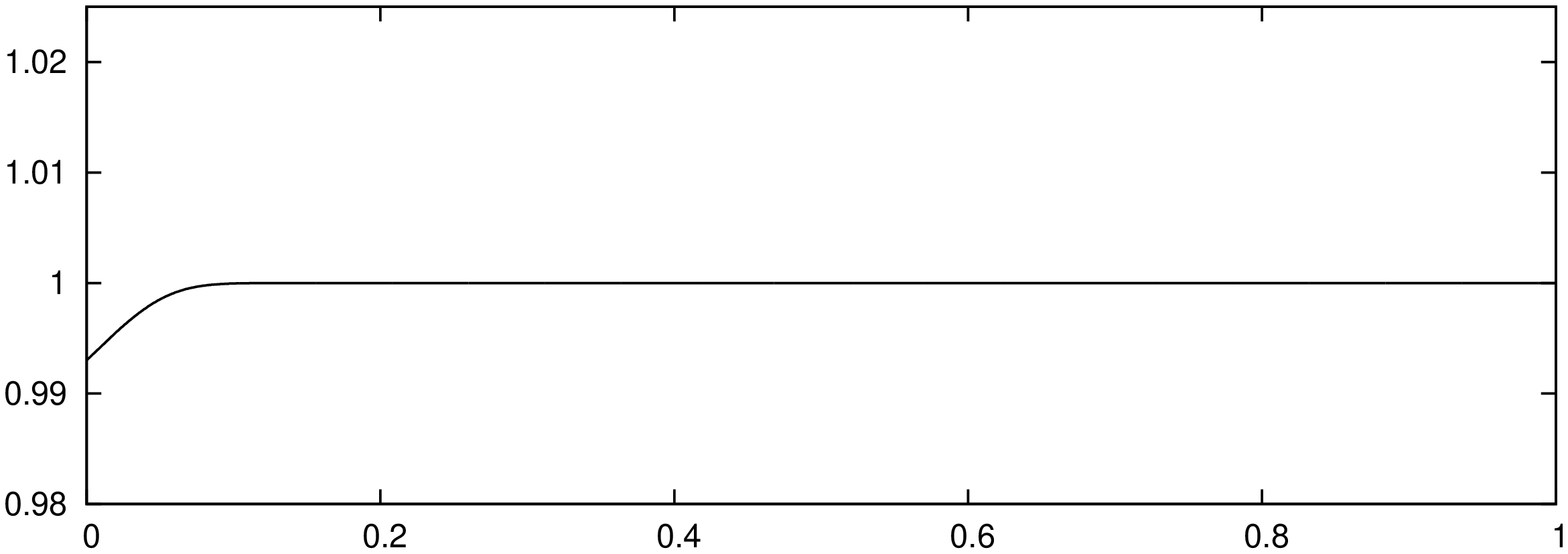}}\\
				(h)\,\,\,\,$\eta$ and $u$ at $t=1.25$ \\
				\subfloat[]{\includegraphics[totalheight=2.7in,width=1.2in,angle=-90]{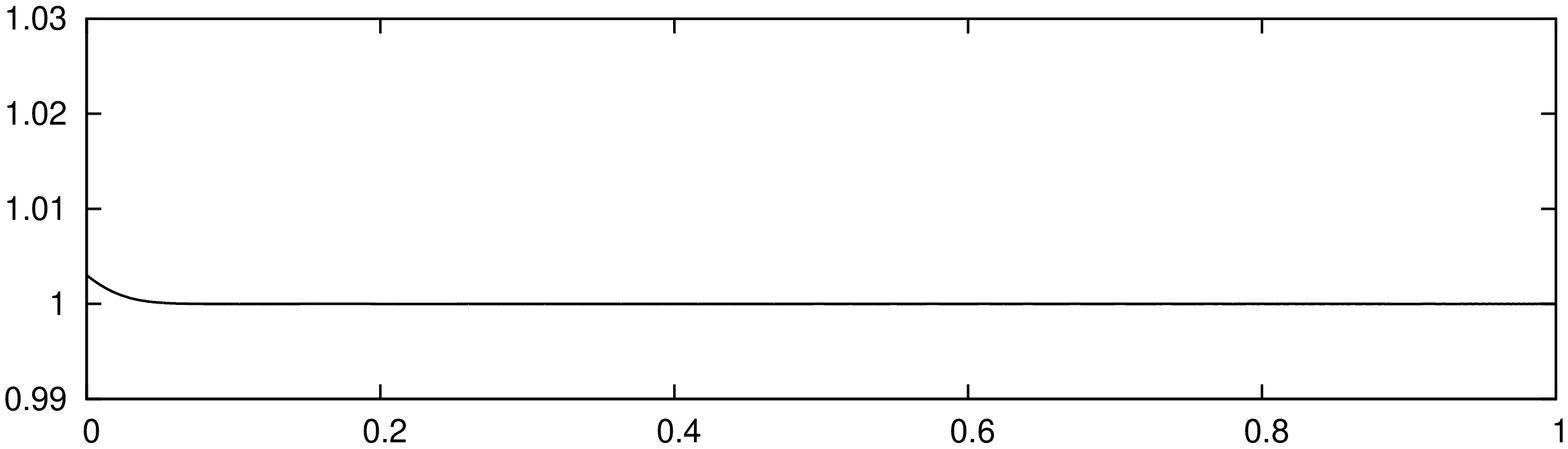}}  \quad \,
				\subfloat[]{\includegraphics[totalheight=2.7in,width=1.2in,angle=-90]{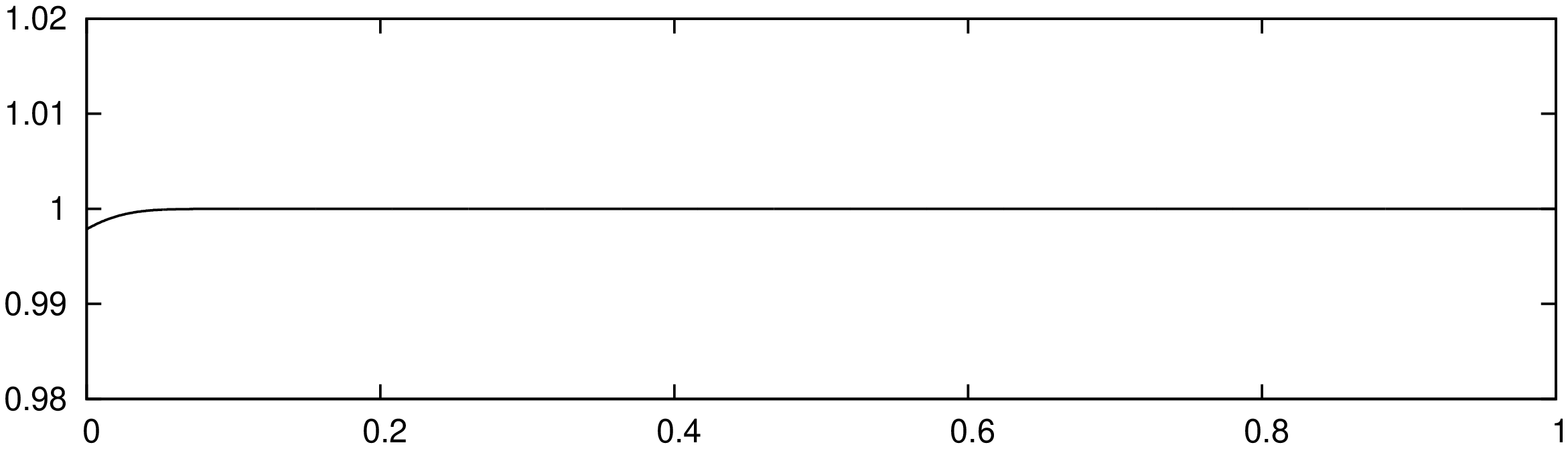}}\\
				(i)\,\,\,\,$\eta$ and $u$ at $t=1.35$ \\
				\subfloat[]{\includegraphics[totalheight=2.7in,width=1.2in,angle=-90]{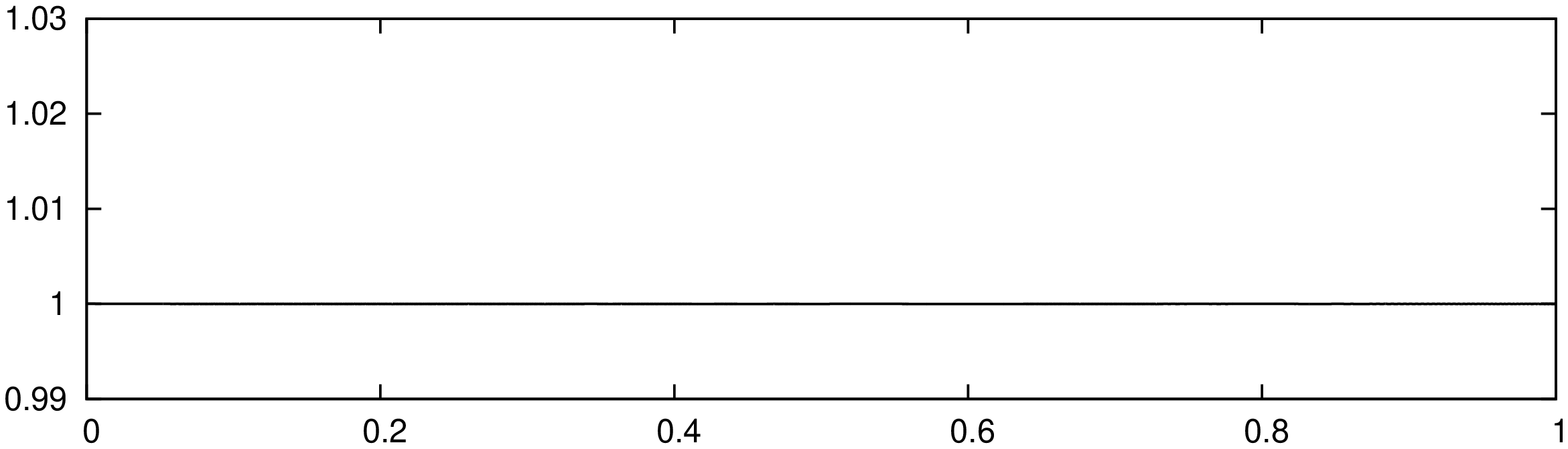}}  \quad \,
				\subfloat[]{\includegraphics[totalheight=2.7in,width=1.2in,angle=-90]{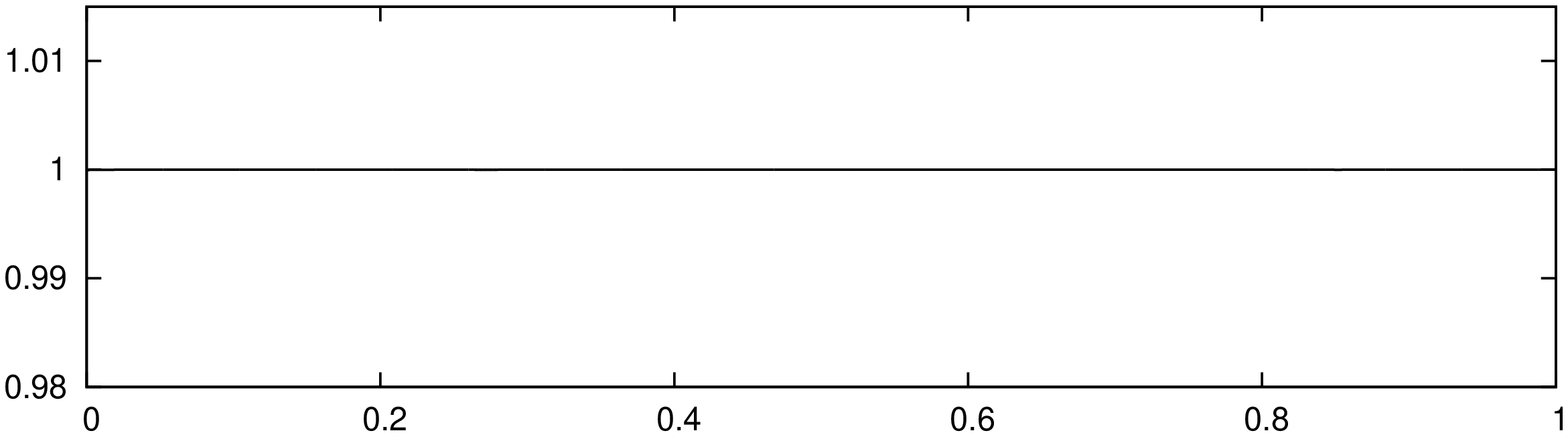}}\\
				(j)\,\,\,\,$\eta$ and $u$ at $t=1.5$\\
				\subfloat[]{\includegraphics[totalheight=2.7in,width=1.825in,angle=-90]{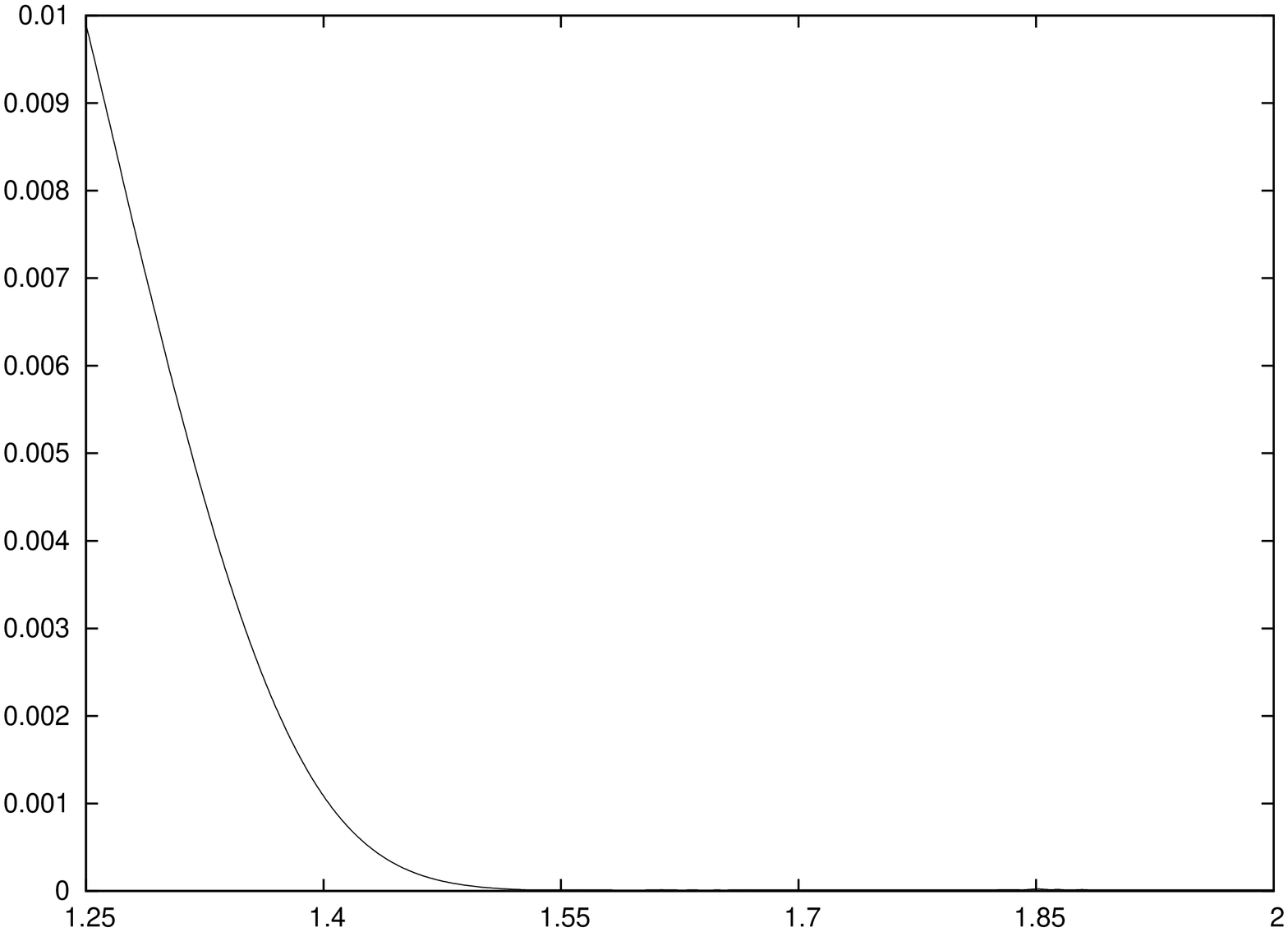}}  \quad \,
			\subfloat[]{\includegraphics[totalheight=2.7in,width=1.825in,angle=-90]{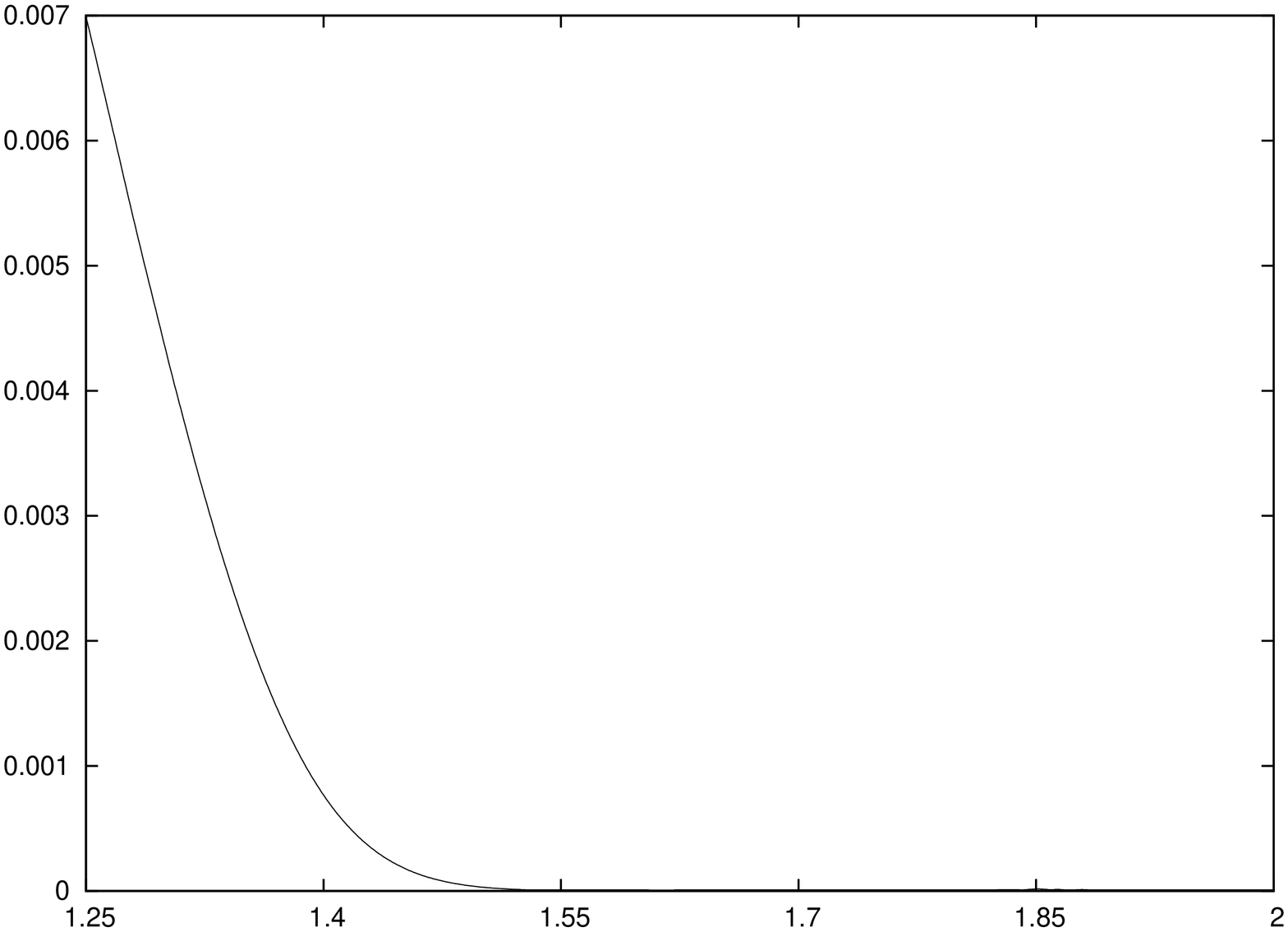}}\\
			(k)\,\,\,\,$\max_{x}\abs{\eta(x,t)-\eta_{0}}$ and $\max_{x}\abs{u(x,t)-u_{0}}$ vs. time  
		\end{center}
	\end{figure}
	\clearpage
	\captionsetup[subfloat]{labelformat=empty,position=bottom,singlelinecheck=false}
	\begin{figure}[t]
		\begin{center}	
			\subfloat[]{\includegraphics[totalheight=2.7in,width=1.825in,angle=-90]{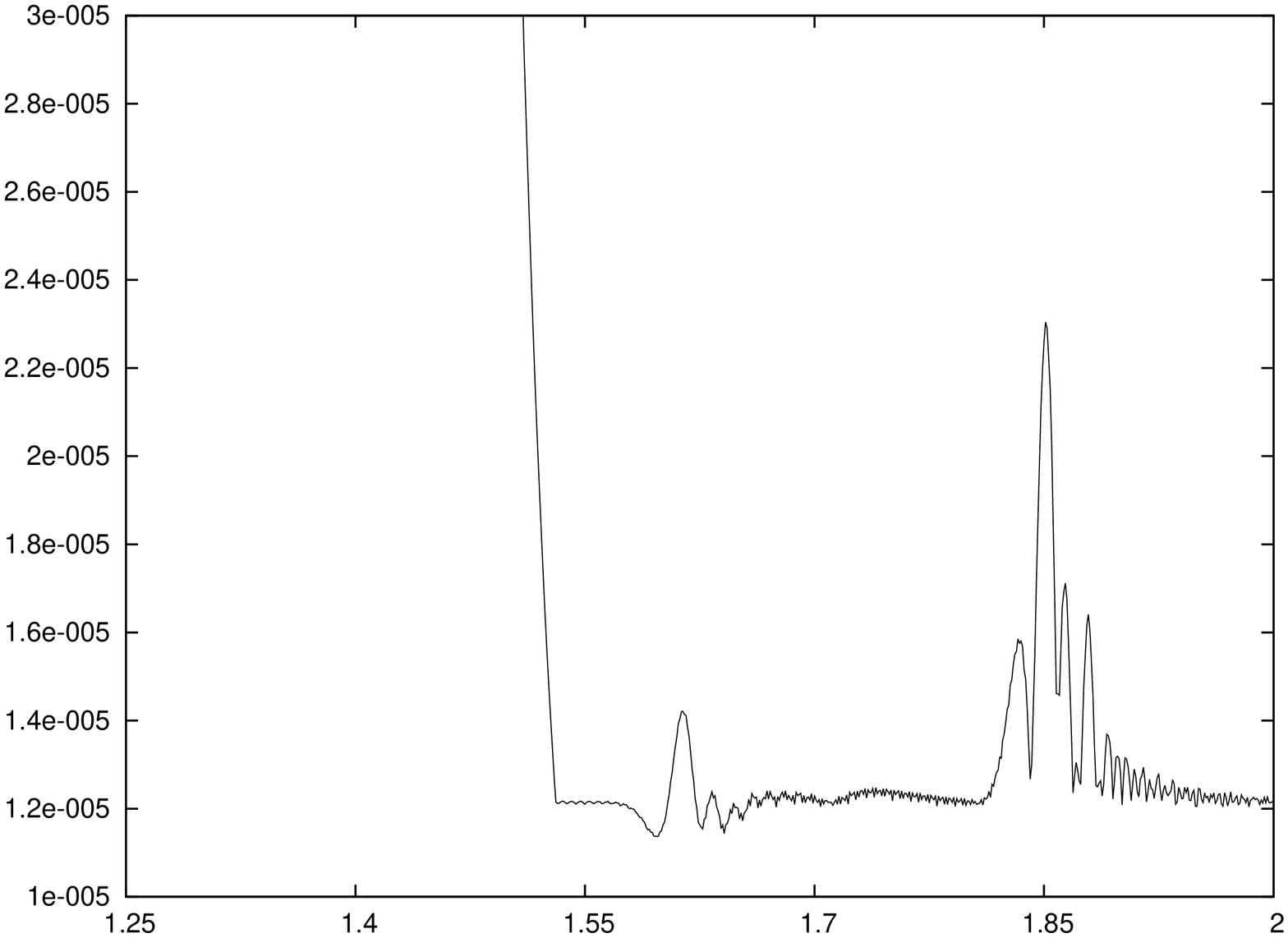}}  \quad \,
			\subfloat[]{\includegraphics[totalheight=2.7in,width=1.825in,angle=-90]{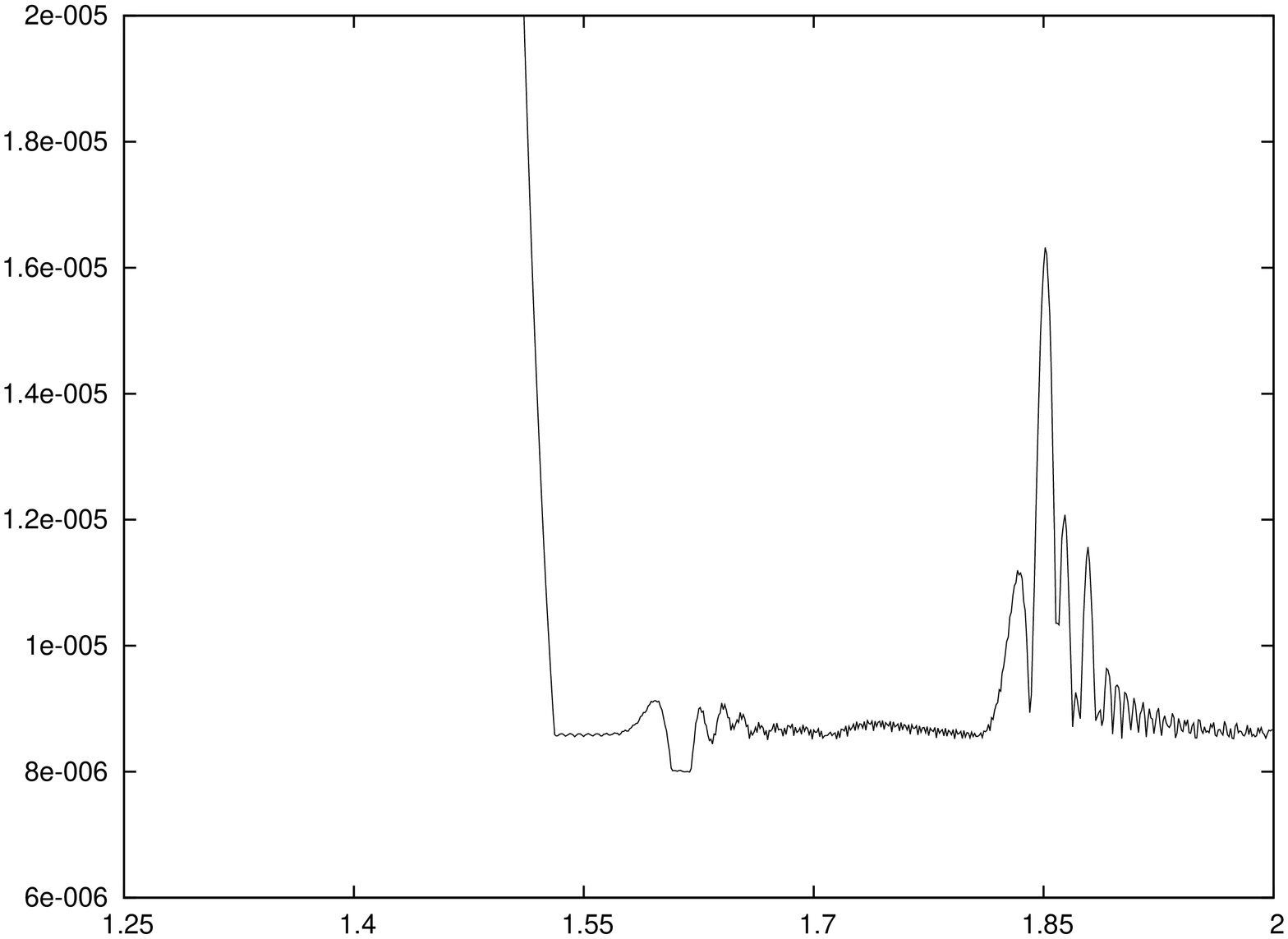}}\\
			(l)\,\,\,\,manification of (k)\\
			\subfloat[]{\includegraphics[totalheight=3in,width=1.9in,angle=-90]{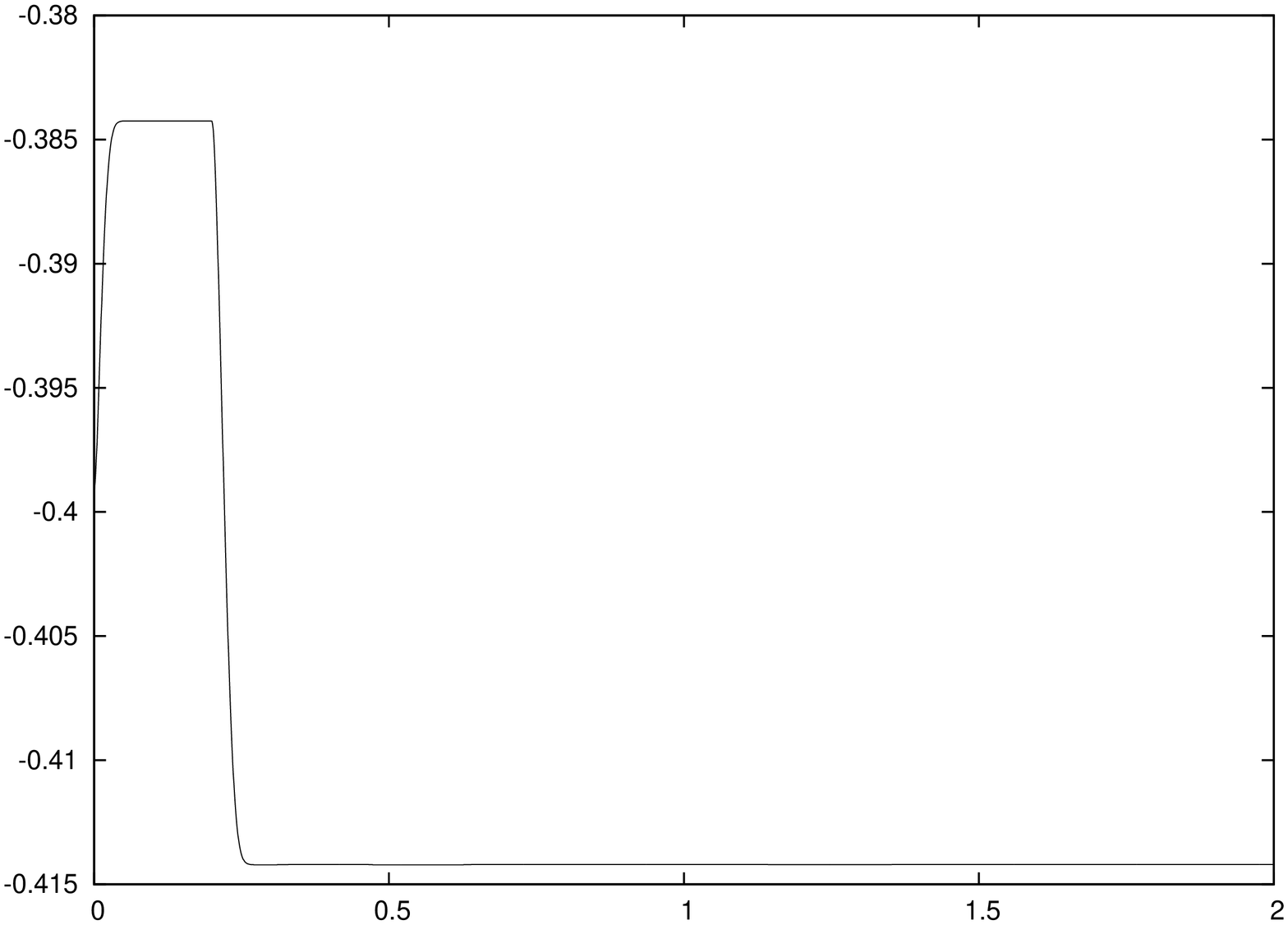}}\\
			(m)\,\,\,\,$\max_{x}(u-\sqrt{1+\eta})$ vs. time
		\end{center}
			\caption{Evolution of Gaussian initial profiles, subcritical case, semidiscretization (\ref{eq43})-(\ref{eq46}).}
			\label{fig42}
	\end{figure}
\noindent %
The initial Gaussian perturbations of the steady state $\eta_{0}=1$, $u_{0}=3$ evolve into 
two unequal pulses for both components of the solution, which travel to the right and left and exit the 
computational  domain
at about $t=0.25$ and $t=1.35$, respectively, without leaving any visible residue as confirmed by 
the temporal  history of the maximum deviations of the approximations of $\eta$ and $u$ from $\eta_{0}$ 
and $u_{0}$ shown in Figures \ref{fig42}(k) and (l). The latter graph shows that the maximum residue after 
exit is of  $O(10^{-5})$, confirming the high degree of absorption of the discrete characteristic boundary 
conditions. The quantity $\max_{x}(u - \sqrt{1 + \eta})$ remains negative, i.e. the numerical solution is 
subcritical, throughout the evolution, cf. Fig. \ref{fig42}(m). We investigated the stability of this fully
discrete scheme by using again as a measure of error the quantity $\max_{x}|\eta - \eta_{0}|$ which
stabilizes at $t=1.55$ to the value $1.21E-05$. The maximum wave speed $c$ for this problem is
about $2.5$, and the observed Courant number restriction was $ck/h\leq 1.67$ in conformity with the
analogous value in the supercritical case. \par
In section 3 we analyzed a different Galerkin semidiscretization of the ibvp under consideration, namely
that given by (\ref{eq36})-(\ref{eq38}), a semidiscrete approximation of the diagonal form of the ibvp, i.e. of
(\ref{eqsw2a}). As the following experiment suggests, this semidiscretization is also $O(h^{2})$ accurate
in $L^{2}$ if we use piecewise linear continuous functions on a uniform mesh. We consider the
inhomogeneous version of (\ref{eqsw2a}) with unknowns $2v$ and $2w$ instead of $v$ and $w$
and take as exact solution the functions $v = u - u_{0} + 2(\sqrt{1 + \eta} - \delta_{0})$,
$w = u - u_{0} - 2(\sqrt{1 + \eta} - \delta_{0})$, where $\delta_{0} = \sqrt{1 + \eta_{0}}$ and 
$\eta(x,t) = (x+1)\mathrm{e}^{-xt}$, $u(x,t) = (2x + \cos (\pi x) - 1)\mathrm{e}^{t} + x A(t) + (1-x)B(t)$, with 
$A(t) = 2\sqrt{1 + \eta(1,t)} + u_{0} - 2\delta_{0}$, $B(t) = -2\sqrt{1 + \eta(0,t)} + u_{0} + 2\delta_{0}$. For
$\eta_{0} = u_{0} =1$ we approximate this ibvp by the nonhomogeneous analog of the semidiscretization
(\ref{eq36})-(\ref{eq38}) (with unknowns $2v_{h}$ and $2w_{h}$) using piecewise linear continuous 
elements with $h=1/N$, and the `classical' 4$^{th}$-order RK scheme for time stepping with $k=h/10$.
\begin{table}[h!]
\setlength\tabcolsep{10.7pt}
\begin{tabular}{@{}ccccc@{}}
\tblhead{$N$   &  $\eta$  &  $order$   &      $u$           &  $order$}  
$40$  &  $2.470369(-3)$   &    --      &   $9.918820(-4)$   &   --        \\ 
$80$  &  $6.172661(-4)$   &  $2.00076$ &   $2.472869(-4)$   &  $2.00398$ \\ 
$160$ &  $1.543038(-4)$   &  $2.00012$ &   $6.179903(-5)$   &  $2.00053$  \\ 
$320$ &  $3.857665(-5)$   &  $1.99997$ &   $1.545737(-5)$   &  $1.99929$ \\  
$480$ &  $1.714531(-5)$   &  $1.99998$ &   $6.870865(-6)$   &  $1.99967$  \\ 
$520$ &  $1.460903(-5)$   &  $1.99999$ &   $5.854663(-6)$   &  $1.99958$ \\ \hline
\end{tabular} \vspace{3.1pt}
\small
\caption{$L^{2}$ errors and spatial orders of convergence, subcritical case,
semidiscretization (\ref{eq36})-(\ref{eq38}).}
\label{tbl43}
 \end{table}
 \normalsize
 The resulting $L^{2}$ errors and rates of convergence at $T=1$ are given in Table \ref{tbl43}. The rates are
 practically equal to 2 as in the previous cases, due to the uniform mesh. The temporal order of convergence
 for this scheme was found to be practically equal to 4; the associated temporal-order calculation with
 $h=1/50$, $k_{ref}=h/200$ was quite robust. 
 \par
 We now compare the solutions of the two semidiscretizations (\ref{eq43})-(\ref{eq46}) and 
 (\ref{eq36})-(\ref{eq38}) by means of a numerical experiment. We consider the ibvp (\ref{eqsw2}) 
 with $\eta_{0}=u_{0}=1$ and initial values $\eta(x,0) = 0.1\exp (-400(x-0.5)^{2}) + \eta_{0}$,
 $u(x,0) = 0.05\exp ( -400(x - 0.5)^{2}) + u_{0}$. (Recall that the temporal evolution of the numerical
 solution of this ibvp generated by (\ref{eq43})-(\ref{eq46}) with $N=2000$, $h=1/N$, $k=h/10$ is
 shown in Figure \ref{fig42}.)  We will compare the numerical solutions of this ibvp with both 
 discretizations, using the same numerical parameters. Let $(\eta_{h},u_{h})$ denote the fully discrete
 approximations at time $t$ produced by (\ref{eq43})-(\ref{eq46}) as underlying semidiscretization. 
 In addition, let $(\eta_{hD},u_{hD})$ be functions in $S_{h}$ with point values computed by the
 formulas (\ref{eq310}), where $v_{h}$, $w_{h}$ are now the fully discrete approximations at $t$
 produced when we use the method (\ref{eq36})-(\ref{eq38}). Define 
 $\ve(t) = \max_{0\leq i\leq N}\abs{\eta_{h}(x_{i},t) - \eta_{hD}(x_{i},t)}$ 
\begin{table}[h!]
		\setlength\tabcolsep{10.7pt}
		\begin{tabular}{@{}cccccccc@{}}
\tblhead{$t$    & $0.0$      & $0.05$     & $0.1$      & $0.15$     & $0.2$     & $0.225$   & $0.25$} 
    $\ve(t)$  & $2.0(-8)$  & $2.0(-8)$  & $2.0(-8)$  & $2.0(-8)$  & $4.0(-8)$ & $3.0(-8)$ & $3.0(-8)$  \\ 
    $e(t)$  & $<10^{-8}$ & $<10^{-8}$ & $<10^{-8}$ & $<10^{-8}$ & $2.0(-8)$ & $2.0(-8)$ & $2.0(-8)$ \\ \hline
   \end{tabular}
\vspace{12pt} \\  
   	\setlength\tabcolsep{6.5pt}	
\begin{tabular}{@{}cccccccc@{}} 
\tblhead{$t$      & $0.35$      &  $0.75$       & $1.05$       & $1.15$     & $1.25$      & $1.35$     & $1.5$} 
   	$\ve(t)$ & $1.232(-5)$ &  $1.301(-5)$  & $1.237(-5)$  & $1.3(-5)$  & $1.224(-5)$ & $1.23(-5)$ & $1.217(-5)$ \\  
   	$e(t)$   & $8.71(-6)$  &  $9.2(-6)$    & $8.75(-6)$   & $9.54(-6)$ & $8.65(-6)$  & $8.75(-6)$ & $8.6(-6)$ \\ \hline
   	\end{tabular}\vspace{3.1pt}
 \small
 	\caption{Comparison of the discretizations (\ref{eq43})-(\ref{eq46}) and (\ref{eq36})-(\ref{eq38}), subcritical case,
 		experiment of Fig. \ref{fig42}}
 	\label{tbl44}
 \end{table}
 \normalsize
  and
 $e(t) = \max_{0\leq i\leq N}\abs{u_{h}(x_{i},t) - u_{hD}(x_{i},t)}$.  The quantities $\ve$ and $e$ 
 for various values of $t=t^{n}\in [0,1]$ are given in Table \ref{tbl44}.  We observe
 that up to about $t=0.25$ the errors are of $O(10^{-8})$ and subsequently increase to $O(10^{-5})$.
 This is due to the fact that about $t=0.25$ the larger, rightwards-travelling wave 
 completes its exit from the
  computational domain (see Figure \ref{fig42}(e), and we expect a small residue to be radiated into $[0,1]$
  because the numerical characteristic boundary conditions of (\ref{eq43})-(\ref{eq46}) are not exactly transparent.
  This residue has a magnitude of $O(10^{-5})$ as evidenced by Fig. \ref{fig42}(l). It is worthwhile to note that
  the `diagonal' approximation (\ref{eq36})-(\ref{eq38}) leaves a practically negligible residue.  
  This is suggested
  by the evidence in Figure \ref{fig43}. In this figure, the two graphs on the left depict what is left in the
   computational domain of the $\eta$-and $u$-components of the numerical solution generated by 
  (\ref{eq43})-(\ref{eq46}) at $t=1.5$, after the waves have exited the domain, while those on the right
  are the analogous $\eta$-and $u$-profiles generated by (\ref{eq36})-(\ref{eq38}). (Thus the graphs
  on the left are magnifications of the graphs of Fig. \ref{fig42}(j).) We observe that the residues
  of the usual Galerkin semidiscretizations are of $O(10^{-5})$ while those of the `diagonal' scheme
  are much smaller. It is clear, at least for this example, that the `diagonal' Galerkin method has practically
  transparent boundary conditions. 	
  \clearpage 
  \captionsetup[subfloat]{labelformat=empty,position=bottom,singlelinecheck=false}
 \begin{figure}[!ht]
 	\begin{center}
 		\subfloat[]{\includegraphics[totalheight=2.7in,width=1.82in,angle=-90]{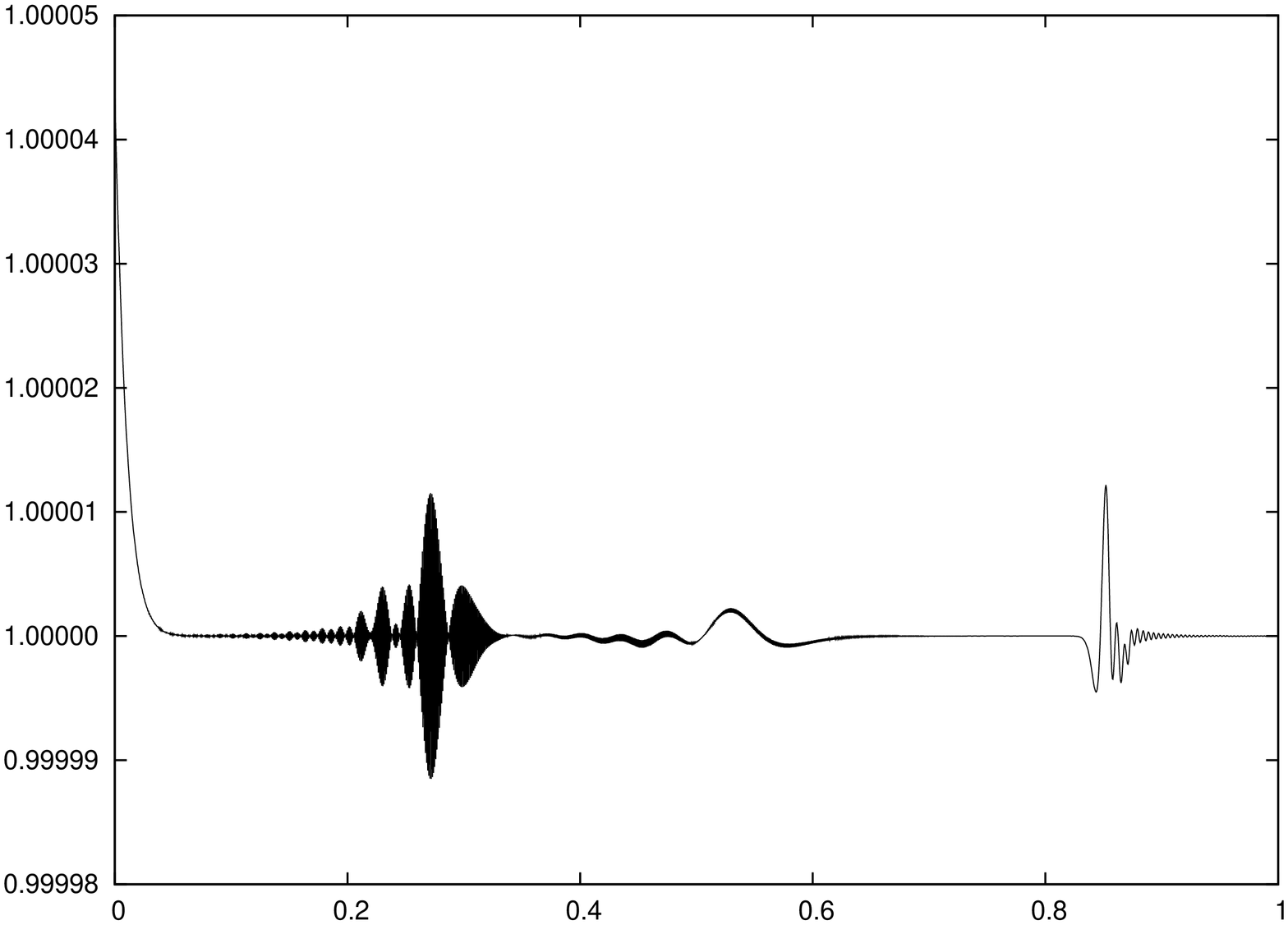}}  \quad \,
 		\subfloat[]{\includegraphics[totalheight=2.7in,width=1.82in,angle=-90]{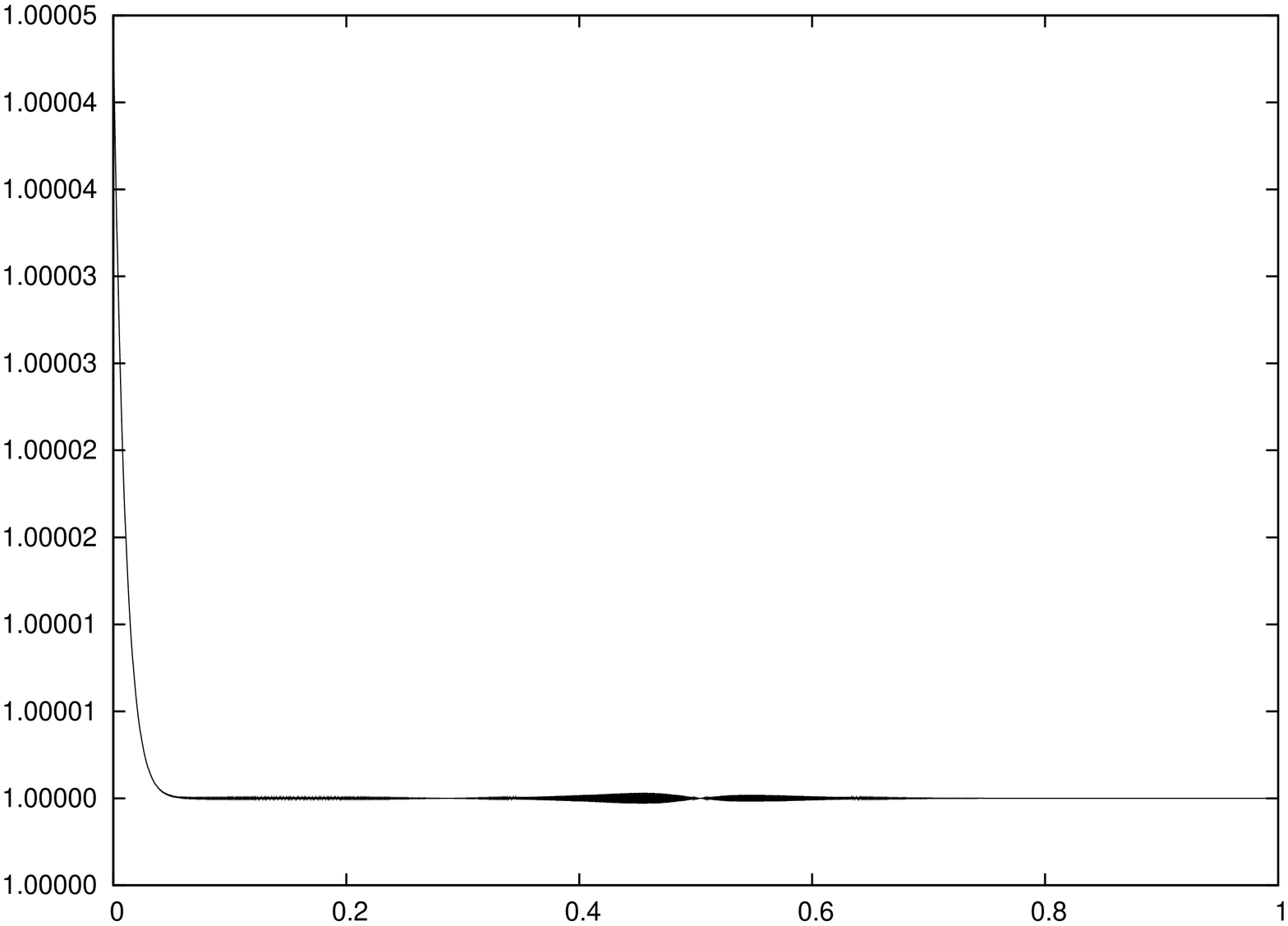}}\vspace{2pt} \\
 		(a)\,\,\,\,$\eta_{h}$ and $\eta_{hD}$, $t=1.5$ \vspace{1pt} \\
 		\subfloat[]{\includegraphics[totalheight=2.7in,width=1.82in,angle=-90]{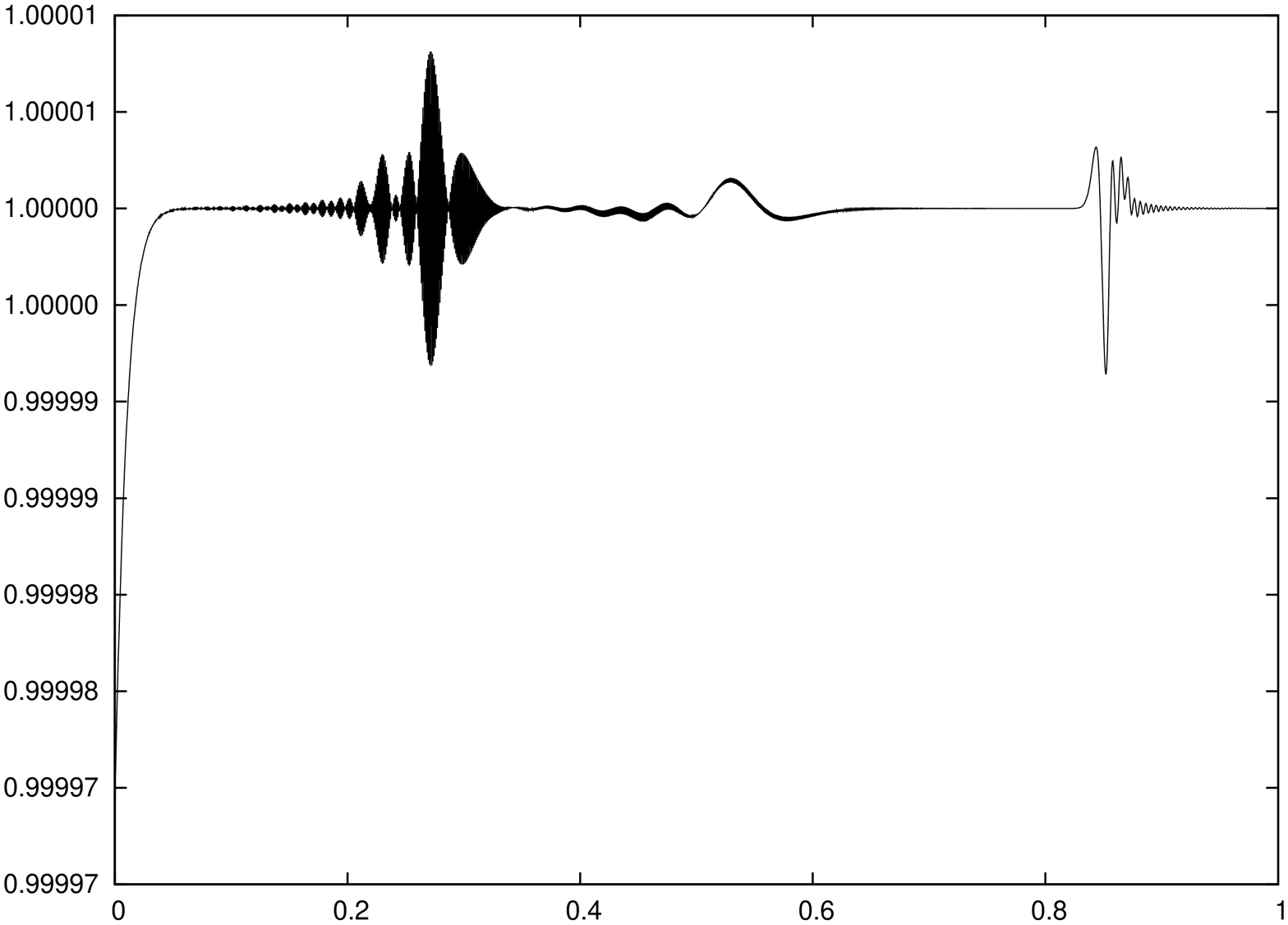}}  \quad \,
 		\subfloat[]{\includegraphics[totalheight=2.7in,width=1.82in,angle=-90]{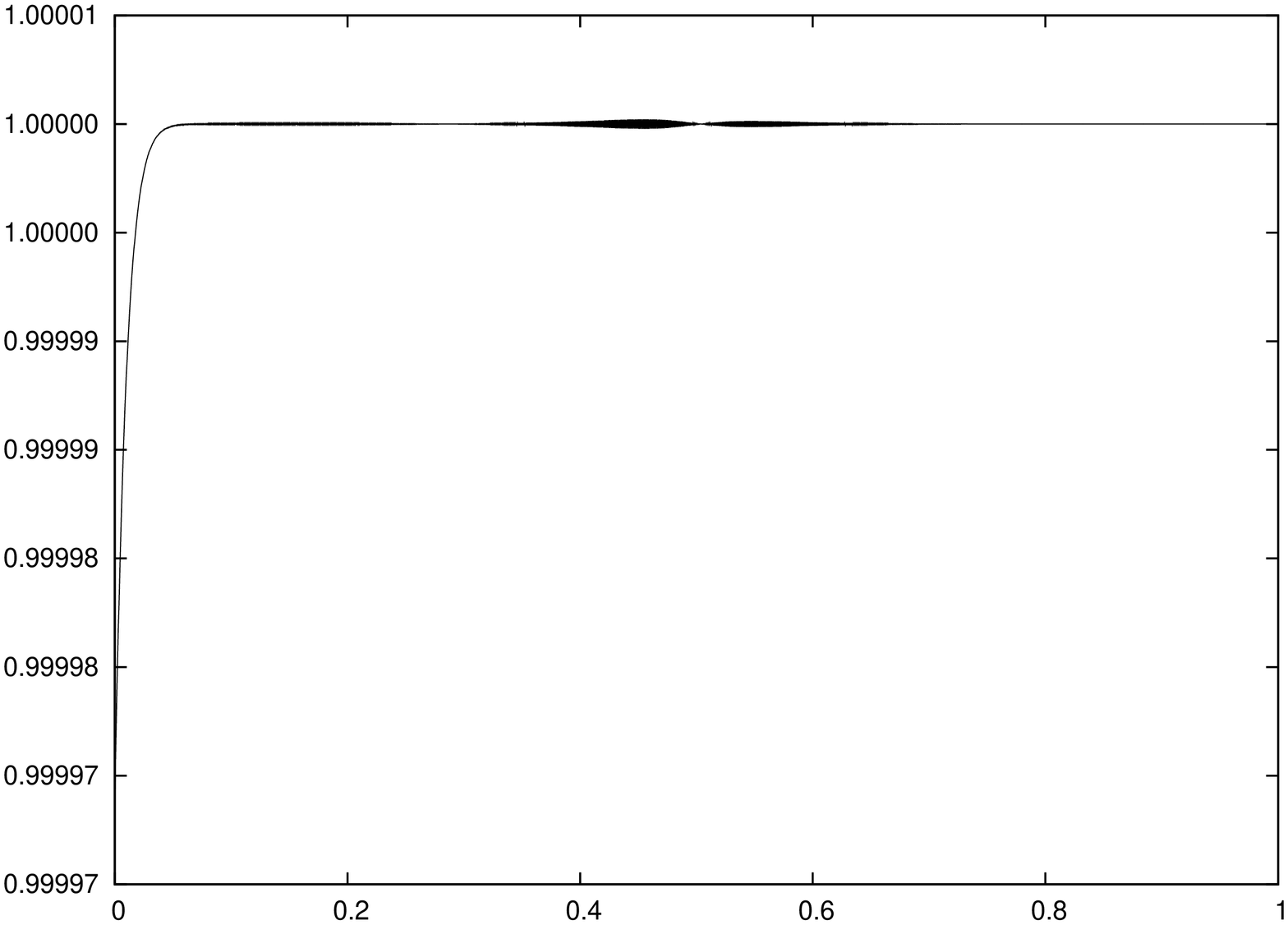}}\vspace{2pt}\\
 		(b)\,\,\,\,$u_{h}$ and $u_{hD}$, $t=1.5$ 
 	\end{center} 
 	\caption{Magnifications of the profiles of the numerical solutions generated by the methods 
 		(\ref{eq43})-(\ref{eq46}) (left) and (\ref{eq36})-(\ref{eq38}) (right) at $t=1.5$. Subcritical case, 
 		evolution as in Fig. \ref{fig42}.}
 	\label{fig43}
 \end{figure}
  \normalsize  	
 \subsection{Linearized vs. nonlinear characteristic boundary conditions in the subcritical case}
 If the elevation of the free surface $\eta$ is a small perturbation of the steady state $\eta_{0}$ one may
 derive \emph{linearized} approximations to the characteristic boundary conditions: Assume that the 
 wave height is given by $1 + \eta= 1 + \eta_{0} + \wt{\eta}$, where $\abs{\wt{\eta}}\ll 1 + \eta_{0}$. Then 
 $\sqrt{1 + \eta} = \sqrt{1 + \eta_{0}} + \tfrac{\eta - \eta_{0}}{2\sqrt{1 + \eta_{0}}} + O(\wt{\eta}^{2})$, 
 and, in the subcritical case (\ref{eqsw2}), from the characteristic boundary condition, for example at 
 $x=1$, it follows that 
 \[
 u(1,t) - \frac{\eta(1,t)}{\sqrt{1 + \eta_{0}}} = u_{0} - \frac{\eta_{0}}{\sqrt{1 + \eta_{0}}} + O(\wt{\eta}^{2}),
 \]
 from which neglecting the $O(\wt{\eta}^{2})$ term we obtain the linearized b.c.
 \begin{equation}
 	u(1,t) - \frac{\eta(1,t)}{\sqrt{1 + \eta_{0}}} = u_{0} - \frac{\eta_{0}}{\sqrt{1 + \eta_{0}}},
 	\label{eq47}
 \end{equation}
 and similarly at $x=0$ the b.c.
 \begin{equation}
 	u(0,t) + \frac{\eta(0,t)}{\sqrt{1 + \eta_{0}}} = u_{0} + \frac{\eta_{0}}{\sqrt{1 + \eta_{0}}}.
 	\label{eq48}
 \end{equation}
 \par
 The linearized boundary conditions (\ref{eq47}) and (\ref{eq48}) have often been used in the computational
 fluid dynamics literature, as was mentioned in the Introduction. It is not hard to see that (\ref{eq47}) and 
 (\ref{eq48}) are transparent for the \emph{linearized} shallow water equations obtained by linearizing 
 the system (\ref{eq11}) about the steady state $(u_{0},\eta_{0})$. This may be seen by diagonalizing the
 linearized system and explicitly computing its solutions propagating along the characteristics.
 In \cite{sltt} it is proved that the `energy' integral $\int_{0}^{1}(u^{2} + \tfrac{1}{1+\eta_{0}}\eta^{2})dx$ of
 the solution $(\eta,u)$ of the linearized system decreases with time in the presence of a class of
 boundary conditions that includes (\ref{eq47})-(\ref{eq48}), a fact that implies the well-posedness of the
 ibvp for the linearized system supplemented by (\ref{eq47}) and (\ref{eq48}). \par
 However, the linearized characteristic boundary conditions are not transparent for the nonlinear system (\ref{eq11}).
 So, there arises a need to study their absorption properties. In \cite{nmf} and \cite{sltt} the linearized
 conditions were compared with the nonlinear ones, appropriately discretized in the finite difference
 schemes used in these two references. The comparison was effected by means of numerical experiments,
 many of which in more complicated instances of hydraulic and geophysical interest, including single- and
 multi- layered flows in the presence of variable bottom topography, cross velocity terms, and Coriolis
 forces. Moreover, the finite difference schemes used in \cite{nmf} and \cite{sltt} allow in general the
 simulation of flows that develop steep fronts and discontinuities. It was confirmed that the linearized 
 conditions in many  examples give rise to spurious oscillations that are reflected backwards into the 
 computational domain.  \par
 In the sequel we will compare the two sets of boundary conditions in the case of smooth subcritical flows 
 discretized by the `direct' fully discrete Galerkin method with piecewise linear continuous functions
in space coupled
with `classical' 4$^{th}$-order RK time stepping. The spatial discretization in the 
 nonlinear b.c. case is given by (\ref{eq43})-(\ref{eq46}); in the linearized case (\ref{eq45})-(\ref{eq46})
 are replaced by their obvious linearized analogs resulting from (\ref{eq47})-(\ref{eq48}). We first check
 the $L^{2}$ errors and associated orders of convergence for the scheme with the linearized b.c.'s.
 Adding appropriate  right-hand sides to the pde's in (\ref{eqsw2}) so that the exact
 solution of the ibvp is $\eta(x,t) = (x+1)\mathrm{e}^{-xt}$, 
 $u(x,t) = (2x + \cos (\pi x) - 1)\mathrm{e}^{t} +x a(t) + (1-x)b(t)$, where 
 $a(t) = u_{0} + \frac{2\mathrm{e}^{-t} - \eta_{0}}{\sqrt{1 + \eta_{0}}}$,
 $b(t) = u_{0} + \frac{\eta_{0} - 1}{\sqrt{1 + \eta_{0}}}$, we take   $u_{0}=\eta_{0}=1$ and compute 
 with $h=1/N$ and $k=h/20$. 
 The resulting $L^{2}$ errors   of (essentially) the spatial discretization
 are shown, accompanied by the associated orders of convergence in Table \ref{tbl45}.
 Due to  the spatial
 uniform mesh the rates are again practically equal to 2. 
 \begin{table}[h]
 \setlength\tabcolsep{10.7pt}
 	\begin{tabular}{@{}ccccc@{}}
\tblhead{$N$  &  $\eta$          &   $order$   &  $u$               &  $order$}
$40$  &  $4.835002(-3)$  &             &  $2.930984(-3)$    &          \\ 
$80$  &  $1.204245(-3)$  &  $2.00539$  &  $7.408500(-4)$    &  $1.98413$ \\ 
$120$ &  $5.350289(-4)$  &  $2.00088$  &  $3.299726(-4)$    &  $1.99472$  \\ 
$160$ &  $3.008683(-4)$  &  $2.00099$  &  $1.858783(-4)$    &  $1.99497$  \\ 
$200$ &  $1.925597(-4)$  &  $1.99991$  &  $1.190431(-4)$    &  $1.99695$  \\ 
$240$ &  $1.337304(-4)$  &  $1.99966$  &  $8.269369(-5)$    &  $1.99835$  \\ 
$280$ &  $9.825114(-5)$  &  $1.99998$  &  $6.077487(-5)$    &  $1.99783$  \\ 
$320$ &  $7.523153(-5)$  &  $1.99920$  &  $4.653473(-5)$    &  $1.99936$  \\ 
$360$ &  $5.944094(-5)$  &  $2.00018$  &  $3.677599(-5)$    &  $1.99820$  \\ 
$400$ &  $4.814899(-5)$  &  $1.99964$  &  $2.979144(-5)$    &  $1.99908$  \\ 
$440$ &  $3.979234(-5)$  &  $2.00006$  &  $2.462384(-5)$    &  $1.99880$  \\ 
$480$ &  $3.343735(-5)$  &  $1.99975$  &  $2.069270(-5)$    &  $1.99898$  \\ 
$520$ &  $2.849223(-5)$  &  $1.99946$  &  $1.763197(-5)$    &  $1.99978$  \\ \hline
\end{tabular}\vspace{3.1pt}
 	\small
 	\caption{$L^{2}$ errors and spatial orders of convergence, subcritical case, semidiscretization 
 		(\ref{eq43}), (\ref{eq44}) with the linearized b.c. (\ref{eq47}), (\ref{eq48}).}
 	\label{tbl45}
 \end{table}
 \normalsize
 \captionsetup[subfloat]{labelformat=empty,position=bottom,singlelinecheck=false}
 \begin{figure}[h]
 	\begin{center}
 		\subfloat[]{\includegraphics[totalheight=2.7in,width=1.825in,angle=-90]{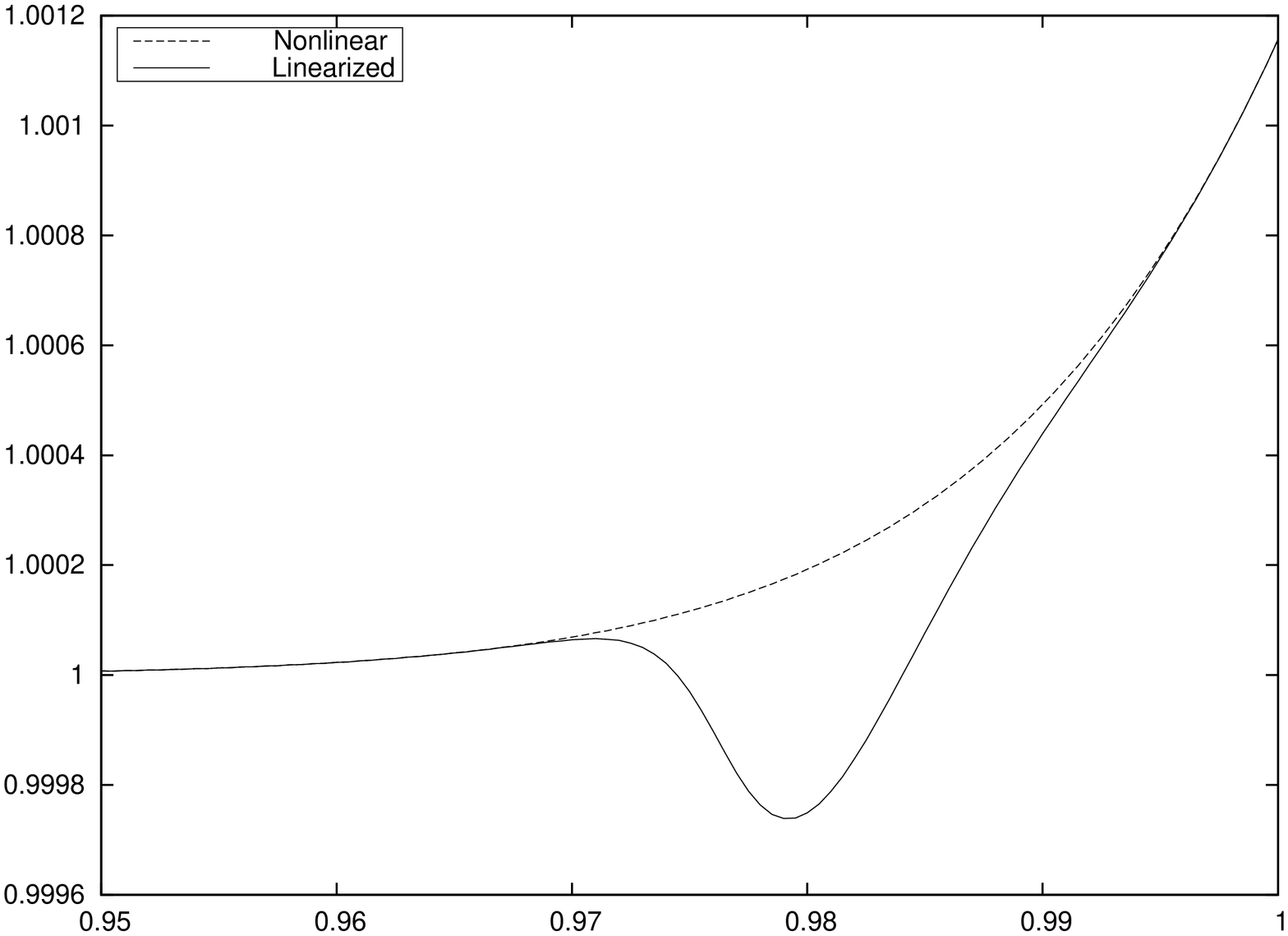}}  \quad \,
 		\subfloat[]{\includegraphics[totalheight=2.7in,width=1.825in,angle=-90]{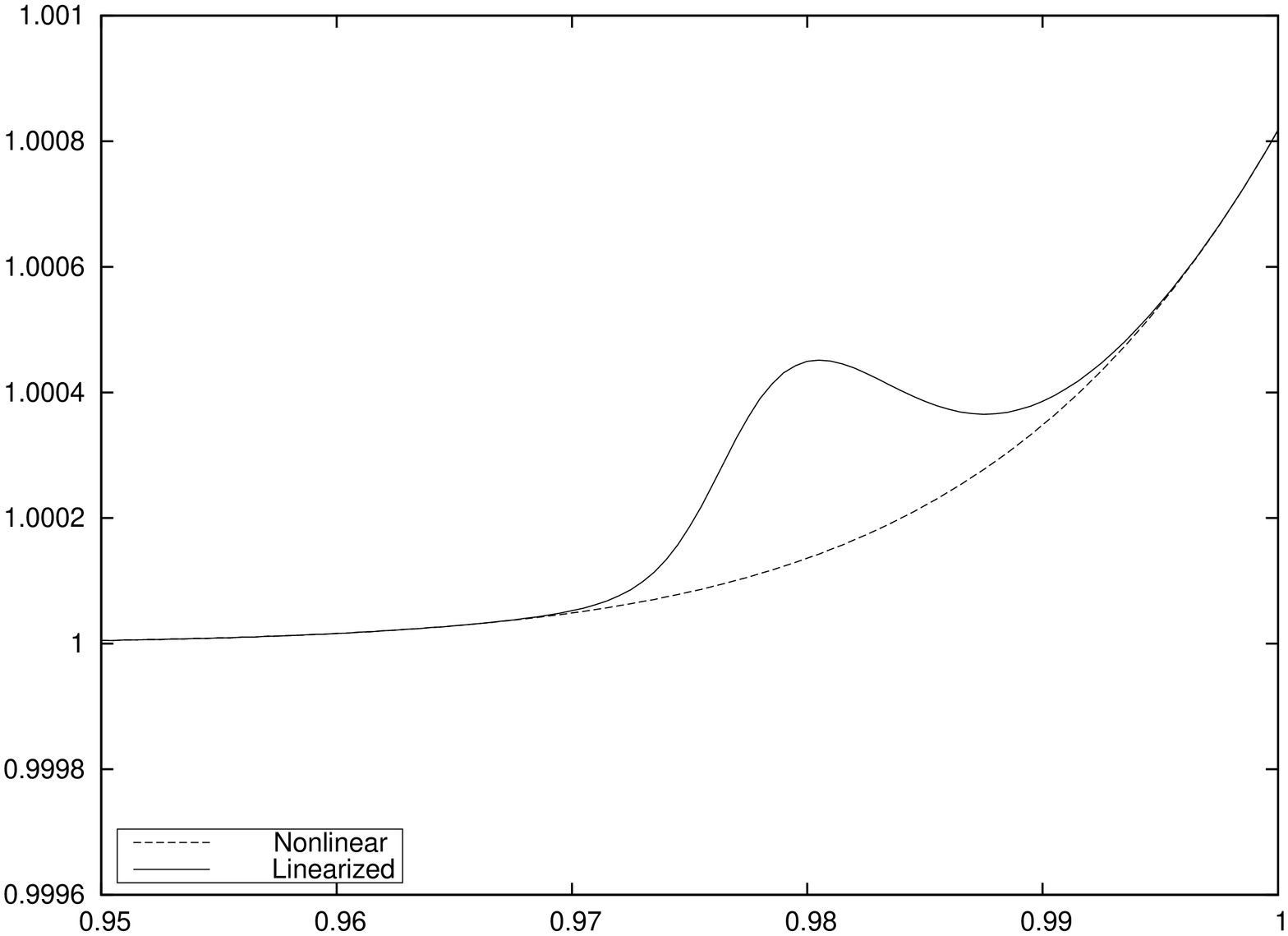}}\\
 		(a)\,\,\,\,$\eta_{lin.}$ vs. $\eta_{Nlin.}$ and $u_{lin.}$ vs. $u_{Nlin.}$ at $t=0.25$ \\ 
 		\subfloat[]{\includegraphics[totalheight=2.7in,width=1.825in,angle=-90]{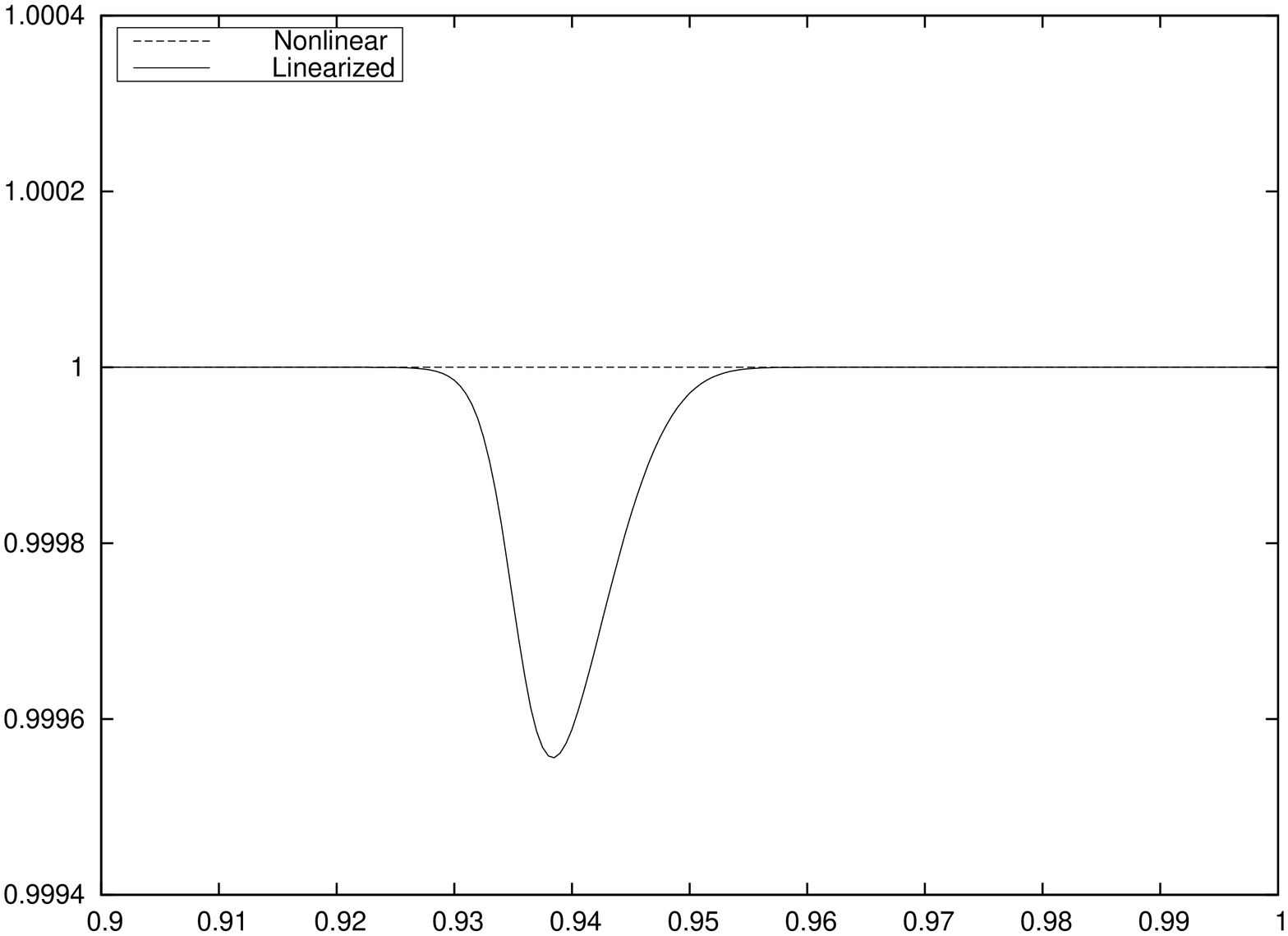}}  \quad \,
 		\subfloat[]{\includegraphics[totalheight=2.7in,width=1.825in,angle=-90]{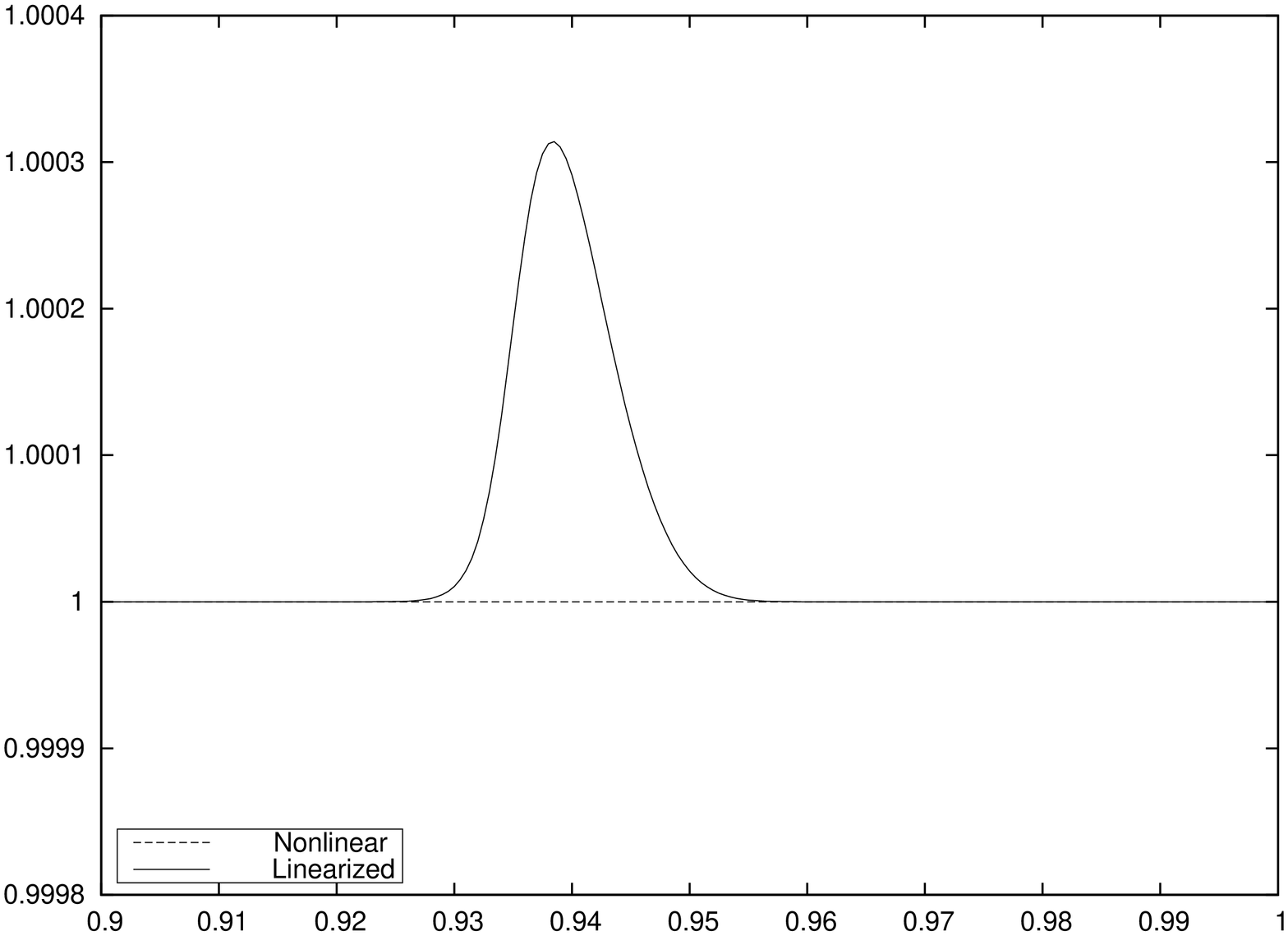}}\\
 		(b)\,\,\,\,$\eta_{lin.}$ vs. $\eta_{Nlin.}$ and $u_{lin.}$ vs. $u_{Nlin.}$ at $t=0.35$ 
 	\end{center}
 	\caption{Reflected spurious pulses due to the linearized b.c.'s (solid lines) superimposed 
 		on the solution of (\ref{eq43})-(\ref{eq46}) (dotted lines); evolution of Fig. \ref{fig42}.}
 	\label{fig44}
 \end{figure}
 \par
 We now consider a numerical experiment discussed in Section 4.2, wherein the fully discrete numerical
 scheme based on the semidiscretization (\ref{eq43})-(\ref{eq46}) gives the evolution depicted in 
 Figure \ref{fig42}. We repeat the experiment with the linearized boundary conditions (\ref{eq47}),
 (\ref{eq48}) instead of (\ref{eq45}),(\ref{eq46}).  In Figure \ref{fig44}  the  
 graphs of $\eta_{lin}$, $u_{lin}$ (solid line), i.e. the numerical solution corresponding to the linearized
 b.c.'s, are superimposed on those of the numerical solution $\eta_{Nlin}$, $u_{Nlin}$ (dotted line) 
 computed with the nonlinear b.c.'s. At $t=0.25$, when the rightwards-travelling pulse has almost
 completed its exit from the computational domain at $x=1$ (recall Fig. \ref{fig42}(g)), we observe 
 that the linearized b.c.'s give rise to a spurious pulse of amplitude about $4\times 10^{-4}$ that
 is reflected inwards and travels (as Fig. \ref{fig44}(b) confirms) to the left with a speed of 
 $-u_{0} + \sqrt{1 + \eta_{0}}\cong 0.41$ as predicted by the analysis of the linearized system of
 which this small-amplitude pulse is an approximate solution. A similar (but much smaller)
 rightwards-travelling pulse is created at $x=0$ when the smaller wave in Fig. \ref{fig42} exits the
 computational interval. (These pulses will eventually exit the interval since the linearized b.c.'s
 are transparent for the linearized system.) This may also be seen in the similar numerical experiment
 shown in Figure \ref{fig45}. We consider (\ref{eqsw2}) with $\eta_{0}=u_{0}=0$
 \begin{figure}[h]
 	\begin{center}
 		\subfloat[]{\includegraphics[totalheight=2.7in,width=1.825in,angle=-90]{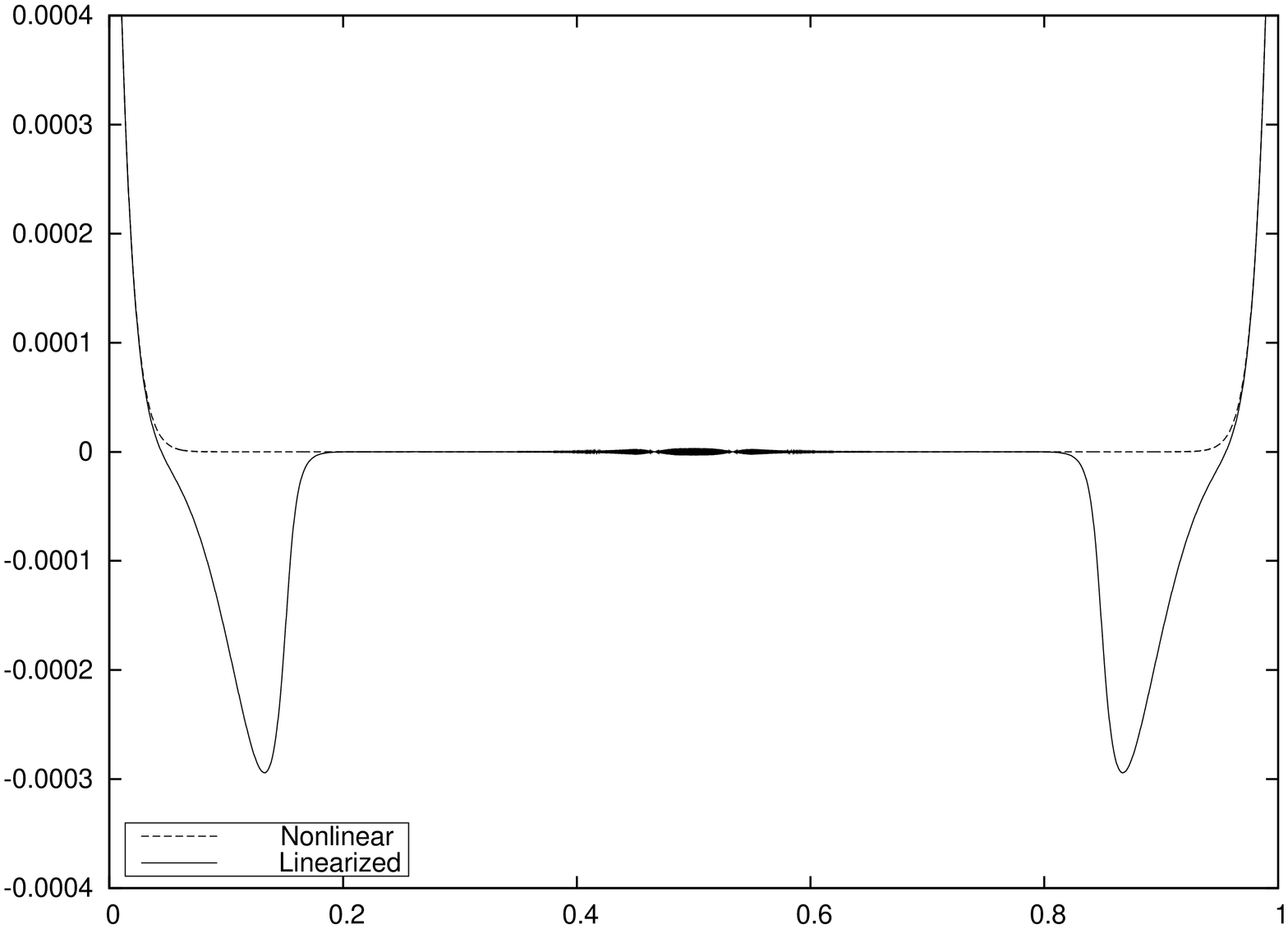}}  \quad \,
 		\subfloat[]{\includegraphics[totalheight=2.7in,width=1.825in,angle=-90]{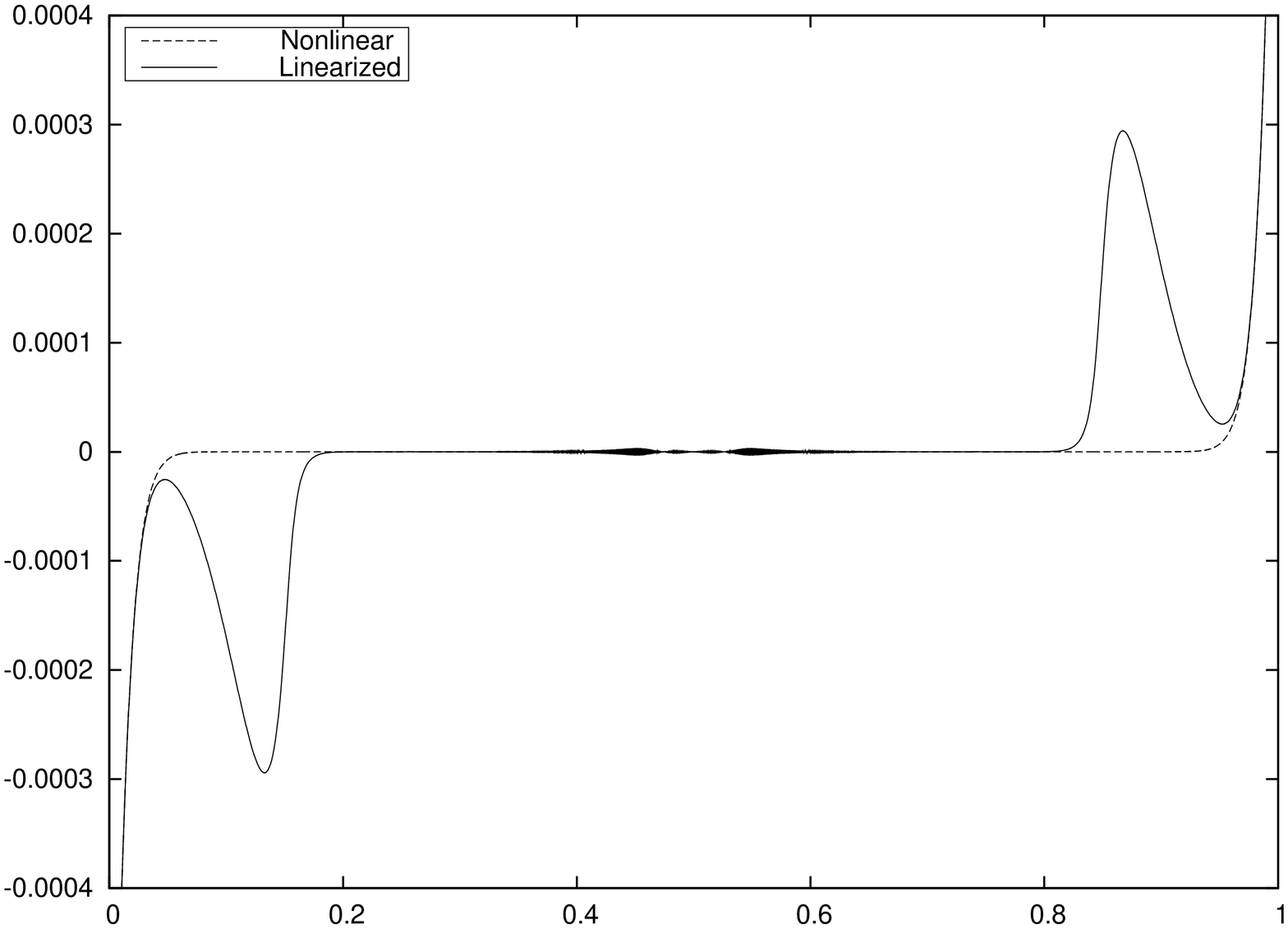}}\\
 		(a)\,\,\,\,$\eta_{lin.}$ vs. $\eta_{Nlin.}$ and $u_{lin.}$ vs. $u_{Nlin.}$ at $t=0.6$ 
 	\end{center}
 \end{figure}
 \clearpage
 \begin{figure}[h]
 	\begin{center}
		\subfloat[]{\includegraphics[totalheight=2.7in,width=1.825in,angle=-90]{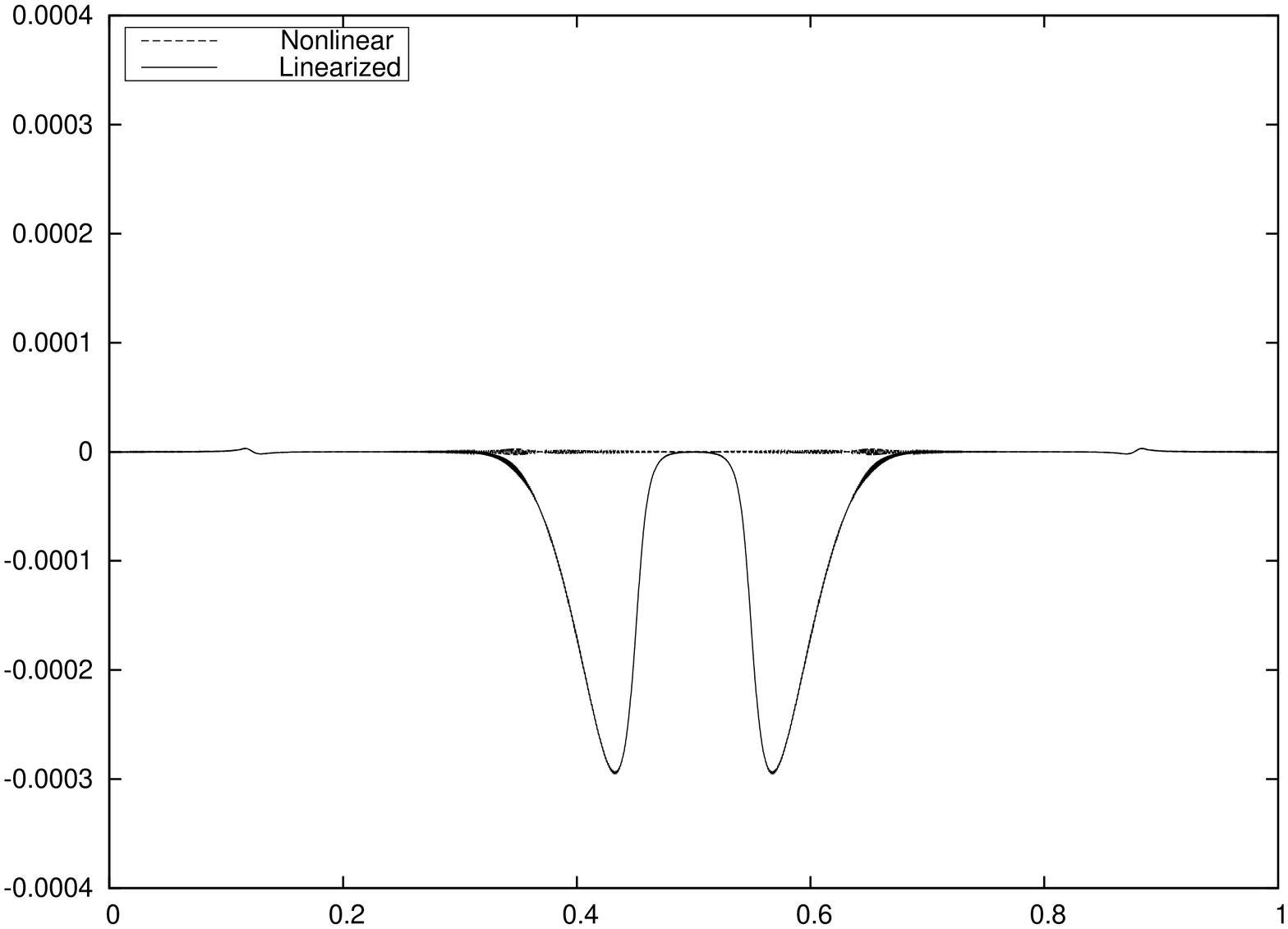}}  \quad \,
 		\subfloat[]{\includegraphics[totalheight=2.7in,width=1.825in,angle=-90]{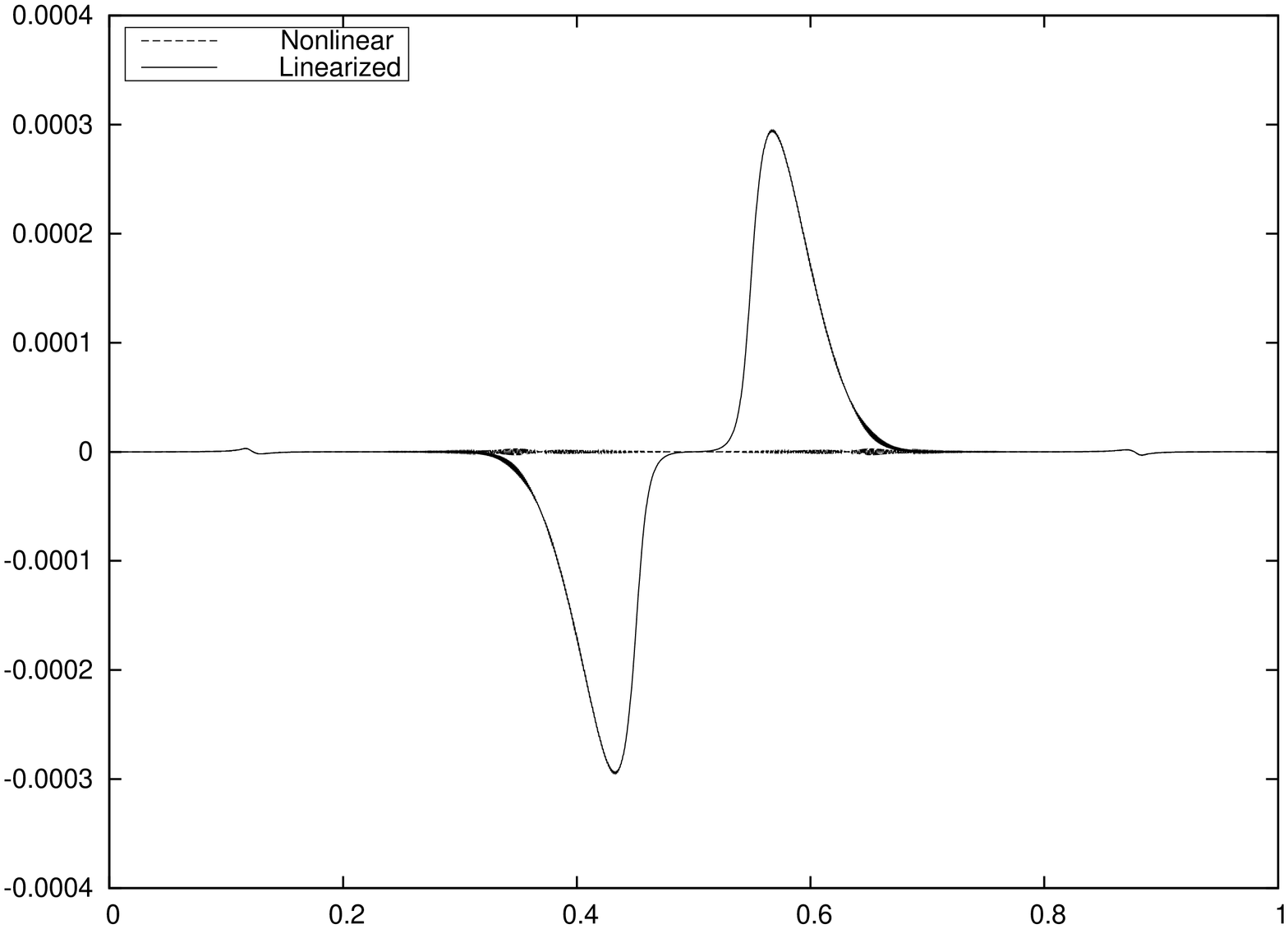}}\\
 		(b)\,\,\,\,$\eta_{lin.}$ vs. $\eta_{Nlin.}$ and $u_{lin.}$ vs. $u_{Nlin.}$ at $t=0.9$ \vspace{13pt} 	\\
 		\subfloat[]{\includegraphics[totalheight=2.7in,width=1.825in,angle=-90]{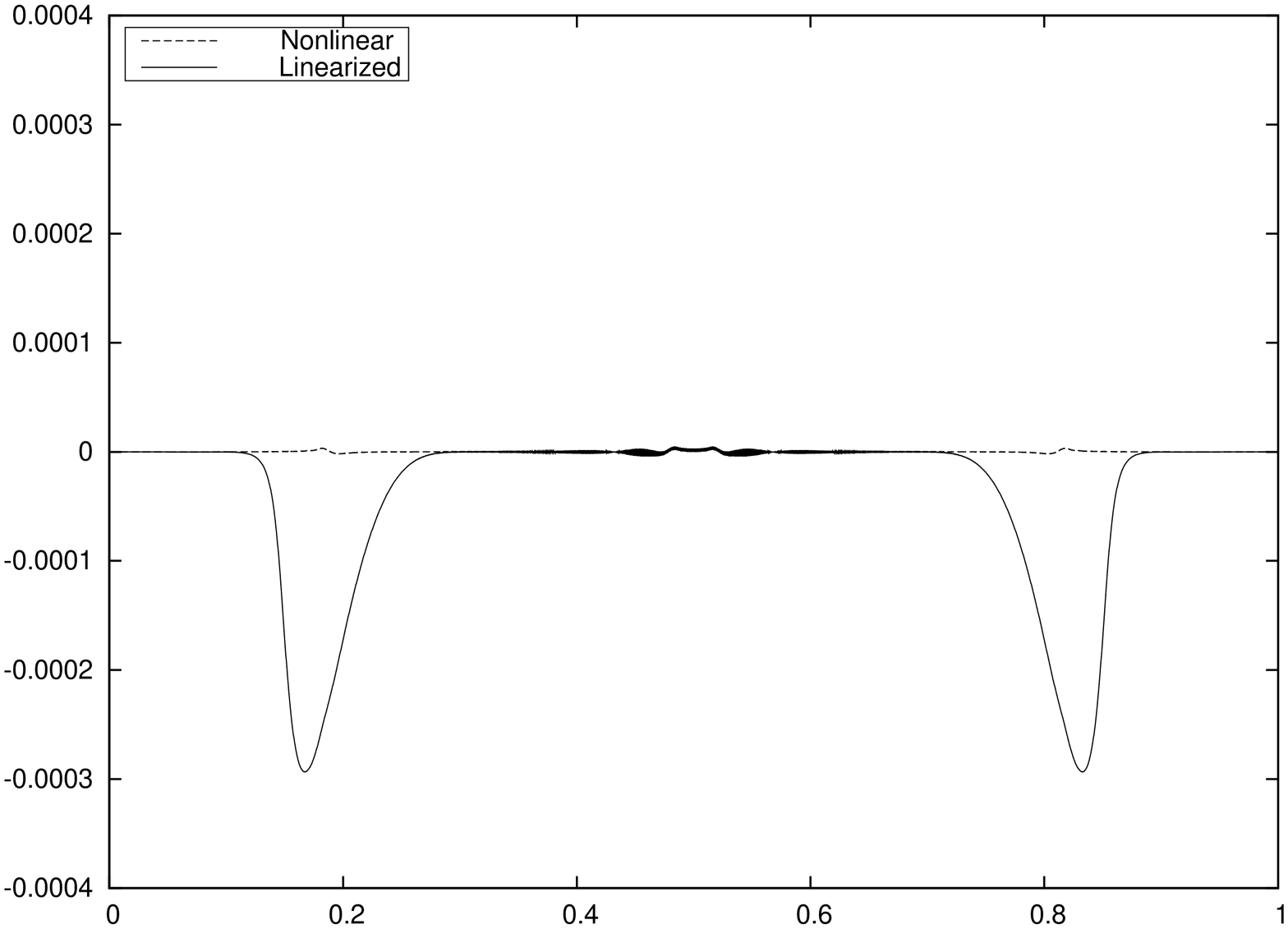}}  \quad \,
 		\subfloat[]{\includegraphics[totalheight=2.7in,width=1.825in,angle=-90]{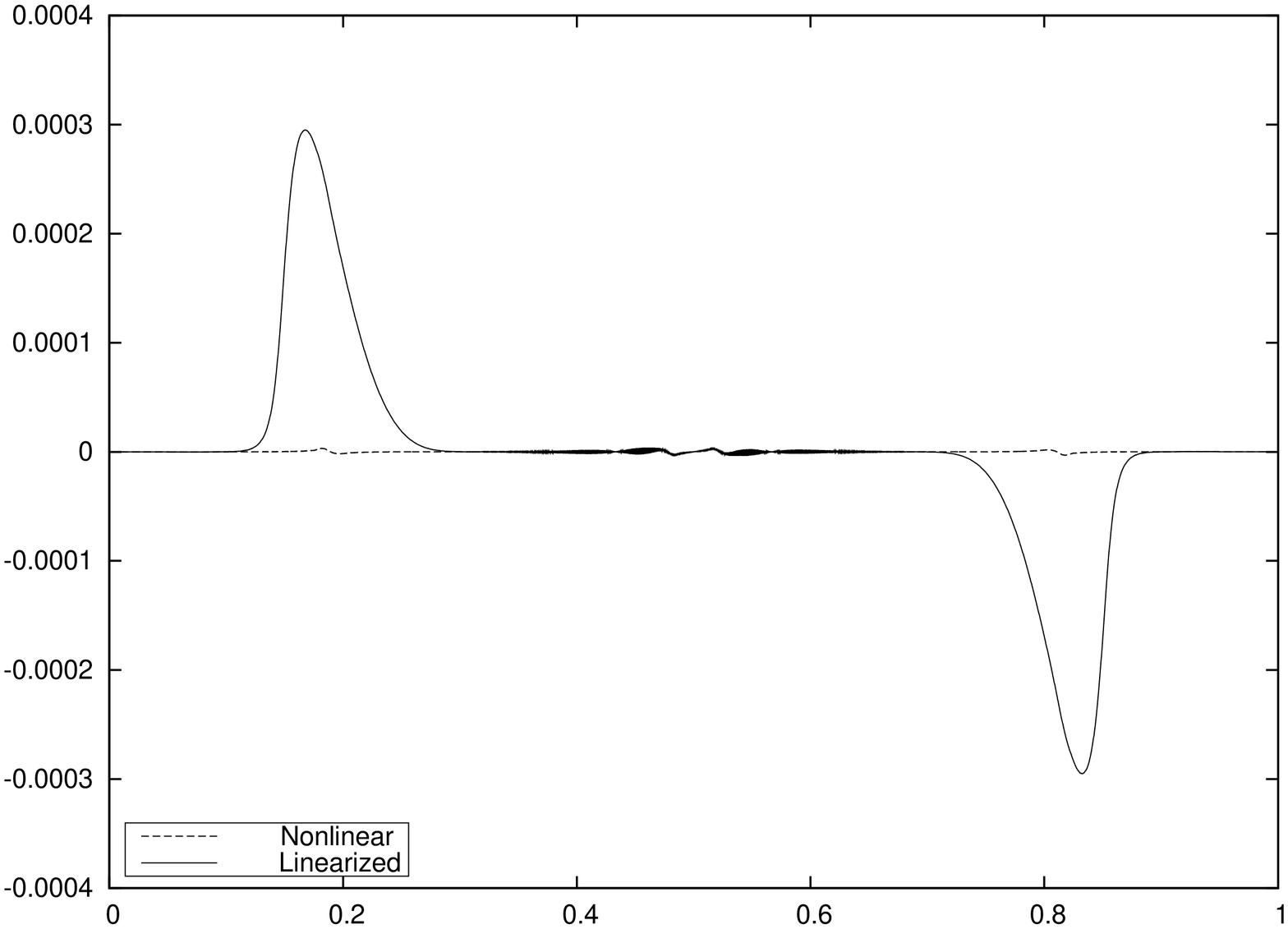}}\\
 		(c)\,\,\,\,$\eta_{lin.}$ vs. $\eta_{Nlin.}$ and $u_{lin.}$ vs. $u_{Nlin.}$ at $t=1.3$ \vspace{13pt} 	\\
 		\subfloat[]{\includegraphics[totalheight=2.7in,width=1.825in,angle=-90]{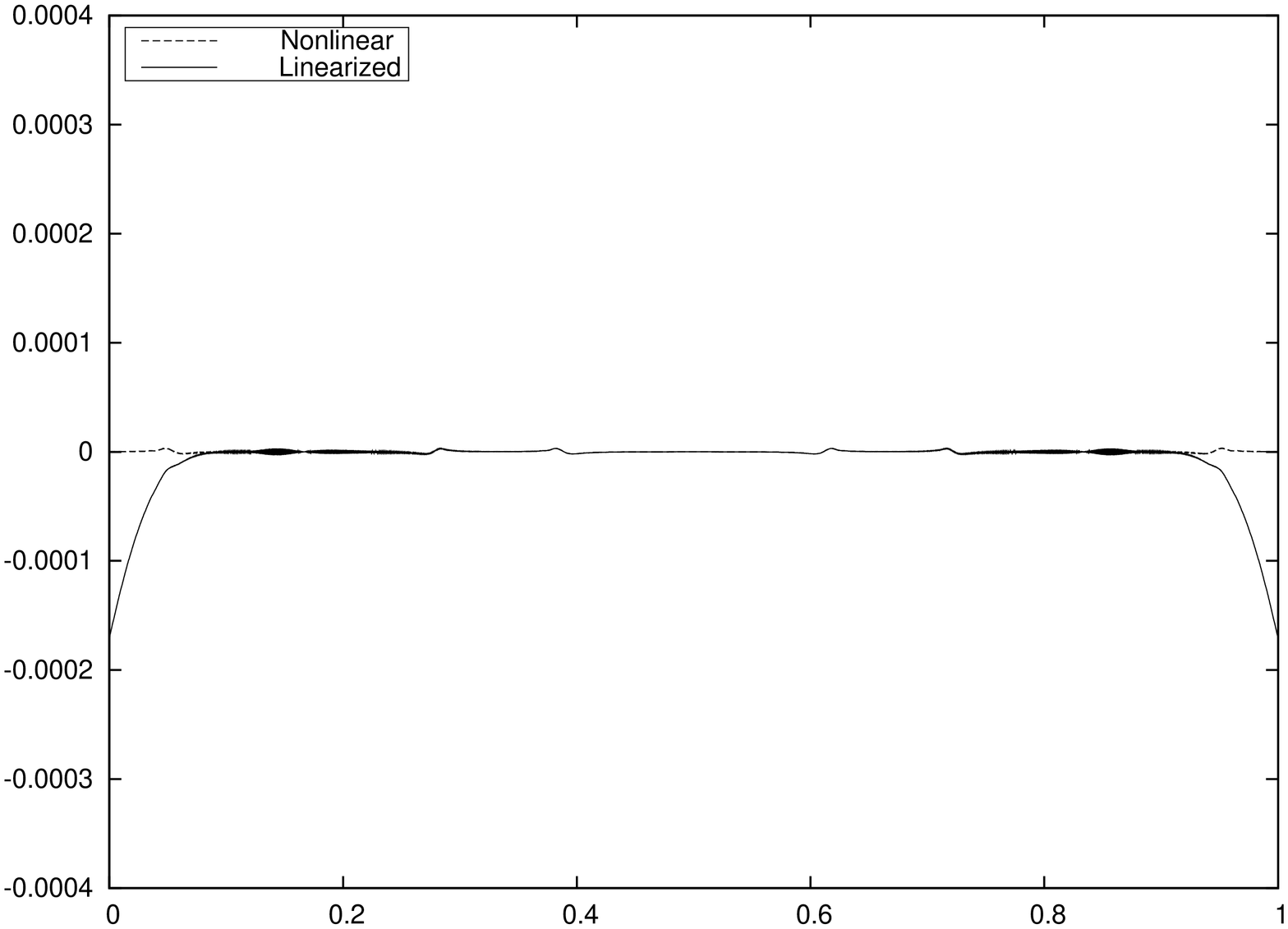}} \quad \,
 		\subfloat[]{\includegraphics[totalheight=2.7in,width=1.825in,angle=-90]{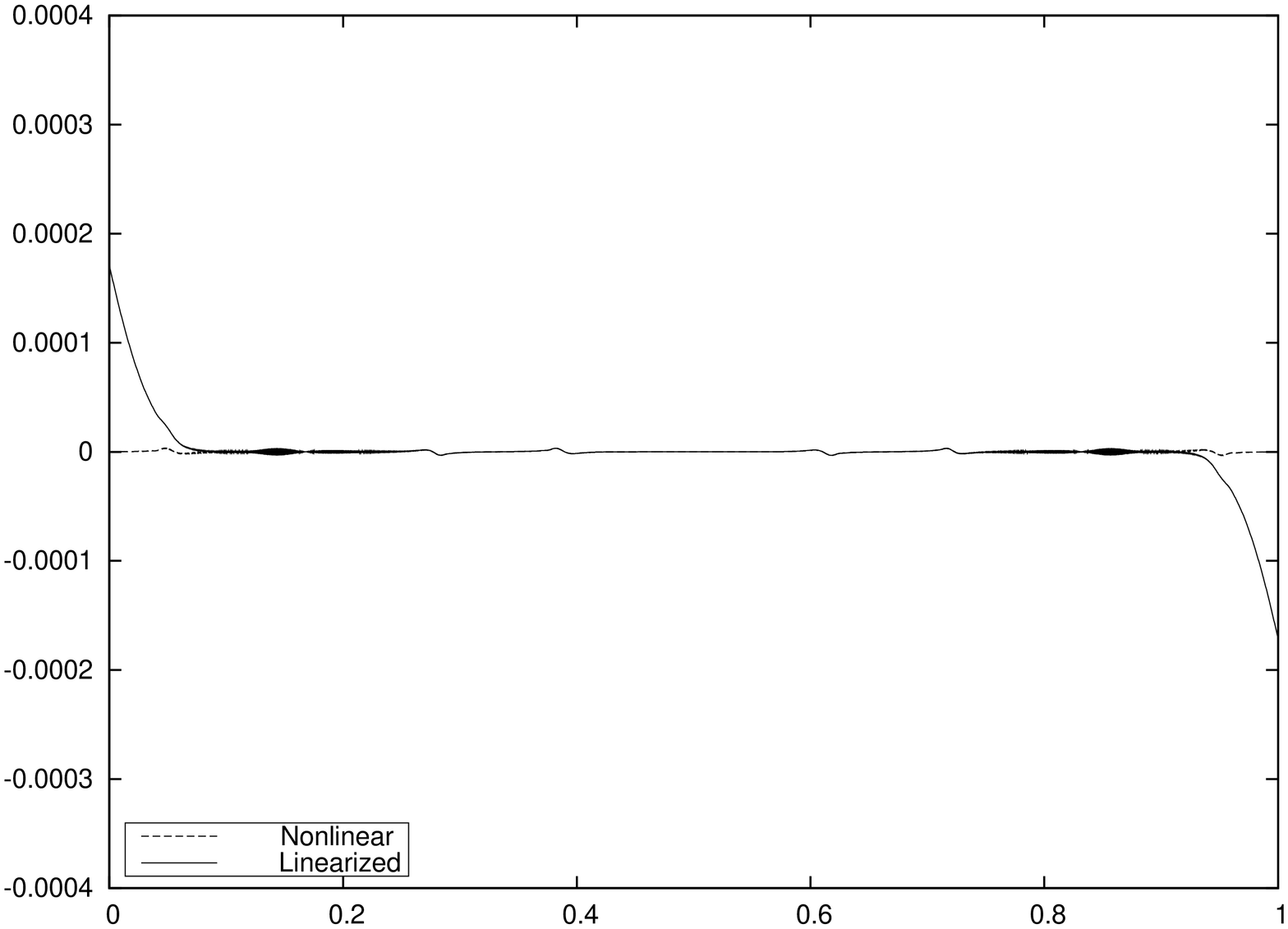}}\\
 		(d)\,\,\,\,$\eta_{lin.}$ vs. $\eta_{Nlin.}$ and $u_{lin.}$ vs. $u_{Nlin.}$ at $t=1.5$ 	
 	\end{center}
 	\end{figure}
 	\clearpage
\begin{figure}[h]
 	\begin{center} 	
 	\subfloat[]{\includegraphics[totalheight=2.7in,width=1.825in,angle=-90]{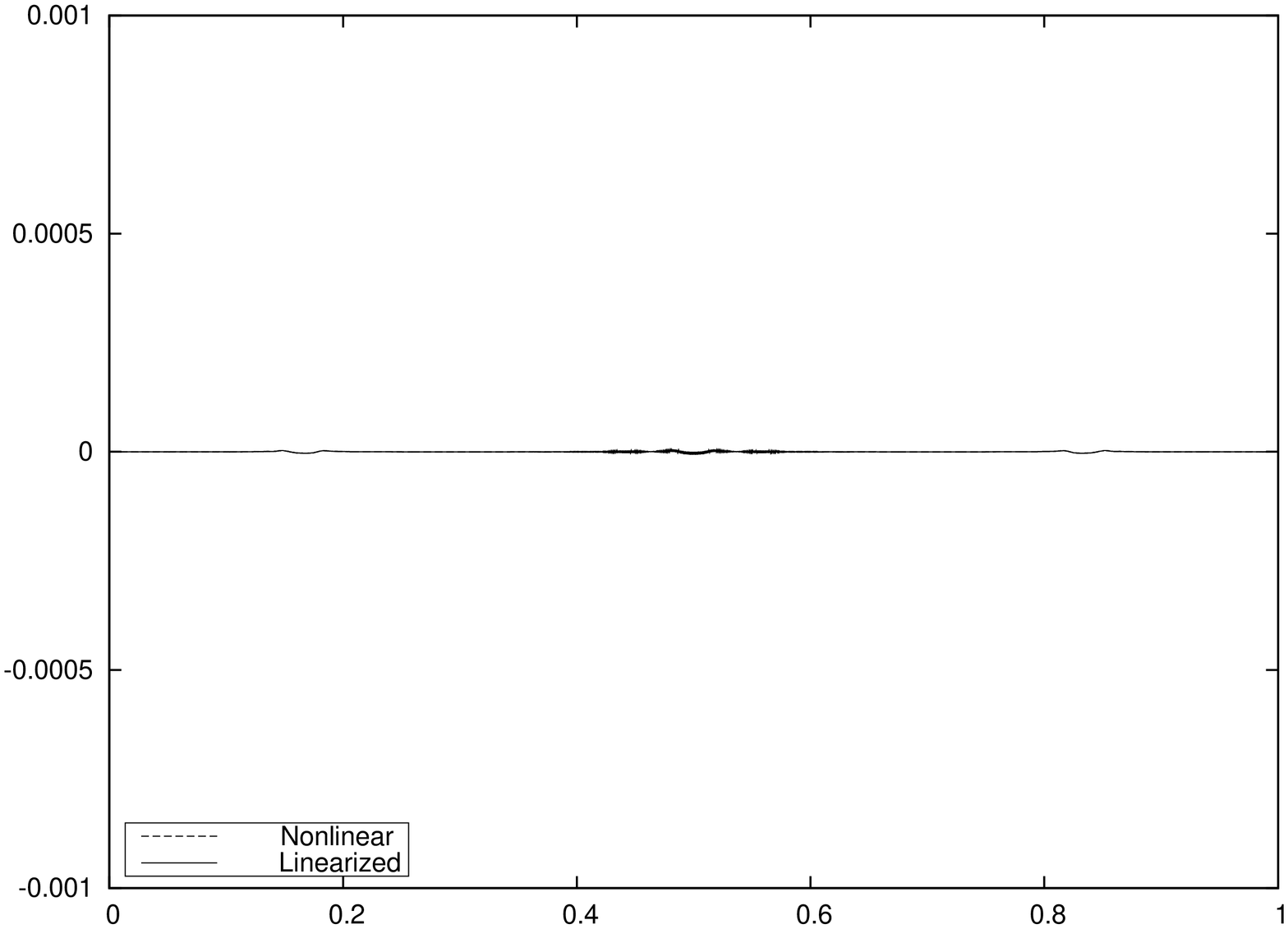}}  \quad \,
 		\subfloat[]{\includegraphics[totalheight=2.7in,width=1.825in,angle=-90]{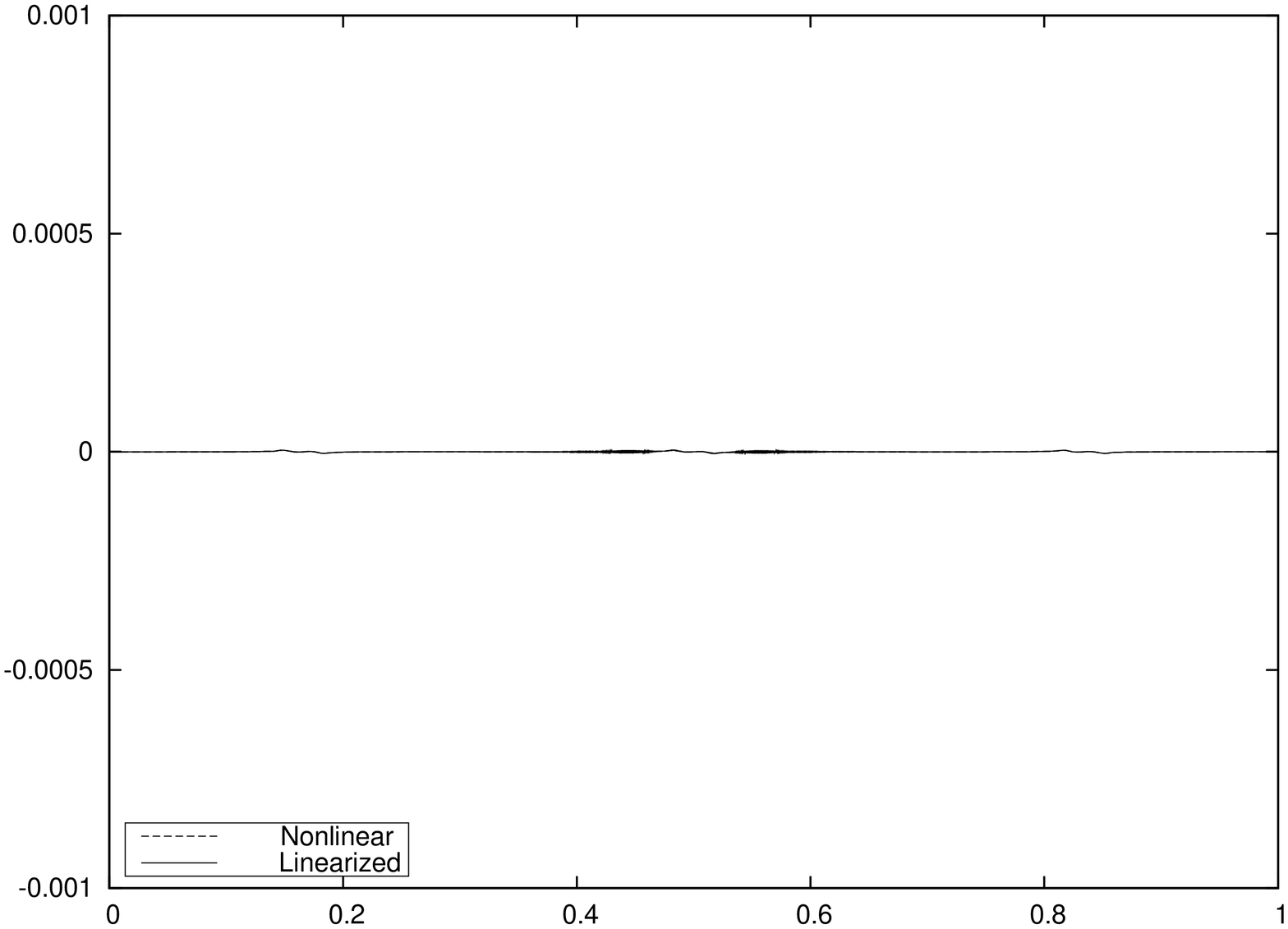}}\\
 		(e)\,\,\,\,$\eta_{lin.}$ vs. $\eta_{Nlin.}$ and $u_{lin.}$ vs. $u_{Nlin.}$ at $t=1.6$	
 	\end{center}
 	\caption{Reflected spurious pulses due to the linearized b.c.'s (solid lines) superimposed on the  
 		solution of (\ref{eq43})-(\ref{eq46}) (dotted lines); evolution resulting from $\eta^{0}=0.1\exp(-400(x-0.5)^{2})$, 
 		$u^{0}=0$, $\eta_{0}=u_{0}=0$.}
 	\label{fig45}
 \end{figure}
 \noindent
 and initial conditions
 $\eta^{0}(x)=0.1\exp (-400(x-0.5)^{2})$, $u^{0}(x)=0$, that we solve by the fully discrete Galerkin
 method corresponding to (\ref{eq43})-(\ref{eq44}) with nonlinear and linearized b.c.'s using $h=1/N$,
 $k=h/10$, $N=2000$. Due to symmetry the initial Gaussian $\eta^{0}(x)$ breaks up into two waves
 that travel with equal speeds in opposite directions and exit the computational domain at about
 $t=0.6$. Figure {\ref{fig45} show magnifications of the ensuing behavior of the numerical solution
 	computed with the linearized b.c.'s (solid line) superimposed on the correct steady state 
 	(dotted line), the result of computing with the nonlinear b.c.'s. Two equal spurious pulses are
 	reflected inwards at both ends and, according to linear theory, travel with unit speed, interact linearly, 
 	and exit cleanly by $t=1.6$. \par
 	We consider next the shallow water equations written in dimensional variables $(h,u)$ where $u$ 
 	is the horizontal velocity and $h$ denotes now the height of the water column above the bottom;
 	the latter is located at a depth $H$ below the level of rest. The system, posed on a channel of length
 	$2L$, is written, in the notation of \cite{nmf}, in the form
 	\begin{equation}
 		\begin{aligned}
 			h_{t} & + (hu)_{x} = 0, \\
 			u_{t} & + gh_{x} + uu_{x} = 0,
 		\end{aligned}
 		\quad -L \leq x\leq L, \quad t\geq 0, 
 		\label{eq49}
 	\end{equation}
 	with initial conditions
 	\begin{equation}
 		h(x,0) = f(x), \quad u(x,0)= v(x), \quad -L\leq x\leq L,
 		\label{eq410}
 	\end{equation}
 	where $g$ is the acceleration of gravity. The system is supplemented by the nonlinear characteristic 
 	boundary conditions now written as
 	\begin{align}
 		u(-L,t) & + 2\sqrt{gh(-L,t)} = a_{E}^{+}, 
 		\label{eq411} \\
 		u(L,t) & - 2\sqrt{gh(L,t)} = a_{E}^{-},
 		\label{eq412}
 	\end{align}
 	where $a_{E}^{\pm} = u_{0} \pm 2\sqrt{gh_{0}}$, and $u_{0}$, $h_{0}$ are the constant values 
 	of $u$ and $h$ outside $[-L,L]$, i.e. the `asymptotic' state of the flow. The linearized boundary
 	conditions (obtained by linearizing (\ref{eq411})-(\ref{eq412}) assuming that $h=H + \wt{h}$, where
 	$\wt{h}$ is small) are
 	\begin{align}
 		u(-L,t) & + \sqrt{\frac{g}{H}} h(-L,t) = b_{E}^{+},
 		\label{eq413}\\
 		u(L,t) & - \sqrt{\frac{g}{H}} h(L,t) = b_{E}^{+},
 		\label{eq414}
 	\end{align}
 	where $b_{E}^{\pm} = u_{0} \pm \sqrt{\frac{g}{H}}h_{0}$. We repeat the numerical experiment 
 	in section 4.1 of \cite{nmf}, using now the fully discrete Galerkin method with piecewise linear 
 	continuous functions in space and the 4$^{th}$-order classical RK method in time, and taking
 	$L=2\,\mathrm{m}$, $H=0.2\,\mathrm{m}$, $g=9.8\,\mathrm{m}/\mathrm{sec}^{2}$, 
 	$h_{0}=0.2\,\mathrm{m}$, $u_{0}=0$, i.e. $a_{E}^{\pm} = \pm 2.8\,\mathrm{m}/\mathrm{sec}$, 
 	$b_{E}^{\pm} = \pm 1.4\,\mathrm{m}/\mathrm{sec}$, and initial values $f(x) = 0.25\,\mathrm{m}$
 	and $v(x) = 0$, with discretization parameters $\Delta x=4/N$, $k=\Delta t=1/(10N)$, $N=8000$.
 	We first solve the problem with the nonlinear b.c.'s (\ref{eq411}), (\ref{eq412}). The evolution of the
 	waveheight $h$ is depicted in Fig. 4.6. Due to the difference of the initial profile $h(x,0)=f(x)$ 
 	\captionsetup[subfloat]{labelformat=empty,position=bottom,singlelinecheck=true}
 	\begin{figure}[h]
 		\begin{center}
 			\subfloat[$t=0.1$]{\includegraphics[totalheight=2.7in,width=1.825in,angle=-90]{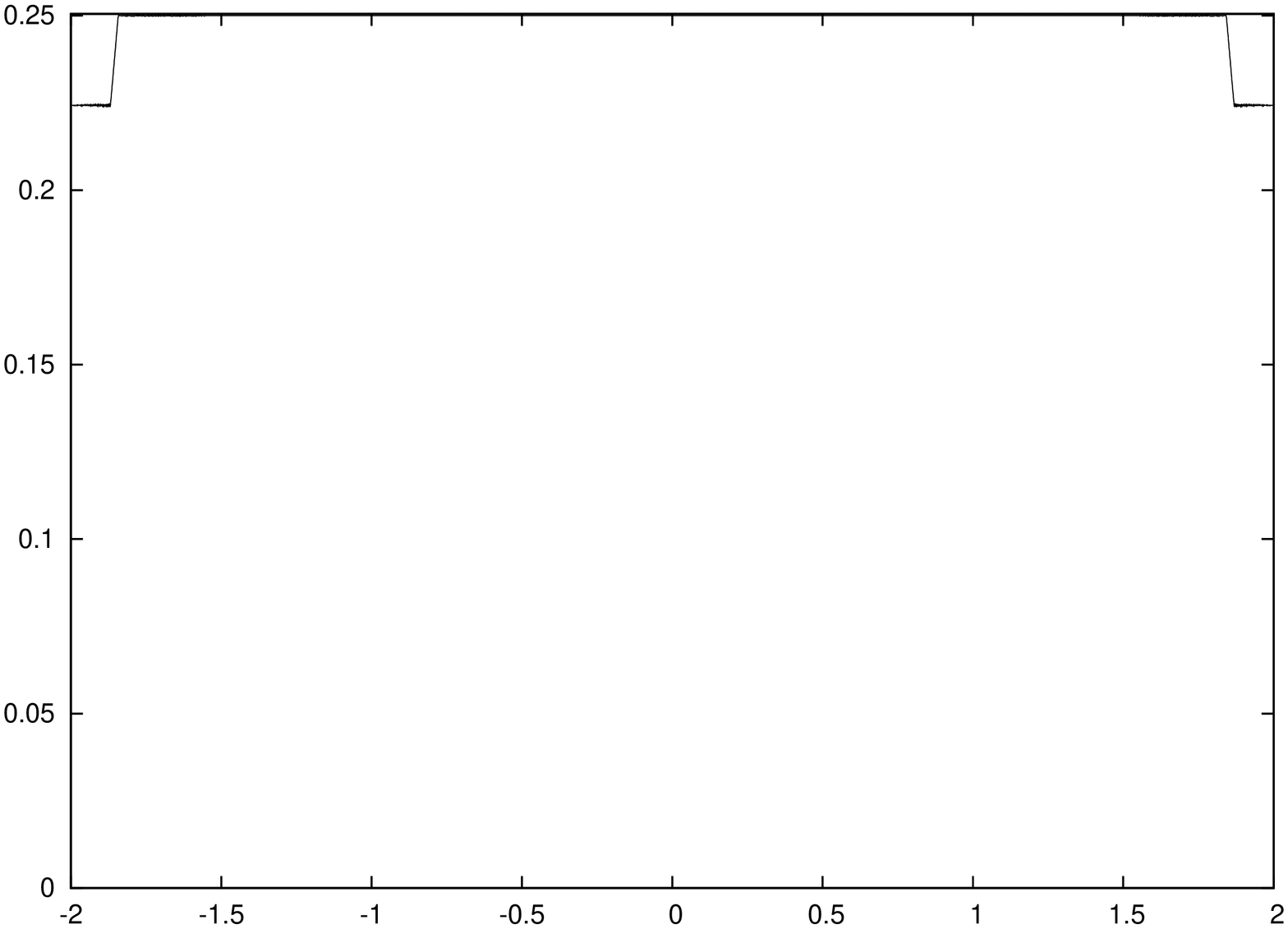}}\qquad
 			\subfloat[$t=0.5$]{\includegraphics[totalheight=2.7in,width=1.825in,angle=-90]{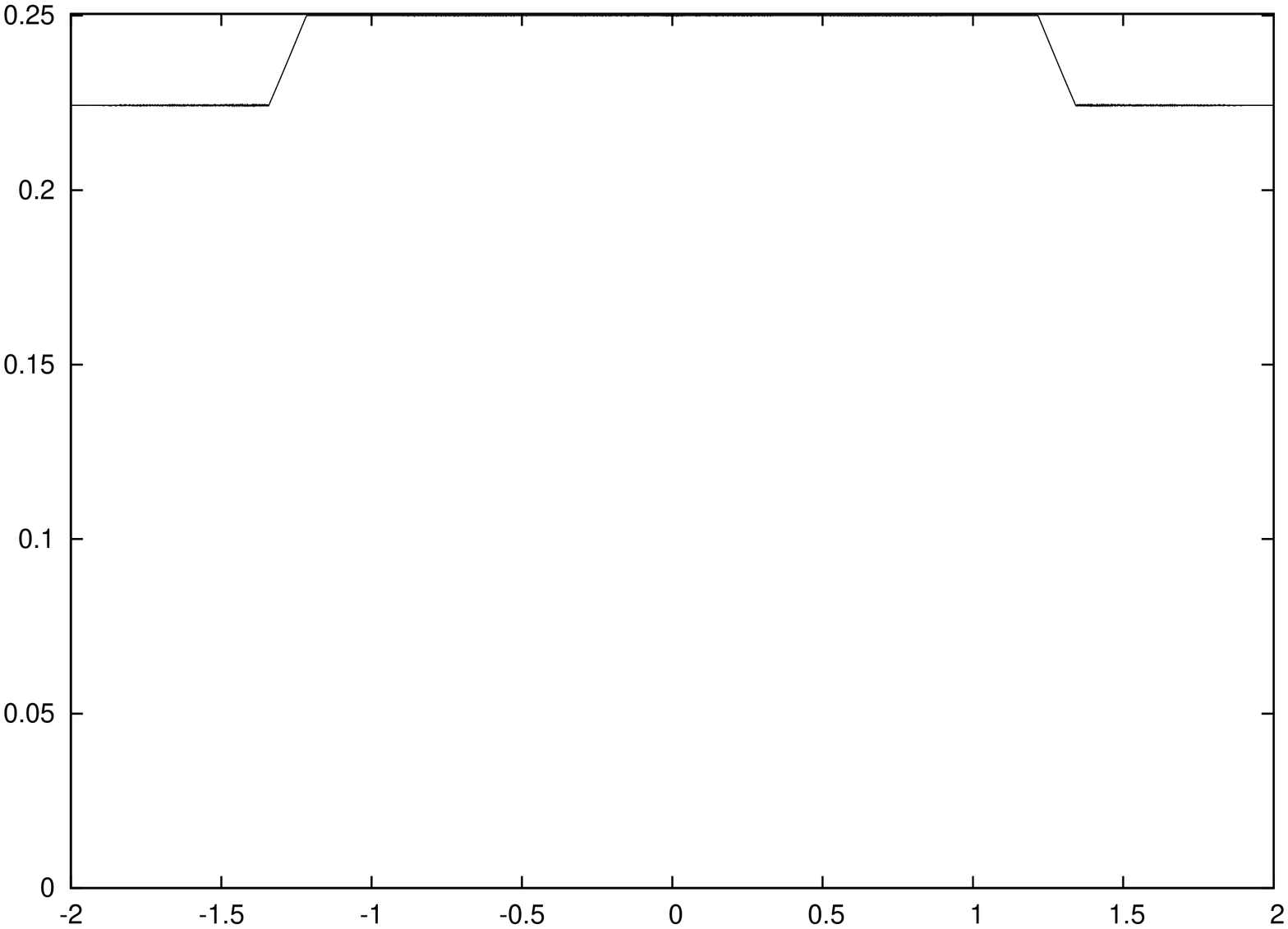}} \\
 			\subfloat[$t=1.0$]{\includegraphics[totalheight=2.7in,width=1.825in,angle=-90]{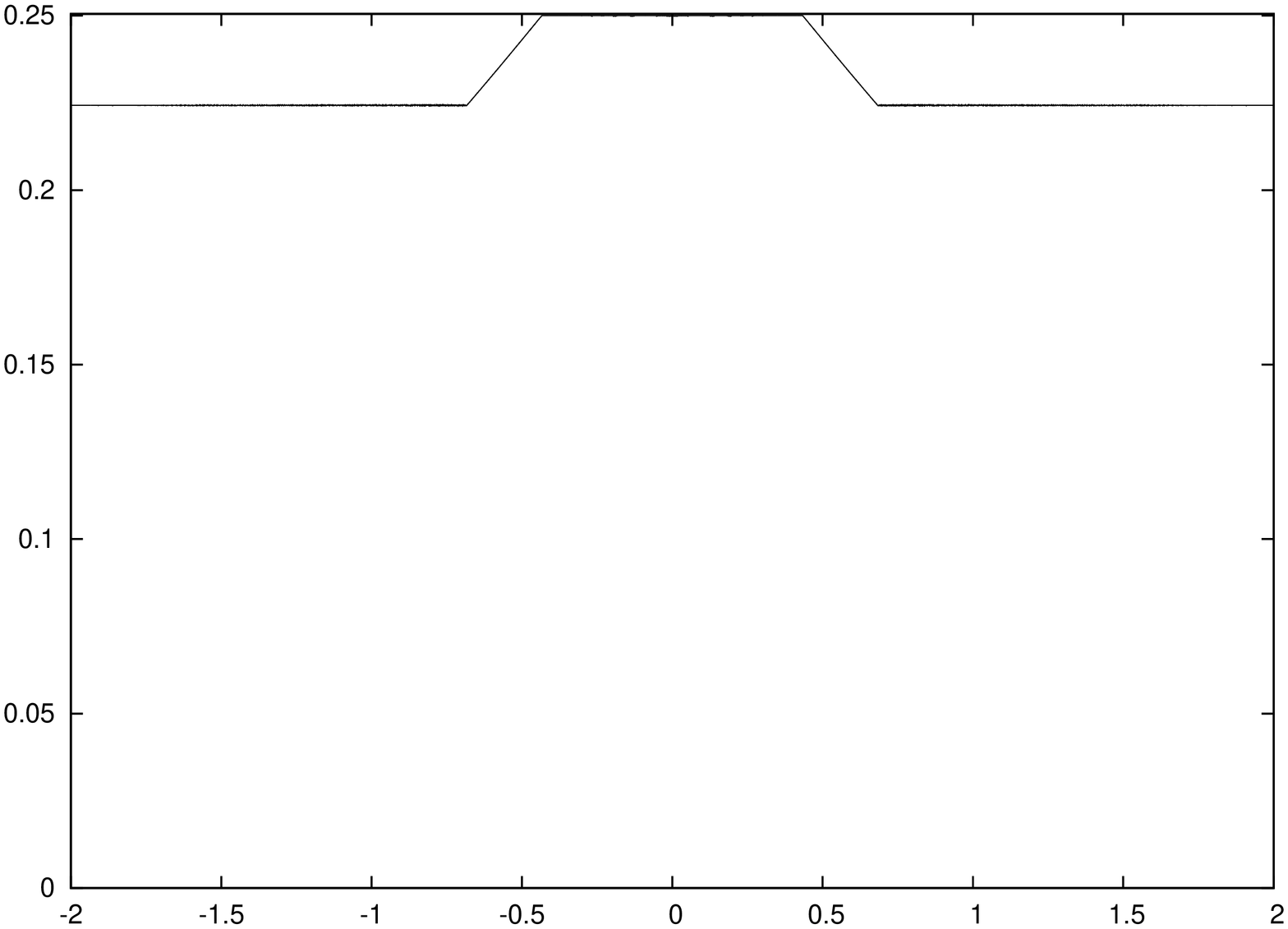}}\qquad
 			\subfloat[$t=1.5$]{\includegraphics[totalheight=2.7in,width=1.825in,angle=-90]{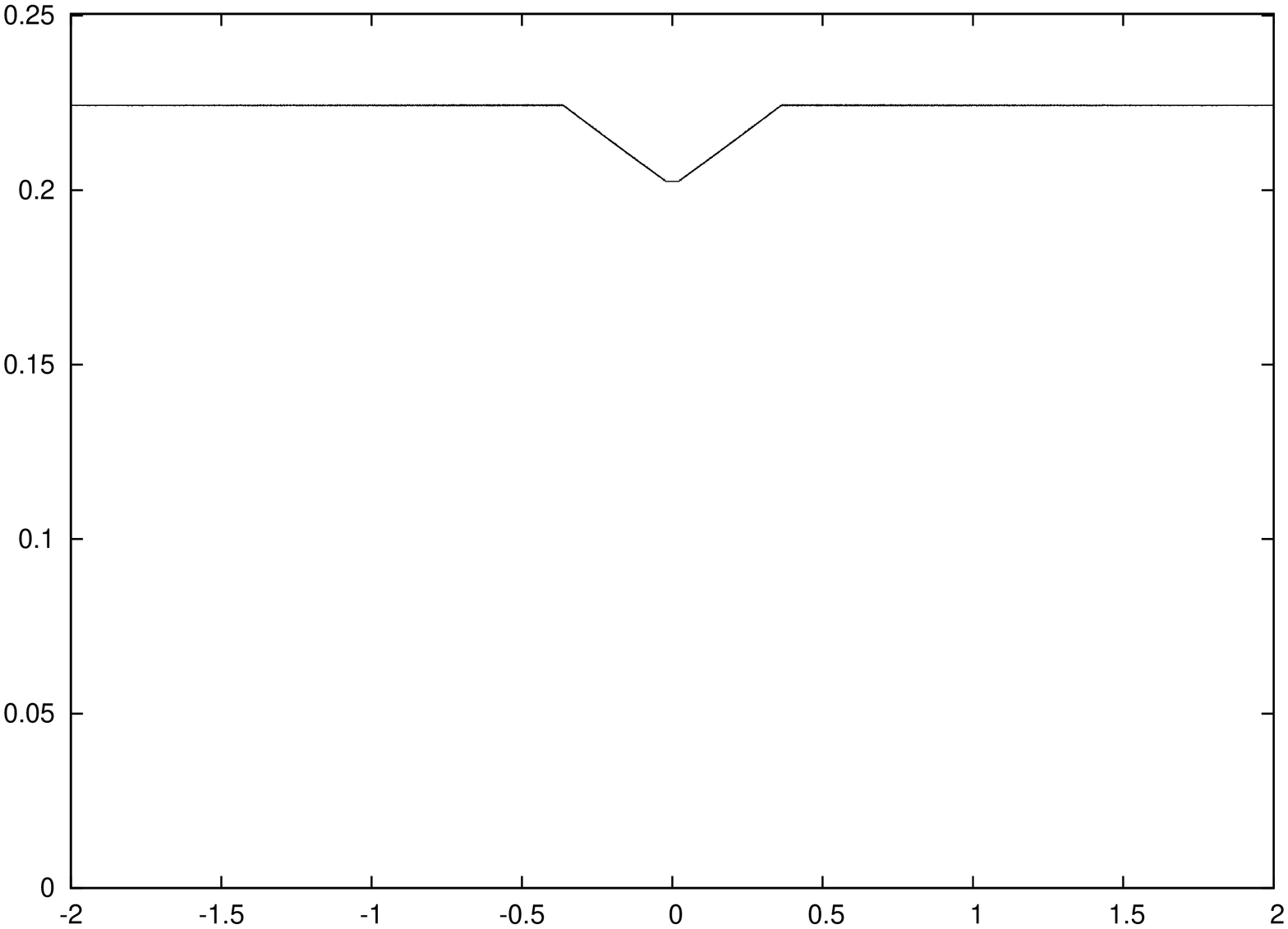}}	\\
 			\subfloat[$t=2.0$]{\includegraphics[totalheight=2.7in,width=1.825in,angle=-90]{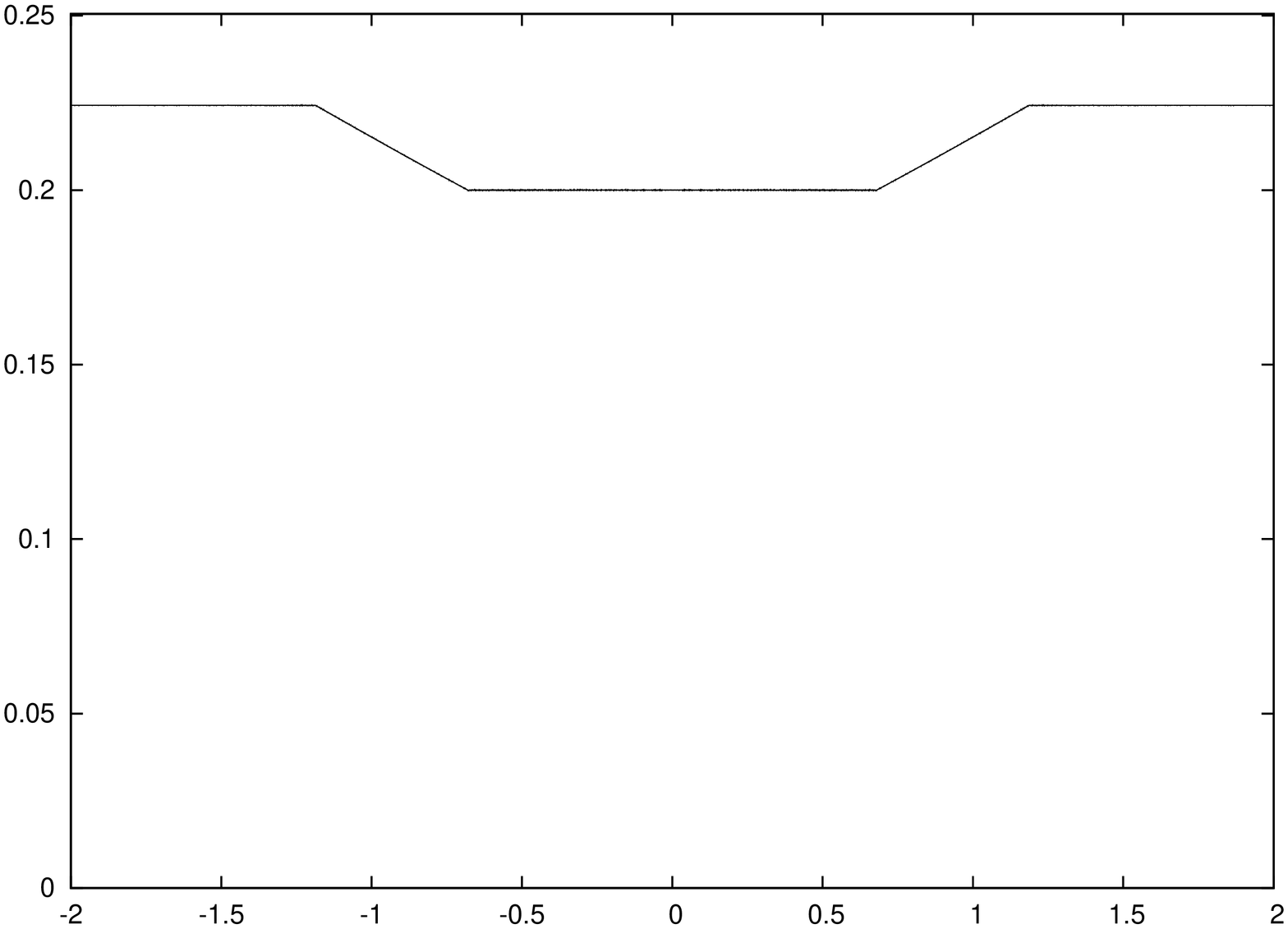}}\qquad
 			\subfloat[$t=3.0$]{\includegraphics[totalheight=2.7in,width=1.825in,angle=-90]{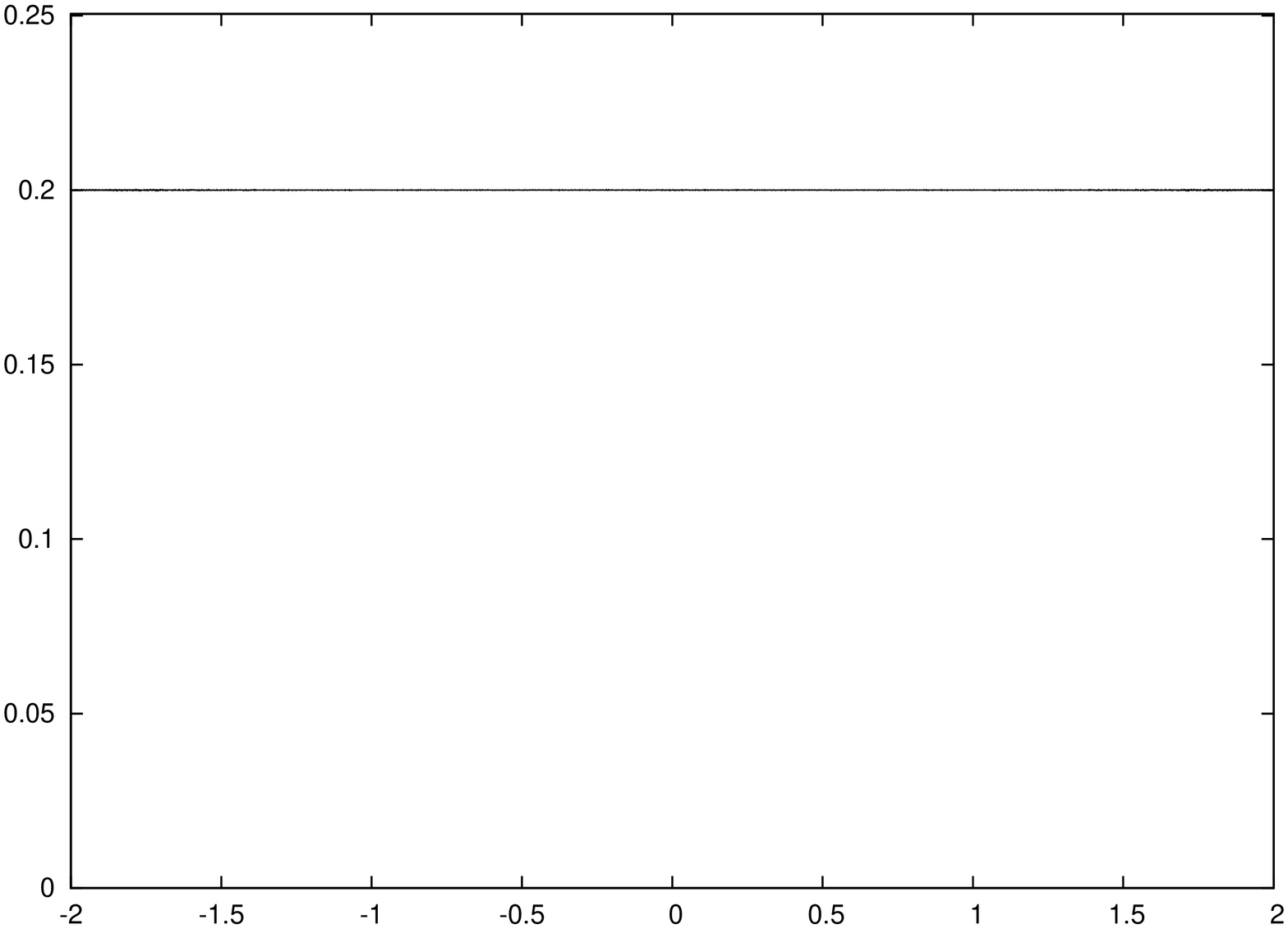}}	
 		\end{center}
 		\caption{Evolution of waveheight $h(x,t)$; numerical solution of (\ref{eq49})-(\ref{eq412}) with 
 			$f(x)=0.25\,\mathrm{m}$, $v(x)=0$, $h_{0}=0.2\,\mathrm{m}$, $u_{0}=0$.}
 		\label{fig46}
 	\end{figure}	
 	\clearpage
 	\begin{figure}[h]
 		\begin{center}		
 			\subfloat[]{\includegraphics[totalheight=2.7in,width=1.825in,angle=-90]{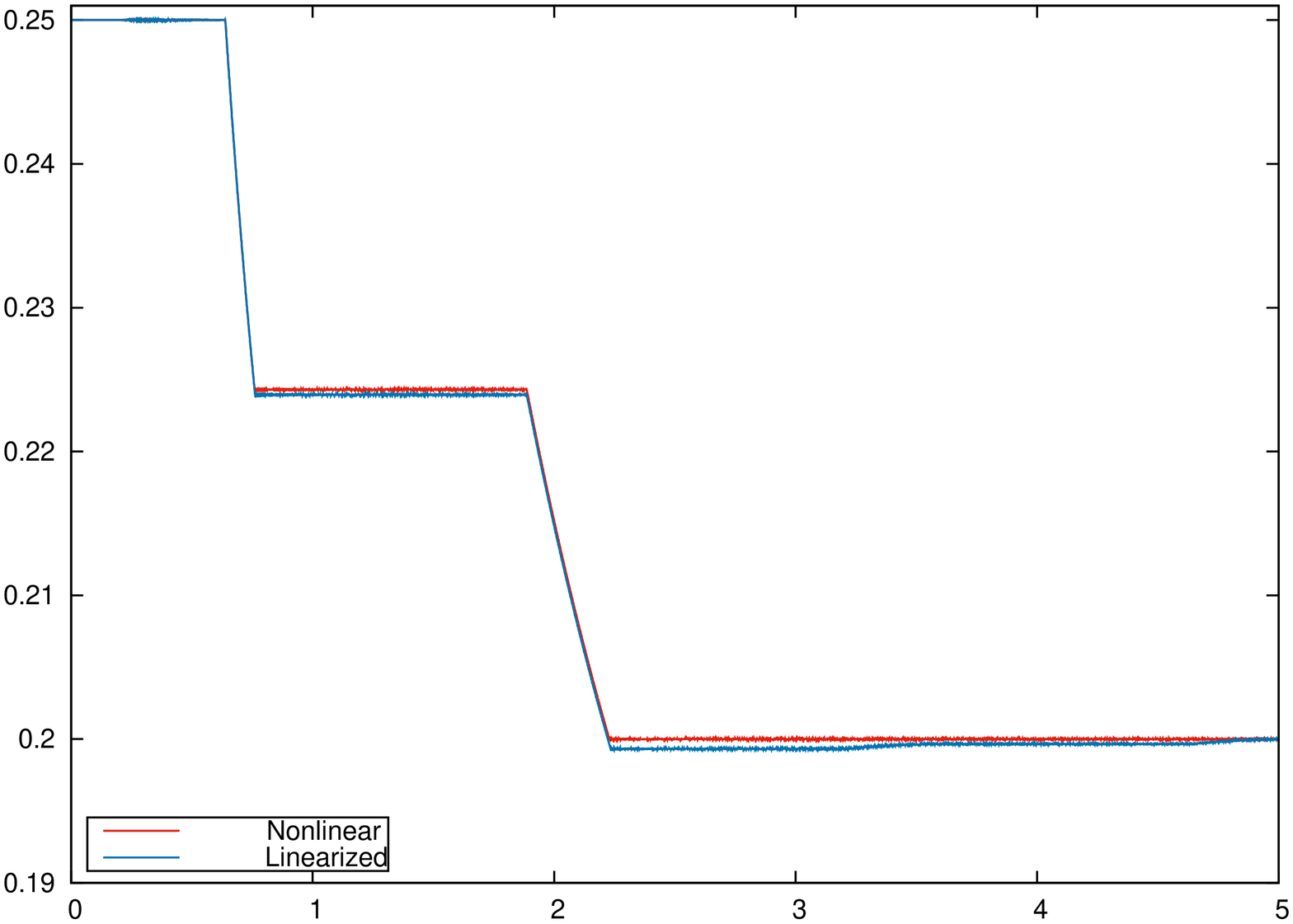}}\qquad
 			\subfloat[]{\includegraphics[totalheight=2.7in,width=1.825in,angle=-90]{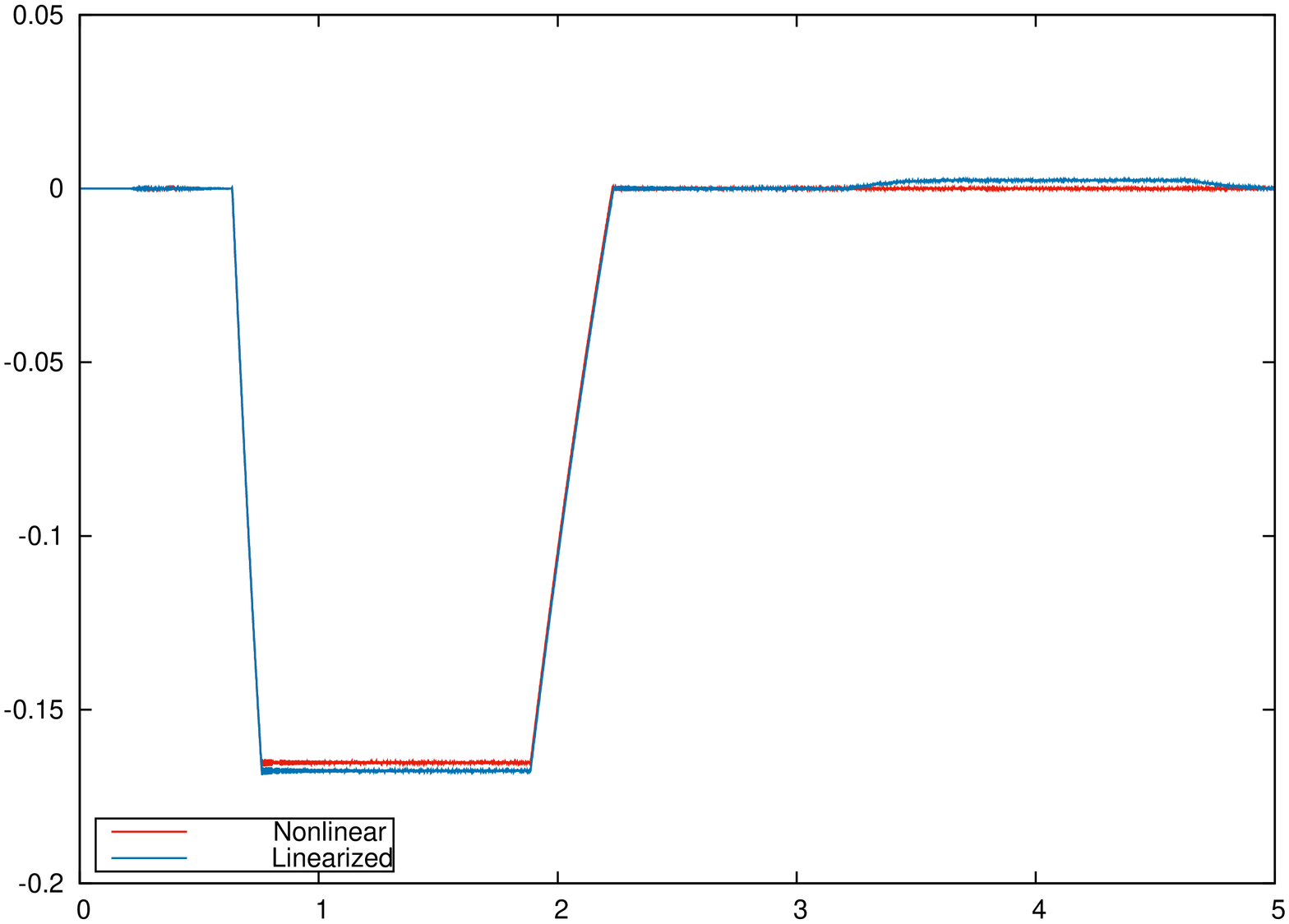}}
 		\end{center}
 		\caption{Graphs of $h(-1,t)$ and $u(-1,t)$ vs. time. Evolution of Figure \ref{fig46}, nonlinear and
 			linearized characteristic b.c.'s}
 		\label{fig47}
 	\end{figure}
 	\noindent 
 	at $x=\pm2$ with the asymptotic state $h_{0}$, fluid exits the domain from both boundary points 
 	and two wavefronts are created that travel in opposite directions, interact, and completely exit
 	the computational domain at about $t=2.25\,\mathrm{sec}$ with no apparent reflections. After this time
 	the asymptotic state  $h_{0}=0.2\,\mathrm{m}$ is achieved throughout the domain. The time history of the
 	solution $(h,u)$ at $x=-1$ is given in Figure \ref{fig47} and corresponds to the three stages of the
 	evolution deduced from Figure \ref{fig46}. The graph of the numerical solution corresponding to the 
 	linearized b.c.'s (\ref{eq413})--(\ref{eq414}) is superimposed on the previous graph. It may be seen 
 	that solution with the linearized b.c.'s is just slightly less absorbing. These results are in agreement 
 	with those of \cite{nmf}. \par
 	In our final numerical experiment we solve again the SW on $[-L,L]$ in their dimensional form 
 	(\ref{eq49}), (\ref{eq410}) with  the nonlinear b.c.'s (\ref{eq411}), (\ref{eq412}) and the linearized
 	ones (\ref{eq413}), (\ref{eq414}), taking now $L=1\,\mathrm{m}$, $H=0.2\,\mathrm{m}$,
 	$g=9.8\,\mathrm{m}/\mathrm{sec}^{2}$, $h_{0}=0.2\,\mathrm{m}$, $u_{0}=0$, i.e.
 	$a_{E}^{\pm}=\pm 2.8 \,\mathrm{m}/\mathrm{sec}$, $b_{E}^{\pm} = \pm 1.4\,\mathrm{m}/\mathrm{sec}$,
 	and as initial values in (\ref{eq410}) $v(x)=0$ and $f(x)$ given by a half-sine pulse centered at $x=0$ and
 	equal to $0.2 + 0.05\sin (\pi (x+0.3)/0/6)\,\mathrm{m}$ if $\abs{x}\leq 0.3\,\mathrm{m}$ and to
 	$0.2\,\mathrm{m}$ for $\abs{x}>0.3\,\mathrm{m}$. The mesh parameters are $\Delta x=2/N$,
 	$N=4000$ and $\Delta t=1/(10N)$. Figure \ref{fig48} shows the ensuing  evolution of the  numerical solution
 	($h$ is shown on the left and $u$ on the right at each temporal instance.) The graphs corresponding to the
 	nonlinear characteristic b.c.'s (dotted line) are superimposed on those computed with the linearized b.c.'s.
 	Two pulses are produced (with no oscillations apparent at the points where the spatial derivative is 
 	discontinuous) that travel to opposite directions and exit the computational interval $[-1,1]$ with no
 	reflections in the case of the nonlinear b.c.'s, shortly after $t=0.8\,\mathrm{sec}$. The linearized b.c.'s
 	give spurious reflected pulses with waveheights of amplitude of $O(10^{-4})$     
 	\begin{figure}[h]
 		\begin{center}
 			\subfloat[]{\includegraphics[totalheight=2.7in,width=1.825in,angle=-90]{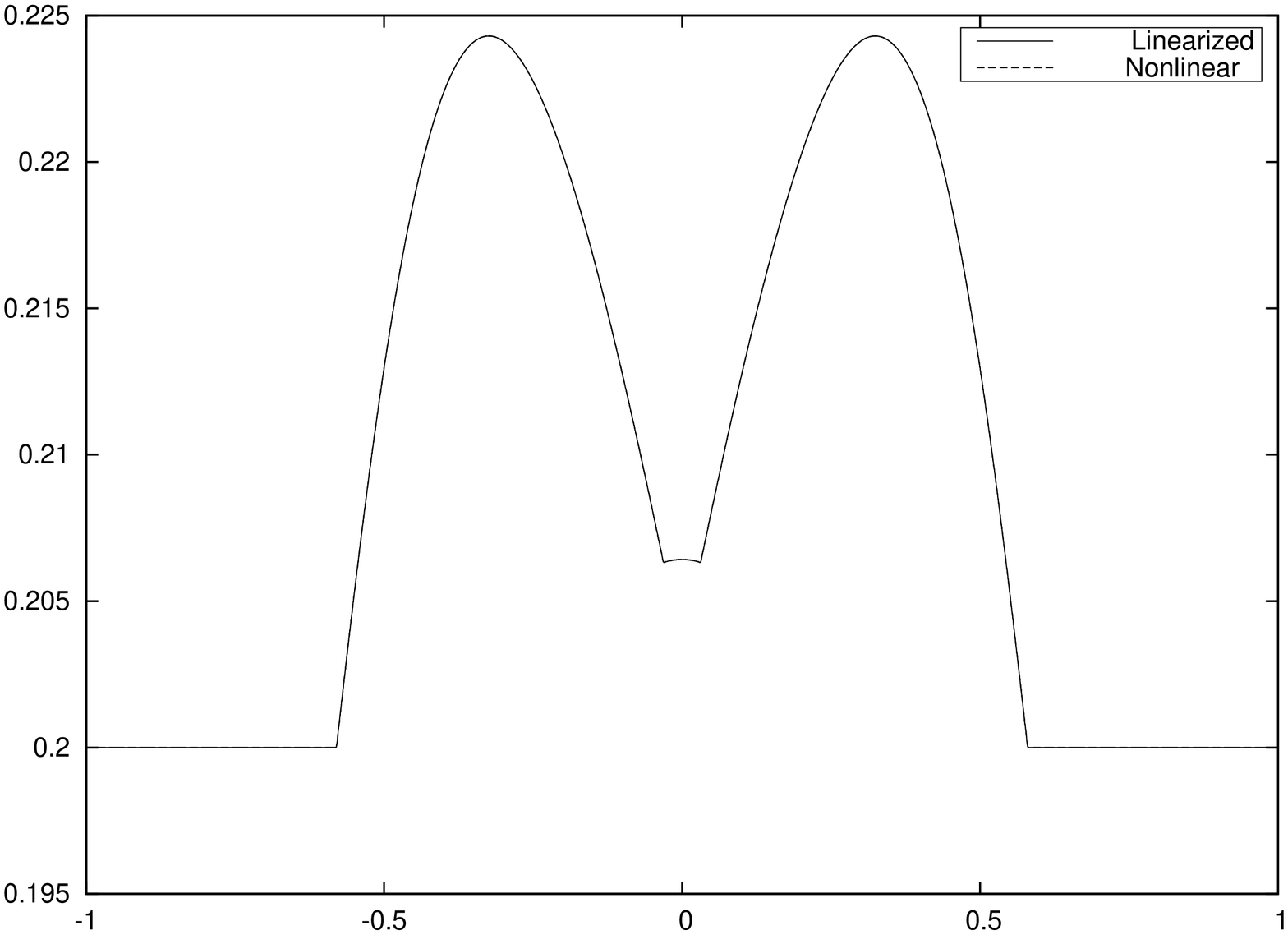}}\qquad
 			\subfloat[]{\includegraphics[totalheight=2.7in,width=1.825in,angle=-90]{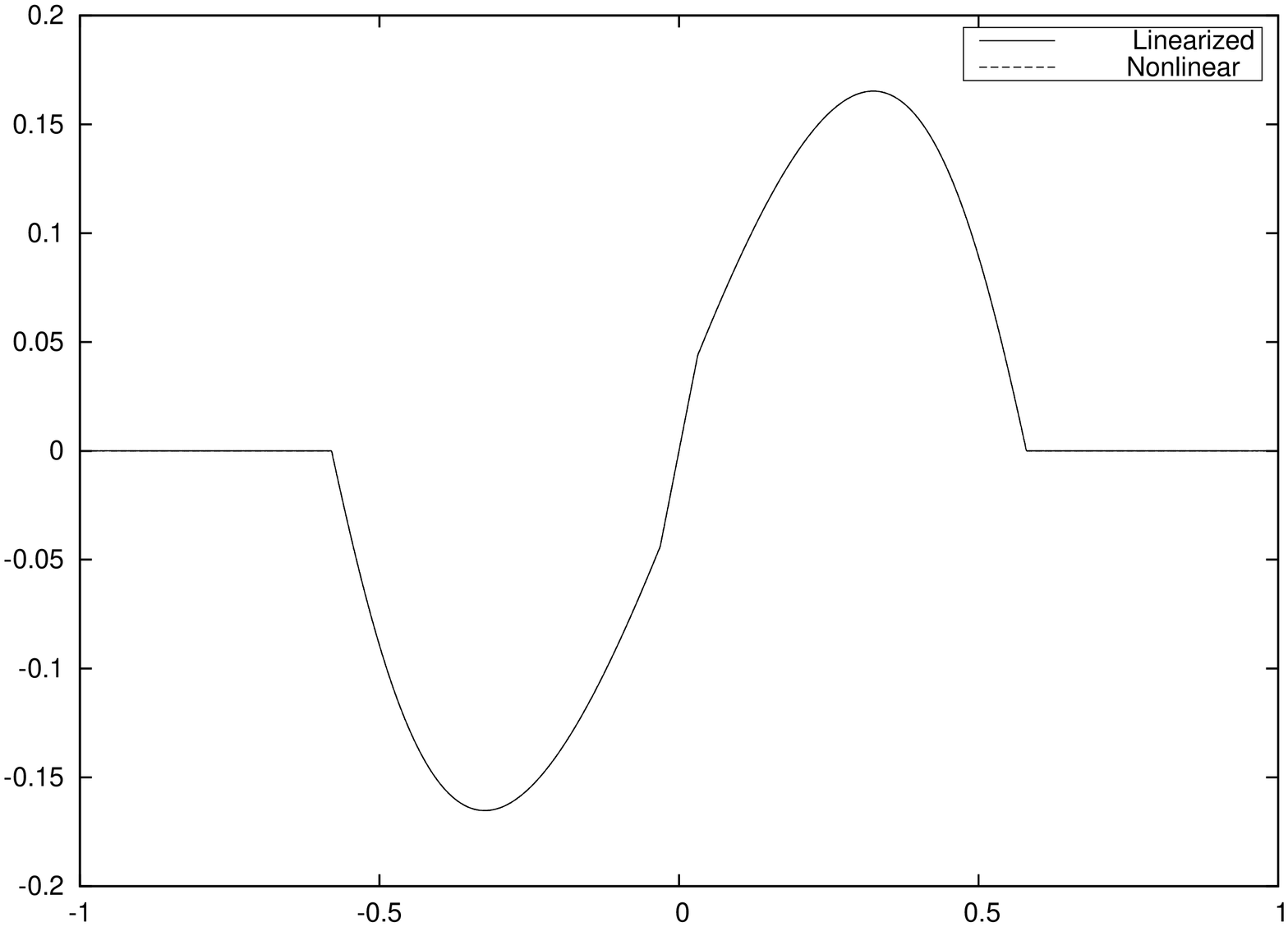}} \\
 			(a)\,\,\,\,$h$ and $u$ at $t=0.2$ %
 		\end{center}
 	\end{figure}
 	\clearpage
 	\begin{figure}[h]
 		\begin{center}
 	    	\subfloat[]{\includegraphics[totalheight=2.7in,width=1.825in,angle=-90]{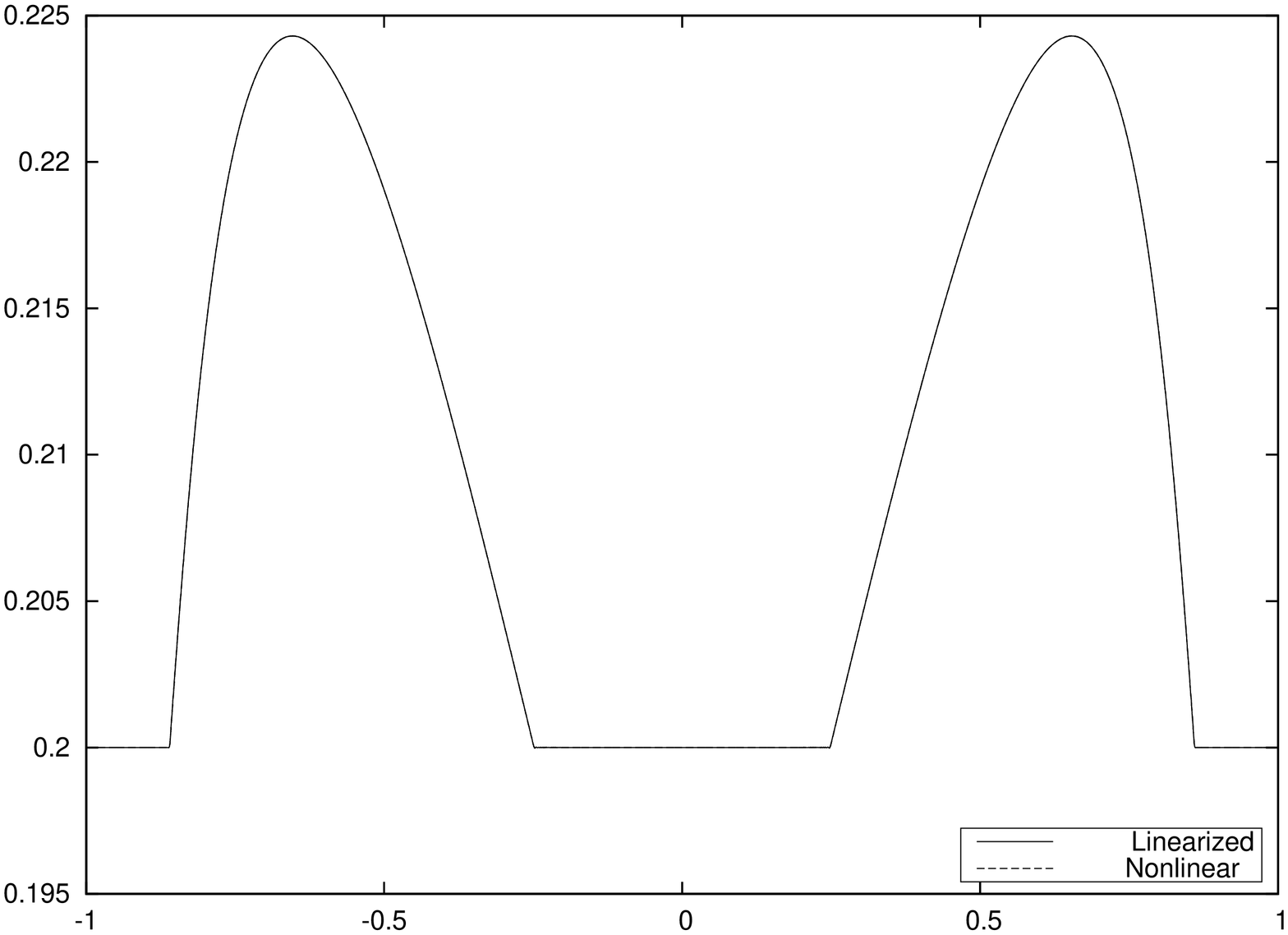}}\qquad
 			\subfloat[]{\includegraphics[totalheight=2.7in,width=1.825in,angle=-90]{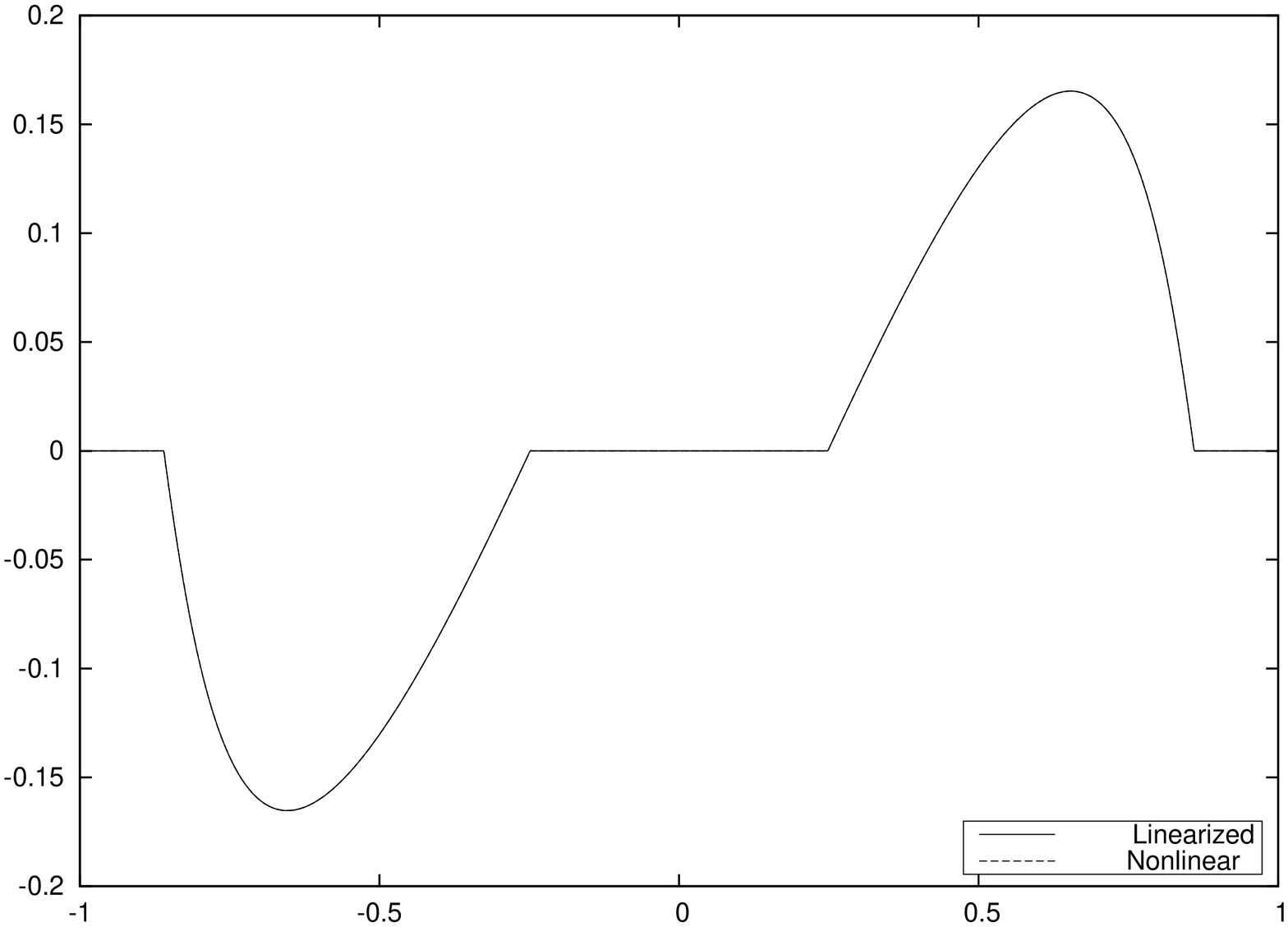}}\\
 			(b)\,\,\,\,$h$ and $u$ at $t=0.4$ \vspace{17pt} \\
 			\subfloat[]{\includegraphics[totalheight=2.7in,width=1.825in,angle=-90]{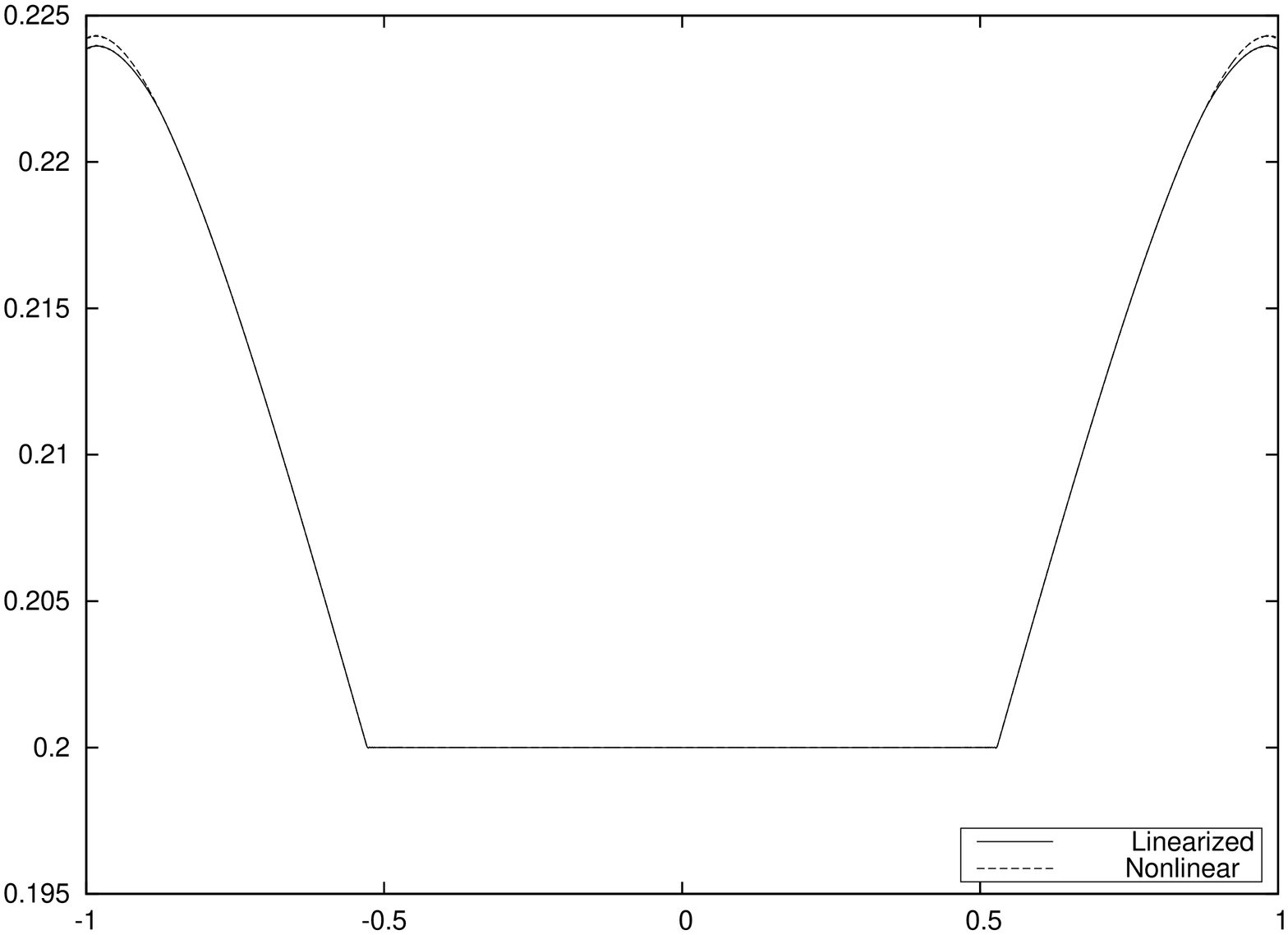}}\qquad
 			\subfloat[]{\includegraphics[totalheight=2.7in,width=1.825in,angle=-90]{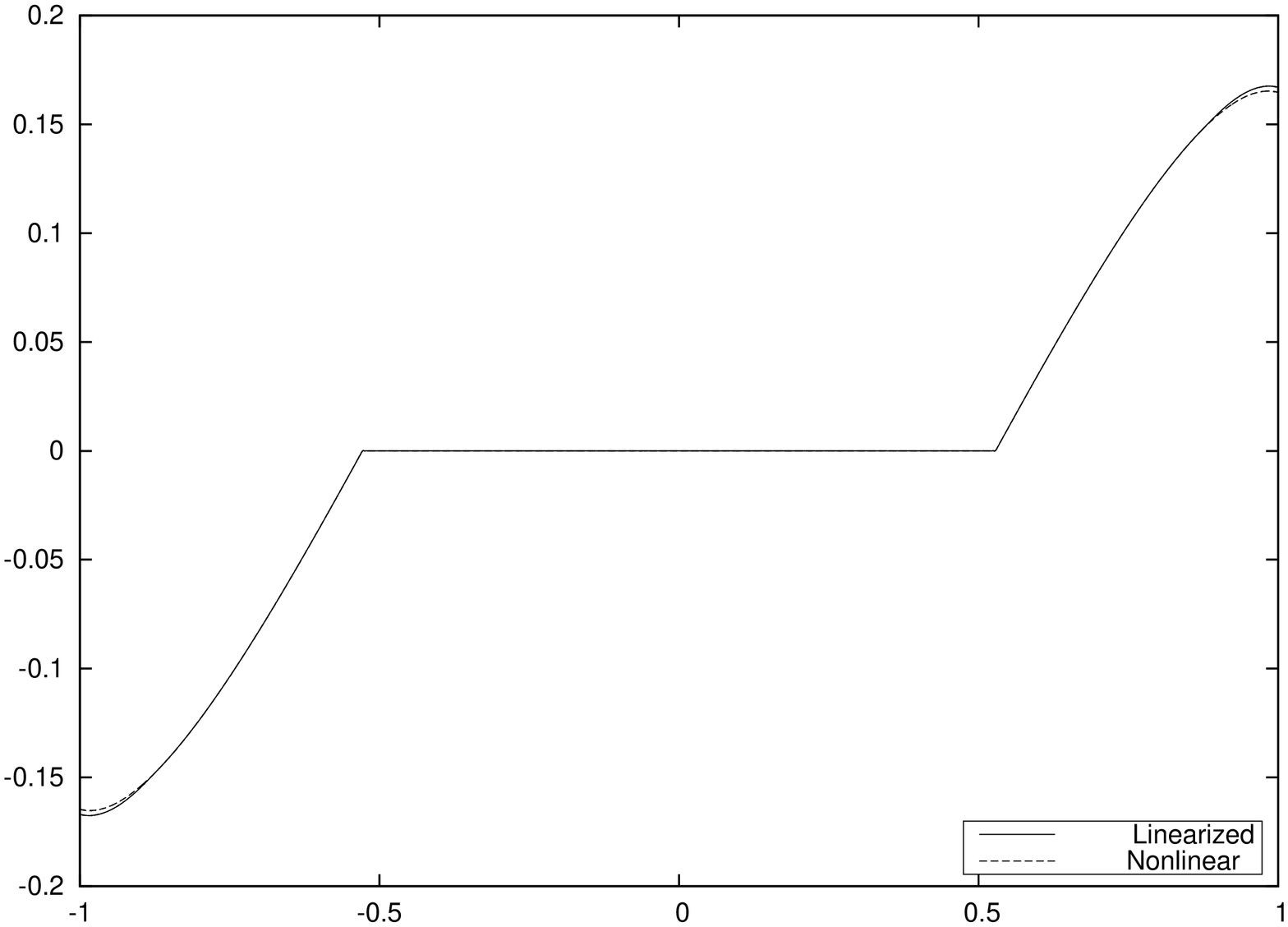}} \\
 			(c)\,\,\,\,$h$ and $u$ at $t=0.6$  	\vspace{17pt} \\
 			\subfloat[]{\includegraphics[totalheight=2.7in,width=1.825in,angle=-90]{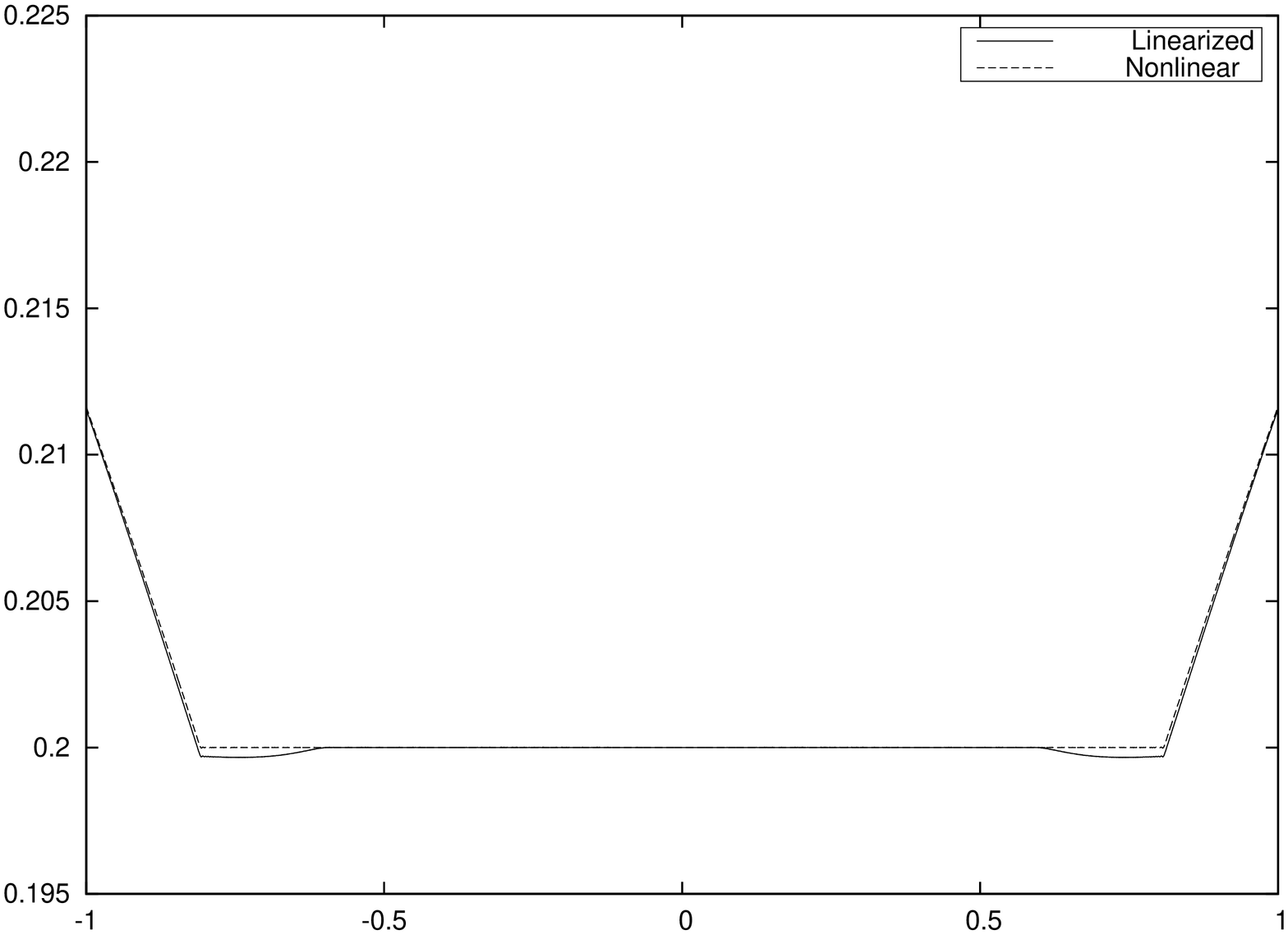}}\qquad
 			\subfloat[]{\includegraphics[totalheight=2.7in,width=1.825in,angle=-90]{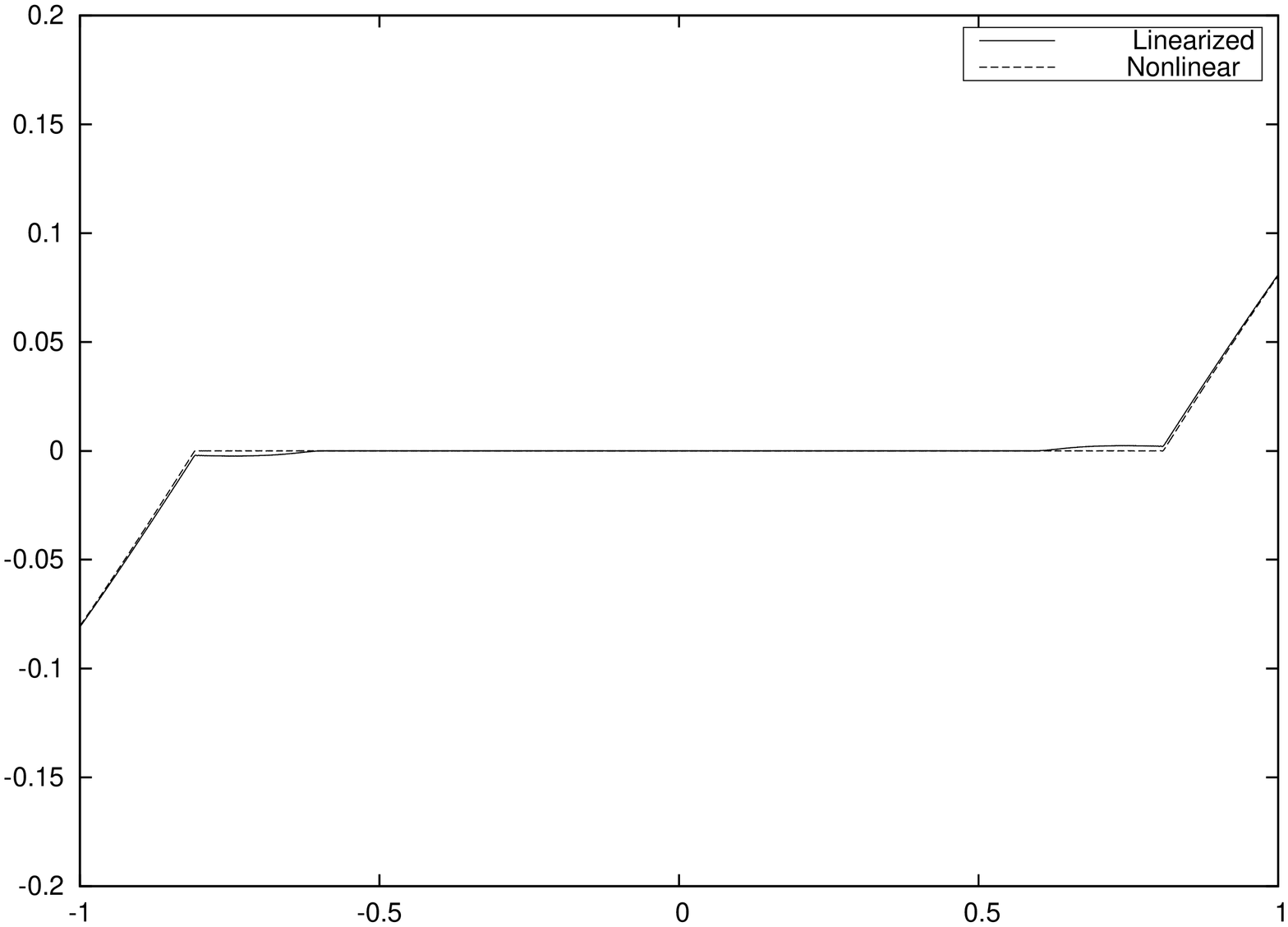}} \\
 			(d)\,\,\,\,$h$ and $u$ at $t=0.8$  %
 		\end{center}
 	\end{figure}
 	\clearpage
 	\begin{figure}[h]
 		\begin{center}
 			\subfloat[]{\includegraphics[totalheight=2.7in,width=1.825in,angle=-90]{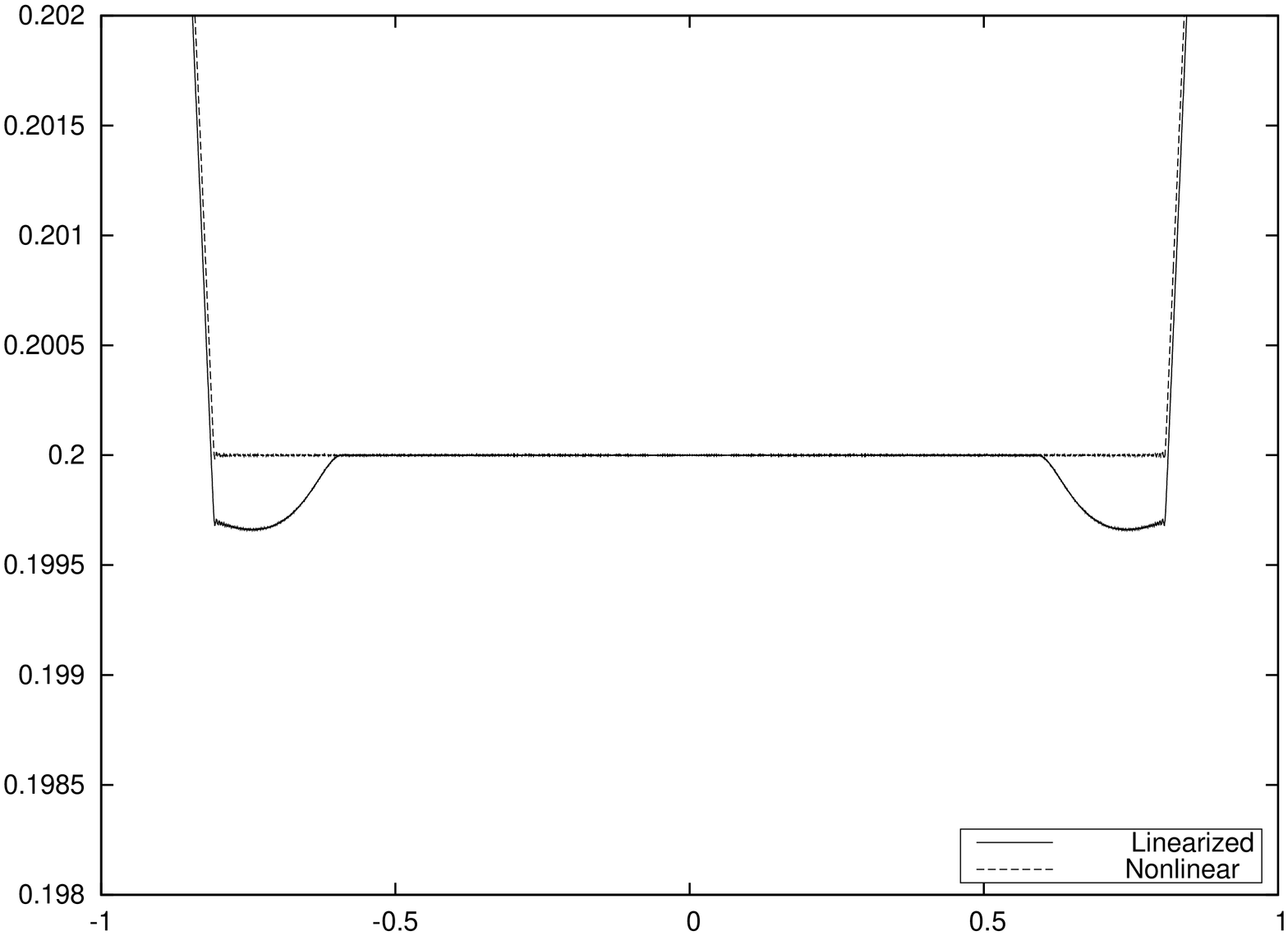}}\qquad
 			\subfloat[]{\includegraphics[totalheight=2.7in,width=1.825in,angle=-90]{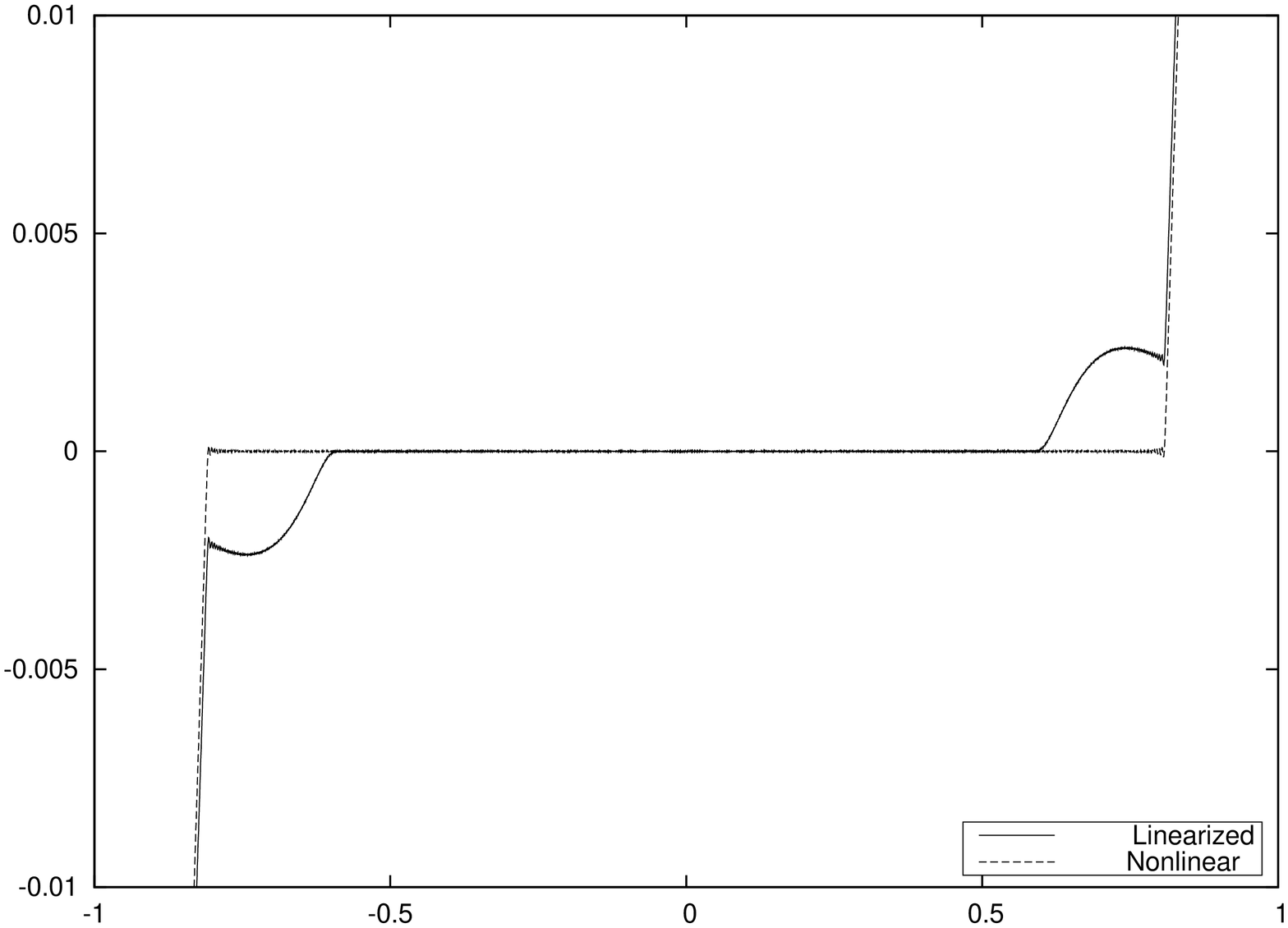}} \\
 			(e)\,\,\,\,magnification of (d) \vspace{17pt} \\
 			\subfloat[]{\includegraphics[totalheight=2.7in,width=1.825in,angle=-90]{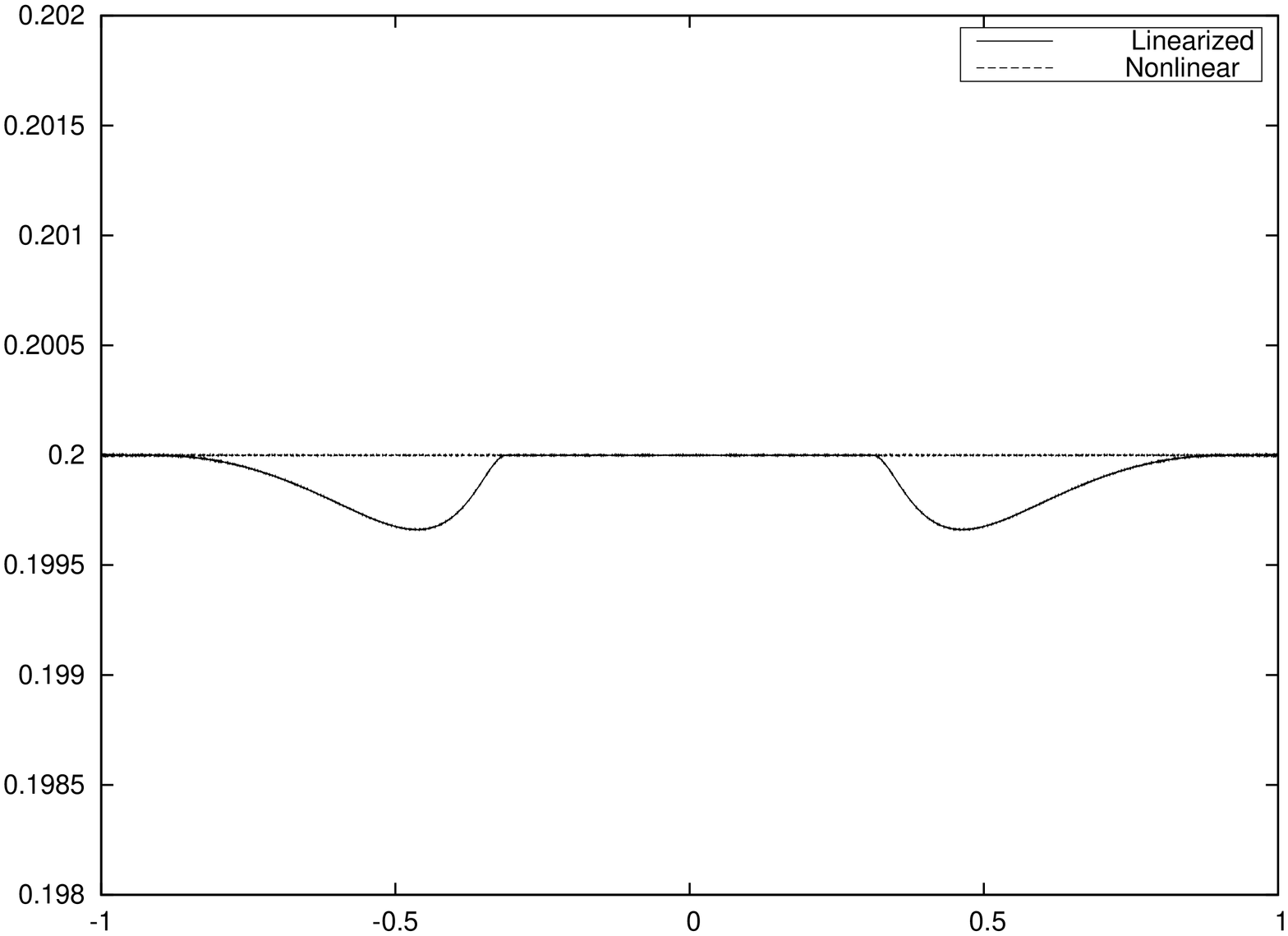}}\qquad
 			\subfloat[]{\includegraphics[totalheight=2.7in,width=1.825in,angle=-90]{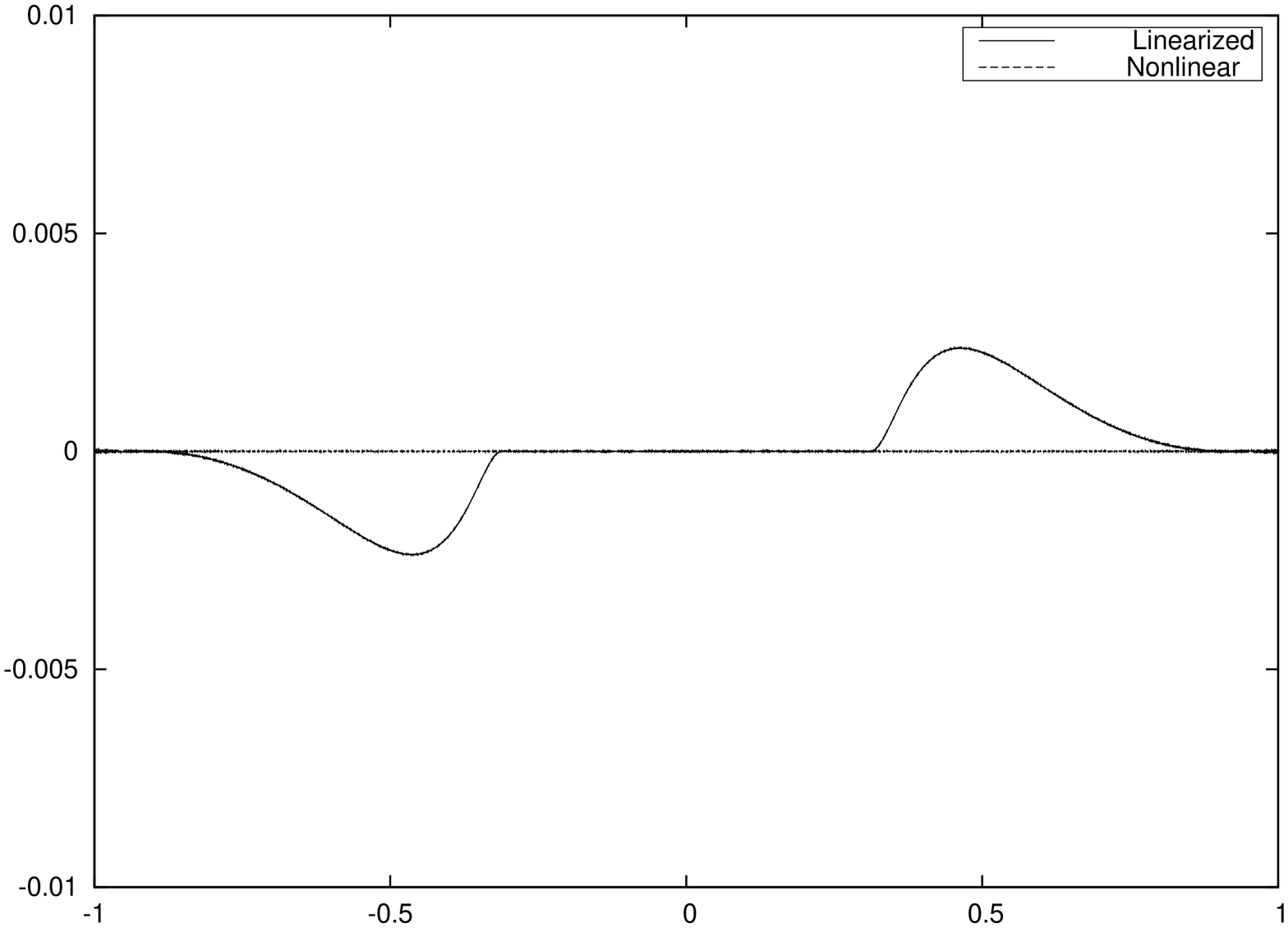}}\\
 			(f)\,\,\,\,$h$ and $u$ at $t=1.0$ 	\vspace{17pt} \\
 			\subfloat[]{\includegraphics[totalheight=2.7in,width=1.825in,angle=-90]{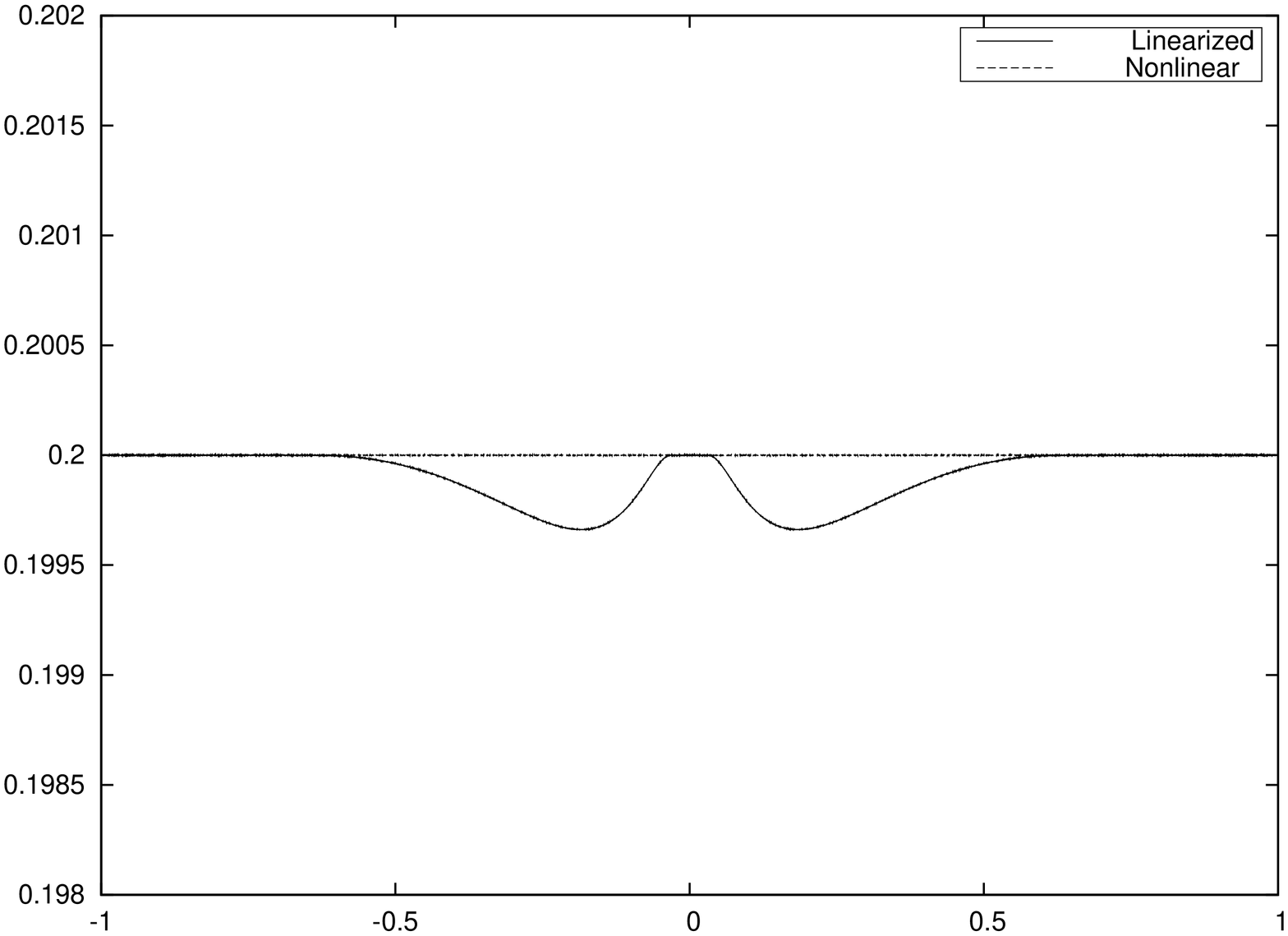}}\qquad
 			\subfloat[]{\includegraphics[totalheight=2.7in,width=1.825in,angle=-90]{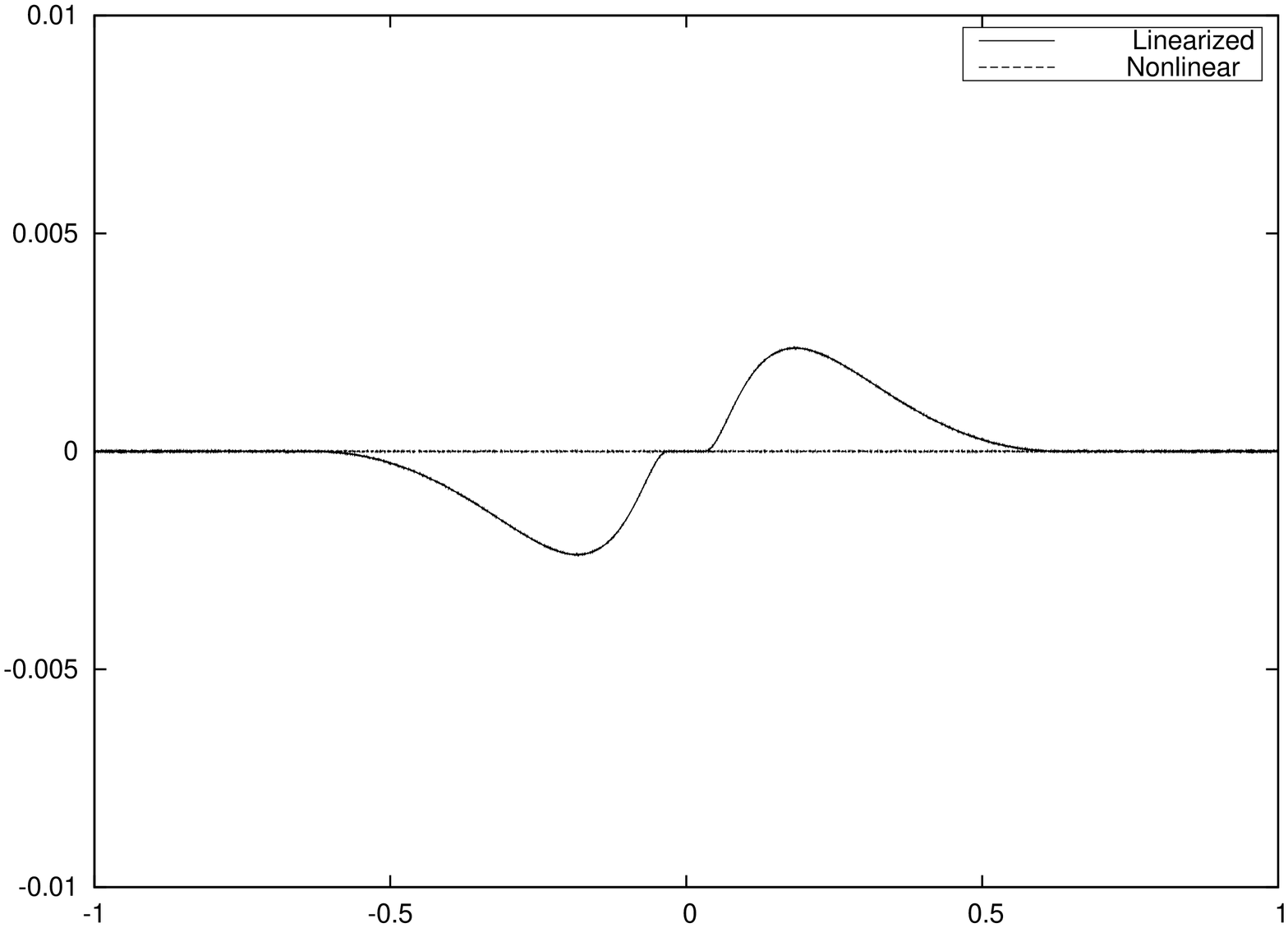}} \\
 			(g)\,\,\,\,$h$ and $u$ at $t=1.2$  %
 		\end{center}
 	\end{figure}
 	\clearpage
 	\begin{figure}[h]
 		\begin{center}
 			\subfloat[]{\includegraphics[totalheight=2.7in,width=1.825in,angle=-90]{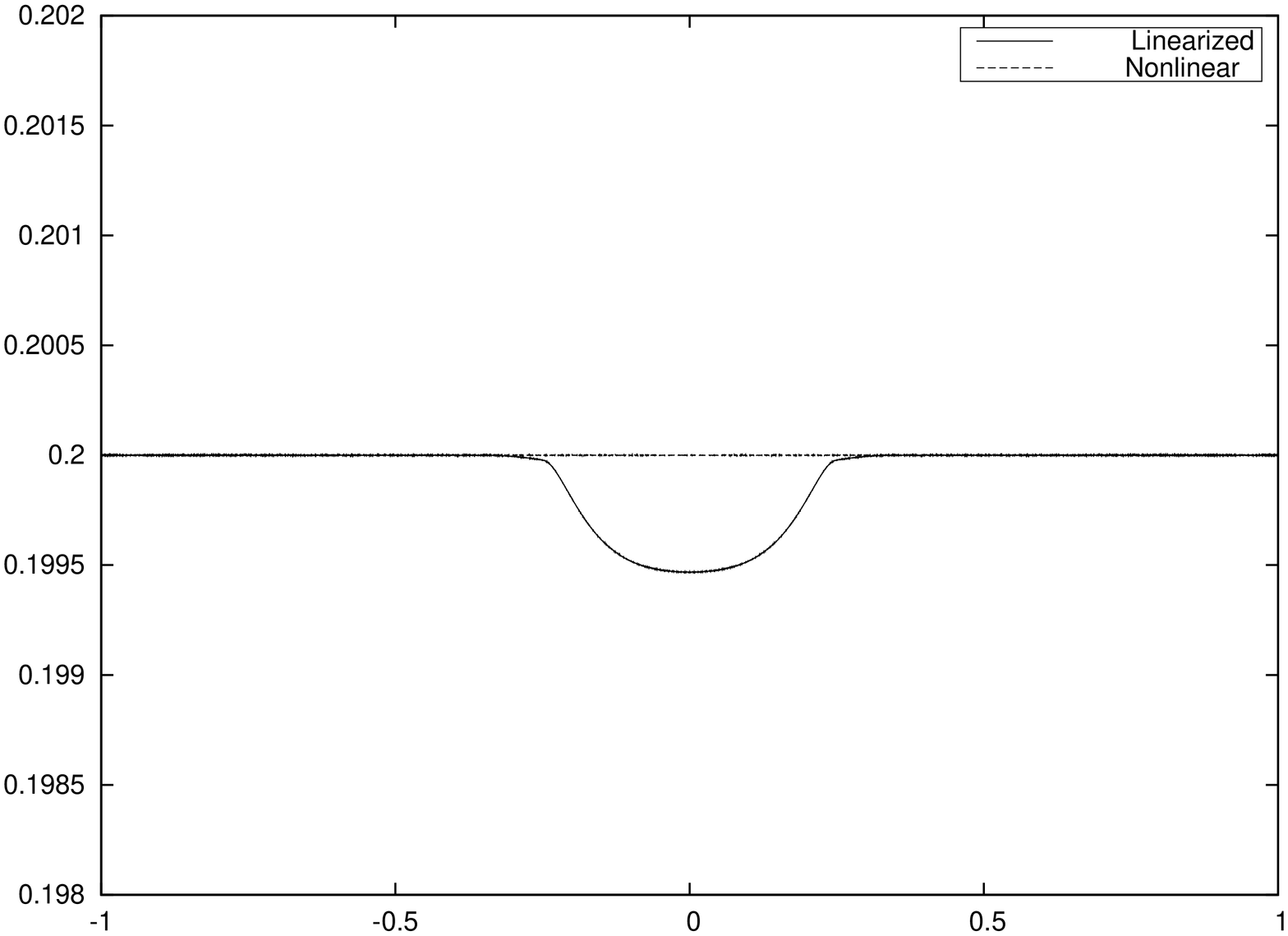}}\qquad
 			\subfloat[]{\includegraphics[totalheight=2.7in,width=1.825in,angle=-90]{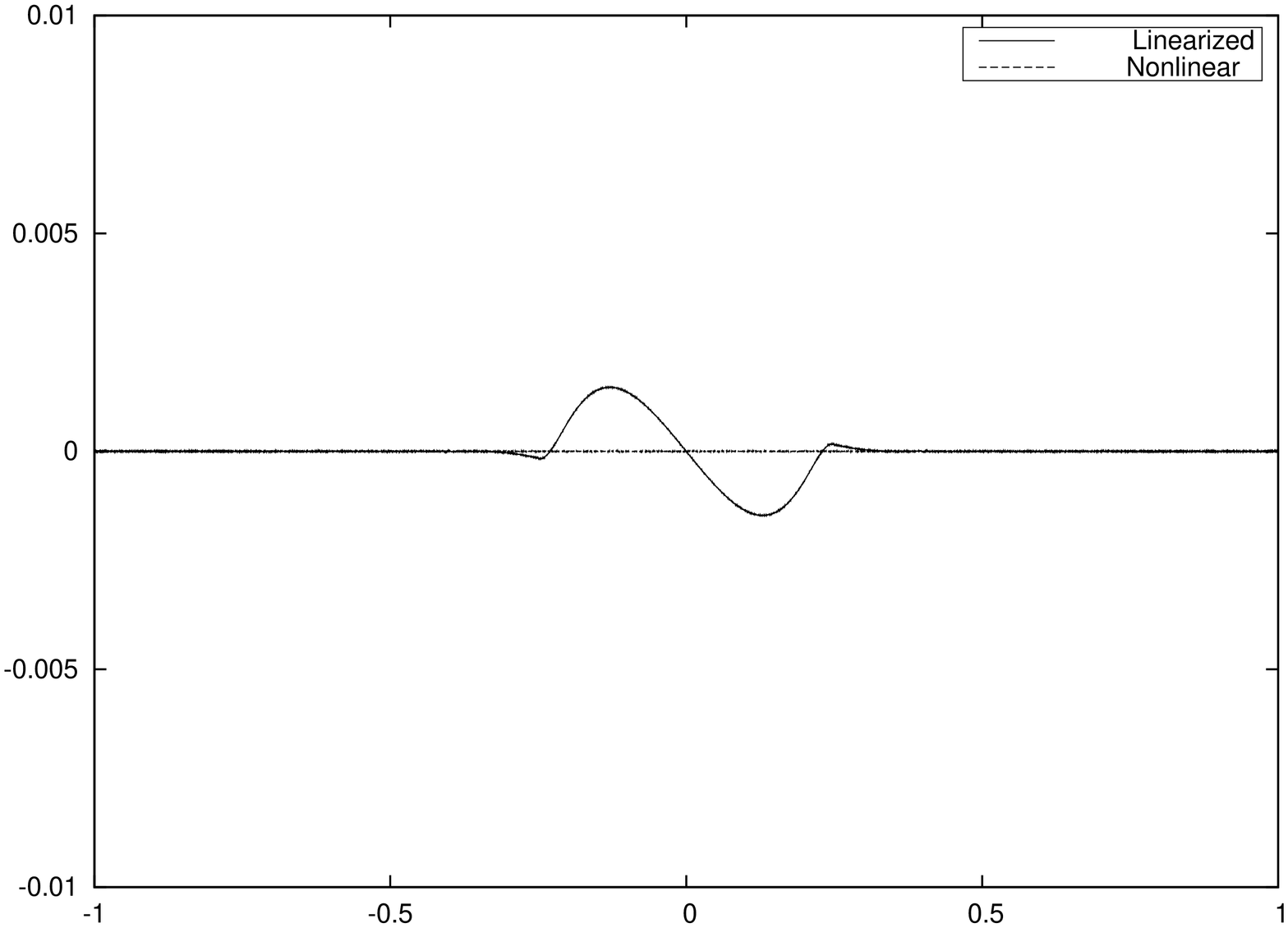}} \\
 			(h)\,\,\,\,$h$ and $u$ at $t=1.4$ \vspace{17pt} \\
 			\subfloat[]{\includegraphics[totalheight=2.7in,width=1.825in,angle=-90]{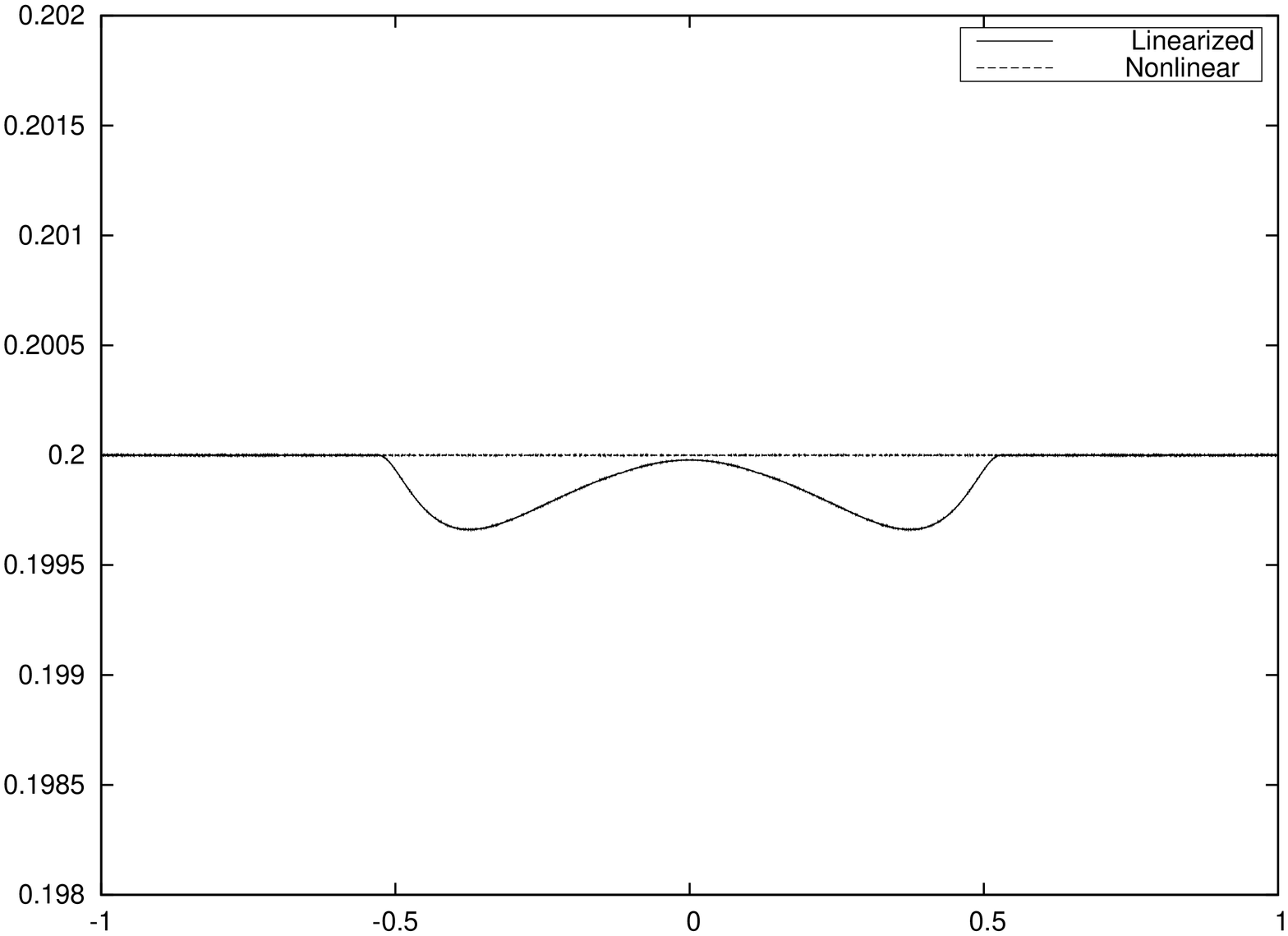}}\qquad
 			\subfloat[]{\includegraphics[totalheight=2.7in,width=1.825in,angle=-90]{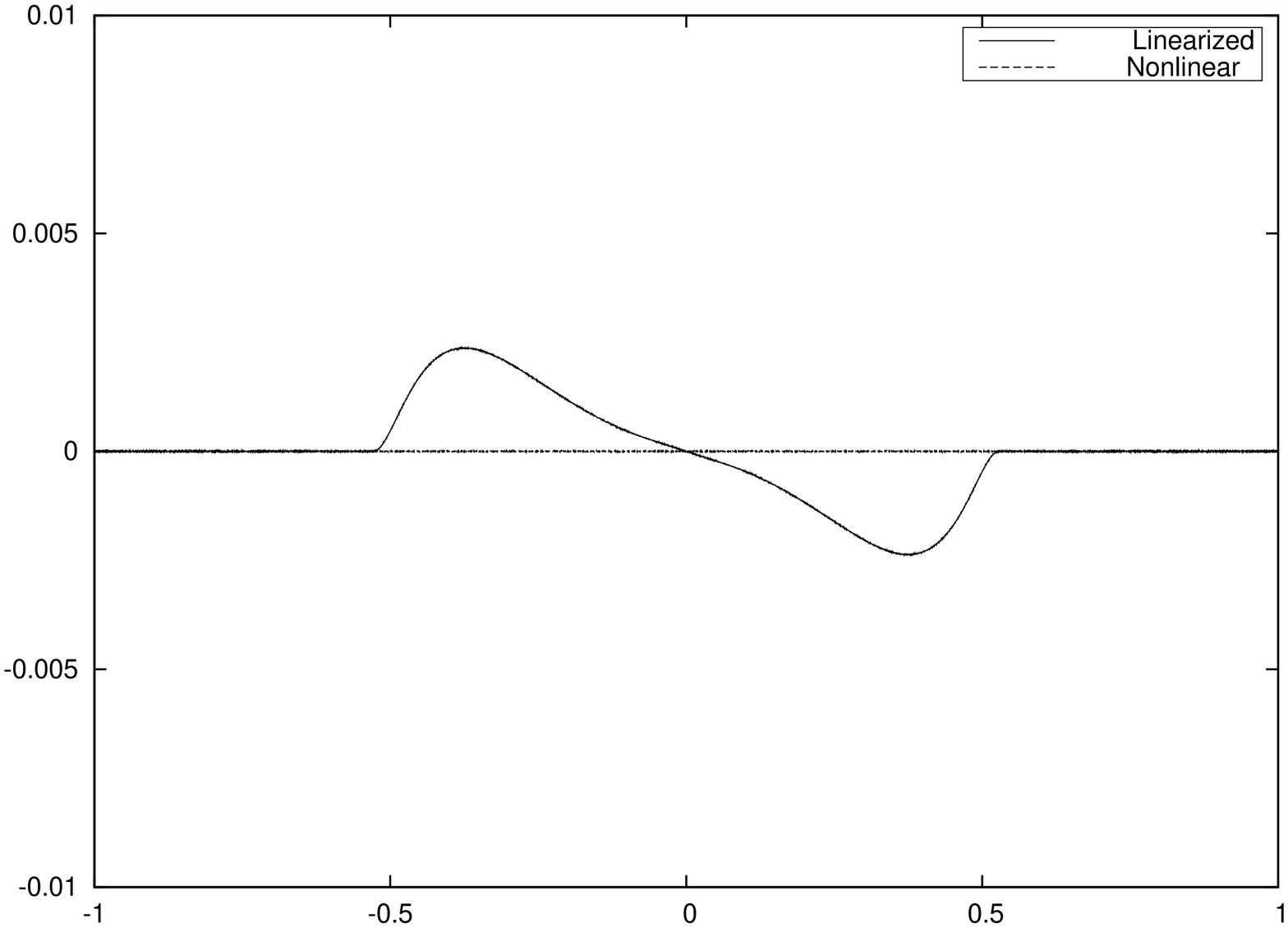}}\\
 			(i)\,\,\,\,$h$ and $u$ at $t=1.6$ \vspace{17pt} \\		
 			\subfloat[]{\includegraphics[totalheight=2.7in,width=1.825in,angle=-90]{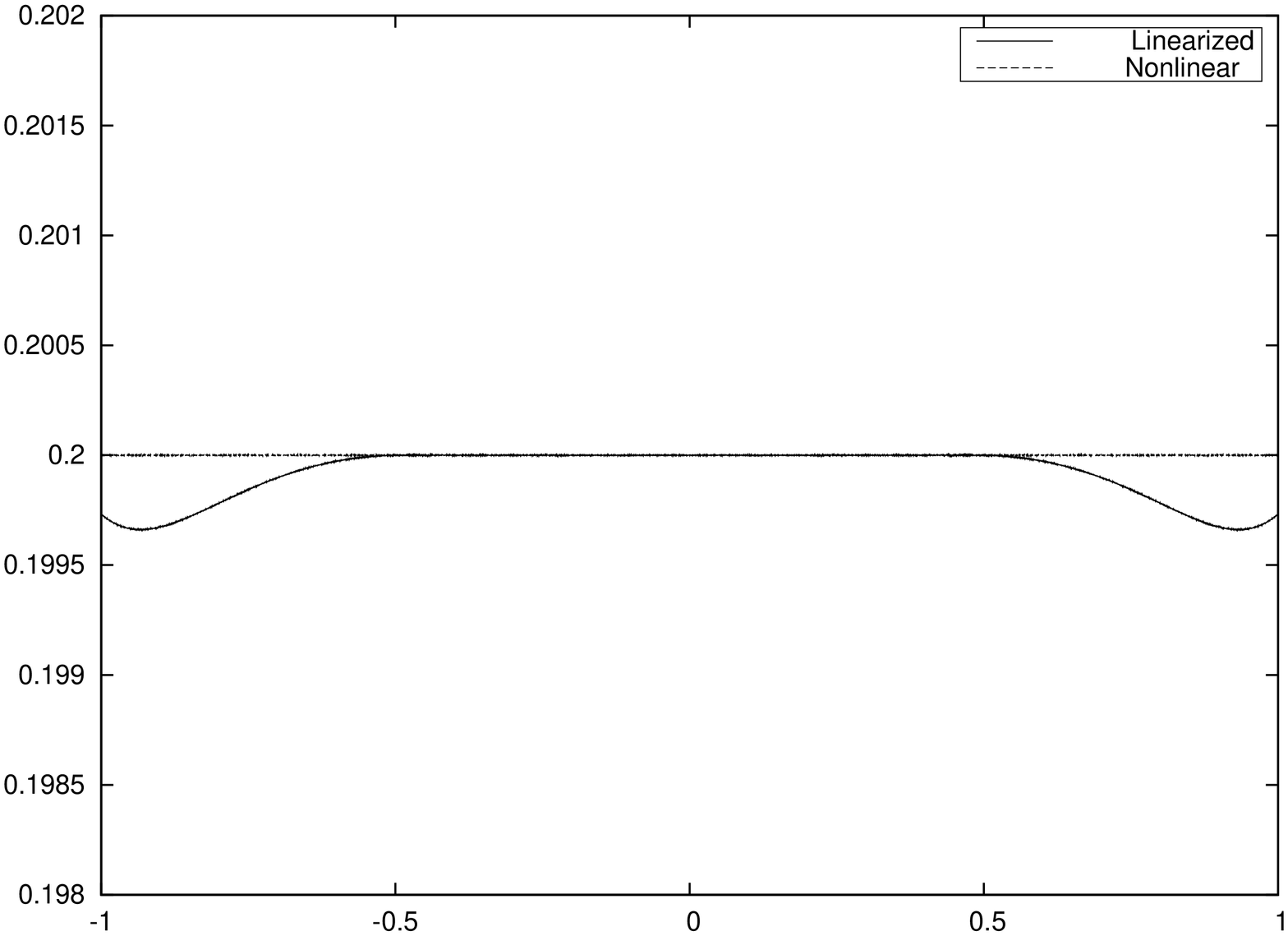}}\qquad
 			\subfloat[]{\includegraphics[totalheight=2.7in,width=1.825in,angle=-90]{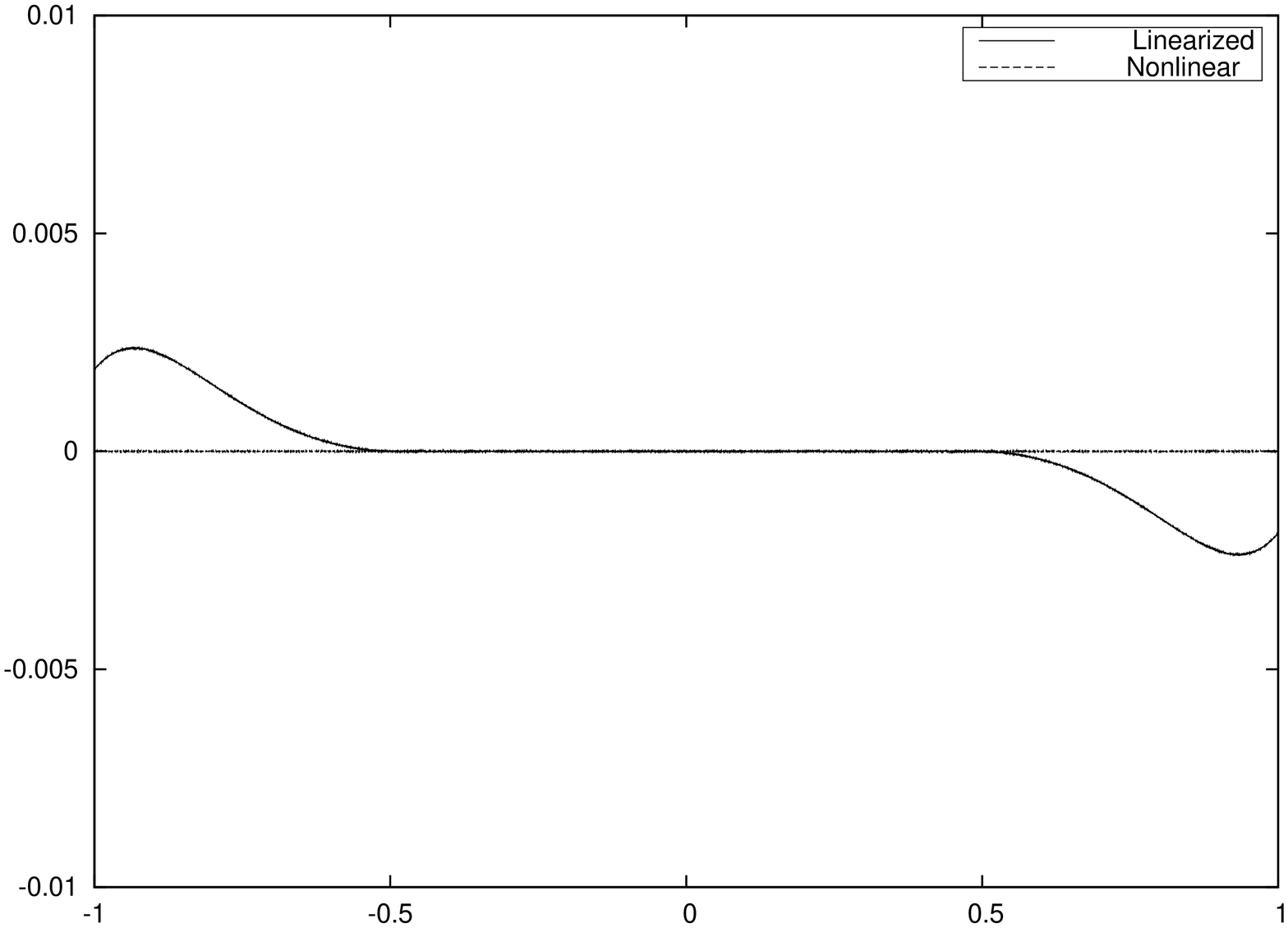}} \\
 			(j)\,\,\,\,$h$ and $u$ at $t=2.0$  
 		\end{center}
 		 	\end{figure}
 	\clearpage
 \begin{figure}[h]
 		\begin{center}
 			\subfloat[]{\includegraphics[totalheight=2.7in,width=1.825in,angle=-90]{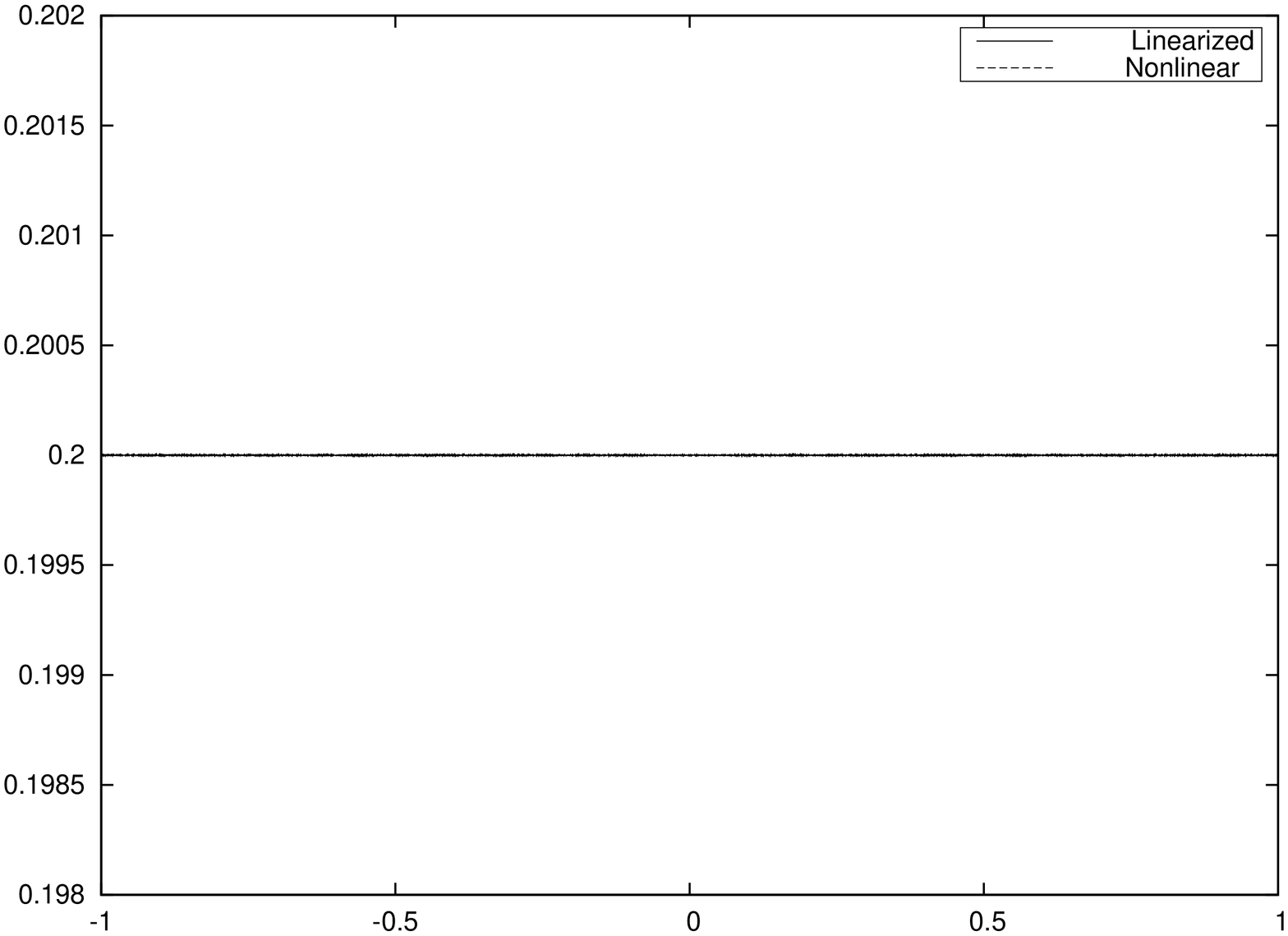}}\qquad
 			\subfloat[]{\includegraphics[totalheight=2.7in,width=1.825in,angle=-90]{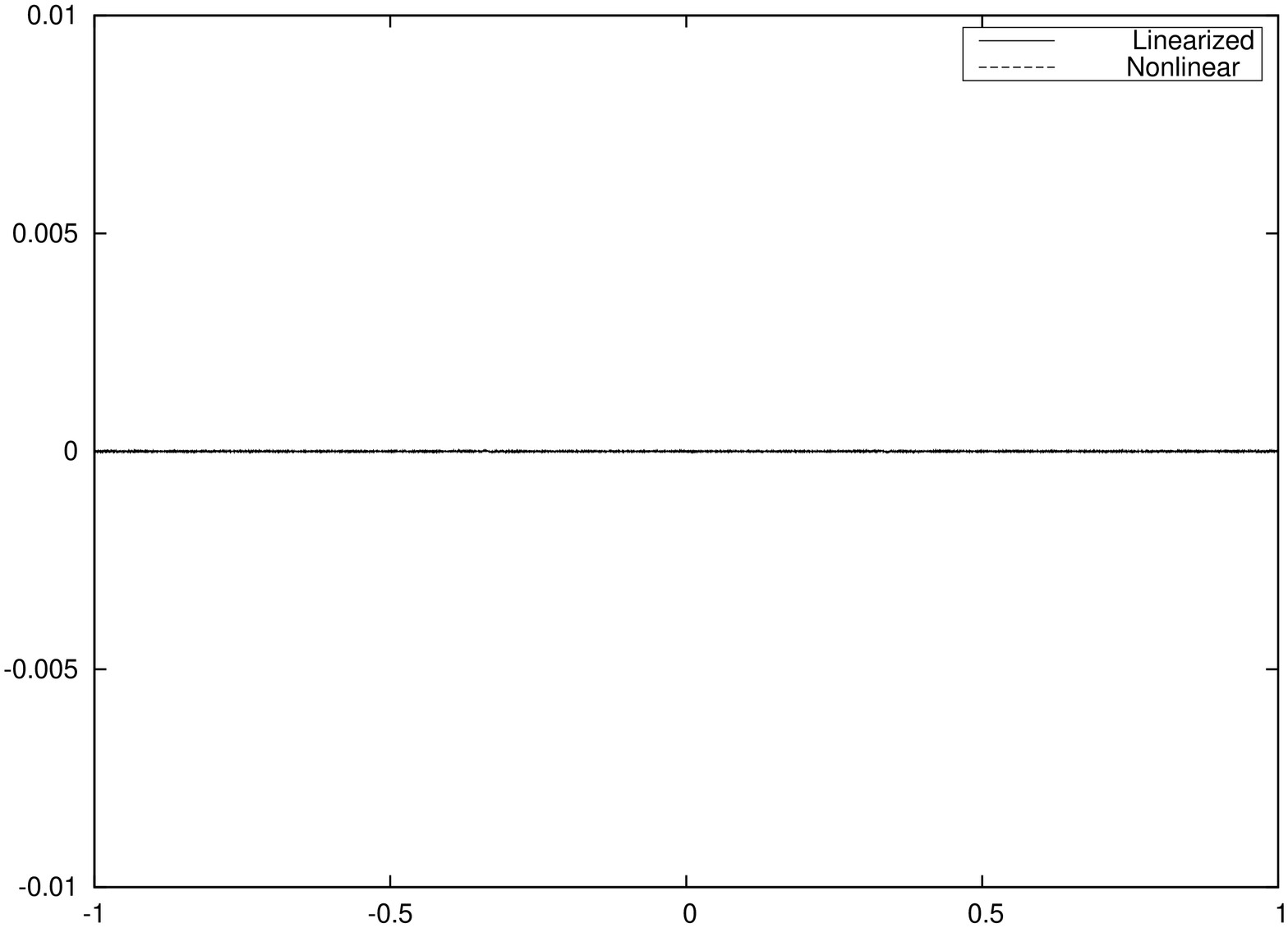}} \\
 			(k)\,\,\,\,$h$ and $u$ at $t=2.4$ 	
 		\end{center}
\small	
 		\caption{Evolution of $h(x,t)$ (left) and $u(x,t)$ (right); numerical solution of (\ref{eq49})-(\ref{eq410})
 			on $[-1,1]$ with nonlinear (dotted line) and linearized (solid line) outflow b.c.'s with $v(x)=0$,
 			$f(x)=0.2 + 0.05\sin (\pi (x+0.3)/0.6)$ if $\abs{x}\leq 0.3$, $f(x)=0.2$ if $\abs{x} \geq 0.3$.}
 		\label{fig48}
 	\end{figure}
 	\noindent  \normalsize
 	that travel and interact approximately linearly until they exit the domain at about $t=2.4\,\mathrm{sec}$. 
 	The spurious reflections may also be observed in the graphs of the temporal history of the waveheight 
 	at a gauge at  
 	\begin{figure}[h]
 		\begin{center}
 			\subfloat[]{\includegraphics[totalheight=2.7in,width=1.825in,angle=-90]{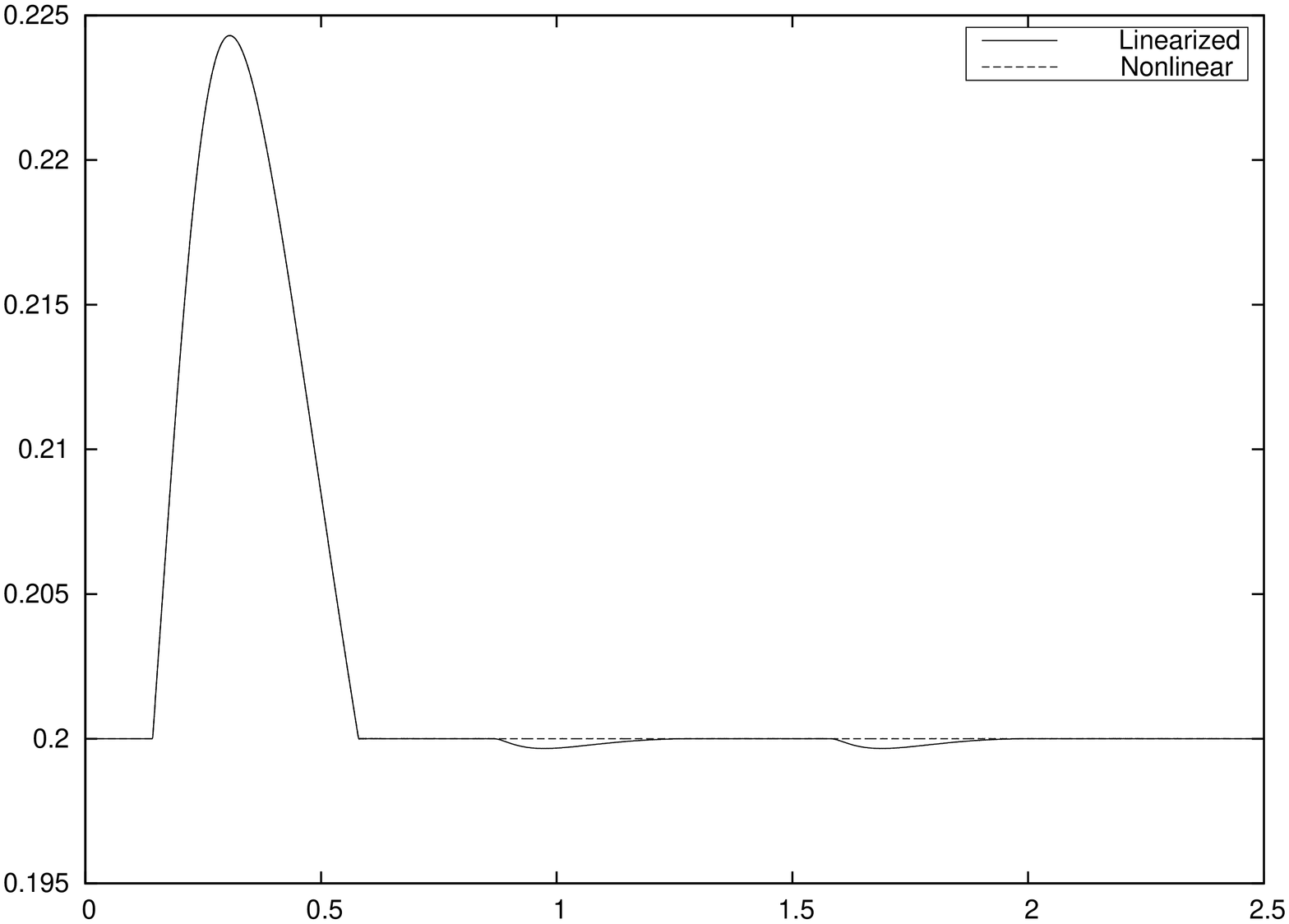}}\qquad
 			\subfloat[]{\includegraphics[totalheight=2.7in,width=1.825in,angle=-90]{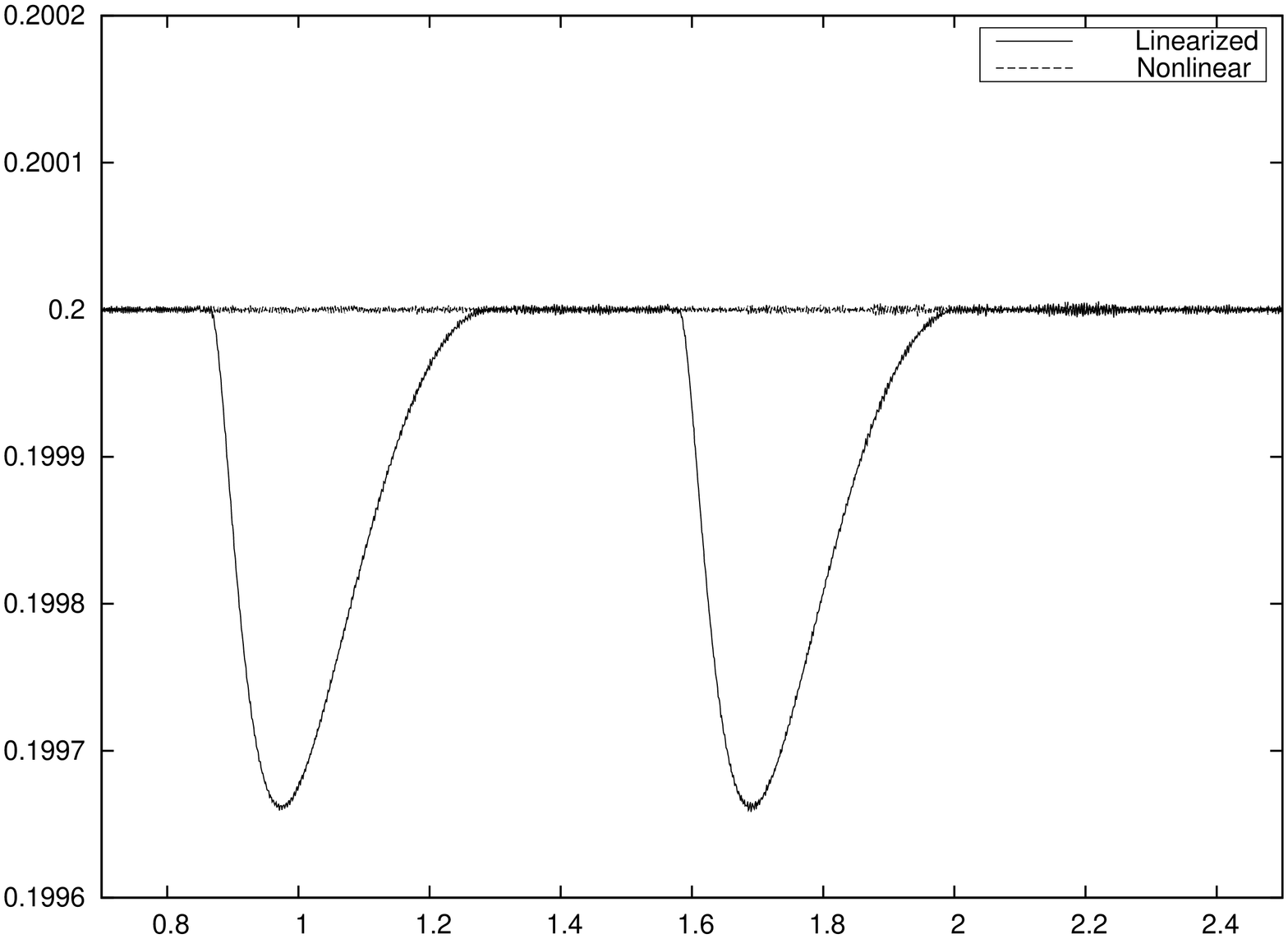}} \\
 			(a)\,\,\,\,$h$ at $x=-0.5$ vs. time  \hspace{45pt} (b)\,\,\,\, magnification of (a)  
 		\end{center}
 		\caption{Graphs of $h(-0.5,t)$ for $0\leq t\leq 2.5$;  evolution of Figure \ref{fig48}. Nonlinear
 			b.c.'s (dotted line) vs. linearized b.c.'s (solid line).}
 		\label{fig49}
 	\end{figure}
 	$x=-0.5\,\mathrm{m}$ computed with the nonlinear (dotted line) and the linearized (solid line) 
 	boundary conditions.
 	The passage of the main pulse at the gauge location is completed at about
 	$t = 0.6\,\mathrm{sec}$. Subsequently the rightwards-travelling reflected spurious small-amplitude
 	pulse due to the linearized b.c.'s is observed to be passing the gauge location at around 
 	$t=1\,\mathrm{sec}$, while the leftwards-travelling companion spurious pulse passes at around 
 	$t=1.7\,\mathrm{sec}$.\\
 	This experiment is qualitatively analogous to that described in Section 4.2 of \cite{nmf}. We simulated
 	here a smaller-amplitude wave to avoid the emergence of discontinuities during the temporal interval of the 
 	evolution of Fig. \ref{fig48}. In \cite{nmf} the analogous maximum waveheight was equal to about $0.3$ as the
 	finite-difference scheme used in that work had artificial viscosity (private communication by  
 	Dr. A. McC. Hogg), that smoothed out the emerging discontinuity. (A more correct simulation of the
 	numerical experiment in \cite{nmf} requires that an initial half-sinusoidal \emph{impulse} is imparted on
 	the water column. This may be modelled by an appropriate forcing term of the form $-F(x,t)/h$ on the
 	right-hand side of the momentum equation, i.e. the second pde in (\ref{eq49}),
 	where $F(x,t)$ may be taken as a multiple of $\sin (\pi t)$ if $(x,t)\in [-0.1,0.1]\times [0,1]$
 	and $F=0$ otherwise, while the initial conditions should now be $f(x) = 0.2\,\mathrm{m}$ and
 	$v(x)=0$. For small enough amplitude of the forcing term $F$ we observed that a qualitatively similar 
 	evolution took place.) 
\subsection*{Acknowledgement} The authors would like to thank Prof. J. Nycander, Dr. A.McC. Hogg, and Dr. L.M. 
Frankcombe for helpful comments on their numerical experiments in \cite{nmf}. 
\bibliographystyle{amsalpha} 

\end{document}